\documentclass[11pt]{article} 

\usepackage{latexsym}
\usepackage{amsmath}
\usepackage{amsthm}
\usepackage{amssymb}
\usepackage{vmargin}
\usepackage{cancel}
\usepackage{amscd}
\usepackage{stmaryrd}
\usepackage{euscript}
\usepackage{mathrsfs}
\usepackage{amscd}
\usepackage[all]{xy}
\usepackage{xr}
\usepackage{imakeidx}
\usepackage{hyperref} 
\usepackage{cleveref}
\usepackage{tikz-cd}
\usepackage{yhmath}
\usepackage{float}

    \definecolor{emphcolor}{RGB}{65,105,225}	
    \let\emph\relax 
    \DeclareTextFontCommand{\emph}{\color{emphcolor}\em}    
\definecolor{vinered}{rgb}{0.77, 0, 0.14}
\hypersetup{hidelinks,
            colorlinks=true,
            breaklinks=true,
            urlcolor= vinered,
            citecolor=vinered,
            linkcolor=vinered,
            bookmarksopen=false,
            pdftitle={Title},
            pdfauthor={Author}}

\crefformat{footnote}{#2\footnotemark[#1]#3}
\newcommand{\catdga}[1][k]{\operatorname{dg_+Alg}_{#1}} 
\newcommand{\catdgm}[1][A_\bullet]{\operatorname{dg_+Mod}_{#1}} 
\newcommand{\catdgmU}[1][A_\bullet]{\operatorname{dgMod}_{#1}} 

\DeclareMathAlphabet{\mathpzc}{OT1}{pzc}{m}{it}

\newtheorem{theorem}{Theorem}[section]
\newtheorem{proposition}[theorem]{Proposition}
\newtheorem{corollary}[theorem]{Corollary}

\newtheorem{lemma}[theorem]{Lemma}
\newtheorem*{theorem*}{Theorem}
\newtheorem*{proposition*}{Proposition}
\newtheorem*{corollary*}{Corollary}
\newtheorem*{lemma*}{Lemma}
\newtheorem*{conjecture*}{Conjecture}

\theoremstyle{definition}
\newtheorem{definition}[theorem]{Definition}
\newtheorem{notation}[theorem]{Notation}

\newtheorem*{definition*}{Definition}

\theoremstyle{remark}
\newtheorem{example}[theorem]{Example}
\newtheorem{examples}[theorem]{Examples}
\newtheorem{remark}[theorem]{Remark}
\newtheorem{construction}[theorem]{Construction}
\newtheorem{warning}[theorem]{Warning}
\newtheorem{remarks}[theorem]{Remarks}
\newtheorem{exercise}[theorem]{Exercise}
\newtheorem{exercises}[theorem]{Exercises}
\newtheorem{property}[theorem]{Property}
\newtheorem{properties}[theorem]{Properties}

\newtheorem{digression}[theorem]{Digression}

\newtheorem*{example*}{Example}
\newtheorem*{examples*}{Examples}
\newtheorem*{remark*}{Remark}
\newtheorem*{warning*}{Warning}
\newtheorem*{remarks*}{Remarks}
\newtheorem*{exercise*}{Exercise}
\newtheorem*{property*}{Property}
\newtheorem*{properties*}{Properties}

\newcommand\da{\!\downarrow\!}

\newcommand\la{\leftarrow}

\newcommand\id{\mathrm{id}}

\newcommand\ten{\otimes}

\newcommand\eps{\epsilon}

\newcommand\CC{\mathrm{C}}
\newcommand\DD{\mathrm{D}}

\renewcommand\H{\mathrm{H}}
\newcommand\z{\mathrm{Z}}
\renewcommand\b{\mathrm{B}}

\newcommand\N{\mathbb{N}}
\newcommand\Z{\mathbb{Z}}
\newcommand\Q{\mathbb{Q}}

\newcommand\R{\mathbb{R}}
\newcommand\Cx{\mathbb{C}}

\newcommand\bA{\mathbb{A}}

\newcommand\bC{\mathbb{C}}

\newcommand\bF{\mathbb{F}}
\newcommand\bG{\mathbb{G}}

\newcommand\bL{\mathbb{L}}

\newcommand\bN{\mathbb{N}}

\newcommand\bP{\mathbb{P}}
\newcommand\bQ{\mathbb{Q}}
\newcommand\bR{\mathbb{R}}

\newcommand\bZ{\mathbb{Z}}

\newcommand\C{\mathcal{C}}

\newcommand\cC{\mathcal{C}}
\newcommand\cD{\mathcal{D}}

\newcommand\cF{\mathcal{F}}

\newcommand\cH{\mathcal{H}}

\newcommand\cJ{\mathcal{J}}

\newcommand\cL{\mathcal{L}}
\newcommand\cM{\mathcal{M}}
\newcommand\cN{\mathcal{N}}

\newcommand\cP{\mathcal{P}}
\newcommand\cQ{\mathcal{Q}}

\newcommand\cS{\mathcal{S}}

\newcommand\cU{\mathcal{U}}
\newcommand\cV{\mathcal{V}}
\newcommand\cW{\mathcal{W}}

\renewcommand\O{\mathscr{O}}

\newcommand\sE{\mathscr{E}}
\newcommand\sF{\mathscr{F}}

\newcommand\sH{\mathscr{H}}
\newcommand\sI{\mathscr{I}}

\newcommand\sL{\mathscr{L}}

\newcommand\sN{\mathscr{N}}
\newcommand\sO{\mathscr{O}}

\newcommand\sT{\mathscr{T}}

\newcommand\fX{\mathfrak{X}}
\newcommand\fY{\mathfrak{Y}}

\renewcommand\L{\Lambda}

\newcommand\m{\mathfrak{m}}

\newcommand\g{\mathfrak{g}}

\renewcommand\hom{\mathscr{H}\!\mathit{om}}

\newcommand\Ho{\mathrm{Ho}}

\newcommand\Alg{\mathrm{Alg}}

\newcommand\Mod{\mathrm{Mod}}

\newcommand\Hom{\mathrm{Hom}}
\newcommand\Map{\mathrm{Map}}

\newcommand\HHom{\underline{\mathrm{Hom}}}

\newcommand\Ext{\mathrm{Ext}}
\newcommand\EExt{\mathbb{E}\mathrm{xt}}

\newcommand\cone{\mathrm{cone}}

\newcommand\coker{\mathrm{coker\,}}

\newcommand\im{\mathrm{Im\,}}

\newcommand\Ob{\mathrm{Ob}\,}
\newcommand\Ch{\mathrm{Ch}}
\newcommand\ch{\mathrm{ch}}

\newcommand\Ab{\mathrm{Ab}}
\newcommand\Top{\mathrm{Top}}

\newcommand\Spec{\mathrm{Spec}\,}

\newcommand\Set{\mathrm{Set}}

\newcommand\Cat{\mathrm{Cat}}

\newcommand\Aff{\mathrm{Aff}}

\newcommand\Sch{\mathrm{Sch}}

\newcommand\Sing{\mathrm{Sing}}

\newcommand\Lim{\varprojlim}
\newcommand\LLim{\varinjlim}
\DeclareMathOperator*{\holim}{holim}
\newcommand\ho{\mathrm{ho}\!}

\newcommand\onto{\twoheadrightarrow}
\newcommand\abuts{\implies}
\newcommand\xra{\xrightarrow}
\newcommand\xla{\xleftarrow}
\newcommand\pr{\mathrm{pr}}

\newcommand\even{\mathrm{even}}

\newcommand\bt{\bullet}
\newcommand\by{\times}

\newcommand\Symm{\mathrm{Symm}}

\newcommand\GL{\mathrm{GL}}

\newcommand\Mat{\mathrm{Mat}}
\newcommand\et{\acute{\mathrm{e}}\mathrm{t}}

\newcommand\an{\mathrm{an}}
\newcommand\hol{\mathrm{hol}}
\newcommand\cart{\mathrm{cart}}

\newcommand\Tot{\mathrm{Tot}\,}
\newcommand\diag{\mathrm{diag}\,}

\newcommand\ev{\mathrm{ev}}
\newcommand\ind{\mathrm{ind}}

\renewcommand\iff{\Longleftrightarrow}

\newcommand\pd{\partial}

\newcommand\Tor{\mathrm{Tor}}

\newcommand\QCoh{\cQ\cC\mathit{oh}}

\newcommand\Hilb{\cH\mathit{ilb}}

\newcommand\cosk{\mathrm{cosk}}

\newcommand\dR{\mathrm{dR}}

\newcommand\op{\mathrm{op}}

\newcommand\co{\colon\thinspace}

\newcommand\oR{\mathbf{R}}
\newcommand\oP{\mathbf{P}}
\newcommand\oL{\mathbf{L}}

\newcommand\oSpec{\mathbf{Spec}\,}

\newcommand\on{\mathbf{n}}

\newcommand\oO{\mathbf{0}}

\newcommand\uleft\underleftarrow
\newcommand\uline\underline
\newcommand\uright\underrightarrow

\usetikzlibrary{arrows}

\makeindex

\begin{document}
\title{An introduction to derived (algebraic) geometry}

\author{Jon Eugster and Jonathan Pridham}
\maketitle

\begin{abstract}
Mostly aimed at an audience with backgrounds in geometry and homological algebra, these notes offer an introduction to derived geometry based on a lecture course given by the second 
    author. The focus is on derived algebraic geometry, mainly in characteristic $0$, but we also see the tweaks which extend most of the content to analytic and differential settings. 
\end{abstract}

\section*{Preface}
 These notes are based on those taken by both authors from a lecture course given by the second 
 author
 in Edinburgh in  2021. Some material from courses given in Cambridge in 2013 and 2011 has been added, together 
 with  details, references and  additional related content (notably some more down-to-earth characterisations of derived stacks). 

 The main background topics assumed are homological algebra, sheaves, basic category theory and algebraic topology, together with some familiarity with typical notation and terminology in algebraic geometry. A lot of the motivation will be clearer for those familiar with moduli spaces, but they are not essential background.
 
  The perspective of the notes is to try to present the subject as a natural, concrete development of more classical geometry, instead of merely as an opportunity to showcase  $\infty$-topos theory (a topic we only encounter indirectly in these notes). 
  The main moral of the later sections is that if you are willing to think of geometric objects in terms of \v Cech nerves of atlases rather than as ringed topoi, then the business of developing higher and derived generalisations becomes much simpler. 
 
 These notes are only intended as an introduction to the subject, and are far from being a comprehensive survey. We have tried to include more detailed  references throughout, with the original references where we know them. Readers may be surprised at how old many of the references are, but the basics have not changed in more than a decade and the fundamentals were established half a century ago,  though
  as 
 terminology becomes more specialised, researchers can
   tend to overestimate the originality of their ideas\footnote{potentially compounded by Maslow's hammer and  Disraeli's maxim on reading books}. We have probably overlooked precursors for many phenomena in the supersymmetry literature, for which we apologise in advance.

Focusing on the characteristic $0$ theory, Section \ref{dgasn} introduces  dg-algebras as the affine building blocks for derived geometry. It then gives a simple characterisation of derived schemes, derived DM stacks  and their quasi-coherent complexes in those terms, explaining their relation with the older notions of dg-schemes and dg-manifolds.

The natural notion of equivalence for dg-algebras is not isomorphism, but quasi-isomorphism. Since quasi-isomorphism classes have poor gluing properties, adequate treatment of morphisms necessitates some flavour of infinity categories.  
 \S \ref{inftycatsn} gives a minimalist overview of the necessary homotopy theory.

In \S \ref{consequencesn}, we start to reap the consequences of those homotopical techniques, introducing derived intersections, the cotangent complex and shifted symplectic structures. The section also features obstruction theory, which was the original motivating application for derived deformation theory, and hence for derived algebraic geometry.

Concrete constructions in homotopy theory tend to feature simplicial objects, and we cover the basic theory in \S \ref{simpsn}. Simplicial sets and simplicial algebras are covered, together with the Dold--Kan correspondence which allows us to think of these as generalising chain complexes beyond the realm of abelian categories. As preparation for the later sections, we also include an account of Duskin--Glenn $n$-hypergroupoids.

\S \ref{stacksn} develops the theory of higher stacks from that perspective, as simplicial schemes satisfying analogues of the hypergroupoid property. It includes a description of morphisms and of the theory of quasi-coherent complexes, as well as  a comparison with the slightly older topos-theoretic definitions.

In \S \ref{dstacksn}, those definitions and constructions are extended to derived stacks, which incorporate enhancements in both derived and stacky directions. Obstruction theory leads to the Artin--Lurie representability theorem, here covered in a simplified form due to the second author. Several examples of derived moduli functors  are then discussed, including concrete representability criteria for various classes of moduli problems parametrising schemes, derived Artin stacks and quasi-coherent complexes.
 
 We would like to thank the audience members, and particularly Sebastian Schlegel-Mejia, for very helpful comments, without which many explanations would be missing from the text. 
 
\subsection*{Notational conventions}

    \begin{itemize}
        
        
        \item We adhere strictly to the standard convention that the indices in chain complexes and simplicial objects, and related operations and constructions, are denoted with subscripts, while those in  cochain complexes and cosimplicial objects are denoted with superscripts; to do otherwise would invite chaos.
      
      \item We intermittently write chain complexes $V$ as $V_{\bt}$ to emphasise the structure, and similarly for cochain, simplicial and cosimplicial structures. The presence or absence of bullets in a given expression should not be regarded as significant.
        
        \item We denote shifts of chain and cochain complexes by $[n]$, and always follow the convention originally developed for cochains, so we have $M[n]^i:=M^{n+i}$ for cochain complexes, but $M[n]_i:=M_{n-i}$ for chain complexes. 
    \end{itemize}

\clearpage

\tableofcontents

\clearpage
%

\section{Introduction and dg-algebras}\label{dgasn}

    The idea behind derived geometries, and in particular  derived algebraic geometry  (DAG for short), is to endow rings of functions with extra structure, making families of geometric objects behave better.  For example, singular points start behaving more like smooth ones as observed in \cite{Kon2,Kon}, a philosophy known as \emph{hidden smoothness}. \index{hidden smoothness}
    
    The most fundamental formulation of the theory would probably be in terms of simplicial rings,
    but in characteristic $0$ these give the same theory as commutative 
    differential graded algebras (dg-algebras),
    which we will focus on most in these notes, as they are simpler to work with.
    
    \begin{remark}
        Spectral algebraic geometry\footnote{often confusingly referred to as derived algebraic geometry following \cite{lurieDAG5}, and originally dubbed  Brave New Algebraic Geometry in \cite{TVbn,hag2}, ``brave new algebra'' then  being well-established Huxleian (or Shakespearean) terminology,  dating at least to Waldhausen's plenary talk ``Brave new rings''  at the conference \cite{waldhausenBN})} (SAG)
        \index{spectral algebraic geometry}\index{SAG|see{spectral algebraic geometry}}
        is another powerful closely related framework and is based on commutative ring spectra; it is studied amongst other homotopical topics in 
        \cite{lurieSAG}.
        Over $\Q$, this gives the same theory as DAG, but different geometric behaviour appears
        in finite and mixed characteristic. While DAG is mostly used to apply methods of algebraic topology to
        algebraic geometry, SAG is mainly used the other way around, an example being elliptic cohomology as in \cite{lurieellipticsurvey}.
        
        One motivation for SAG is that cohomology theories come from symmetric spectra,
        and you try to cook up more exotic cohomology theories by replacing rings
        in the theory of schemes/stacks with $E_{\infty}$-ring  spectra.
        There's a functor $H$ embedding discrete rings in $E_{\infty}$-ring  spectra
        \cite[p.~185]{hag2}, but it doesn't preserve smoothness: even 
        the morphism $H(\bF_p) \to H(\bF_p[t])$ is not formally smooth.
        
        This is just a side note and we won't use spectra in these lecture notes, although most of  the
        results of \S \ref{dstacksn} also hold in  SAG. For example, \cite{stacks2},
        from which the results of \S\S\ref{stacksn}--\ref{dstacksn} are mostly taken,
        was explicitly couched in sufficient generality to apply to ring spectra, too.
    \end{remark}

    \begin{notation}
        Henceforth (until we start using simplicial rings), we fix a commutative ring $k$ containing $\bQ$, i.e. we work
        in equal characteristics $(0,0)$.\footnote{Cdga don't behave nicely in other characteristics
        because the symmetric powers of \cref{exp:free-graded} don't preserve quasi-isomorphisms of chain complexes.} 
    \end{notation}

\subsection{dg-Algebras}
    In this section we define dg-algebras and affine dg-schemes, as well as analogues in differential and analytic geometry.
    \begin{definition}
        A \emph{differential graded $k$-algebra}\index{differential graded algebra}\index{dga|see{differential graded algebra}}
        (dga or dg-algebra for short)
        $A$ consists of a chain complex with a unital associative multiplication.
        Concretely, that is a family of
        $k$-modules $\{A_i\}_{i\in\bZ}$, an associative $k$-linear
        multiplication $(-\cdot-): A_i \times A_j \to A_{i+j}$ (for all $i,j$), a unit $1 \in A_0$ and a differential
        $\delta: A_i \to A_{i-1}$ (for all $i$) which is $k$-linear, satisfies $\delta^2 = 0$ and is a
        derivation with respect to the multiplication, which means
        $\delta(a\cdot b) = \delta(a)\cdot b + (-1)^{\deg(a)} a \cdot \delta(b)$.
         
        An object $A$ with all the structures above except the differential $\delta$ is simply called  a \emph{graded algebra}\index{graded algebra}.
        
        A graded $k$-algebra $A$ is \emph{graded-commutative}\index{graded-commutative}
        if $a\cdot b = (-1)^{\deg(a)\cdot\deg(b)} b\cdot a$.
        We write \emph{cdga}\index{cdga|see{graded-commutative}} for differential graded-commutative algebras.
    \end{definition}
    \begin{definition}
        A dg-algebra $A$ is \emph{discrete}\index{discrete dg-algebra} if $A_n = 0$ for all $n \ne 0$.
    \end{definition}
    \begin{notation}
        We usually denote a graded algebra by $A_\ast := \{A_i\}_i$ while we use the notation $A_\bullet := (\{A_i\}_i, \delta)$ to denote a differential graded algebra, where usually the $\delta$ is implicit/suppressed. However, the presence or absence of bullets in a given expression should not be regarded as significant.
        
        Moreover, we implicitly identify rings with  discrete (differential) graded algebras; given a ring $A$,  we simply denote the associated discrete dg-algebra   by $A$, which corresponds to its degree zero term (all other terms being $0$). 
    \end{notation}
    \begin{remark}
        Usually we restrict ourselves to the case where these cdga are \emph{concentrated
        in non-negative chain degree}\index{concentrated},  i.e. $A_i = 0$ for all $i<0$.
        
        Algebraic geometers more often work with cochains  instead of chains;
        our main reasons for using chain notation here are to assist the comparison with simplicial objects and to help distinguish the indices from those arising from sheaf cohomology.
    \end{remark}
    \begin{notation}
        In concrete examples we will often denote cdga concentrated
        in non-negative chain degree
        by $(A_0 \leftarrow A_1 \leftarrow A_2 \leftarrow \ldots)$, assuming that the
        first written entry is degree zero. For example, if $f:A \onto B$ is a 
        surjective map of rings then $(A \hookleftarrow \ker(f) \leftarrow 0 \leftarrow \ldots)$
        would be a chain complex with $A$ in degree zero, $\ker(f)$ in degree $1$ and $0$ 
        everywhere else. This
        chain complex is quasi-isomorphic to the discrete dg algebra $B$.
        
    \end{notation}
    \begin{example}\label{exp:free-graded}
        Let $M_\ast$ be a graded $k$-module.
        The \emph{free graded-commutative $k$-algebra}\index{free graded-commutative algebra} generated by $M_\ast$ is
        $k[M_\ast] :=(\bigoplus \Symm^n M_{\even})\ten (\bigoplus \bigwedge^nM_{odd})$, with the degree of a product of elements being the sum of the degrees of those elements.
    \end{example}
    \begin{example}\label{exp:cdga-001}
        Take a free graded-commutative $k$-algebra $A_\ast$ on three generators $X,Y,Z$ (i.e. on the $k$-module $k.X \oplus k.Y \oplus k.Z$), where $\deg(X)=0, \deg(Y)=\deg(Z)=1$.
        Then we get
        \begin{itemize}
            \item $A_0 = k[X]$
            \item $A_1 = k[X]Y \oplus k[X]Z$
            \item $A_2 = k[X]YZ$
            \item $A_i = 0$ for $i <0$ and $i\ge 3$.
        \end{itemize}
        which we can see by computing that $ZY=-YZ$ and $Y^2=Z^2 = 0$.
        A differential of $A_\ast$ is then completely determined
        by its values $f:=\delta(Y), g:=\delta(Z) \in A_0 = k[X]$.
        So for example for $a,b,c\in k[X]$ we have $\delta(aY+bZ)=af + bg$ and
        $\delta(cYZ) = c(Zf - Yg)$.\footnote{Some readers might recognise this
        as a variant of a Koszul complex.}

        In fact, we get $\H_0(A_\bt) = k[X]/(f,g)$, and we thing of  $A_\bt$ as the ring of functions
        on the \emph{derived vanishing locus}\index{derived vanishing locus} of the map
        $(f,g): \bA^1 \to \bA^2, x \mapsto (f(x), g(x))$.
    \end{example}

\subsubsection{Differential and analytic analogues}

    \Cref{exp:cdga-001} is set in the world of  algebraic geometry. However, it is straightforward
    to adjust the example to differential or analytic geometry. All that's needed is
    to put extra structure on $A_0$.
    For differential geometry, $A_0$ ought to be a \emph{$\C^\infty$-ring}\index{C-ring@$\C^\infty$-ring} \cite{dubuc}, 
    which means that for any $f \in \C^\infty(\bR^n, \bR)$ there is an $n$-ary operation
    $A_0 \times\ldots\times A_0 \to A_0$, and these operations need to satisfy some
    natural consistency conditions.
   
    \begin{example}\label{Cinfty}
        Finitely generated
        $\C^\infty$-rings just take the form $\C^\infty(\bR^m, \bR)/I$ where $I$ is an ideal;
        these include $\C^{\infty}(X,\R)$ for manifolds $X$.
        Hadamard's lemma ensures that
        the operations descend to the quotient.
        
        A $\C^{\infty}$-ring homomorphism $\C^\infty(\bR^m, \bR)/I \to \C^\infty(\bR^n, \bR)/J $ is
        then just given by  elements $f_1, \ldots, f_m \in \C^\infty(\bR^n, \bR)/J$ satisfying
        $g(f_1, \ldots,f_m)=0$ for all $g \in I$; think of this as a smooth morphism from the
        vanishing locus of $J$ to the vanishing locus of $I$.
        
        Arbitrary $\C^{\infty}$ rings arise as quotient rings of nested unions
        $\C^{\infty}(\R^S,\R):= \bigcup_{\substack{T \subset S \\ \text{finite}}} \C^\infty(\bR^T, \bR)$ for infinite sets $S$.
    \end{example}
     
    This approach allows for singular spaces, and is known as \emph{synthetic}\index{synthetic differential geometry} differential geometry.
        
    For  analytic geometry, $A_0$ should be a \emph{ring with entire functional calculus} (\emph{EFC-ring}),
    \index{EFC-ring}\index{ring with entire functional calculus|see{EFC-ring}}
    meaning for any holomorphic function $f: \bC^n \to \bC$ there is again an operation
    $A_0 \times\ldots\times A_0 \to A_0$, with these satisfying some natural consistency conditions
    --- there are analogous definitions for non-Archimedean analytic geometry.  
    For more details and further references on this approach, see \cite{CarchediRoytenberg,nuitenThesis} in
    the differential setting, and \cite{DStein} in the analytic setting.
    The latter shows that this
    is equivalent to the approach via
    pregeometries in \cite{lurieDAG9},  classical theorems in analysis rendering
    most of the pregeometric data redundant.

\subsubsection{Morphisms and quasi-isomorphisms}

    \begin{definition}
        As with any chain complex, we can define the \emph{homology}\index{homology} $H_\ast(A_\bullet)$
        of a dg-algebra $A_\bullet$ by
        $H_i(A_\bullet) = \ker(\delta:A_i \to A_{i-1})/\im(\delta: A_{i+1}\to A_i)$, which is a
        graded-commutative algebra when $A$ is a cdga.
    \end{definition}
    \begin{definition}
        A \emph{morphism of dg-algebras}\index{dga!morphism}
        is a map $f: A_\bullet \to B_\bullet$ that respects the differentials
        (i.e. $f\delta_A = \delta_B f:A_i \to B_{i-1}$ for all $i \in \Z$), and the multiplication
        (i.e. $f(a\cdot_A b) = f(a) \cdot_B f(b) \in B_{i+j}$ for all $a\in A_i, b\in A_j$ for all $i,j$).
    \end{definition}    
    \begin{definition}
        We denote by $\catdga$ the
        \emph{category of graded-commutative differential graded $k$-algebras}\index{dg+Algk@$\catdga$|see {dga}}
        which are concentrated in non-negative degree. 
        The opposite category $(\catdga )^\op$ is the
        \emph{category of affine dg-schemes}\index{dg-scheme!affine}\index{dgaff@$\operatorname{DG^+\Aff}$|see {affine dg-scheme}},
         denoted by $\operatorname{DG^+\Aff}$.
        We denote elements in this opposite category formally by $\Spec(A_\bullet)$.\index{dg-scheme}
    \end{definition}
    \begin{notation}
        For $R_\bullet \in \catdga$ we write $\catdga[R_\bullet]$ for the category $R_\bullet \da (\catdga)$, i.e. 
        cdgas $A_\bullet \in \catdga$ with morphism $R_\bullet \to A_\bullet$.
        Further, for $A_\bullet \in \catdga[R_\bullet]$, an \emph{$A_\bullet$-augmented $R_\bullet$-algebra}\index{augmented} is an object
        $B_\bullet$ of the category $(\catdga[R_\bullet]) \da A_\bullet$, i.e. 
        $B_\bullet \in \catdga$ with a morphisms  $R_\bullet \to B_\bullet \to A_\bullet$ of dgas.
    \end{notation}
    \begin{remark}
        The notation $\Spec(A_\bullet)$ is
        used to stress the similarity to rings and affine schemes. However, at this stage the construction of an affine  dg-scheme
        is purely in a categorical sense, meaning we do not use any of the explicit constructions such as
        the prime spectrum of a ring or locally ringed spaces.    
    \end{remark}
    \begin{remark}
     In geometric terms, one should think of the ``points'' of a dg-scheme just as the points in
     $\Spec(\H_0(A_\bullet))$ (which is a classic affine spectrum). The rest of the structure of a dg-scheme is
     in some sense infinitesimal.
     
     In analytic and $\C^{\infty}$ settings, we can make similar definitions for dg analytic spaces or dg $\C^{\infty}$ spaces, but it is usual to impose some restrictions on the EFC-rings and $\C^{\infty}$-rings being considered, since not all are of geometric origin; we should restrict to those coming from {\it closed} ideals in affine space, with some similar restriction on the $A_0$-modules $A_i$.
    \end{remark}

    \begin{definition}
        Let $A_\bullet, B_\bullet \in \catdga[R]$.
        A morphism $f: A_\bullet \to B_\bullet$ of $R$-cdga is a \emph{quasi-isomorphism}\index{quasi-isomorphism} (or \emph{weak equivalence}\index{weak equivalence}) 
        if it induces an isomorphism on
        homology $H_\ast(A_\bullet) \xrightarrow{\simeq} H_\ast(B_\bullet)$. We say that $R$-cdga
        $A_\bullet$ and $B_\bullet$ are \emph{quasi-isomorphic}\index{quasi-isomorphic} if there exists a diagram  $A_\bullet \la C_\bt   \to B_\bullet$ of quasi-isomorphisms in $\catdga[R]$. 
    \end{definition}
    
        

\subsection{Global structures}
    As a next step, one would like to globalise the concept of an affine dg-scheme to get
    a dg-scheme (or a  dg analytic space or dg $\C^{\infty}$-space in other contexts). There's a straightforward approach to achieve this: instead of a ring in
    degree $0$ and more structure above it, we can take a scheme (or analogous geometric object)  in degree $0$ and a sheaf of
    dg-algebras above it. This definition is due to \cite{Quot} after Kontsevich 
    \cite[Lecture 27]{Kon}.\footnote{These dg-schemes should not be confused with the DG-schemes of \cite{GaitsgoryIndCoh}, which are an alternative characterisation of the derived schemes of Definition \ref{def:dersch}.}
        
    \begin{definition}\label{def:dg-scheme}
        A \emph{dg-scheme}\index{dg-scheme} consists of a scheme $X^0$ and quasi-coherent sheaves $\O_X:=\{\O_{X,i}\}_{i\ge 0}$ on $X^0$ such that
        $\O_{X, 0} = \O_{X^0}$ (i.e. the structure sheaf of $X^0$), equipped
        with a cdga structure, i.e. $\delta: \O_{X,i}\to \O_{X,i-1}$ and
        $\cdot:\O_{X,i} \otimes \O_{X,j} \to \O_{X,i+j}$ satisfying the usual conditions.
    \end{definition}
    
    Although we have given this  definition  in the algebraic setting, obvious analogues exist replacing schemes with other types of geometric object in $\C^{\infty}$ and analytic settings.
    
    \begin{definition}
        A \emph{morphism of dg-schemes}\index{dg-scheme!morphism} $f: (X^0,\O_X) \to (Y^0, \O_Y)$
        consists of a morphism of schemes $f^0:X^0\to Y^0$ and a morphism of sheaves of cdga
        $f^\sharp: f^{-1}\O_Y \to \O_X$.
    \end{definition}
    \begin{definition}
        Define the \emph{underived truncation}
        $\pi^0X \subseteq X^0$ to be $\underline{\Spec}_{X^0}(\H_0(\O_X))$,  \index{underived truncation $\pi^0X$}\index{pi0X@$\pi^0X$|see {underived truncation}}
        the closed subscheme of $X^0$ on which $\delta$ vanishes, or equivalently defined by the
        ideal $\delta \O_{X,1} \subseteq \O_{X,0}$.\footnote{In \cite{Quot}, the notation $\pi_0$ is used for this construction,
        but subscripts are more appropriate for quotients than for kernels, and using $\pi_0$ would cause confusion when we come to combine these with simplicial constructions.} 
    \end{definition}
    The underived truncation $\pi^0X$ is also known as the \emph{classical locus} of $X$\index{classical locus}. 
    \begin{definition}
        A morphism of dg-schemes is a \emph{quasi-isomorphism}\index{quasi-isomorphism} if $\pi^0f:\pi^0X \to \pi^0Y$
        is an isomorphism of schemes  and $\sH_*(\O_Y)\to \sH_*(\O_X)$ (homology taken in the category of sheaves) is an isomorphism of quasi-coherent sheaves on $\pi^0X=\pi^0Y$. 
    \end{definition}
    \begin{remark}
        A problem with \cref{def:dg-scheme} is that $X^0$ has no geometrical meaning, in the sense that we
        can replace it with any open subscheme containing $\pi^0X$ and get a
        quasi-isomorphic dg-scheme. Moreover, the ambient scheme $X^0$ can get in the way when we want
        to glue multiple
        dg-schemes together, since we cannot usually choose the ambient scheme consistently on overlaps.
        
        Gluing tends  not to be an issue for analogous constructions in differential geometry, because a
        generalised form of  Whitney's embedding theorem holds:  a derived $\C^{\infty}$ space  has a
        quasi-isomorphic dg $\C^{\infty}$ space  with  $X^0=\R^N$ whenever its underived truncation
        $\pi^0X$ admits a closed embedding in $\R^N$, by an obstruction theory argument along the
        lines of \S \ref{morphismsrevsn}. 
        However, in algebraic and analytic  settings this definition turns out to be too restrictive
        in general, which can be resolved by working with derived schemes.
    \end{remark}
         
    The following definition incorporates the flexibility needed to allow gluing constructions,
    and gives a taste of the sort of objects we will be encountering towards the end of the notes.
    
    \begin{definition}\label{def:dersch}
        A \emph{derived scheme}\index{derived scheme} $X$ consists of a scheme $\pi^0X$ and a
        presheaf $\O_X$ on the site of affine open subschemes of $\pi^0X$, taking values in $\catdga$, 
         such that $\H_0(\O_X) =\O_{\pi^0X}$ in degree zero and
        all $\H_i(\O_X)$ are quasi-coherent $\O_{\pi^0X}$-modules for all $i\ge 0$.\footnote{
            Here, $\O_{\pi^0X}$ is the structure sheaf of the scheme $\pi^0X$ and $H_i(\O_X)$ is
            a presheaf of homology groups.
        }
    \end{definition}
    In other words, this says that for every inclusion 
        $U \hookrightarrow V$ of open affine subschemes in $\pi_0X$, the maps
\[
 \H_0\O_X(U)\ten_{\H_0\O_X(V)} \H_i\O(V) \to \H_i\O(U)
\]
are isomorphisms.
    
    \begin{construction}
      To get from a dg-scheme to a derived scheme one looks at the canonical embedding
      $i: \pi^0X \hookrightarrow X^0$ and takes $(\pi^0 X, i^{-1}\O_X)$, which is a
      derived scheme.
      
      In the other direction, observe that on each open affine subscheme $U \subseteq \pi^0X$, we have an affine dg-scheme $\Spec (\sO_X(U))$, but that the schemes $\Spec (\sO_{X,0}(U))\supseteq U$ will not in general glue together to give an ambient affine scheme $X^0 \supseteq \pi^0X$.
    \end{construction}
    \begin{remark}[Alternative characterisations]
        By \cite[Theorem~\ref{stacks2-dgshfthm}]{stacks2}, 
        the derived schemes of Definition \ref{def:dersch}  are  equivalent to objects usually described in a much fancier way: those derived Artin or Deligne--Mumford $\infty$-stacks in the sense of \cite{hag2,lurie} whose underlying underived stacks are schemes.
        
 A similar characterisation, using sheaves instead of presheaves, was later stated without proof or reference as \cite[Definition 3.1]{toenICM}. Such an object can be obtained from our data $(\pi^0X,\sO_X)$ by sheafifying each presheaf $\sO_{X,n}$ individually. However, this na\"ive sheafification procedure destroys a hypersheaf property enjoyed by our presheaf $\sO_X$, so the quasi-inverse functor is not just given by forgetting the sheaf property, instead requiring fibrant replacement in a local model structure.      

        \label{dschfootnote}
Also beware that these are not the same as the derived schemes of \cite[Definition 4.5.1]{lurie}, which gives a notion more general than a derived algebraic space (see \cite[Proposition 5.1.2]{lurie}),  out of step with the rest of the literature.
   \end{remark}

        To generalise the definition to derived algebraic spaces (or even derived Deligne--Mumford $1$-stacks), let $\pi^0X$ be an algebraic space (or DM stack) and let $U$ run over affine schemes \'etale over $\pi^0X$. \index{derived algebraic space} \index{derived Deligne--Mumford stack} 
        
        In fact, only a basis for the topology is needed, which also works for derived Artin stacks (see \S \ref{dstacksn})\index{derived Artin stack}. However, we cannot  generalise Definition \ref{def:dersch} so easily to derived Artin stacks, because unlike the \'etale sites, $\pi^0$ 
        does not give an $\infty$-equivalence between the lisse-\'etale sites of $X$ and of $\pi^0X$.

    
    \begin{definition}\label{dgmfddef}
        A dg-scheme is a \emph{dg-manifold}\index{dg-manifold}\footnote{
            The ``manifold'' terminology alludes to the locally free generation of $\sO_{X,\bt}$ by co-ordinate variables.
        }
        if $X^0$ is smooth and as a graded-commutative algebra $\O_X$ is freely generated
        over $\O_{X^0}$ by a finite rank projective module (i.e. a graded vector bundle).
    \end{definition}
    \begin{remarks}
        Note that  the second condition  says that the morphism $\O_{X,0} \to \O_{X, \bt}$ is given by
        finitely generated cofibrations of cdgas.
        Every affine dg-scheme with perfect cotangent complex is quasi-isomorphic to an affine
        dg-manifold. Further, we can drop perfect condition if we drop finiteness in the definition of
        a dg-manifold.
    \end{remarks}
    
    \begin{digression}\label{LRinfty}       
        There is a more extensive literature on dg-manifolds in the setting of differential geometry, often in order to study supersymmetry and  supergeometry in mathematical physics; these tend to be $\Z/2$- or $\Z$-graded and are often known as $Q$-manifolds \index{Q-manifold@$Q$-manifold}
        (their $Q$ corresponding to our differential $\delta$), following \cite{AKSZ,kontsevichPoisson}; also see \cite{DeligneMorganSupersymm,voronovMackenzie,kapranovSupergeometry}. 
        The $Q$-manifold literature tends to place less emphasis on homotopy-theoretical phenomena (and especially quasi-isomorphism invariance) than the derived geometry literature.  
        
        When the sheaf $\sO_{X^0}$ of functions is enriched in the opposite direction to Definition \ref{dgmfddef}, i.e. $\delta \co \sO_{X,0} \to \sO_{X,-1} \to \ldots$, the resulting object behaves very differently from the dg-manifolds we will be using, and corresponds to a stacky (rather than derived) enrichment, related to quotient spaces rather than subspaces. It gives a form of derived Lie algebroid\index{derived Lie algebroid}, closely related to  strong homotopy (s.h.) Lie--Rinehart algebras\index{Lie--Rinehart algebra} ($LR_{\infty}$-algebras). In differential settings, these tend to be known as NQ-manifolds or (confusingly) dg-manifolds. For more  on their relation to derived geometry, see \cite{nuitenThesis,DQDG,smallet} and references therein; see the end of \S \ref{derivedDRsn} below  for a brief explanation of their role establishing Poisson geometry for stacks. Such objects are also closely related to foliations, with a universal characterisation in \cite{LaurentGengouxLavauStroblLieInftyAlgdFoliation}. 
    \end{digression}
    
\subsection{Quasi-coherent complexes}\label{qucohcx}    
    
    \begin{definition}
        Let $A_\bullet \in \catdga$. An \emph{$A_\bullet$-module in chain complexes}\index{module in chain complexes}\index{dgmod@$\catdgm$|see {module in chain complexes}}
        consists of a chain complex
        $M_\bullet$ of $k$-modules and a scalar multiplication $(A\otimes_k M)_\bullet \to M_\bullet$
        which is compatible with the multiplication on $A_\bullet$.
        
        Explicitly, for all $i,j$ we have a $k$-bilinear map
        $A_i \times M_j \to M_{i+j}$ satisfying $(ab)m=a(bm)$, $1m = m$, and the chain map
        condition $\delta_M(am) = \delta_A(a)m + (-1)^{\deg(a)}a\delta_M(m)$.

        We denote the category of  $A_\bullet$-modules in chain complexes by $\catdgmU$, and the subcategory of modules concentrated in non-negative chain degrees  by $\catdgm$.
    \end{definition}
    \begin{definition}
        A morphism of $A$-modules $M_\bullet \to N_\bullet$ is a \emph{quasi-isomorphism}\index{quasi-isomorphism}\index{$\simeq$|see {quasi-isomorphism, weak equivalence}},
        denoted $M_\bullet \simeq N_\bullet$,
        if it induces an
        isomorphism on homology $\H_\ast (M_\bullet) \cong \H_\ast (N_\bullet)$.
    \end{definition}
    \begin{definition}[Global version]\label{derschCart}
        Let $(\pi^0X, \O_X)$ be a derived scheme.
        We can look at $\O_X$-modules $\mathscr{F}$
        in complexes of presheaves. We say they are \emph{homotopy-Cartesian modules}\index{homotopy-Cartesian!module} (following \cite{hag2}), or \emph{quasi-coherent complexes}\index{quasi-coherent complex} (following \cite{lurie}), if the homology {\it presheaves} $\H_i(\sF)$ are all quasi-coherent $\O_{\pi^0X}$-modules.
        \end{definition}
In other words, this says that for every inclusion 
        $U \hookrightarrow V$ of open affine subschemes in $\pi_0X$, the maps
\[
 \H_0\O_X(U)\ten_{\H_0\O_X(V)} \H_i\sF(V) \to \H_i\sF(U)
\]
are isomorphisms. The conditions on $\O_X$ mean that is  equivalent to saying that, for the derived tensor product $\ten^{\oL}$ of Definition \ref{tenoLdef} below, the maps
\[
	        \O_X(U)\otimes^{\oL}_{\O_X(V)} \mathscr{F}(V) \to \sF(U)
        \]
        are quasi-isomorphisms, which is the characterisation  favoured in the original sources.

\subsection{What about morphisms and gluing?}

   
 	We want to think of derived schemes $X,Y$ as equivalent if they can be connected by
	a zigzag of quasi-isomorphisms.
    \[\xymatrix{
        X && W_1 &&& \ldots  & W_n  &&   Y\\
        & Z_0\ar[ul]^{\sim} \ar[ur]_{\sim} && Z_1\ar[ul]^{\sim} \ar[ur]_{\sim} &\ldots & 
        \ar[ur]_{\sim} && Z_n\ar[ul]^{\sim} \ar[ur]_{\sim}
    }\]
    
    How should we define morphisms compatibly with this notion of equivalence?\footnote{The first global constructions \cite{Kon,Quot,Hilb} of derived moduli spaces did not come with  functors of points partly because morphisms are so hard to define; it was not until \cite{dmc} that those early constructions were confirmed to parametrise the ``correct'' moduli functors.}  What about gluing data?
    
    We could forcibly invert
    all quasi-isomorphisms, giving the ``homotopy category'' $\Ho(\catdga)$ \index{homotopy category}
    in the affine case. That doesn't have limits and colimits, or behave well with gluing.
    For any small category $I$, we might also want to look at the category $\catdga^I$ of $I$-shaped diagrams of cdgas (e.g. taking $I$ to be a poset of open subschemes as in the definition
    of a derived scheme). There is then a homotopy category 
            $\Ho(\catdga^I)$ of diagrams, given by inverting objectwise quasi-isomorphisms.
    But, unfortunately, the natural functor
    \[
        \Ho(\catdga^I) \to \Ho(\catdga)^I
    \]
    (from the homotopy category of diagrams to diagrams in the homotopy category)
    is seldom an equivalence;
    it goes wrong for everything but for discrete diagrams, i.e. when $I$ is a set. This means that constructions such as sheafification are doomed to fail if we try to formulate everything in terms of the homotopy category $\Ho(\catdga)$.
    
    To fix this, we will need some flavour of infinity (i.e. $(\infty,1)$) category, this description in terms of diagrams being closest to
    Grothendieck's derivators. An early attempt to address the problems of morphisms and gluing for dg-schemes  was \cite{behrendDGSch2}, which used $2$-categories to avoid the worst pathologies.
  
\clearpage
\section[Infinity categories and model categories]{Infinity categories and model categories (a bluffer's guide)}\label{inftycatsn}
\subsection{Infinity categories}

    There are many equivalent notions of $\infty$-categories. We start by looking at a few
    different ones as it can be quite useful to have different ways to think about
    $\infty$-categories at hand. 
    This entire section is meant to merely give an overview of the more accessible notions of
    $\infty$-categories and is in no way meant to be a complete or rigorous introduction.
    For equivalences between these and some other models of $\infty$-categories, see for instance \cite{joyaltierney,joyalQCatSCat}. For the general theory of $\infty$-categories, with slightly different emphasis, see \cite{hinichLectInftyCats,cisinskiHigherCats}.

\subsection{Different notions of \texorpdfstring{$\infty$}{infinity}-categories}

    We continue with some constructions of $\infty$-categories.

    \begin{enumerate}
        \item Arguably \emph{topological categories}\index{topological category}
            are conceptually the easiest notion. 
            A \emph{topological category} is a category enriched in topological spaces
            (i.e. for any two objects $X, Y\in \mathcal{C}$, the morphisms between them
             form a topological space $\Hom_{\mathcal{C}}(X, Y)$ and composition is a
            continuous operation).
            
            Given a topological category, $\C$,  the \emph{homotopy category}\index{homotopy category}
            $\Ho(\mathcal{C})$ of $\mathcal{C}$ \index{homotopy category}
            is the category with the same objects as $\mathcal{C}$ and the morphisms are given by path
            components of morphisms in $\mathcal{C}$, i.e. $\pi_0\Hom_{\cC}(X,Y)$.
            
            A  functor $\mathcal{F}: \mathcal{C} \to \mathcal{D}$
            (assumed to respect the extra structure, so everything is continuous)  is a \emph{quasi-equivalence}\index{quasi-equivalence} if
            \begin{enumerate}
                \item for all $X, Y\in \mathcal{C}$ the map
                    $\Hom_{\mathcal{C}}(X, Y) \to \Hom_{\mathcal{D}}(\mathcal{F}(X), \mathcal{F}(Y))$
                    is a weak homotopy equivalence of topological spaces
                    (i.e. induces isomorphisms on homotopy groups).
                \item $\mathcal{F}$ induces an equivalence on the homotopy categories
                    $\Ho(\cF) \co \Ho(\C) \to \Ho(\cD)$.
            \end{enumerate}
        \item Topological spaces contain a lot of information,  so a more combinatorially efficient model
            with much of the same intuition is given by \emph{simplicial categories}\index{simplicial!category},
            which have a simplicial set of morphisms between each pair of objects.
            We will be defining simplicial sets in \ref{simplicial-sets}. The behaviour is much the same
            as for topological categories but simplicial categories have much less data
            to handle.
        \item\label{item:relative-categories} By far the easiest to construct are
            \emph{relative categories}\index{relative category}
            \cite{simploc2,DKEquivsHtpyDiagrams,BarwickKanRelative}.
            These consist of pairs $(\cC,\cW)$ where
            $\cC$ is a category and $\cW$ is a subcategory. That's it!\footnote{ignoring
            cardinality issues/Russell's paradox} The idea is that the morphisms in
            $\cW$ should encode some notion of equivalence
            weaker than isomorphism.\index{weak equivalence} The homotopy
            category $\Ho(\C)$ \index{homotopy category}
            is a localisation of $\cC$ given by forcing
            all the morphisms in  $\cW$ to become isomorphisms, and the
            associated simplicial category $L_{\cW}\cC$ arises as a fancier form of localisation, 
            whose path components of morphisms recover
            the homotopy category.
            
            Examples of subcategories $\cW$ are homotopy equivalences
            or weak homotopy equivalences for topological spaces, quasi-isomorphisms
            of chain complexes and of cdgas, and equivalences of categories.
            
            The main drawback is that quasi-equivalences of relative categories are hard to describe.
            
        \item \emph{Grothendieck's derivators}\index{Grothendieck's derivators}  provide another useful 
            perspective: Given a small category $I$, we can look at the $\infty$-category of
            $I$-shaped diagrams $\C^I$ in an $\infty$-category $\C$, and then there is a
            natural functor $\Ho(\C^I)\to \Ho(\C)^I$ from the homotopy category of
            diagrams to diagrams in the homotopy category, which is
            usually \textbf{not an equivalence}; instead, these data essentially determine
            the whole $\infty$-category.
            
            Concretely,
            a \emph{derivator}\index{derivator} is an assignment $I \mapsto \Ho(\C^I)$ for all small categories $I$. There are several accounts of the theory written by Maltsiniotis and others.
            It turns out that a derivator determines the $\infty$-category structure on $\C$, up to essentially unique quasi-equivalence, by \cite{renaudin}.
            This can be a  useful way to think about $\infty$-functors $\C \to \cD$, since they amount to giving compatible functors $\Ho(\C^I) \to \Ho(\cD^I)$ for all $I$.
    \end{enumerate}
    \begin{remarks}
        Especially (\ref{item:relative-categories}) illustrates how little data one needs to specify
        an $\infty$-category. While topological categories suggest that there are entire topological
        spaces to choose,  relative categories show that in practice once a notion of
        weak equivalence has been picked, everything else is determined.

        Model categories don't belong in this list. They are relative categories equipped with some
        extra  structure (two more subcategories in addition to $\cW$) which makes
        many calculations feasible  --- a bit like a presentation for a group ---  and avoids
        cardinality issues. See \cite{QHA,hovey,Hirschhorn} and \S \ref{modelsn} below.
    \end{remarks}
    
       If anyone gives you an infinity category, you can assume it's a topological or simplicial category,
        while if someone asks you for an infinity category, it's enough to give them a relative category.     

\subsection{Derived functors}

    Although derived functors are often just defined in the setting of model categories,  they only depend on relative category structures, as in the  approach of \cite{DwyerHirschhornKanSmith}:

    \begin{definition}
        If $(\C,\cW)$ and  $(\cD, \cV)$ are relative categories and  $F\co \C \to \cD$ is a functor of
        the underlying categories, we say that $F'\co \Ho(\C) \to \Ho(\cD)$ is
        a \emph{right-derived functor}\index{derived functor}\index{right-derived functor $\oR F$}
        of $F$, and denote $F'$ by $\oR F$, if:
        \begin{enumerate}
            \item There is a natural transformation
                $\eta\co \lambda_{\cD}\circ  F  \to F'\circ \lambda_{\C}$,
                for $\lambda_{\C} \co \C \to \Ho(\C)$
                and $\lambda_{\cD} \co \cD \to \Ho(\cD)$.
            \item Any natural transformation $\lambda_{\cD}\circ  F  \to G\circ \lambda_{\C}$
                factors through $\eta$, and this factorisation is unique up to
                natural isomorphism in $\Ho(\cD)$ 
                --- this condition ensures that $F'$ is unique up to weak equivalence.
                  
        \end{enumerate}
        The dual notion is called \emph{left-derived functor}\index{left-derived functor $\oL F$} and denoted by $\oL F$.
    \end{definition}
    \begin{warning}
        The notation $\oR F$ is also used to denote derived $\infty$-functors
        $L_{\cW}\C \to L_{\cV}\cD$. See  \cite[\S 4.1]{riehlHtpicalCats} for more on this 
        view of derived functors. The results there are stated for homotopical categories,
        which are relative categories with extra restrictions (almost always satisfied in practice).
%
    \end{warning}    
    
    \begin{examples} \label{cohoex}
        Most homology/cohomology theories arise as left/right derived functors.

        \begin{enumerate}
            \item Consider the global sections functor $\Gamma$ from sheaves of non-negatively graded cochain complexes 
                on a topological space $X$ to cochain complexes of abelian groups. If we take weak equivalences being quasi-isomorphisms on both sides, then $\Gamma$ has a right-derived functor $\oR\Gamma$, whose cohomology groups are just sheaf cohomology.
            \item For any category $\C$, consider the functor $\Hom: \C^{\op} \times \C \to \Set \subset \Top$
                (or simplicial sets, if you prefer).
                For a subcategory $\cW \subseteq \C$ and for $\pi_\ast$-equivalences in $\Top$, we get
                a right-derived functor, the derived mapping space $\oR\Map: \Ho(\C)^{\op} \times \Ho(\C)\to \Ho(\Top)$.\index{mapping space $\oR\Map$}
                That's essentially how simplicial and topological categories are associated to relative categories --- the spaces of morphisms in the topological category associated to the relative category $(\C, \cW)$ are then 
                just $\oR\Map(X, Y)$ (up to weak homotopy equivalence).
            \item For a category $\C$ of chain complexes, we  have an enriched $\Hom$ functor
                $\HHom\co \C^{\op}\times \C \to \operatorname{CochainCpx}$ (with $\Hom=\z^0\HHom$). If we take weak equivalences to be quasi-isomorphisms on both sides, this then leads to a right-derived functor $\oR\HHom \co \C^{\op}\times \C \to \operatorname{CochainCpx}$, which  has cohomology groups
                $H^i\oR\HHom(X,Y) \cong \EExt^{i}(X,Y)$. 
                The space $\oR\Map$ is then just the topological space associated to the good truncation of this complex, which satisfies $\pi_j\oR\Map(X,Y) \cong \Ext^{-j}_{\C}(X, Y)$ for $j \ge 0$.\footnote{
                    This last statement follows by combining the Dold--Kan equivalence with composition of right-derived functors, 
                    using that the right-derived functor of $\z^0$ is the good truncation $\tau^{\le 0}$.
                }
            \item\label{tophomologyitem} As a more exotic example, if $F$ is the functor sending a topological space $X$ to the free
                topological abelian group generated by $X$, then taking weak equivalences to be $\pi_*$-isomorphisms on both sides, we have a  left-derived functor $\oL F$ with homotopy groups $\pi_i\oL F(X) \cong \H_i(X)$ given by singular homology.        
        \end{enumerate}       
    \end{examples}
    
\subsection{Model categories}\label{modelsn}

    A standard reference for this section is \cite{hovey}.
    The idea is to endow a relative category with extra structure aiding computations.
    This is similar in flavour to presentations of a group; once
    the weak equivalences are chosen all the homotopy theory is determined, but
    the extra structure (classes of fibrations and cofibrations) makes it much more accessible.

    \begin{definition}
        A \emph{model category}\index{model category} is a relative category $(\C, \cW)$ together with two choices of 
        classes of morphisms, called \emph{fibrations} and \emph{cofibrations}.\index{fibration}\index{cofibration}
        These classes of morphisms are required to satisfy several further axioms.
        
        A \emph{trivial (co)fibration} is a (co)fibration that is also in $\cW$, i.e. also 
        a weak equivalence.\index{trivial fibration}\index{trivial cofibration}
    \end{definition}
    
    \begin{definition}
        Let $f\co X \to Y$ and $g\co A \to B$  be morphisms in a category. We say that $f$ has the \emph{left lifting property}\index{left lifting property}
        with respect to $g$ (LLP for short)\index{LLP}, or dually that $g$ has the \emph{right lifting property}\index{right lifting property} with respect to $f$ (RLP for short) \index{RLP}
        if for any commutative diagram of the form
        below, there is a lift as indicated.
        \[\begin{tikzcd}
            X \arrow[d, "f"]\arrow[r] & A\arrow[d, "g"]\\
            Y\arrow[r]\arrow[ru, dotted, swap, "\exists"] & B
        \end{tikzcd}\]
        
        We say that a morphism has LLP (resp. RLP) with respect to a class $S$ of morphisms if it has LLP (resp. RLP) with respect to all members of $S$.
            \end{definition}
    \begin{example}
        Any category with limits and colimits has a trivial model structure, in which all morphisms are both fibrations and cofibrations, while the weak equivalences are just the isomorphisms.
    \end{example}
    \begin{example}[Model structure on $\catdga$]\label{exp:model-struct-dgAlg} \index{model structure on cdgas!standard}
        There is a model structure on $\catdga$, due to Quillen \cite{QRat}.\footnote{
            Quillen's proof is for dg-Lie algebras, but he observed that the same  proof works for other types of algebras. For associative algebras and other algebras over non-symmetric operads, our characteristic $0$ hypothesis becomes unnecessary.
        }
        On $\catdga$ weak equivalences are quasi-isomorphisms. Fibrations are maps
        which are surjective in strictly positive chain degree, i.e. $f: A_i \to B_i$ is surjective
        for all $i>0$.
        Cofibrations are maps $f: P_\bullet \to Q_\bullet$ which have
        the left lifting property with respect to
        trivial fibrations. Explicitly, if $Q_\bullet$ is  \emph{quasi-free}\index{quasi-free} over $P_\bullet$ in the sense that it is freely generated as a
        graded-commutative algebra, then $f$ is a cofibration. An arbitrary cofibration is a retract of a quasi-free map.
  
    \end{example}
    
    \begin{example}[Model structure on $\operatorname{DG^+\Aff}$]
        When considering the opposite category
        $\operatorname{DG^+\Aff}=(\catdga)^\op$ one 
        takes the opposite model structure, so cofibrations in $\catdga $ correspond to fibrations in $DG^+\Aff$ and vice versa.
    \end{example}

    \begin{example}\label{projmodelstr}\index{model structure on modules!projective}
        Another model structure is the
        projective model structure on non-negatively graded chain complexes of modules over a ring $R$:
        weak equivalences are quasi-isomorphisms, fibrations are surjective in strictly positive
        chain degrees,  and cofibrations are maps $f: M \hookrightarrow N$ such that
        $N/M$ is a complex of projective $R$-modules.
        
        The resulting homotopy category\footnote{
            Here, we are using ``homotopy category'' in the homotopy theory sense of inverting weak equivalences (i.e. quasi-isomorphisms);
            beware that this clashes with the usage in  homological algebra which refers to the category $\mathbf{K}(R)$ of \cite[\S 10.1]{W}
            in which only strong homotopy equivalences are inverted.
        }
        is the full subcategory of the  derived category $\cD(R)$ on non-negatively graded chain complexes.
    \end{example}
        \begin{example}\label{injmodelstr}\index{model structure on modules!injective}
        Dually there is an injective model structure for non-negatively graded  cochain complexes of $R$-modules:
        weak equivalences are quasi-isomorphisms, cofibrations have trivial kernel in strictly
        positive degrees, and fibrations are surjective maps with levelwise injective kernel.
        The resulting homotopy category is the full subcategory of the  derived category $\cD(R)$ on non-negatively graded cochain complexes.
    \end{example}
    
    \begin{remark}
        There are also $\Z$-graded versions of the two examples above, but cofibrations
        (resp. fibrations) have extra restrictions.
        Specifically in the projective case, there should exist an ordering on the generators $x$ by some
        ordinal such that each $\delta x$ lies in the span of generators of lower order.
        For complexes bounded below, we can just order by degree; in general,
        the total complex of a Cartan--Eilenberg resolution as in \cite[\S 5.7]{W} is cofibrant.
        In both cases, the resulting homotopy category is the  derived category $\cD(R)$.\index{derived category}
    \end{remark}
    \begin{properties}\label{rmk:model-structure-axioms}
        Here we list some of the key properties of model structures, though this is not an exhaustive list of required axioms: \index{model category}
        \begin{itemize}
            \item (Lifting A)
                Cofibrations have the LLP with respect to all trivial fibrations.
            \item (Lifting B)
                Trivial cofibrations have the LLP with respect to all fibrations.
            \item (Lifting B')
                Dually, fibrations have the RLP with respect to trivial cofibrations.
            \item (Lifting A')
                Dually, trivial fibrations have the RLP with respect to all cofibrations.    
            \item (Factorisation A)
                Every morphism $f: A\to B$ can be factorised  as
                $A \to \tilde A \to B$ where the first map is a trivial cofibration
                and the second one a fibration. (In \Cref{injmodelstr}, this gives rise to injective resolutions.)
            \item (Factorisation B)
                Every morphism $f: A\to B$ can be factorised  as
                $A \to \hat B \to B$ where the first map is a cofibration
                and the second one a trivial fibration. (In \Cref{projmodelstr}, this gives rise to projective resolutions.)    
        \end{itemize}
    \end{properties}
    \begin{example}
        Let $R$ be a commutative $k$-algebra, $a\in R$ not a zero-divisor, and consider the map $R \to R/(a) =:S$.
        Here is a factorisation  $R\to \hat S \to S$  in $\catdga$ of the quotient map
       such that $R\to \hat S$ is a cofibration and $\hat S \to S$
        a trivial fibration.
For an element $t$ of degree $1$,    we set $\hat S:= (R[t], \delta t = a)$
        so this is a chain complex of the form $0 \to R.t\xrightarrow{\delta} R$.
        The cofibration $R \to \hat S$ is then just the canonical inclusion and the
        trivial fibration sends $t$ to $0$.
    \end{example}
    \begin{remark}
        We will follow the modern convention for model categories in assuming that the factorisations A and B above can be chosen functorially. However, beware that the functorial factorisations tend to be huge.
    \end{remark}
    \begin{example}
        On topological spaces, there is a model structure in which weak equivalences are $\pi_\ast$-equivalences; note that these really are weak, not distinguishing between totally disconnected (e.g. $p$-adic) and discrete topologies. Fibrations are Serre fibrations, which have RLP with respect to the inclusions  $S^n_+ \to B^{n+1}$ of the closed upper $n$-hemisphere in an $n$-ball, for $n\ge 0$ (see Figure \ref{fig:fibrations}).
        Cofibrations are then defined via LLP, or generated by $S^{n-1} \to B^{n}$ for $n \ge 0$ --- these include all relative CW complexes.
    \end{example}
    \begin{example}
        We've already seen commutative dg $k$-algebras in non-negative chain degree. There are variants for dg EFC and $\C^{\infty}$-algebras. Weak equivalences are quasi-isomorphisms. Fibrations are maps
        which are surjective in strictly positive chain degree, i.e. $f: A_i \to B_i$ is surjective
        for all $i>0$.
        
        Cofibrations are again defined by LLP, the property being satisfied whenever the morphism is freely generated as a graded EFC or $\C^{\infty}$-algebra. Freely graded $\C^{\infty}$-algebras over $\R$ are $\C^{\infty}(\R^S,\R)[x_i ~:~ i \in I]$, with $\deg x_i >0$ (taking exterior powers for odd variables) for sets $S,I$ and for $\C^{\infty}(\R^S,\R)$ as in \Cref{Cinfty}. 
    \end{example}

    
\subsection{Computing the homotopy category using model structures}
    \begin{definition}\label{cofreplace}
        We say that an object in a model category is \emph{fibrant}\index{fibrant} if the map to the final object is a fibration, and \emph{cofibrant}\index{cofibrant} if the map from the initial object is a cofibration.
            
        Given an object $A$ and a weak equivalence $A \to \hat{A}$ with $\hat{A}$ fibrant, we refer to $\hat{A}$ as a \emph{fibrant replacement} of $A$. Dually, if we have a weak equivalence $\tilde{A} \to A$ with $\tilde{A}$ cofibrant, we refer to $\tilde{A}$ as a \emph{cofibrant replacement}\index{replacement!(co)fibrant} of $A$.
    \end{definition}
    \begin{example}
        With the model structure on $\catdga$ from \cref{exp:model-struct-dgAlg} every object is fibrant.
    \end{example}
    \begin{definition}
        Given a fibrant object $X$, a \emph{path object}\index{path object}\index{PX@$PX$|see {path object}} $PX$ for $X$ is an object $PX$ together with a diagram	 
        \[
        \begin{tikzcd}
            X \arrow[r, "w"] \arrow[dr, swap, "\textrm{diag.}"] & PX \arrow[d, "f"]\\
            & X \by X\\
        \end{tikzcd}
        \]
        where $w$ a weak equivalence and $f$ a fibration.
    \end{definition}
    \begin{remark}
        Note that path objects always exist, by applying the factorisation axiom in
        \ref{rmk:model-structure-axioms} to the diagonal $X \to X \by X$.
    \end{remark}
    \begin{theorem}[Quillen]\label{thm:maps-homotopy-cat}
        Let $A\in \C$ be a cofibrant object and $X\in C$ a fibrant object. Morphisms in the homotopy category $\Ho(\C)$ are given by $\Hom_{\Ho(\C)}(A,X)$ being the coequaliser (i.e. quotient) of the diagram
        \[
            \Hom_{\C}(A,PX) \rightrightarrows \Hom_{\C}(A,X)
        \]
        induced by the two possible projections $PX \to X\by X \rightrightarrows X$.
    \end{theorem}
    \begin{example}
        In the model category $\catdgmU[A_\bullet]$, of $A_\bt$-modules in (unbounded) chain complexes, a path object $PM_\bt$ for $M_\bt$ is given by $(PM)_n:= M_n \oplus M_n \oplus M_{n+1}$, with $\delta (a,b,c)=(\delta a, \delta b, \delta c + (-1)^n (a-b))$. The map $M_\bt \to PM_\bt$ is $a \mapsto (a,a,0)$, and the map $PM_\bt \to M_\bt \by M_\bt$ is $(a,b,c) \mapsto (ab)$.
        
        Thus for cofibrant $Q_\bt$  (e.g. levelwise projective and bounded below in chain degrees) 
        two morphisms $f,g \co Q_\bt \to M_\bt$ are homotopic if and only if there exists a graded morphism $h \co Q_\bt \to M[-1]_\bt$ such that $f- g = \delta \circ h + h \circ \delta$.

        To modify this example for chain complexes concentrated in non-negative degrees, apply
        the \emph{good truncation}\index{good truncation}\index{tau@$\tau_{\ge 0}$, $\tau^{\le 0}$|see {good truncation}} $\tau_{\ge 0}$ in non-negative chain degrees, replacing $PM_\bt$ with $\tau_{\ge 0}PM_\bt$ so that it is still in the same category; the description of homotopic morphisms is unaffected.
        
        Explicitly, good truncation is $(\tau_{\ge 0}V)_i=\begin{cases} V_i & i>0 \\\z_0V & i=0 \\ 0 & i<0 \end{cases}$, where $\z_0V=\ker(\delta \co V_0 \to V_1)$.
        
    \end{example}
    \begin{example}
        In topological spaces, we can just take $PX$ to be the space of paths in $X$, i.e. the space of continuous maps $[0,1] \to X$, with
        \[
            X \xra{\text{constant}} PX \xra{(\ev_0,\ev_1)} X \by X.
        \]
        Thus morphisms in the homotopy category are just homotopy classes of morphisms.
    \end{example}
    \begin{example}\label{dgalgpath}
        In $\catdga$, a choice of path object is given by taking $PA_\bullet = \tau_{\ge 0}(A_\bullet[t,\delta t])$, for $t$ of degree $0$. The map $A_\bullet \to PA_\bullet$ is the  inclusion of constants, and the map $PA_\bullet \to A_\bullet \by A_\bullet$ given by $a(t) \mapsto (a(0),a(1))$ and $b(t)\delta t \mapsto 0$.
        Explicitly, 
        \[
            (PA)_n =
            \begin{cases}
                A_n[t] \oplus A_{n+1}[t]\delta t     & n>0\\
                \ker(\delta \co A_0[t] \oplus A_{n+1}[t]\delta t \to A_0[t]\delta t) & n=0,
            \end{cases}
        \]
        where $\delta (\sum_i a_i t^i)= \sum_i (\delta a_i) t^i + \sum_i (-1)^{\deg a_i} i a_i t^{i-1}\delta t$.   
         
        Thus for $C_\bullet$ cofibrant, $\Hom_{\Ho(\catdga)}(C_\bullet,A_\bullet)$ is the quotient 
        \[
            \Hom_{\catdga}(C_\bullet,A_\bullet)/\Hom_{\catdga}(C_\bullet,PA_\bullet). 
        \]
    \end{example}
    \begin{digression}\label{hensel}
        Taking a cofibrant replacement can be nuisance, but there are Quillen equivalent model structures with more cofibrant objects but fewer fibrant objects, with the fibrant replacement functor sending a  cdga $A_\bt$ to its completion, Henselisation or localisation over $\H_0(A_\bt)$;
        existence of all these  follows from \cite[Lemma \ref{stacks2-dgshrink}]{stacks2}, with details for the complete case in \cite[Proposition \ref{drep-cNhat}]{drep} and the others (including $\C^{\infty}$ and analytic versions) in \cite[Proposition \ref{DStein-locmodelprop}]{DStein}.\index{model structure on cdgas!complete} For the complete and Henselian model structures, all smooth $k$-algebras are cofibrant.   

            Specifically, cofibrations in the local (resp. Henselian) model structure are generated by cofibrations in the standard model structure together with localisations (resp. \'etale morphisms) of discrete algebras. Fibrations are those fibrations $A_\bt \to B_\bt$ in the standard model structure for which $A_0 \to B_0\by_{\H_0(B_\bt)}\H_0(A_\bt)$ is conservative (resp. Henselian) in the terminology of \cite[\S 4]{anel}. The identity functor from the standard model structure to the local or Henselian model structure is then a left Quillen equivalence.\index{model structure on cdgas!local}\index{model structure on cdgas!Henselian}
    \end{digression}
    
    \begin{example}\label{ex:odot}
        For dg $\C^{\infty}$-algebras, there is a similar description of path objects, with $PA = \tau_{\ge 0}(A \odot\C^{\infty}(\R)[\delta t])$, for $t \in \C^{\infty}(\R)$ the co-ordinate and $\odot$ the $\C^{\infty}$ tensor product, given by 
        \[
            (\C^{\infty}(\R^m)/(f_1,f_2, \ldots)) \odot  (\C^{\infty}(\R^n)/(g_1,g_2, \ldots))\cong (\C^{\infty}(\R^{m+n})/(f_1,g_1, f_2,g_2, \ldots)),
        \]
        so in particular $\C^{\infty}(X)\odot\C^{\infty}(Y)\cong \C^{\infty}(X \by Y)$.
        Again, the map $A \to PA$ is given by inclusion of constants, and the map $PA \to A \by A$ by evaluation at $t=0$ and $t=1$. 
        There is an entirely similar description for EFC algebras using analytic functions. 
    \end{example}

\subsection{Derived functors}\label{derivedfnsn}

    Model categories also give conditions for derived functors to exist and provide a way to compute them.
    
    \begin{definition}
        A functor $G \co \C \to\cD$ of model categories is
        \emph{right Quillen}\index{Quillen functor}\index{right Quillen}
        if it has a left adjoint $F$  and  preserves fibrations and trivial fibrations.
        Dually, $F$ is \emph{left Quillen}\index{left Quillen} 
        if it has a right-adjoint and $F$ preserves cofibrations and trivial cofibrations.
        $F\dashv G$ is in that case called a \emph{Quillen adjunction}\index{Quillen adjunction}.
    \end{definition}
    
    The following lemma, an exercise in lifting properties, shows that the notion of Quillen adjunction is well defined.
    
    \begin{lemma}\label{lemma:left-quillen-right-quillen}
        Let $F\dashv G$ be an adjunction of functors of model categories.
        $F$ is left Quillen if and only if $G$ is right Quillen
    \end{lemma}
    \begin{theorem}[Quillen]
        If $G$ is right Quillen, then the right-derived functor $\oR G$ exists and is given on objects by $A \mapsto  G\hat{A}$, for $A \to \hat{A}$ a  fibrant replacement. 
        Dually, left Quillen functors give left-derived functors by cofibrant replacement.
    \end{theorem}
    \begin{remark}
        To get a functor, we can take fibrant replacements functorially, but on objects the choice of fibrant replacement doesn't matter (and in particular need not be functorial), because it turns out that right Quillen functors preserve weak equivalences between fibrant objects. The proof is an exercise with path objects.
    \end{remark}
    \begin{example}
        We can thus interpret sheaf cohomology in terms of derived functors, because fibrant replacement in the model category of non-negatively graded cochain complexes of sheaves corresponds to taking an injective resolution.
    \end{example}
    \begin{definition}
        A Quillen adjunction $F \dashv G$ is said to be a \emph{Quillen equivalence}\index{Quillen equivalence} if
        $\oR G \co \Ho(\C) \to \Ho(\cD)$ is an equivalence of categories, in which case it has quasi-inverse $\oL F$.
    \end{definition}
        Explicitly, this says that for all fibrant objects $A \in \C$ and cofibrant objects $B \in \cD$, the unit and co-unit
        give rise to weak equivalences $F(\widetilde{GA}) \to A$ and $B \to G(\widehat{FB})$, where $\widehat{(\;)}$ and $\widetilde{(\;)}$ 
        are fibrant and cofibrant replacement.
    \begin{remark}
        Note that Quillen equivalences give equivalences on simplicial localisations. This is because in addition to the equivalence $\Ho(\C) \simeq \Ho(\cD)$, we have weak equivalences  $\oR \Map_{\C}(A,C) \simeq \oR \Map_{\cD}(GA,GC)$ of mapping spaces. 
        
        From the derivator perspective, we have an equivalence of $\infty$-categories because there are systematic techniques for endowing diagram categories $\C^I$ with model structures such that $G \co \C^I \to \cD^I$ is also a right Quillen equivalence, giving us equivalences $\Ho(\C^I) \to \Ho(\cD^I)$ for all small diagrams $I$.   
    \end{remark}
    

\subsection{Homotopy limits and fibre products}

    \begin{definition} \index{homotopy limit $\ho\Lim$ or $\oR\Lim$}
        \emph{Homotopy limits} $\ho\Lim_I$ or $\oR\Lim_I$ are right-derived functors of the limit functors $\Lim_I \co \C^I \to \C$ (weak equivalences in $\C^I$ defined objectwise).
        For diagrams of the form $X \to Y \la Z$, we denote the \emph{homotopy fibre product} by $X\by^h_YZ$. \index{homotopy fibre product $X\by^h_YZ$}
    \end{definition}
    
    
    
    \begin{lemma}
        If $Y$ is fibrant, the homotopy fibre product $X\by_Y^hZ$ is given by $\hat{X}\by_{Y}\hat{Z}$, where $\hat{X}\to Y$ and $\hat{Z} \to Y$ are fibrant replacements over $Y$. In \emph{right proper}\index{right proper}\footnote{Almost everything we work with will satisfy this; it says that weak equivalence is preserved by pullback along fibrations.}
        model categories, it suffices to take $\hat{X}\by_YZ$.
    \end{lemma}
    
    
    
    \begin{example}\label{lesfibreprod}
       For fibrant objects,  an explicit construction of $X\by^h_YZ$ is given in terms of the path object by  $X\by_{Y, \ev_0}PY\by_{\ev_1,Y}Z$. In particular, for topological spaces we get $\{x\}\by^h_Y\{z\}\simeq P(Y;x,z)$, the space  of paths from $x$ to $z$ in $Y$. Hence $\{y\}\by^h_Y\{y\}= \Omega(Y;y)$, the space of loops based at $Y$.
       
        For homotopy fibre products of topological spaces, we have a long exact sequence
        \[
            \pi_i(X\by^h_YZ) \to \pi_i X \by \pi_i Z \to \pi_iY \to  \pi_{i-1}(X\by^h_YZ) \to \ldots \to \pi_0Y 
        \]
        of homotopy groups and sets.\index{long exact sequence of groups and sets}\footnote{Exactness of this sequence can be deduced from a double application of the long exact sequence of  \cite[Lemma I.7.3]{sht} by observing that the homotopy fibres of $X\by^h_YZ \to Z$ and $X \to Y$ are equivalent.}

Explicitly, this means that the maps $\pi_i(X\by^h_YZ)\to \pi_iX\by_{\pi_iY}\pi_iZ$ (for compatible choices of basepoint) are all surjective, with $\coker(\pi_{i+1}X\oplus \pi_{i+1}Z \to \pi_{i+1}Y)$ acting transitively on the fibres for $i>0$, while when $i=0$ the fibre  over $([x],[z])$ is isomorphic to the set $\pi_1(X,x)\backslash \pi_1(Y; \bar{x},\bar{z})/ \pi_1(Z,z)$ of double cosets.
    \end{example}
    \begin{example}
        Similarly, in cdgas, we get  a long exact sequence
        \[
            \H_i(A_\bt\by^h_{B_\bt}C_\bt) \to \H_i(A_\bt) \by \H_i(B_\bt) \to \H_i(C_\bt) \to  \H_{i-1}(A_\bt\by^h_{B_\bt}C_\bt) \to \ldots .
        \]
        We can evaluate this as $A_\bt \by_{B_\bt}(PB_\bt)\by_{B_\bt}C_\bt$, though $\hat{A}_\bt\by_{B_\bt}C_\bt$ for any fibrant replacement $\hat{A}_\bt \to B_\bt$ will do. 
    \end{example}
     
\clearpage
\section{Consequences for dg-algebras}\label{consequencesn}

    We have an embedding of algebras in cdgas
    \begin{align*}
        \Alg_k &\subseteq \catdga[k]\\
        A &\mapsto (A \la 0 \la 0 \la \ldots )
    \end{align*}
    which induces a map $\Alg_k \to \Ho(\catdga)$ by composition
    with the map $\catdga \to \Ho(\catdga)$.
    
    \begin{lemma}\label{discretefull}
        The induced functor $\Alg_k \to \Ho(\catdga)$ is full and faithful.
    \end{lemma}
    \begin{proof}
        First, we observe the following.
        For any $A_\bullet \in \catdga$ and $B \in \Alg_k$ we have 
        \[
            \Hom_{\catdga}(A_\bullet,B) =\Hom_{\Alg_k}(\H_0(A_\bullet),B)
        \]
        because for any $f\in \Hom_{\catdga}(A_\bullet,B)$
        anything positive $a\in A_{>0}$ has to map to zero, $f(a) = 0 \in B_i$, and thus
        $f(\delta a') = \delta f(a') = 0$ for all $a' \in A_1$.
        In particular we can also replace $A_\bullet$ with a cofibrant replacement
        $\tilde A_\bullet$ to obtain
        \[
            \Hom_{\catdga}(\tilde{A_\bullet},B) =
            \Hom_{\Alg_k}(\H_0(\tilde{A_\bullet}),B) =
            \Hom_{\Alg_k}(\H_0(A_\bullet),B)
        \]
        
        Next, we observe that for any $B \in \Alg_k$ the map $B \to B\by B$ is a
        fibration (as there is nothing in positive degrees), so any such $B$ is a
        path object for itself in $\catdga$.
        
        With these two observations we can show that the functor is full;
        let $A_\bullet \in \catdga$
                (for the proof it would be enough to take $A \in \Alg_k$),
        and $B \in \Alg_k$.
        We calculate
        \begin{align*}
            \Hom_{\Ho(\catdga)}(A_\bullet, B) &= \Hom_{\Ho(\catdga)}(\tilde A_\bullet, B)\\
            &= \operatorname{coeq}(\Hom_{\catdga}(\tilde A_\bullet, B) \rightrightarrows \Hom_{\catdga}(\tilde A_\bullet, B))\\
            &= \Hom_{\catdga}(\tilde A_\bullet, B)\\
            &= \Hom_{\Alg_k}(\H_0(A_\bullet), B)\\
        \end{align*}
        The second step is \cref{thm:maps-homotopy-cat} together with the observation that $B$
        is a path object for itself. Step three is then the observation that
        both maps in this coequaliser are simply the identity.
        
        Faithfulness follows because for $A \in \Alg_k$, we have $\H_0(A)=A$.
    \end{proof}
    
        The same result and proof hold for $\C^{\infty}$-rings and EFC-rings.
    
    \begin{remark}
        Geometrically, we can rephrase Lemma \ref{discretefull} as saying that
        given an affine scheme $X$ and a derived affine scheme $Y$, we have
        \[
            \Hom_{\Ho(DG^+\Aff)}(X,Y) \cong \Hom_{\Aff}(X,\pi^0Y);
        \]
         a similar statement holds for non-affine $X$ and $Y$.
    \end{remark}

\subsection{Derived tensor products (derived pullbacks and intersections)}
    
    \begin{definition}\label{def:derived-tensor}
        Let $A_\bullet,B_\bullet \in \catdga$. The \emph{graded tensor product}\index{tensor product!graded}
        $(A\ten_k B)_\bullet$ is defined by
        \[
            (A\ten_k B)_n = \bigoplus_{i+j=n} A_i \ten_k B_j
        \]
        with differential $\delta(aa\ten b) := \delta a   + (-1)^{\deg(a)} \delta b$,
        and multiplication $(a\ten b)\cdot (a'\ten b') := (-1)^{\deg(a')\deg(b)}(aa'\ten bb') $.
    \end{definition}
    \begin{lemma}
        The functor $\ten_k: \catdga \by \catdga \to \catdga$ is left Quillen, with right adjoint $A \mapsto (A,A)$.
    \end{lemma}
    \begin{proof}
        It is immediate to see that this is the correct right adjoint functor:
        \[
            \Hom_{\catdga}((A\ten_k B)_\bt, C_\bullet) \cong
            \Hom_{\catdga\by\catdga}((A_\bullet, B_\bullet), (C_\bullet, C_\bullet))
        \]
        This right adjoint is right Quillen as it clearly preserves fibrations
        and trivial fibrations. Thus, by \cref{lemma:left-quillen-right-quillen}
        the left adjoint is left Quillen.
    \end{proof}
    
    Section \ref{derivedfnsn}
    told us that a functor $F$ being left Quillen means
    that the left-derived functor $LF$ exists.
    \begin{definition}\label{tenoLdef}
        Define the \emph{derived graded tensor product}\index{tensor product!derived}\index{$\ten^{\oL}$|see {derived tensor product}}
        $\ten^{\oL}_k: \Ho(\catdga)\by \Ho(\catdga) \to \Ho(\catdga)$
        to be the left-derived functor of
        $\ten_k \co \catdga \by \catdga \to \catdga$.
    \end{definition}
    \begin{remark}
        Recall that our base ring $k$ can be any $\Q$-algebra, not only a field. Therefore
        this construction is less trivial that it might seem at first glance.
    \end{remark}
    

    From this, one could expect that one would need to take cofibrant replacements on
    both sides to calculate $\ten_k^\oL$, which could be really complicated. The following
    simplifying lemma shows that one gets away with much less.
  
    \begin{definition}\label{def:quasiflat}
        Given a cdga $A_\bullet \in \catdga$ and an $A_\bullet$-module $M_\bullet$
        in chain complexes, say that $M_\bullet$ is \emph{quasi-flat}\index{quasi-flat}
        if the underlying graded module $M_\ast$ is flat over the graded algebra
        underlying $A_\ast$.\footnote{We say ``quasi-flat'' rather than just ``flat'' to avoid a clash with \Cref{strongdef}.}
    \end{definition}
    \begin{definition}
        Let $A_\bullet \in \catdga$ and $X_\bullet \in \Ho(\catdga)$. We say $A_\bullet$ is \emph{a model}\index{model for a homotopy class} for $X_\bullet$
        if $A_\bullet \simeq X_\bullet$ are quasi-isomorphic (i.e. isomorphic in $\Ho(\catdga)$).

        Note that here $A_\bullet$ is defined up to isomorphisms while $X_\bullet$ is defined up to quasi-isomorphisms.
    \end{definition}
    \begin{lemma}\label{lem:quasiflatreplace}
        To calculate $(A\ten_k^\oL B)_\bullet$ it is enough to take
        a quasi-flat replacement of one of the two factors.
        In particular, if $A_\bullet$ is a complex of flat $k$-modules, then
        $(A \ten_k B)_\bullet$ is a model for $(A \ten_k^\oL B)_\bullet$.
    \end{lemma}
    \begin{proof}
        The assumptions imply that the $i^{\textit{th}}$ homology groups of the tensor product are
        simply
        \[
            H_i((A \ten_k B)_\bullet) = \Tor_i^k(A_\bullet, B_\bullet)
        \]    
        Now if $\tilde A_\bullet, \tilde B_\bullet$ are
        cofibrant replacements, they also satisfy the flatness condition, so we get
        \[
            H_i((A \ten_k^\oL B)_\bullet)
            = H_i((\tilde A \ten_k \tilde B)_\bullet)
            = \Tor_i^k(A_\bullet, B_\bullet)
        \]
        Therefore $(A \ten_k^\oL B)_\bullet \to (A \ten_k B)_\bullet$
        is a quasi-isomorphism.
    \end{proof}
    
    One can generalise this result by choosing an arbitrary base
    $R_\bullet \in \catdga$ instead of $k$. This just induces another grading but
    the proof goes through the same way:
    \begin{lemma}
        If $A_\bullet$ is quasi-flat  over $R_\bullet$, then
        $(A \ten_{R}B)_\bullet$ is a model for $(A \ten^\oL_{R}B)_\bullet$.
    \end{lemma}
    \begin{definition}
        In the opposite category we denote these as \emph{homotopy pullbacks}\index{homotopy pullback}, i.e. we write
        $X\times^h_Z Y := \Spec((A\ten^\oL_{C}B)_\bullet)$ where
        $X = \Spec(A_\bullet), Y= \Spec(B_\bullet), Z = \Spec(C_\bullet)$.    
    \end{definition}
    \begin{example}
        Consider the self-intersection $$\{0\}\by^h_{\bA^1}\{0\}$$ of the origin in the affine line,
        or equivalently look at $k\ten^\oL_{k[t]}k$.
        There is a quasi-flat (in fact cofibrant) resolution of $k$ over $k[t]$
        given by
        $(k[t]\cdot s \to k[t])$ with $\delta s = 1$. In other words, this is the graded
        algebra $k[t,s]$ with $\deg(t) = 0$, $\deg(s)=1$ and $\delta s=1$
        (and since we are in a commutative setting we automatically have $s^2=0$).
         We calculate
        \[
            k[t,s]\ten_{k[t]} k = k[s]
        \]
        where $\deg(s)=1$ and $\delta s =0$.
        
        The underived intersection corresponds to an underived tensor product, taking $\H_0(-)$ of this to just give $k$, corresponding to $\Spec(k) \cong \{0\}$. 
        On the other hand, the virtual number of points of our derived intersection scheme $\Spec k[s]$ is given by taking the Euler characteristic, giving $1-1=0$, so we can think of this as a negatively thickened point.
        
        It also makes sense to talk of the virtual dimension of an example such as this, informally given by taking the Euler characteristic of the generators. Since $s$ is in odd degree, the virtual dimension of $\Spec k[s]$ is $-1$, which is consistent with the usual  dimension rules for intersections since it is the derived intersection of codimension $1$ subschemes of a scheme of dimension $1$.

        These enumerative properties are instances of a general phenomenon, namely that properties which hold generically in the classical world tend to hold everywhere in the derived setting.
        \end{example}
  
    \begin{example}
        More generally, we can look at the derived intersection $\{a\} \by^h_{\bA^1} \{ 0\}$ = $\Spec(k \ten^{\oL}_{a,k[t],0} k)$.
   
        By \Cref{lem:quasiflatreplace}, to compute $k \ten^{\oL}_{k[t]} k$ we need to replace one of the copies of $k$ with a quasi-flat $k[t]$-algebra that is quasi-isomorphic to $k$. For this, consider the cdga $A_\bullet$ generated by variables $t,s$ with $\deg(t) = 0$ and $\deg(s)= 1$ and differential defined by $\delta s = t-a$. We have $A_0 = k[t]$ and $A_1 = k[t]s$ and $A_i = 0$ for $i > 1$. Thus the morphism $f\co k[t] \to A_\bullet$ is  quasi-flat, and is in fact cofibrant: $A_\bullet$ is free as a graded algebra over $k[t]$. Now we can compute the derived intersection 
        \begin{equation*}
            \{a\} \times_{\bA^1} \{0\} = \Spec(k[t]/(t-a)\ten^{\oL}_{k[t],0} k) = \Spec(A_\bullet \ten_{k[t],0} k)
            = \Spec(k[s], \delta s=-a).
        \end{equation*}
        When $a$ is a unit, this means the derived intersection is  quasi-isomorphic to $\Spec(0)=\emptyset$, but when $a=0$ we have $k[s]=k \oplus k.s$ with $\delta s=0$.
     
        The Euler characteristic of $k[s]$ is equal to zero, regardless of $\delta$. If we think of the Euler characteristic of a finite dimensional cdga as the (virtual) number of points, then this corresponds to our intuition for intersecting two randomly chosen points in $\bA^1$. 
     
        Contrast this with the classical intersection, which is not constant under small changes, since $\{a\} \by_{\bA^1} \{0\}$ is $\emptyset$ if $a \neq 0 $ and $\{0\}$ if $a = 0$. Our derived self-intersection is categorifying Serre's intersection numbers \cite{serreAlgebreLocale}.
    \end{example} 
    \begin{definition}\label{def:loopspace}
        The \emph{derived loop space}\index{loop space!derived}\index{LX@$\cL X$|see {derived loop space}} of $X\in \operatorname{DG^+\Aff}$ is
        $\cL X := X \times^h_{X\times X} X$, i.e. the pullback via the diagonal.\footnote{These loop spaces don't look like loop spaces
        in topology; the reason is that here the notion of equivalence is a
        completely different one.}
    \end{definition}
    \begin{example}
       Look at $\cL\bA^1 =\bA^1 \by_{\bA^1 \times \bA^1}^h \bA^1$,
       i.e. a self-intersection of a line in a plane.
       Equivalently, we are looking at
       $$(k[x,y]/(x-y)) \ten^\oL_{k[x,y]} (k[x,y]/(x-y)).$$
       
       A cofibrant replacement for $k[x,y]/(x-y)$ over $k[x,y]$ is given by
       $k[x,y,s]$ with $\deg(x)=\deg(y)=0$, $\deg(s)=1$ and $\delta s = x-y$.
       Then $\bA^1 \by_{\bA^1 \times \bA^1}^h \bA^1$ is $\Spec(k[x,s])$
       with $\deg(s)=1$ and $\delta s = x-x = 0$.
    \end{example}
    
    More generally, what happens if we take the loop space $X \by_{X\by X}^h X$ in $DG^+\Aff$?

    \begin{example}\label{ex:loop}
        For any smooth affine scheme $X$ of dimension $d$, 
        we can calculate $\cL X$ as
        \[
            \cL X= \Spec (\O_X \xla{0} \Omega^1_X \xla{0} \Omega^2_X \ldots \xla{0} \Omega^d_X).       
        \]
        
        This is  a strengthening of the  HKR 
        isomorphism relating Hochschild homology and differential forms; for more details and generalisations, see e.g. \cite{benzvinadlerlooplanglands,TVchern}, which were inspired by precursors in the supergeometry literature, where the right-hand side corresponds to $\Pi TX=\Map(\R^{0|1},X)$, as in \cite[Lectures 4 \& 5]{Kon} or \cite[\S 7]{kontsevichPoisson}.
    \end{example}
    
    \begin{remark} 
        The cotangent complex $\bL^{A/k}$ of \S \ref{cotsn} gives a generalisation of \Cref{ex:loop} to all cdgas $A_\bullet$, with $(A\ten^{\oL}_{(A\ten^{\oL}_k A)} A)_\bullet \simeq \bigoplus_p \L^p\bL^{A/k}[p]$. The easiest way to prove this is to observe that the functors have derived right adjoints sending $B_\bullet$ to  $(B\by^h_{(B \by B)}B)_\bullet$ and $B_\bullet \oplus B_\bullet[1]$ respectively; for $PB_\bullet$ as in \Cref{dgalgpath}, inclusion of constants then gives a quasi-isomorphism $B_\bullet \oplus B_\bullet[1] \to (PB\by_{(B\by B)} B)_\bullet \simeq (B \by^h_{(B \by B)}B)_\bullet$.
    \end{remark}
    
\subsubsection{Analogues in differential and analytic contexts}

    There are analogues of derived tensor products for the $\C^\infty$-case and EFC-case; one needs to tweak things
    slightly but not very much.
    
    The basic problem is that the abstract tensor product of two rings of smooth or analytic
    functions won't be a ring of smooth or analytic functions. So there are
    $\C^\infty$ and EFC tensor products $\odot$\index{tensor product!$\C^\infty$}\index{$\odot$|see {$\C^\infty$ tensor product}} as in \Cref{ex:odot}, satisfying
    \[
        \C^\infty(X) \;\odot\; \C^\infty(Y) = \C^\infty(X\times Y)
    \]
    and similarly for EFC-rings. To extend these to dg-rings, we set
    \[
        A \odot B :=
        A \ten_{A_0} (A_0 \odot B_0) \ten_{B_0} B;
    \]
    for example 
    \[
        \C^\infty(X)[s_1, s_2,\ldots] \;\odot\; \C^\infty(Y)[t_1, t_2, \ldots]
        = \C^\infty(X \times Y)[s_1, t_1, s_2, t_2, \ldots]
    \]
    %
    %
    There are similar expressions for EFC rings, and indeed for any Fermat theory in the sense of \cite{CarchediRoytenberg,DubucKock}).
    
\subsection{Tangent and obstruction spaces}\label{tgtsn1}\label{obssn}
  
    An area where derived techniques are particularly useful is obstruction theory.
    To begin with, we recall the dual numbers and how they give rise to tangent spaces.

    \begin{definition}
        For any commutative ring $R$ we define the \emph{dual numbers}\index{dual numbers}\index{Repsilon@$R[\eps]$|see {dual numbers}} 
       $R[\eps]$  by 
        setting $\deg(\eps)=0$ and $\eps^2=0$,
        so $R[\eps]=R \oplus R\eps $.
    \end{definition}     
    \begin{remark}
        Note that this is naturally a $\C^{\infty}$-ring when $R:=\R$ and
        an EFC-ring when $R:=\Cx$, since
        \[
            \C^{\infty}(\R)/(t^2)\cong \R[\eps], \qquad \sO^{\hol}(\Cx)/(z^2) \cong \Cx[\eps].
        \]
        for co-ordinates $t$ on $\R$ and $z$ on $\Cx$.
    \end{remark}
    \begin{construction}
        If $X$ is a smooth scheme, a $\C^{\infty}$-space (e.g. a manifold)
        or a complex analytic space (e.g. a complex manifold), then
        maps $\Spec k[\eps] \to X$ correspond to \emph{tangent vectors}\index{tangent vectors}. That means
        \[
            X(k[\eps]) \cong \{(x,v)~:~ x \in X(k), ~v \text{ a tangent vector at } x\}.
        \]
        i.e. the set of $k[\eps]$-valued points forms a \emph{tangent space}\index{tangent space}.

        More generally, for any ring $A$ and any $A$-module $I$, we have that $X(A \oplus I)$
        consists of $I$-valued tangent vectors at $A$-valued points
        of $X$.
        In this construction the ring $A\oplus I$ has multiplication determined by
        setting $I \cdot I=0$.

    \end{construction}    
    \begin{definition}
        A \emph{square-zero extension}\index{square-zero extension} of commutative rings is
        a surjective map $f \co A \onto B$ such that
        $xy=0$ for all $x,y \in \ker(f)$.
    \end{definition}
    \begin{notation}
        For the rest of \cref{tgtsn1} we pick two commutative rings $A$ and $B$, a square-zero extension $f: A\onto B$ 
        and we define $I:=\ker(f)$.
    \end{notation}
    
    Note that any \emph{nilpotent surjection}\index{nilpotent surjection} of rings can be written as a composite
    of finitely many square-zero extensions, which is why deformation theory
    focuses on the latter.
    There is a way of thinking about square-zero extensions in terms of torsors.
    Note that 
    \begin{align*}
        A\by_BA &\cong A\by_B(B \oplus I)\\
        (a, a') &\mapsto (a, (f(a), a-a'))
    \end{align*}
    which is a ring homomorphism.
    For a smooth scheme $X$ this means that 
    \begin{align*}
        X(A) \times_{X(B)} X(A) &\cong X(A\times_B A)\\
        &\cong X(A\times_B (B\oplus I))\\
        &\cong X(A) \times_{X(B)} X(B \oplus I)
    \end{align*}
    so we get $I$-valued tangent vectors (from the tangent space $X(B\oplus I)$) acting transitively on the fibres of $X(A) \to X(B)$.
    
    Note that $X(A) \to X(B)$ is only surjective for $X$ smooth (assuming finite type). Singularities in $X$ give obstructions to lifting $B$-valued points to $A$-valued points.  It had long been observed (since at least \cite{LS}) that obstruction spaces tend to exist, measuring this failure to lift.  Specifically, the image of $X(A)\to X(B)$ tends to be the vanishing locus of a section of some bundle over $X(B)$, known as the obstruction space.

    Here is an analogy with homological algebra. If $f:A^\bullet \onto B^\bullet$ is a surjective map of cochain complexes with kernel $I^{\bt}$, then in the derived category we have a map $B^\bullet\to I^\bullet[1]$ with homotopy kernel $A^\bullet$. If the image of $H^0(B^\bullet) \to H^0(I^\bullet[1]) = H^1(I^\bullet)$ is non-zero, it then gives an obstruction to lifting elements from $H^0(B^\bullet)$ to $H^0(A^\bullet)$.
    
    Now we want to construct a non-abelian version of this, leading to the miracle of derived deformation theory: that tangent spaces are obstruction spaces.\index{miracle  of derived deformation theory} This accounts for the well-known phenomenon that when a tangent space is given by a cohomology group, the  natural obstruction space tends to be the next group up. 
    
    Almost everything we have seen in the lectures  so far is essentially due to Quillen. However, the first instance of our next argument is apparently  \cite[proof of Theorem 3.1, step 3]{Man2}, 
    although its consequences already featured in \cite[III 1.1.7]{Ill1}, with a more indirect proof.
 
    \medskip
    
    We start with an analogue  of the homological construction above. Given a square-zero extension\footnote{For simplicity, you can assume  that $A$ and $B$ are commutative rings, but exactly the same argument holds for cdgas and for (dg) $\C^{\infty}$ or EFC rings.}
    $A \onto B$  with kernel $I$, let $\tilde B_\bullet:=\cone(I\to A)$, i.e. $\tilde B_\bullet = (A \hookleftarrow I \leftarrow 0 \leftarrow \ldots) \in \catdga$; the multiplication on $\tilde B_\bullet $ is the obvious one. There is a natural quasi-isomorphism $\tilde B_\bullet \to B$. 
    
    Now we have a cdga map $u:\tilde B_\bullet \to (B \xleftarrow{0} I \leftarrow 0 \leftarrow \ldots) =: B \oplus I[1]$ where we just kill the image of $I$.\footnote{This is where we need $I$ to be square-zero; otherwise, the map would not be multiplicative.}
    Observe that $u$ is surjective and that
    \[
        \tilde B_\bullet \times_{u, (B\otimes I[1]), 0} B = A
    \]
    which gives us 
    \begin{equation*}
        A = \tilde B_\bullet \times^h_{B\otimes I[1]} B \in \catdga. \tag{$\dagger$}
    \end{equation*}

    For a sufficiently nice functor on $\catdga$, we can use  this to generate obstructions to lifting elements. The first functors we can look at are representable functors on the homotopy category $\Ho(\catdga)$, i.e. $\Hom_{\Ho(\catdga)}(S_\bullet,-)$ for cdgas $S_\bullet$, the functors associated to derived affine schemes.

    Limits in the homotopy category tend not to exist, but we do have homotopy fibre products, which have a weak limit property and permit the following definition (c.f. \cite{heller}).

    \begin{definition}\label{halfexact}
        A functor $F: \Ho(\catdga) \to \Set$
        is \emph{half-exact}\index{half-exact}
        if we have
        \begin{enumerate}
            \item $F(0) \cong \ast$,
            \item $F((A \by B)_\bullet) \cong F(A_\bullet) \by F(B_\bullet)$ for any
        $A_\bullet, B_\bullet \in \catdga$, 
            \item $F((A\by_{B}^h C)_\bullet) \onto F(A_\bullet)\by_{F(B_\bullet)}F(C_\bullet)$ for all diagrams $A_\bt \to B_\bt \la C_\bt$ in $\catdga$. 
        \end{enumerate}
    \end{definition}   
     \begin{remark}
         If restricting to Artinian objects, readers may notice the similarity of the  half-exactness property to Schlessinger's conditions \index{Schlessinger's conditions}\cite{Sch} in the underived setting (also see \cite{descent,Artin}), and to Manetti's characterisation of extended deformation functors in \cite{Man2}.
     \end{remark}
     
    \begin{lemma}
        Any representable functor $F$  on $\Ho(\catdga)$ is half-exact.
    \end{lemma}
    \begin{proof}[Proof (sketch)]
        The reason for this is that $\Hom_{\Ho(\catdga)}(S_\bullet,-)$ is given by path components $\pi_0$ of a topological space-valued functor $\oR\Map_{\catdga}(S_\bullet,-)$, with the latter preserving homotopy limits. The first two 
        properties then follow quickly, with the final property following by noting that if we take a homotopy fibre product of spaces, then its path components map surjectively onto the fibre product of the path components:
        \[
            \pi_0(X\by^h_Y Z) \onto \pi_0(X)\by_{\pi_0(Y)}\pi_0(Z).\qedhere
        \]
   \end{proof}

    %
        

    \begin{remark}
        It will turn out that non-affine geometric objects such as derived schemes and stacks $F$ still satisfy a weakened  half-exactness property, with the  final condition of Definition \ref{halfexact} only holding when $A_\bt \onto B_\bt$ is a nilpotent surjection, which is all we will need for the consequences  in this section to hold.   
    \end{remark}   

    \begin{construction}[Obstruction theory]\label{obstrnth}
    Returning to the obstruction question, if we  apply a half-exact functor $F$ to our square-zero extension $A \onto B$, then the expression ($\dagger$) gives
    \[
        F(A) \onto F(\tilde B_\bullet) \times_{u, F(B\oplus I[1]), 0} F(B)
        \cong F(B) \times_{u, F(B\oplus I[1]), 0} F(B),
    \]
    so the theory has given us a map $u: F(B) \to F(B\oplus I[1])$ such that 
    \[
        u(x) = (x,0)  \quad \text{if and only if}  \quad x \in \im(F(A) \to F(B)).
    \]
    In other words, by working over $\catdga$, we have acquired an obstruction theory  $$(F(B\oplus I[1]),u)$$ for free. In contrast to classical deformation theory, this means obstruction spaces  exist automatically in derived deformation theory.
\end{construction}
  
    \begin{remark}
        Whereas $F(B \oplus I)$ is a tangent space over $F(B)$, we think of $F(B\oplus I[1])$ as a higher degree   tangent space. In due course, we'll work with tangent complexes instead of tangent spaces, and  this then becomes the first cohomology group $\H^1$.
    \end{remark}  

\subsection{Postnikov towers}\label{postnikovsn}
    Pick for this entire section a cdga $A_\bullet \in \catdga$. Postnikov towers will give us the  justification for thinking of derived structure as being infinitesimal.
    
    \begin{notation}
        We recall the notations $\b_nA := \im(\delta:A_{n+1} \to A_n)$ for the image of the differential and $\z_nA := \ker(\delta:A_n \to A_{n-1})$ for the kernel. \index{Bn@$\b_n$}\index{Zn@$\z_n$}
        In particular we have that $\b_nA \cong A_{n+1}/\z_{n+1}A$ and $H_n(A_\bullet) = \z_n A / \b_n A$.
    \end{notation}

    \begin{definition}\label{coskdef}
        The \emph{$n^\textit{th}$ coskeleton}\index{coskeleton $\cosk$}\index{cosknA@$(\cosk_n A)_\bullet$|see {coskeleton}} $(\cosk_n A)_\bullet \in \catdga$ of $A_\bullet$ is given by
        \[
            (\cosk_nA)_i= 
            \begin{cases}
                A_i     & i < n+1\\
                \z_nA   & i = n+1\\
                0       & i > n+1
            \end{cases}
        \] 
        with the differential in degrees $i<n$ being the differential of $A_\bullet$ (i.e. $\delta_{(\cosk_nA)} = \delta_A: A_{i+1} \to A_i$) and the differential $\delta_{(\cosk_nA)}: (\cosk_nA)_{n+1} \to (\cosk_nA)_n $ in degree $n$ being  given by the inclusion $\z_nA \hookrightarrow A_n$. The multiplication on $(\cosk_n A)_\bullet$ is given by
        \[
            a \cdot b =
            \begin{cases}
                ab      & \deg(a)+\deg(b) < n+1\\
                \delta_A(ab)   & \deg(a)+\deg(b) = n+1\\
                0       & \deg(a)+\deg(b) > n+1
            \end{cases}
        \]
        The canonical map $A_\bullet \to (\cosk_n A)_\bullet$ is given in degree $n+1$ by $\delta_A:A_{n+1} \to \z_nA$ and by the identity in degrees $\le n$.
    \end{definition}
    
    \begin{remark}
        The idea of coskeleta is to give quotients truncating $A_\bullet$ without changing its lower homology groups, i.e. $\H_i((\cosk_n A)_\bullet) = \H_i(A_\bullet)$ for $i< n$ and $\H_i((\cosk_n A)_\bullet) = 0$ for $i\ge n$.
    \end{remark}

    
    The following gives an adjoint characterisation of the coskeleton:
    
    \begin{lemma}
        $\Hom_{\catdga}(A_\bullet, (\cosk_nB)_\bullet) \cong \Hom_{\catdga}((A_{\le n})_\bullet,(B_{\le n})_\bullet)$, where $(A_{\le n})_\bullet$ is the \emph{brutal truncation}\index{brutal truncation}\index{An@$(A_{\le n})_\bullet$|see {brutal truncation}} in degrees $\le n$.
    \end{lemma}

 The  brutal truncation functor from $\catdga$ to its subcategory of objects  concentrated in degrees $[0,n]$ also has a left adjoint, the $n$-skeleton\index{nskeleton@$n$-skeleton}. There are analogous adjunctions for simplicial algebras and sets, where the skeleton/coskeleton terminology is more common.
    
    \begin{definition}
        Let $A_\bullet \in \catdga$. The \emph{Moore-Postnikov tower}\index{Moore-Postnikov tower} is the family of cdgas $\{(P_nA)_\bullet\}_{n\in \bN}$ given by $(P_nA)_\bullet= \im(A_\bullet \to (\cosk_nA)_\bullet) = \im((\cosk_{n+1}A)_\bullet \to (\cosk_nA)_\bullet) \in \catdga$, so 
        \[
            (P_nA)_i =
            \begin{cases}
                A_i     & i \le n\\
                \b_nA   & i = n+1\\
                0       & i> n+1.
            \end{cases}
        \]
        These form a diagram with maps
        \[
            A_\bullet \to \ldots \to (P_nA)_\bullet \to (P_{n-1}A)_\bullet \to \ldots \to (P_0A)_\bullet.
        \]
    \end{definition}
    \begin{lemma}\label{postfact}
        The morphism $(P_nA)_\bullet \to (P_{n-1}A)_\bullet$ is the composition of a trivial fibration and a square-zero extension.
    \end{lemma}
    \begin{proof}
        Define $C_\bullet \in \catdga$ by 
        \[
            C_i :=
            \begin{cases}
                A_i         & i < n\\
                A_n/\b_nA   & i = n\\
                0           & i> n,
            \end{cases}
        \]
        and note that the map $(P_nA)_\bullet \to (P_{n-1}A)_\bullet$ factors as $(P_nA)_\bullet \to C_\bullet \to (P_{n-1}A)_\bullet$, with $(P_nA)_\bullet \to C_\bullet$ a trivial fibration, and $C_\bullet \to (P_{n-1}A)_\bullet$ a square-zero extension (with kernel $(\H_n(A_\bullet))[-n]$).
    \end{proof}
    \begin{remark}
        We can thus think of $\Spec(A_\bullet)$ as like a formal infinitesimal neighbourhood of $\Spec(\H_0(A_\bullet))$, since we have characterised it as a direct limit of a sequence of homotopy square-zero thickenings.
    \end{remark}
    
    Assuming some finiteness conditions, we now strengthen these results, relating $A_\bullet$ to a genuine completion over $\H_0(A_\bullet)$.  

    \begin{definition}
        Let $A_\bullet \in \catdga$. The \emph{completion}\index{completion}\index{Ahat@$\hat A_\bullet$|see {completion}} of $A_\bullet$
        is given by \index{completion of a cdga}
        \[
            \hat A_\bullet := \varprojlim_n A_\bullet/I^nA_\bullet
        \]
        where $I:= \ker(A_0 \to \H_0(A_\bullet))$ and $I^n=\underbrace{I\cdot I \cdots  I}_n$.
    \end{definition}
    \begin{lemma}\label{dgshrink}
        If $A_0$ is Noetherian and each $A_n$ is a finite $A_0$-module, then the natural map $A_\bullet \to \hat{A_\bullet}$ is a quasi-isomorphism.
    \end{lemma}
    \begin{proof}
        This is \cite[Lemma \ref{stacks2-dgshrink}]{stacks2}, proved using fairly standard commutative algebra.
        If $A_0$ is Noetherian, then \cite[Thm.~8.8]{Mat} implies that $A_0 \to \hat{A}_0$ is flat. If $A_n$ is a finite $A_0$-module, then \cite[Thm~8.7]{Mat} implies that $\hat{A}_n = \hat{A}_0 \ten_{A_0}A_n$. Thus 
        \[
            \H_\ast(\hat{A}_\bullet)\cong \H_*(A_\bullet)\ten_{A_0}\hat{A}_0,
        \]
        and applying \cite[Thm~8.7]{Mat} to the $A_0$-module $\H_0(A_\bullet)$ gives that
        $\H_*(\hat{A}_\bullet)\cong \H_*(A_\bullet)$, as required. 
    \end{proof}
    
    

    
\subsection{The cotangent complex}\label{cotsn}
    
    The cotangent complex is one of the earliest applications of abstract homotopy theory, due to  Quillen \cite{Q}\footnote{The two manuscripts with the greatest influence on derived geometry are probably \cite{Q} and \cite{Kon}, though practitioners tend to encounter their contents indirectly.}, using \cite{QHA}. Until then, tangent and obstruction spaces for relative extensions only fitted in the nine-term long exact sequence of \cite{LS}. For more history, see \cite{barrearlycoho}.

    \begin{definition}
        Given a morphism $R_\bullet \to A_\bullet$ in $\catdga$, the complex $\Omega^1_{A/R} \in \catdgm$ of \emph{K\"ahler differentials}\index{K\"ahler differentials}\index{Omega1@$\Omega^1_{A/R}$|see {K\"ahler differentials}} is given by $I/I^2$, where $I = \ker((A\ten_{
        R} A)_\bullet \to A_\bullet)$.

        There is a \emph{derivation}\index{derivation} $d \co A_\bullet \to \Omega^1_{A/R}$ given by $a \mapsto a\ten 1-1\ten a +I^2$.

    \end{definition}    

    \begin{example}
        If $A_\bullet=(R_{\bt}[x_1, \ldots, x_n], \delta)$ for variables $x_i$ in various degrees, then $\Omega^1_{A/R} = (\bigoplus_{i=1}^n (A_\bt).dx_i, \delta)$.
    \end{example}

    \begin{definition}\label{cotdef}
        Given a morphism $R_\bullet \to A_\bullet$ in $\catdga$, the 
        \emph{cotangent complex}\index{cotangent complex}\index{LAR@$\bL^{A/R}$|see {cotangent complex}} is defined as
        \[
            \bL^{A/R}:= (\Omega^1_{\tilde{A}/R}\ten_{\tilde{A}}A)_\bullet \in \catdgm[A_\bt]
        \]
        where $\tilde{A}_\bullet \to A_\bullet$ is a cofibrant replacement in $\catdga[R_\bullet]$.          
    \end{definition}
    
    The idea behind the cotangent complex is that we want to take the left-derived functor of $A_\bullet \mapsto \Omega^1_{A/R}$,
    but this isn't a functor as such, since the codomain depends on $A_\bullet$.
    
    Instead, we take the slice category $(\catdga[R_\bullet]) \da A_\bullet$ of $A_\bullet$-augmented $R_\bullet$-algebras, and 
    look at the functor $B_\bullet \mapsto (\Omega^{B/R}\ten_{B}A)_\bullet$ from $(\catdga[R_\bullet]) \da A_\bullet$ to $\catdgm[A_\bt]$; this  is left adjoint to the functor $M_\bullet \mapsto A_\bullet \oplus (M_\bullet).\eps$, for $\eps^2=0$. These form a Quillen pair, and taking the left-derived functor gives the cotangent complex  $\bL^{A/R}:= \oL(C \mapsto (\Omega^1_{C/R}\ten_{C}A)_\bullet)(A_\bullet)$, which we calculate by evaluating the functor on a cofibrant replacement $\tilde{A}_\bullet \to A_\bullet$.

    \begin{remark}
        Note that $\H_0(\bL^{A/R}) = \Omega_{\H_0(A_\bullet)/\H_0(R_\bullet)}$.
    \end{remark}
    
    \begin{lemma} The cotangent complex
        $\bL^{A/R}$ can be calculated by letting $J := \ker((\tilde{A}\ten_{R} A)_\bullet \to A_\bullet)$, and setting $\bL^{A/R}:= J/J^2$. 
    \end{lemma}
    
    \begin{remark}
        It follows from results below (see \Cref{rem:qusmooth}) that in Definition \ref{cotdef} we can relax the condition on  $\tilde{A}_\bullet$ to  just require that $R_0\to \tilde{A}_0$ be  ind-smooth (i.e. a filtered colimit of smooth morphisms) and that $\tilde{A}_\bullet$ be cofibrant over $(R\ten_{R_0}\tilde{A}_0)_\bt$ (i.e. underlying graded freely generated by a graded projective module).
    \end{remark}
    
    Also note that the functor $-\ten_{\tilde{A}}A \co \catdgm[\tilde{A}_\bullet]\to \catdgm[A_\bt]$ is a left Quillen equivalence, and in particular that $\Omega^1_{\tilde{A}/R}\to  (\Omega^1_{\tilde{A}/R}\ten_{\tilde{A}}A)_\bullet$ is a quasi-isomorphism of $\tilde{A}_\bullet$-modules, but beware that $\Omega^1_{\tilde{A}/R}$ itself is not an $A_\bullet$-module.
    
    \begin{definition}
        \emph{Andr\'e--Quillen}  \emph{cohomology}\index{cohomology!Andr\'e--Quillen} \index{cohomology!Harrison} (or Harrison cohomology --- they agree in characteristic $0$) is defined to be  $D^i_{R_\bullet}(A_\bullet,M_\bullet):=\EExt^i_{A_\bullet}(\bL^{A/R},M_\bullet)$.\index{Andr\'e--Quillen cohomology!$D^*(A,M)$}
    \end{definition} 

    In interpreting this, note that $\Hom_{\catdgm[A_\bullet]}(\Omega^1_{A/R},M_\bullet)$ consists of $R_\bullet$-linear derivations from $A_\bullet$ to $M_\bullet$.

    The homotopy fibre of $\oR\Map_{\catdga[R_\bt]}(S_\bullet,B_\bullet \oplus M_\bullet) \to \oR\Map_{\catdga[R_\bt]}(S_\bullet,B_\bullet)$ over $f \co S_\bt \to B_\bt$ has homotopy groups $\pi_i = D^i_{R_\bullet}(S_\bullet,f_*M_\bullet)$, where $f_*M_\bt$ is the chain complex $M_\bt$ equipped with the $S_\bt$-module structure induced by $f$.
    In particular, the $I$-valued  obstruction space (\Cref{obstrnth}) of $\Hom_{\Ho(\catdga[R_\bt])}(S_\bt,-)$  is $D^0_{R_\bullet}(S_\bullet,I[1])=D^1_{R_\bullet}(S_\bullet,I)$.

    \begin{theorem}[Quillen]\label{smoothcotthm} 
        If $R \to S$ is  a smooth morphism of $k$-algebras (concentrated in degree $0$), then $\bL^{S/R} \simeq \Omega^1_{S/R}$.
  
        Moreover, if $T=S/I$, for $I$ an  ideal generated by  a regular sequence $a_1, a_2, \ldots$, then  
        \[
            \bL^{T/R} \simeq \cone(I/I^2 \to  \Omega^1_{S/R}\ten_ST).
        \]
    \end{theorem}

    Before proving the theorem, we first state a key lemma from \cite{Q}, which follows from universal properties of derived functors.
    
    \begin{lemma}\label{cotses}
        Given morphisms $A_\bullet \to B_\bullet \to C_\bullet$ in $\catdga$, we have an exact triangle of dgas
        $$
        \bL^{C/B}[-1]  \to   (\bL^{B/A}\ten_{B}C)_\bullet \to \bL^{C/A} \to \bL^{C/B}. 
        $$
    \end{lemma}
    \begin{proof}[Sketch proof of theorem]\ 
        \begin{enumerate}
            \item If $S=R[x_1, \ldots, x_n]$, then it is cofibrant over $R$, so the conclusion holds.
            \item Next, reduce to the \'etale case (smooth of relative dimension $0$). A smooth morphism is \'etale locally affine space, in the sense that $\Spec S$ admits an \'etale  cover by  affine schemes $U$ admitting factorisations of the form:
                \[
                \begin{CD} U @>f>{\text{\'etale}}> \Spec S \\ 
                    @VgV{\text{\'etale} }V  @VVV\\
                    \bA^n_R  @>>> \Spec R.
                \end{CD}
                \]
                If the conclusion of the theorem holds for  \'etale morphisms, then the complexes $\bL^{U/S}$ and $\bL^{U/\bA^n_R}$ are both quasi-isomorphic to $0$, so    the lemma gives $f^*\bL^{S/R}\simeq \bL^{U/R}$ and $\bL^{U/R}\simeq g^*\bL^{\bA^n_R/R}\simeq \Omega^1_{U/R}$.
            
                Thus the map $\bL^{S/R} \to \Omega^1_{S/R}$ is a quasi-isomorphism \'etale locally, so must be a quasi-isomorphism globally.
            \item Now, reduce to open immersions. If $U \to Y$ is an \'etale map of affine schemes, then the relative diagonal
                \[
                    \Delta \co  U \to U\by_YU
                \]
                is an open immersion. If the conclusion holds for open immersions, then $\bL^{(U\by_YU)/U}\simeq 0$, so   Lemma \ref{cotses} gives
                \[
                    \Delta^*\bL^{(U\by_YU)/Y} \cong \bL^{U/Y}.
                \]
                But $\bL^{(U\by_YU)/Y} \cong \pr_1^*\bL^{U/Y} \oplus \pr_2^*\bL^{U/Y}$, so we would then have
                \[
                    \bL^{U/Y} \oplus \bL^{U/Y}\cong \bL^{U/Y},
                \]
                and thus $\bL^{U/Y} \simeq 0 = \Omega^1_{U/Y}$.
                
            \item Every  open immersion is given by repeated composition and pullback of the open
                immersion $\Spec(k[x,x^{-1}]) \to \Spec(k[x])$, so to prove the statement for smooth morphisms, it suffices to prove the theorem for this one morphism.
            \item Abstract nonsense has taken us this far,
                but now we have to dirty our hands. A cofibrant replacement $\tilde{A}_\bullet$ for $A := k[x,x^{-1}]$ over $B := k[x]$ is given by $k[x,y,t]$ with $\delta (t) = xy-1$, for $t$ of degree $1$ and $y$ of degree $0$, i.e. 
                \[
                 \tilde{A}_\bullet = (  k[x,y] \xla{\delta} k[x,y]t).
                \]
                
                Then $\Omega^1_{\tilde{A}/B} = (\tilde{A}_\bullet).dy \oplus (\tilde{A}_\bullet).dt$, with $\delta(dt)= x \cdot dy$. Thus $\bL^{A/B} \simeq (A.dy \oplus A.dt; ~ \delta(dt) = x \cdot dy)$. Since $x \in A$ is a unit, this gives $\bL^{A/B} \simeq 0 =\Omega^1_{A/B}$, completing the proof for smooth algebras.
                
            \item For the statement on regular sequences, we observe that a cofibrant
                replacement $\tilde{T}_\bullet$ for $T=S/(a_1, a_2, \ldots)$ over $S$ is given by $(S[t_1, t_2, \ldots],\delta)$ with $\delta(t_i) =a_i$, for $T_i$ of degree $1$; this is effectively a Koszul complex calculation. Then
                \[
                    \Omega^1_{\tilde{T}_\bullet/S}\cong (\bigoplus_i (\tilde{T}_\bullet).dt_i, \delta), 
                \]
                and $\bL^{T/S} \simeq (\Omega^1_{\tilde{T}/S}\ten_{\tilde{T}}T)_\bt \cong \bigoplus_i T.dt_i \cong (I/I^2)[1]$.
                
     Now  $\Omega^1_{\tilde{T}/R}\ten_{\tilde{T}}T \cong \cone(I/I^2 \to  \Omega^1_{S/R}\ten_ST)$, our desired expression. If we calculate $\bL^{T/R}$ by taking a cofibrant replacement of $\tilde{T}$ over $R$, we get a natural map $\alpha \co \bL^{T/R} \to  \Omega^1_{\tilde{T}/R}\ten_{\tilde{T}}T$. The calculations above and Lemma \ref{cotses} give both terms as  homotopy  extensions  of  $(I/I^2)[1]$ by $ \Omega^1_{S/R}\ten_ST$, so $\alpha$ is indeed a quasi-isomorphism. \qedhere
            \end{enumerate}
        \end{proof}
    
    \begin{remark}\label{rem:qusmooth}
    As a consequence, in general, we don't need cofibrant replacement to calculate $\bL^{A/R}$, it suffices for $R_\bullet \to A_\bullet$ to be the composite of a cofibration and a  smooth morphism.
    
    The natural name for this concept, as used for instance in \cite{Kon,Man2,ddt1} is \emph{quasi-smoothness}\index{quasi-smooth}, and it was simply called smoothness in \cite{Quot,Hilb}. However, quasi-smooth is more commonly used in the later DAG literature (apparently originating with \cite{toenseattle}) to mean virtually LCI\index{virtually LCI} in the sense that the cotangent complex is generated in degrees $0,1$, so the term is unfortunately now best avoided altogether despite its utility. 
    
    Both usages have their roots in the \emph{hidden smoothness}\index{hidden smoothness} philosophy
    of \cite{Kon2} and \cite[Lecture 27]{Kon}, with the motivating examples from the former (but {\it not} the latter) being virtually LCI as well as quasi-smooth in the original sense.
    \end{remark}
    
    \begin{remark}
        Analogues in differential and analytic settings work in much the same way for all the results in this section, giving $\Omega^1_{A/R}$ as
        the module of smooth or analytic differentials. The definition just uses the analytic or $\C^{\infty}$ tensor product $\odot$ instead of $\ten$.\footnote{In fact, cotangent modules were formulated in \cite{Q} for arbitrary algebraic theories, taking values in Beck modules \cite{beckThesis}.} 
        For instance, $\Omega^1_{\C^{\infty}(\R^n)}$ is given by $\bigoplus_i \C^{\infty}(\R^n)dx_i$. 
    
        The proof of the last theorem also works much the same: simpler in the differential setting, 
        but harder in the analytic setting.\footnote{In the $\C^{\infty}$ setting, the final step of the proof is to analyse 
        the restriction map from smooth functions on the line to smooth functions on an open interval,
        but the latter is already cofibrant. The analytic setting instead looks at the morphism from analytic functions on the plane to analytic functions on an open disc; calculating the latter's cotangent complex relies on its characterisation as a domain of holomorphy.} 
    In particular, the $\C^{\infty}$-cotangent complex of the $\C^{\infty}$-ring $\C^{\infty}(X,R)$ of smooth functions on a manifold $X$ is quasi-isomorphic to the module of K\"ahler $\C^{\infty}$-differentials of $\C^{\infty}(X,\R)$, which in turn  just consists  of global smooth $1$-forms on $X$. Similarly, the EFC-cotangent complex of the EFC ring of complex analytic functions on a Stein manifold  consists of its global analytic $1$-forms. 
    \end{remark}
    \begin{property}
        A map $f\co A_\bullet \to B_\bullet$ in $\catdga$ is a weak equivalence if and only if $\H_0(f)$ is an isomorphism and $(\bL^{B/A}\ten^{\oL}_{B}\H_0(B))_\bt \simeq 0$. The ``only if'' direction follows by definition; to prove the ``if'' direction, look at maps from both to arbitrary objects $C\in \Ho(\catdga)$, and use the Postnikov tower to break $C$ down into square-zero extensions over $\H_0(C)$. For details, see Lemma \ref{detectweak}.
    
        The same is true for dg $\C^{\infty}$-rings and dg EFC-rings, with exactly the same reasoning.
    \end{property}
    \begin{definition}\label{strongdef} (\cite{hag2})
        A morphism $f\co A_\bullet \to B_\bullet$ in $\catdga$ is \emph{strong}\index{strong morphism} if $\H_i(B_\bullet) \cong \H_i(A_\bullet)\ten_{\H_0(A_\bullet)}\H_0(B_\bullet)$. We then say a morphism is \emph{homotopy-}(\emph{flat}, resp. \emph{open immersion}, resp. \emph{\'etale}, resp. \emph{smooth})\index{homotopy-flat}\index{homotopy open immersion}\index{homotopy-\'etale}\index{homotopy-smooth} if it is strong and the morphism  $\H_0(A_\bullet) \to \H_0(B_\bullet)$ is is flat (resp. open immersion, resp. \'etale, resp. smooth).
    \end{definition}
    
    \cite[Def 1.2.7.1 and Theorem 2.2.2.6]{hag2}  then characterise homotopy-\'etale and homotopy-smooth morphisms as follows:
    \begin{align*}
        A_\bullet \to B_\bullet \text{ is homotopy-\'etale} &\iff \bL^{B/A}\simeq 0 \\
        A_\bullet \to B_\bullet \text{ is homotopy-smooth} &\iff (\bL^{B/A}\ten^{\oL}_{B}\H_0(B))_\bt\;\simeq\\
         &\phantom{\iff} \text{a projective } \H_0(B_\bullet)\text{-module in degree } 0. 
    \end{align*}
    \begin{remark}[Terminology]
        In \cite{hag2}, these properties are simply called flat, smooth, \'etale, etc., but we prefer to emphasise their homotopy-invariant nature and avoid potential confusion with notions such as Definition \ref{def:quasiflat}, or the smoothness of \cite{Quot,Hilb} (i.e. quasi-smoothness in the original sense of \Cref{rem:qusmooth}).
    \end{remark}
    
\begin{exercises}\label{strongex}\
 \begin{enumerate}
  \item If $f\co A_\bullet \to B_\bullet$ and $g \co  B_\bullet \to C_\bullet$ are morphisms in $\catdga$ such that $g \circ f$ and $f$ are both strong, then $g$ is also strong.
  
  \item The property of being homotopy-flat is closed under homotopy pushouts in $\catdga$, i.e. if  $f\co A_\bullet \to B_\bullet$ is a homotopy-flat morphism and $A_{\bt} \to A'_\bt$ any map, then the induced morphism $A'_\bt \to A'_\bt\ten^{\oL}_{A_\bt}B_\bt$ is also homotopy-flat. (Hint: use the algebraic Eilenberg--Moore spectral sequence $\Tor^{\H_*A}_p(\H_*A', \H_*B)_q \abuts \H_{p+q}(A'\ten_A^{\oL}B)$.)
  \end{enumerate}
\end{exercises}

    %
    
    \begin{remark}
        The cotangent complex functor $\bL$ can be constructed using functorial cofibrant replacements, so it sheafifies (Illusie \cite{Ill1, Ill2}).       
    \end{remark}

    \begin{lemma}\label{cotsheaflemma}
        For any morphism $f\co X \to Y$ of derived schemes, the presheaf $\bL^{X/Y}:=\bL^{\O_{X,\bt}/f^{-1}\O_{Y,\bt}}$ \index{cotangent complex}
        is a homotopy-Cartesian dg $\O_{X,\bt}$-module.
    \end{lemma}
    \begin{proof}
        For any inclusion $U \to V$ of open affines in $\pi^0X$, the map $\O_{X,\bt}(V) \to\O_{X,\bt}(U) $ is homotopy-open immersion, so $\bL^{\O_{X,\bt}(U)/\O_{X,\bt}(V) }\simeq 0$, and $\bL^{f^{-1}\O_{Y,\bt}(U)/f^{-1}\O_{Y,\bt}(V) }\simeq 0$ similarly. The exact triangle for $\bL$ (Lemma \ref{cotses}) thus gives $$\bL^{\O_{X,\bt}(U)/f^{-1}\O_{Y,\bt}(U)}\simeq  \O_{X,\bt}(U)\ten^{\oL}_{\O_{X,\bt}(V)}\bL^{\O_{X,\bt}(V)/f^{-1}\O_{Y,\bt}(V)},$$ as required.
    \end{proof}

    Although defined in terms of deformations of morphisms, the cotangent complex also governs deformations of objects:  
    \begin{lemma}
        Given $A_\bullet, B_\bullet \in \catdga$, $S_\bullet \in \catdga[B_\bullet]$, and a surjection $A_\bullet \onto B_\bullet$ with kernel $I_\bullet$, the potential obstruction to lifting $S_\bullet$ to a cdga $S'_\bullet \in \catdga[A_\bullet]$ with $(S'\ten^{\oL}_{A}B)_\bullet \simeq S_\bullet$ lies in $\Ext^2_{S_\bullet}(\bL^{S/B}, (S\ten_{B}^{\oL}I)_\bullet)$. If the obstruction vanishes, then the set of equivalence classes of lifts is a torsor for $\Ext^1_{S_\bullet}(\bL^{S/B}, (S\ten_{B}^{\oL}I)_\bullet)$.  
    \end{lemma} 
    \begin{proof}
        This is essentially contained in \cite[III 1.2.5]{Ill1}, but there is  a more direct proof based on \cite[Lectures 13--14]{Kon}.
        Without loss of generality, we may assume that $S_\bullet$ is quasi-free, since every quasi-isomorphism class contains quasi-free cdgas. Since free algebras don't deform, there is a free graded-commutative $A_\ast$-algebra $S'_\ast$ with $ S'_\ast\ten_{A_\ast}B_\ast\cong S_\ast$.
    
        The obstruction in $\Ext^2$  then comes from  lifting $\delta_S$ to a derivation $\delta_{S'}$ on $S'$ and looking at $(\delta_{S'})^2$, while the parametrisation in terms of $\Ext^1$ comes from different choices of lift $\delta_{S'}$. Most of the work is then  in checking quasi-isomorphism-invariance. For details of this argument and global generalisations, see \cite[\S 8.2]{stacks2}; also see \cite[\S 4]{hinichDefsHtpyAlg}.
    \end{proof}    
    
    %
    %
    %
    %
    
    %
    %
    %
    %
    %
    
\subsection{Derived de Rham cohomology}
    
    This originates in \cite[\S VIII.2]{Ill2}.
    We have  a functor from cdgas to double complexes (a.k.a. bicomplexes),\index{double complex} \index {bicomplex|see {double complex}} sending $A_\bullet$ to
    \begin{eqnarray*}
        \Omega^{\bt}_A:=&(  A_\bullet \xra{d} \Omega^1_A \xra{d} \Omega^2_A  \xra{d} \ldots)\\
        = &\left(\!\begin{gathered}\xymatrix{ \vdots \ar[d]^{\delta} & \vdots\ar[d]^{\delta}  & \vdots\ar[d]^{\delta} & \ \adots \\
        A_2 \ar[r]^d \ar[d]^{\delta} & \Omega^1_{A,2} \ar[r]^{d}\ar[d]^{\delta}  & \Omega^2_{A,2}  \ar[r]^{d}\ar[d]^{\delta}  & \ldots \\ 
        A_1 \ar[r]^d \ar[d]^{\delta} & \Omega^1_{A,1} \ar[r]^{d}\ar[d]^{\delta}  & \Omega^2_{A,1}  \ar[r]^{d}\ar[d]^{\delta}  & \ldots \\
        A_0 \ar[r]^d & \Omega^1_{A,0} \ar[r]^{d} & \Omega^2_{A,0}  \ar[r]^{d} & \ldots  }  \end{gathered} \right)
    \end{eqnarray*}
    
    where $\Omega^p_A:=\L^p_A \Omega^1_A$ is the \emph{alternating power}\index{alternating power}, taken in the graded sense. Beware that when $A_\bullet$ has terms of odd degree, alternating powers go on forever. 
    
    Our notion of weak equivalence for double complexes will be quasi-isomorphism on the columns, so 
    \begin{eqnarray*}
        (U^0 \xra{d} U^1 \xra{d} \ldots) \to  (V^0 \xra{d} V^1 \xra{d} \ldots)
    \end{eqnarray*}
    is an equivalence if $\H_*(U^i) \cong \H_*(V^i)$ for all $i$.
    
    The idea behind derived de Rham cohomology is to then take the left derived functor, giving the  double complex $\oL\Omega^{\bt}_A:= \Omega^{\bt}_{\tilde{A}}$, for a cofibrant replacement $\tilde{A}_\bullet$  of $A_\bullet$ (cofibrant over smooth as in \Cref{rem:qusmooth} suffices --- we just need $\Omega^1 \simeq \bL$). Note that $ \Omega^{p}_{\tilde{A}} \simeq \L^p_A \bL^{A/R}$ (just apply the left Quillen equivalence $(- \ten_{\tilde{A}}A)_\bullet$), but also that the de Rham differential $d$ does not descend to the latter objects. 
    
    Then we take the \emph{derived de Rham complex} \index{derived de Rham complex}  to be the  \emph{product} total complex\index{product total complex}
    $ \Tot^{\Pi}\oL\Omega^{\bt}_A$, i.e. $\Tot^{\Pi}(V)_i := (\prod_p V^p_{p+i}, \delta \pm d)$, where $\pm$ is the Koszul sign $(-1)^{p+i}$ on $V^p_{p+i}$. 
In fact, for our notion of weak equivalences, $\Tot^{\Pi}$ is just the right-derived functor of the functor $\z^0\co V \mapsto \ker(d \co V^0 \to V^1)$ on double complexes in non-negative cochain degrees; it preserves weak equivalences by \cite[\S 5.6]{W}.
    
    \begin{theorem}[\cite{Ill2} (with restrictions), \cite{FeiginTsygan} (omitting details, cf. \cite{emmanouil}),  \cite{BhattDerivedDR}]
        The cohomology groups 
        $\H^*(\pi^0X,  \Tot^{\Pi}\oL\Omega^{\bt}_X)$ are Hartshorne's algebraic de Rham cohomology groups \cite{HartshorneAlgDeRham}\footnote{The algebraic de Rham cohomology of $Z$ is defined by taking a closed embedding of $Z$ 
        in a smooth scheme $Y$, then looking at the completion $\hat{\Omega}^{\bt}_Y$ of the de Rham complex of $Y$ with respect to the ideal $\sI_Z$ and taking hypercohomology.}. In particular, these are the singular cohomology groups  $\H^*(X(\Cx)_{\an},\Cx)$ of $X(\Cx)$ with the analytic topology  when working over $\Cx$.
    \end{theorem}
    
    One proof proceeds by taking a cofibrant resolution and killing variables $x$  of non-zero degree, thus identifying $dx$ with $\pm \delta x$; this generates power series in $\delta x$ when $\deg x =1$, giving the comparison with \cite{HartshorneAlgDeRham}. 
    The same arguments work in differential and $\Cx$-analytic settings, giving equivalences with real and complex Betti cohomology respectively.
    
\subsubsection{Shifted symplectic structures}  \label{sympsn}  
    Any complex or  double complex admits a filtration by brutal truncation\index{brutal truncation}, i.e.
    \[
        F^p(V^0 \xra{d} V^1 \xra{d} \ldots ) = (0 \to  \ldots \to 0 \to V^p \xra{d} V^{p+1} \xra{d} \ldots);
    \]
    on the de Rham (double) complex, this is called the \emph{Hodge filtration}\index{Hodge filtration}. 
    Then $(\Tot^{\Pi}F^p)[p]$ calculates the
    right-derived functor $\oR \z^p$ of $\z^p\co V \mapsto \ker(V^p \to V^{p+1})$\index{Zp@$\z^p$|see {closed $p$-forms}}\index{closed $p$-forms},\footnote{Explicitly, we have quasi-isomorphisms $V^q \to W^q:=\cone(\Tot^{\Pi}F^{q+1}\to\Tot^{\Pi}F^q)[q]$ given by $v\mapsto (\pm dv,v)$, combining to give a columnwise quasi-isomorphism $V \to W$ of double complexes, with $\z^pW=(\Tot^{\Pi}F^pV)[p]$.} so  the 
    homologically
    correct analogue of closed $p$-forms is given by the complex  $(\Tot^{\Pi}F^p\oL\Omega^{\bt}_{A})[p]$. 
    
    \begin{example}
        Classically, when $X$ is a smooth scheme (in the algebraic setting) or a  manifold (in the $\C^{\infty}$ and analytic settings), then we just have $\oL\Omega^p_X \simeq \Omega^p_X $, and hence $(\Tot^{\Pi}F^p\oL\Omega^{\bt}_{X})[p]\simeq F^p\Omega^{\bt}_{X}[p]$.
        
        In  $\C^{\infty}$ and analytic settings, we can say more, because the Poincar\'e lemma implies that $F^p\Omega^{\bt}_{X}[p]$ is quasi-isomorphic to the sheaf $\z^p\Omega^{\bt}_{X}=\ker(d \co \Omega^p_X \to\Omega^{p+1}_X   )$ of   closed $p$-forms on $X$, so the derived constructions reduce to the na\"ive underived object.
        
        In algebraic settings, the sheaf $\z^p\Omega^{\bt}_{X}$ of closed algebraic $p$-forms on the Zariski site is poorly behaved, but the GAGA principle \cite{GAGA} applied to the graded pieces shows that for smooth proper complex varieties $X$, analytification gives a quasi-isomorphism
        \[
        \oR\Gamma(X, F^p\Omega^{\bt}_{X})[p] \simeq \oR\Gamma(X(\Cx)_{\an}, F^p\Omega^{\bt}_{X_{\an}})[p]\simeq \oR\Gamma(X(\Cx)_{\an}, \z^p\Omega^{\bt}_{X_{\an}}),
        \]
        and hence an isomorphism   between hypercohomology of the algebraic Hodge filtration and cohomology of closed analytic $p$-forms. 
        
        Thus even in the absence of derived structure, one immediately looks to the  Hodge filtration in algebraic geometry  when seeking to mimic closed forms in analytic  geometry. 
    \end{example}    
    
    Classically, a symplectic structure is a closed non-degenerate $2$-form. This notion was generalised in \cite{AKSZ}, which introduced the notion of a $QP$-manifold  as a $\Z/2$-graded dg-manifold equipped with a non-degenerate $2$-form of odd degree which is closed under both $d$ and $\delta$. Replacing $\z^2$ with $\oR\z^2$ leads to the following:
    
    \begin{definition}[\cite{KhudaverdianVoronov,bruceGeomObjects,PTVV}]
        The complex of    \emph{$n$-shifted pre-symplectic structures} \index{shifted pre-symplectic structure} is $\tau^{\le 0}((\Tot^{\Pi}F^2\oL\Omega^{\bt}_{A})[n+2])$.
        
        Hence the set of homotopy classes of such structures is $\H^{n+2}(\Tot^{\Pi}F^2\oL\Omega^{\bt}_{A})$,  
        each element consisting of an infinite sequence $(\omega_i \in (\Omega^i_{\tilde{A}})_{n-2+i})$ with $d\omega_i = \pm \delta \omega_{i+1}$, where $\tilde{A}$ is a cofibrant (or cofibrant over smooth) replacement for $A$.
        
        We say    $\omega$ is \emph{shifted symplectic} \index{shifted symplectic structure} if it is  non-degenerate in the sense that the maps $\EExt^i_{\tilde{A}}(\Omega^1_{\tilde{A}},\tilde{A}) \to \H_{-i-n}(\Omega^1_{\tilde{A}})$ from the tangent complex to the cotangent complex induced by contraction with $\omega_2 \in \H_{-n}(\Omega^2_{\tilde{A}})$ are isomorphisms. (In particular, this implies that $n$-shifted symplectic structures on derived schemes only exist for $n\le 0$; positively shifted structures can however exist on derived Artin stacks.)
        
        In the global case (for a derived scheme, or even derived algebraic space or DM stack), the complex of \emph{$n$-shifted pre-symplectic structures}\index{pre-symplectic structures}  is 
        $$
        \tau^{\le 0}\oR\Gamma(X,(\Tot^{\Pi}F^2\oL\Omega^{\bt}_X)[n+2]),
        $$
        so  homotopy classes are elements of $\H^{n+2}(X, \Tot^{\Pi}F^2\oL\Omega^{\bt}_{X})$, and are regarded as symplectic if they are locally non-degenerate. (Derived Artin stacks are treated similarly, but non-degeneracy becomes a more global condition --- see Definition \ref{sympdefstack}.) 
    \end{definition}
    
    \begin{remark}[Terminology]\label{PTVVterminology}
        In \cite{KhudaverdianVoronov},  working in  the $\C^{\infty}$ setting,  chain complexes were $\Z/2$-graded rather than $\Z$-graded, only even shifts were considered,  $\delta$ was zero, and their homotopy symplectic structures also permitted linear terms. They outlined an extension of their   definition and results to odd shifts, with the details appearing in \cite{bruceGeomObjects}. The  definition in 
        \cite{PTVV} is only formulated inexplicitly as a homotopy limit, obscuring the similarity with earlier work;  the Hodge filtration is not mentioned.
        
        Our terminology  follows \cite{poisson}, differing slightly from both sources. In \cite{PTVV}, pre-symplectic structures are called closed $2$-forms, terminology we avoid because it refers more naturally to $\z^2$ than to $
        \oR \z^2$. Also beware that 
        ibid. refers to double complexes as ``graded mixed complexes''.\index{graded mixed complex|see {double complex}}
        
        The opposite convention is adopted in \cite[Definition 2.7]{BorisovJoyce}, where a presymplectic $2$-form is taken to be non-degenerate but not closed, usage which clashes with classical terminology.
    \end{remark}
    
    \begin{example}
        For $Y$ a smooth scheme,  the \emph{shifted cotangent bundle} \index{shifted cotangent bundle} 
        $$
        T^*Y[-n]:=\oSpec_Y (\Symm_{\sO_Y}(\sT_Y[n]), \delta=0)
        $$ 
        is $(-n)$-shifted symplectic, with symplectic form $\omega$ given in local co-ordinates by $\sum_i dy_i\wedge d\eta_i$, for $\eta_i=\pd_{y_i} \in \sT_Y$, the tangent sheaf. Thus $\omega \in \Omega^2$ with $\delta\omega=0$ and $d\omega=0$.
        
        There are also twisted versions: if we  twist $T^*Y[-1]$ by taking the differential $\delta$ to be  given by contraction with $df$,  we still have a $(-1)$-shifted symplectic structure. That derived scheme is the  \emph{derived critical locus} \index{derived critical locus} of $f$, i.e. the derived vanishing locus of $df$, 
        \[
            Y\by^h_{df, T^*Y,0}Y.
        \]
    \end{example}
    
    \begin{remark}[Shifted Poisson structures and quantisation]
        There is a related notion of   shifted Poisson structures \cite{KhudaverdianVoronov,poisson,CPTVV}\footnote{introductory slides  available at \href{https://www.maths.ed.ac.uk/~jpridham/edbpoisson.pdf}{www.maths.ed.ac.uk/$\sim$jpridham/edbpoisson.pdf}}. In this setting, such a structure just amounts to a shifted $L_{\infty}$-algebra structure on $\sO_X$, with the brackets all being multiderivations, assuming we have chosen a   cofibrant (or cofibrant over smooth) model for $\sO_X$. The equivalence between shifted symplectic and non-degenerate shifted Poisson structures is interpreted in \cite{KhudaverdianVoronov} as a form of Legendre transformation, and the comparison in \cite{poisson} can be interpreted as a homotopical generalisation 
        of a Legendre transformation; 
        the comparison in \cite{CPTVV} (covering the same ground as \cite{poisson}) takes a much less direct approach.
        
        There are also notions of deformation quantisation for $n$-shifted Poisson structures, mostly summarised in \cite{DQDG}. For $n>0$ (generally existing on derived stacks rather than schemes), quantisation is an immediate consequence of formality of the little $(n+1)$-discs operad \cite[Theorem 2]{kontsevichOperads}, as observed in \cite[Theorem 3.5.4]{CPTVV}. The problem becomes increasingly difficult as $n$ decreases, unless one is willing to break the link with BV quantisation and redefine quantisation for $n<0$ as in \cite[Definition 3.5.8]{CPTVV} so that it also becomes a formality.
    \end{remark}
    
\clearpage 
\section {Simplicial structures}\label{simpsn}
    
    References for this section include \cite[\S 8]{W} and \cite{sht}, among others.
    
\subsection{Simplicial sets}\label{simplicial-sets}    
    
    {\bf Motivation:} 
    %
    \begin{itemize}
        \item Half-exact functors don't behave well enough to allow gluing, so we'll need to work with some flavour of $\infty$-categories instead of homotopy categories.
        \item The category $s\Set$ of simplicial sets is  much more manageable to work with than the category  $\Top$ of topological spaces.
    \end{itemize} 
    
    In algebraic geometry, the idea of looking at simplicial set-valued functors to model derived phenomena goes back at least as far as \cite{hinstack}.
    
    %
    %
    %
    %
    %
    %
    %
    
    \begin{definition}
        Let $|\Delta^n| \subset \R_{\ge 0}^{n+1}$ be the subspace $\{(x_0, \ldots, x_n)\,:\, \sum x_i=1\}$;
        this is the \emph{geometric $n$-simplex}\index{geometric $n$-simplex}\index{Deltan@$\mid\Delta^n \mid$|see {geometric $n$-simplex}}. See \Cref{fig:simplicesonly}.
        \end{definition}
    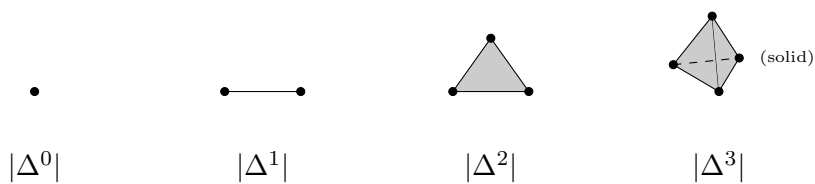
\begin{figure}[H]
        \centering
        \begin{tikzpicture}[scale=1]
            \newcommand\ptSize{1.5pt}
            
            \node at (9,-1) {$|\Delta^3|$};
            \draw[fill=black!20] (8.4,.356) -- (8.911,1) -- (9,0) -- cycle;
            \draw[fill=black!20] (9,0) -- (9.267,.445) -- (8.911,1);
            \draw[dashed] (8.4,.356) -- (9.267,.445);
            \draw[fill=black] (8.4,.356) circle (\ptSize);
            \draw[fill=black] (8.911,1) circle (\ptSize);
            \draw[fill=black] (9,0) circle (\ptSize);
            \draw[fill=black] (9.267,.445) circle (\ptSize);
            \node[anchor=west] at (9.4,.445) {\tiny (solid)};
            
            \node at (6,-1) {$|\Delta^2|$};
            \draw[fill=black!20] (5.5, 0) -- (6.5, 0) --  (6, 0.707) -- cycle;
            \draw[fill=black] (5.5, 0) circle (\ptSize);
            \draw[fill=black] (6.5, 0) circle (\ptSize);
            \draw[fill=black] (6, 0.707) circle (\ptSize);
            
            \node at (3,-1) {$|\Delta^1|$};
            \draw (2.5,0) -- (3.5,0);
            \draw[fill=black] (2.5,0) circle (\ptSize);
            \draw[fill=black] (3.5,0) circle (\ptSize);
            
            \node at (0,-1) {$|\Delta^0|$};
            \draw[fill=black] (0, 0) circle (\ptSize);
        \end{tikzpicture}
        \caption{Geometric $n$-simplices.}
        \label{fig:simplicesonly}
    \end{figure}
    \begin{definition}\label{Singdef}
        Given a topological space $X$, we then have a system 
        $$
        \Sing(X)_n := \Hom(|\Delta^n|, X)
        $$
        of sets, known as the \emph{singular functor}\index{singular functor $\Sing$},
        fitting into a diagram
        \[
            \xymatrix@1{ \Sing(X)_0 \ar@{.>}[r]|{\sigma_0}& \ar@<1ex>[l]^{\pd_0} \ar@<-1ex>[l]_{\pd_1} \Sing(X)_1 \ar@{.>}@<0.75ex>[r] \ar@{.>}@<-0.75ex>[r]  & \ar[l] \ar@/^/@<0.5ex>[l] \ar@/_/@<-0.5ex>[l] 
            \Sing(X)_2 &  &\ar@/^1pc/[ll] \ar@/_1pc/[ll] \ar@{}[ll]|{\cdot} \ar@{}@<1ex>[ll]|{\cdot} \ar@{}@<-1ex>[ll]|{\cdot}  \Sing(X)_3 & \ldots&\ldots,}
        \]   
        where the maps  $\pd_i: X_n \to X_{n-1}$  come from inclusion of the \emph{$i^{\textit{th}}$ face map}\index{face map} $ \pd^i \co |\Delta^{n-1}| \to |\Delta^n|$, and the maps $\sigma_i:X_n \to X_{n+1}$ come from the \emph{$i^{\textit{th}}$ degeneracy map\index{degeneracy map}} $\sigma^i \co |\Delta^{n+1}| \to |\Delta^n|$ given by collapsing the  edge $(i,i+1)$. 
        
        These operations  satisfy the following identities:
        \[
            \pd_i\pd_j=\pd_{j-1}\pd_i\quad \text{ for } \quad i<j,
        \]
        \[
            \sigma_i\sigma_j=\sigma_{j+1}\sigma_i \quad \text{ for }\quad i\le j,
        \]
        and
        \[
            \pd_i\sigma_j= \left\{\begin{matrix} \id & i=j,j-1 \\ \sigma_{j-1}\pd_i & i<j\\ \sigma_j\pd_{i-1} & i>j+1. \end{matrix}\right.       
        \]
    \end{definition} 
    \begin{definition}
        We denote by $\Delta$ the \emph{ordinal number category}\index{ordinal number category}\index{Delta@$\Delta$|see {ordinal number category}} which
        has objects $\on :=\{0,1,\ldots,n\}$ for $n \ge 0$, and morphisms $f$ given by non-decreasing maps between them (i.e. $f(i+1) \ge f(i)$ for every $i \in [0,n]$).
    \end{definition}
    
  $\Sing(X)_{(-)}$ has given us a contravariant functor from $\Delta$ to the category of sets. The correspondence comes by labelling the vertices of $|\Delta^n|$ from $0$ to $n$ according to the non-zero co-ordinate, allowing us to regard $\on$ as a subset of $|\Delta^n|$. The face and degeneracy maps then correspond to morphisms in the ordinal number category, and indeed every morphism in $\Delta$ is expressed as a composition of degeneracy and face maps, by \cite[Lemma 8.1.2]{W}, unique modulo the relations above.

      This motivates the following definition.
       
    \begin{definition}
        The category $s\Set$ of  \emph{simplicial sets}\index{simplicial!sets}\index{sset@$s\Set$|see {simplicial sets}} consists of functors $Y \co \Delta^{\op} \to \Set$. Write $Y_n$ for $Y(n)$. Thus objects are just diagrams
        \[
            \xymatrix@1{ Y_0 \ar@{.>}[r]|{\sigma_0}& \ar@<1ex>[l]^{\pd_0} \ar@<-1ex>[l]_{\pd_1} Y_1 \ar@{.>}@<0.75ex>[r] \ar@{.>}@<-0.75ex>[r]  & \ar[l] \ar@/^/@<0.5ex>[l] \ar@/_/@<-0.5ex>[l] 
            Y_2 &  &\ar@/^1pc/[ll] \ar@/_1pc/[ll] \ar@{}[ll]|{\cdot} \ar@{}@<1ex>[ll]|{\cdot} \ar@{}@<-1ex>[ll]|{\cdot}  Y_3 & \ldots&\ldots,}
        \]   
        satisfying the relations of Definition \ref{Singdef}.
    \end{definition}
    

    %
    
    \begin{definition}
        Define the \emph{combinatorial $n$-simplex}\index{combinatorial $n$-simplex}\index{Deltan@$\Delta^n$|see{combinatorial $n$-simplex}} $\Delta^n \in s\Set$ 
        by $\Delta^n:= \Hom_{\Delta}(-, \on)$.
    \end{definition}
    
    \begin{example}
        $\Delta^0$ is the constant diagram
        \[
            \xymatrix@1{ \bt \ar@{.>}[r]|{\sigma_0}& \ar@<1ex>[l]^{\pd_0} \ar@<-1ex>[l]_{\pd_1} \bt \ar@{.>}@<0.75ex>[r] \ar@{.>}@<-0.75ex>[r]  & \ar[l] \ar@/^/@<0.5ex>[l] \ar@/_/@<-0.5ex>[l] 
            \bt &  &\ar@/^1pc/[ll] \ar@/_1pc/[ll] \ar@{}[ll]|{\cdot} \ar@{}@<1ex>[ll]|{\cdot} \ar@{}@<-1ex>[ll]|{\cdot}  \bt & \ldots&\ldots}
        \]
        on the one-point set; note that this constant diagram is the smallest possible simplicial set with an element in degree $0$, since the degeneracy maps are necessarily injective.
        
        Meanwhile $(\Delta^1)_{i}$ has $i+2$ elements, of which only the two elements in $(\Delta^1)_{0}$ and one of those in $(\Delta^1)_{1}$ are \emph{non-degenerate}\index{non-degenerate} (i.e. not in the image of any  degeneracy map $\sigma_i$).            
    \end{example}    
    \begin{lemma}  
        The functor $\Sing: \Top \to s\Set$ has  a left adjoint $Y \mapsto |Y|$, determined by $\Delta^n \mapsto |\Delta^n|$, and the need to preserve coproducts and pushouts.
    \end{lemma}
    
    Explicitly, $|Y|$ is the quotient of $\coprod_n (Y_n \by |\Delta^n|)$ by the relations $(\pd_i y , a) \sim (y, \pd^ia)$ and  $(\sigma_i y , a) \sim (y, \sigma^ia)$.
    
\subsubsection{The Kan--Quillen model structure}
    
    \begin{definition}
        We say that a morphism $X \to Y$ in $s\Set$ is a  \emph{weak equivalence}\index{weak equivalence!of simplicial sets} if $|X|\to |Y|$ is a weak equivalence (i.e. $\pi_*$-equivalence) of topological spaces.
    \end{definition}
    
    \begin{theorem}[Quillen, \cite{QHA}]
        There is a model structure on $s\Set$ with the weak equivalences above, with cofibrations  just being  maps
        $f \co X \to Y$ which are injective in each level. Fibrations are then those maps with RLP with respect to all trivial cofibrations (i.e. cofibrations
        which are weak equivalences):
        \[
            \xymatrix{ A\ar[d]|{\text{triv.cof.}} \ar[r] & X \ar[d]|{\text{fib.}}\\
            B \ar@{.>}[ur]\ar[r] & Y.
            }       
        \]
    \end{theorem}
    \begin{definition}
        For $n \ge 0$, define the \emph{boundary} $\pd\Delta^n \subset \Delta^n$ to be $\bigcup_i \pd^i(\Delta^{n-1})$.\index{boundary of the $n$-simplex}\index{dDeltan@$\pd\Delta^n$|see{boundary of the $n$-simplex}}
        See Figure \ref{fig:boundarysimplices}. 
    \end{definition}

    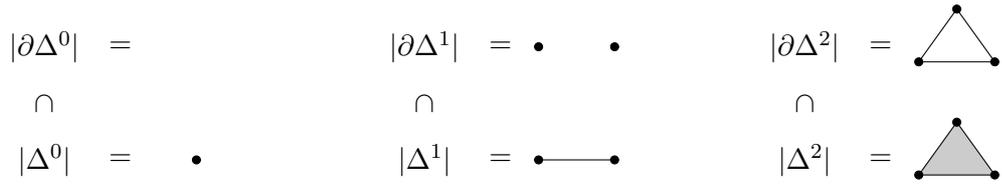
\begin{figure}[H]
        \centering
        \begin{tikzpicture}[scale=1]
            \newcommand\ptSize{1.5pt}
            \node at (-5, 1.5) {$|\pd\Delta^0|$};
            \node at (-5, 0) {$|\Delta^0|$};
            \node[rotate=-90] at (-5, .75) {$\subset$};
            \node at (-4, 1.5) {$=$};
            \node at (-4, 0) {$=$};
            \draw[fill=black] (-3,0) circle (\ptSize);
            
            \node at (0, 1.5) {$|\pd\Delta^1|$};
            \node at (0, 0) {$|\Delta^1|$};
            \node[rotate=-90] at (0, .75) {$\subset$};
            \node at (1, 1.5) {$=$};
            \node at (1, 0) {$=$};
            \draw[fill=black] (1.5, 1.5) circle (\ptSize);
            \draw[fill=black] (2.5,1.5) circle (\ptSize);
            \draw[fill=black] (1.5,0) circle (\ptSize);
            \draw[fill=black] (2.5,0) circle (\ptSize);
            \draw (1.5, 0) -- (2.5, 0);
            
            \draw[fill=black!20] (7, .5) -- (6.5, -.2) -- (7.5, -.2) -- cycle;
            \node at (5, 1.5) {$|\pd\Delta^2|$};
            \node at (5, 0) {$|\Delta^2|$};
            \node[rotate=-90] at (5, .75) {$\subset$};
            \node at (6, 1.5) {$=$};
            \node at (6, 0) {$=$};
            \draw[fill=black] (7, 2) circle (\ptSize);
            \draw[fill=black] (6.5, 1.3) circle (\ptSize);
            \draw[fill=black] (7.5,1.3) circle (\ptSize);
            \draw (7, 2) -- (6.5, 1.3) -- (7.5,1.3) -- cycle;
            \draw[fill=black] (7,.5) circle (\ptSize);
            \draw[fill=black] (6.5,-.2) circle (\ptSize);
            \draw[fill=black] (7.5,-.2) circle (\ptSize);
        \end{tikzpicture}
        \caption{Realisations of  $\pd\Delta^n$}
        \label{fig:boundarysimplices}
    \end{figure}
    \begin{definition}
        For $n \ge 1$,  define the  \emph{$k^{\textit{th}}$ horn} $\Lambda^{n,k}\subset \Delta^n$ to be  $\bigcup_{i\ne k} \pd^i(\Delta^{n-1})\subset \Delta^n$ ($n \ge 1$).
        See Figure \ref{fig:horns}.\index{horn $\L^{n,k}$}\index{Lambdank@$\L^{n,k}$|see{horn}}
    \end{definition}
    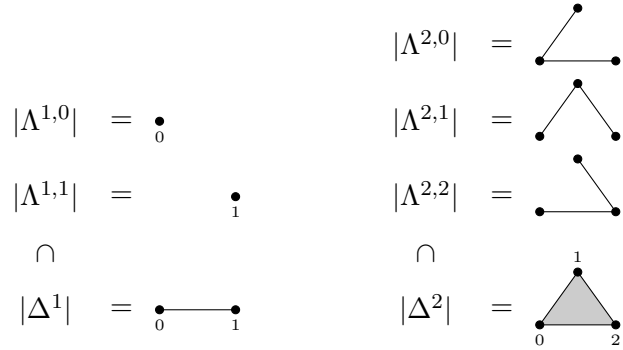
\begin{figure}[H]
        \centering
        \begin{tikzpicture}[scale=1]
            \newcommand\ptSize{1.5pt}
            
            \node at (0, 2.5) {$|\Lambda^{1,0}|$};
            \node at (0, 1.5) {$|\Lambda^{1,1}|$};
            \node at (0, 0) {$|\Delta^1|$};
            \node[rotate=-90] at (0, .75) {$\subset$};
            \node at (1, 2.5) {$=$};
            \node at (1, 1.5) {$=$};
            \node at (1, 0) {$=$};
            \draw[fill=black] (1.5, 2.5) circle (\ptSize);
            \draw[fill=black] (2.5, 1.5) circle (\ptSize);
            \draw[fill=black] (1.5,0) circle (\ptSize);
            \draw[fill=black] (2.5,0) circle (\ptSize);
            \draw (1.5, 0) -- (2.5, 0);
            \node[anchor=north] at (1.5,0) {\tiny $0$};
            \node[anchor=north] at (1.5,2.5) {\tiny $0$};
            \node[anchor=north] at (2.5,0) {\tiny $1$};
            \node[anchor=north] at (2.5,1.5) {\tiny $1$};
            
            \draw[fill=black!20] (7, .5) -- (6.5, -.2) -- (7.5, -.2) -- cycle;
            \node at (5, 3.5) {$|\Lambda^{2,0}|$};
            \node at (5, 2.5) {$|\Lambda^{2,1}|$};
            \node at (5, 1.5) {$|\Lambda^{2,2}|$};
            \node at (5, 0) {$|\Delta^2|$};
            \node[rotate=-90] at (5, .75) {$\subset$};
            \node at (6, 3.5) {$=$};
            \node at (6, 2.5) {$=$};
            \node at (6, 1.5) {$=$};
            \node at (6, 0) {$=$};
            \draw[fill=black] (7, 4) circle (\ptSize);
            \draw[fill=black] (6.5, 3.3) circle (\ptSize);
            \draw[fill=black] (7.5,3.3) circle (\ptSize);
            \draw (7, 4) -- (6.5, 3.3) -- (7.5,3.3);
            \draw[fill=black] (7, 3) circle (\ptSize);
            \draw[fill=black] (6.5, 2.3) circle (\ptSize);
            \draw[fill=black] (7.5,2.3) circle (\ptSize);
            \draw  (7.5,2.3) -- (7, 3) -- (6.5, 2.3) ;
            \draw[fill=black] (7, 2) circle (\ptSize);
            \draw[fill=black] (6.5, 1.3) circle (\ptSize);
            \draw[fill=black] (7.5,1.3) circle (\ptSize);
            \draw (7, 2) -- (7.5,1.3) --  (6.5, 1.3);
            \draw[fill=black] (7,.5) circle (\ptSize);
            \draw[fill=black] (6.5,-.2) circle (\ptSize);
            \draw[fill=black] (7.5,-.2) circle (\ptSize);
            \node[anchor=south] at (7,.5) {\tiny $1$};
            \node[anchor=north] at (6.5,-.2) {\tiny $0$};
            \node[anchor=north] at (7.5,-.2) {\tiny $2$};
        \end{tikzpicture}
        \caption{Realisations of  $\L^{n,k}$.}
        \label{fig:horns}
    \end{figure}
    \begin{theorem}
        Fibrations, resp.  trivial fibrations, in $s\Set$ correspond to maps with RLP with respect to $\L^{n,k} \to \Delta^n$ (generating trivial cofibrations), resp. $\pd \Delta^n \to \Delta^n$ (generating  cofibrations); these are known as (trivial) Kan fibrations.\index{Kan fibrations}
        {\rm See Figure \ref{fig:fibrations}.}
    \end{theorem}
    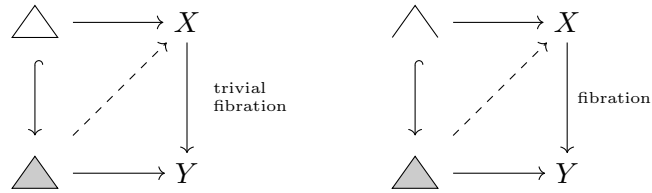
\begin{figure}[H]
        \centering
        \begin{tikzpicture}[scale=1]
            \newcommand\ptSize{1.5pt}
            
            \node (X1) at (2, 2) {$X$};
            \node (X2) at (7, 2) {$X$};
            \node (Y1) at (2, 0) {$Y$};
            \node (Y2) at (7, 0) {$Y$};
            
            \draw[] (-.3, 1.8) -- (.3, 1.8) -- (0, 2.224) -- cycle;
            \draw[fill=black!20] (-.3, -.2) -- (.3, -.2) -- (0, .224) -- cycle;
            \draw[] (5.3, 1.8) -- (5, 2.224) -- (4.7, 1.8);
            \draw[fill=black!20] (4.7, -.2) -- (5.3, -.2) -- (5, .224) -- cycle;
            
            \draw[->] (.5,2) -- (X1);
            \draw[->] (5.5,2) -- (X2);
            \draw[->] (.5,0) -- (Y1);
            \draw[->] (5.5,0) -- (Y2);
            \draw[->] (X1) -- (Y1);
            \draw[->] (X2) -- (Y2);
            \draw[right hook->] (0,1.5) -- (0, .5);
            \draw[right hook->] (5,1.5) -- (5,.5);
            \draw[->, dashed] (.5,.5) -- (X1);
            \draw[->, dashed] (5.5,.5) -- (X2);
            
            \node[anchor=west] at (2,1) {\tiny \begin{tabular}{l} trivial\\ fibration\end{tabular}};
            \node[anchor=west] at (7,1) {\tiny fibration};
        \end{tikzpicture}
        \caption{Existence of boundary-fillers and horn-fillers.}
        \label{fig:fibrations}
    \end{figure}

    \begin{definition}

        Say that a simplicial set is a \emph{Kan complex}\index{Kan complex} if it is fibrant:
        \[
            \xymatrix{ \L^{n,k}\ar[d] \ar[r] & X  & \forall n,k\\
            \Delta^n. \ar@{.>}[ur].
            }       
        \]
    \end{definition}
    \begin{theorem}[Kan, \cite{kanAdjointFunctors}]
        The adjunction 
        \[
            \xymatrix@1{ \Top \ar@<.5ex>[r]^{\Sing} & s\Set  \ar@<.5ex>[l]^{|-|} }
        \]
        is a Quillen equivalence.
        In particular, $\Ho(\Top) \simeq \Ho(s\Set)$.
    \end{theorem}
    
    This also gives rise to an equivalence between the  category of topological categories and the category of simplicial categories, up to weak equivalence in both cases. Here, a \emph{simplicial category}\index{simplicial!category} is a category enriched in simplicial sets, meaning that
    for any two objects $X, Y\in \mathcal{C}$ there is a simplicial set
    $\Hom_{\mathcal{C}}(X, Y)$ of morphisms between them, and a  composition operation defined levelwise.
    
    The \emph{homotopy category} $\pi_0\C$\index{homotopy category!of a simplicial category} of a simplicial or topological category has the same objects, but morphisms given by path components $\pi_0\C(x,y)$.             
    A functor $F \co \C \to \cD$ is then a \emph{weak equivalence}\index{weak equivalence!of simplicial categories} of simplicial or topological categories if the functor $\pi_0F\co \pi_0\C \to \pi_0\cD $ is an equivalence of categories and the maps $\C(X,Y) \to \C(FX,FY)$ are all weak equivalences of simplicial sets or topological spaces.
    
\subsubsection{Matching objects}  \label{matchsn}
    
    \begin{definition}
        Given $X \in s\Set$, we define the \emph{$n^{\textit{th}}$ matching space}\index{matching!space} by $M_{\pd \Delta^n}(X):=\Hom_{s\Set}(\pd\Delta^n,X)$;
        this is often simply denoted by $M_n(X)$. Explicitly, this means
        $$
        M_{\pd \Delta^n}(X) = \{ x \in \prod_{i=0}^n X_{n-1} \,:\, \pd_ix_j= \pd_{j-1}x_i \text{ if } i<j\}
        $$ 
        for $n>0$, with $M_{\pd\Delta^0}X=\ast$.
        
        Define the \emph{$(n,k)^{\textit{th}}$ partial matching space}\index{partial matching!space} by $M_{\L^{n,k}}(X):=\Hom_{s\Set}(\L^{n,k},X)$. Explicitly, this means
        $$
        M_{\L^{n,k}}(X) = \{ x \in \prod_{i=0, i \ne k}^n X_{n-1} \,:\, \pd_ix_j= \pd_{j-1}x_i \text{ if } i<j\}.
        $$
    \end{definition}
    
    The inclusions $\pd\Delta^n \to \Delta^n$  and $\L^{n,k} \to \Delta^n$ induce matching maps \index{matching!map} and partial matching maps \index{partial matching!map}
    $X_n \to M_{\pd \Delta^n}(X)$ and $X_n\to  M_{\L^{n,k}}(X)$, sending $x$ to  
    $
    (\pd_0x, \pd_1x, \ldots, \pd_nx)$ and $(\pd_0x, \pd_1x, \ldots, \cancel{\pd_kx}, \ldots \pd_nx)$, respectively.
    Thus $X \to Y$ being a Kan fibration says the relative partial matching maps \index{relative partial matching map}
    \[
        X_n \to Y_n \by_{ M_{\L^{n,k}}(Y)}M_{\L^{n,k}}(X)
    \]
    are all surjective, while $X \to Y $ being a trivial Kan fibration says the relative  matching maps 
    \[
        X_n \to Y_n \by_{ M_{\pd \Delta^n}(Y)}M_{\pd \Delta^n}(X)
    \]
    are all surjective.   
    
    %
    %
    %
    
\subsubsection{Diagonals}
    A \emph{bisimplicial set}\index{bisimplicial set} is just a simplicial simplicial set, i.e. a functor $X \co (\Delta \by \Delta)^{\op} \to \Set$. The category
    of bisimplicial sets is denoted $ss\Set$.
    There is then a \emph{diagonal functor}\index{diagonal functor $\diag$} 
    \[
        \diag \co ss\Set \to s\Set
    \]
    from bisimplicial sets to simplicial sets given by $\diag(X)_n:=X_{n,n}$ with the maps $\pd_i \co \diag(X)_{n+1} \to \diag(X)_n$ and $\sigma_i \co \diag(X)_{n-1} \to \diag(X)_n $ given by composing the corresponding horizontal maps ($\pd^h_i \co X_{m,n} \to X_{m-1,n}$, $\sigma^h_i \co X_{m,n} \to X_{m+1,n}$) and vertical maps ($\pd^v_i \co X_{m,n} \to X_{m,n-1}$, $\sigma^v_i \co X_{m,n} \to X_{m,n+1}$) in $X$.
    
    It turns out that $\diag(X)$ is a model for the homotopy colimit $\ho\LLim_{n \in \Delta^{\op}} (X_{n,\bt})$. As a consequence, its homotopical behaviour is   just like the total complex of a double complex, even though the diagonal seems much larger. (The analogous statements for semi-simplicial sets are not true: although the degeneracy maps $\sigma_i$ might feel superfluous much of the time, they are vital  for results such as these  to hold.)
    
\subsection{The Dold--Kan equivalence} 
    
    If $A$ is a simplicial abelian group, then $\delta:= \sum (-1)^i\pd_i$ satisfies $\delta^2=0$, so $(A,\delta)$ becomes a chain complex.
    
    \begin{definition}
        The \emph{normalisation} $NA$ \index{normalisation functor $N$} \index{N@$N$|see {normalisation functor}} of a simplicial abelian group is the chain complex given by $N_mA:= \{a \in A_m \mid \pd_ia=0 \quad\forall i>0\}$, with differential given by $\pd_0 \co N_{m+1}A \to N_mA$ (it squares to zero because $\pd_0(\pd_0a)=\pd_0(\pd_1a)=\pd_0(0)$).
    \end{definition}
    
    In fact, the inclusion $NA \to (A,\delta)$ is a quasi-isomorphism of chain complexes. Also, the homology groups $\H_*(NA)$ are just the homotopy groups $\pi_*(A,0):=\pi_*(|A|,0)$ of the simplicial set underlying $A$. 
    
    (These should not be confused with the homology groups $\H_*(X,\Z)$ of a simplicial set $X$, which correspond to homotopy groups of the free simplicial abelian group  $\Z.X$ on generators $X$, with $N(\Z.X)$ then being the complex of normalised chains on $X$. The free abelian and forgetful functors form a Quillen adjunction.)  
    
    \begin{theorem}[Dold--Kan] \index{Dold--Kan equivalence}
        The functor $N$ gives an equivalence of categories between simplicial abelian groups and chain complexes in non-negative degrees.
    \end{theorem}
    The inverse functor $N^{-1}$ is just given by throwing in degenerate elements $\sigma_{i_1}\cdots \sigma_{i_n}a$. 


\subsection{The Eilenberg--Zilber correspondence}\label{shufflesn}
    
    Given a bisimplicial abelian group $A$, we can normalise in both directions to get a double complex $\uline{N}A$,
    and we can also take the diagonal to give a simplicial abelian group $\diag(A)$.
    
    There is a quasi-isomorphism, known as the \emph{Eilenberg--Zilber shuffle map},\index{Eilenberg--Zilber shuffle map}
    \[
        \nabla \co \Tot\uline{N} A \to N\diag  A,
    \]
    given by summing signed shuffle permutations of the horizontal and vertical degeneracy maps $\sigma_i$ in $A$.
    
    This map is symmetric with respect to swapping the horizontal and vertical bisimplicial indices. The homotopy inverse of $\nabla$ is  given by the  Alexander--Whitney cup product\index{Alexander--Whitney cup product},
    which sums the maps
    \[
        ( \pd_{i+1}^h)^j(\pd_0^v)^i \co A_{i+j,i+j}\to A_{ij},
    \]
    and is not symmetric. 
    
    \begin{construction}
        One consequence of the shuffle map is to give a functor from simplicial commutative rings $A$ (i.e. each $A_i$ a commutative ring) to cdgas. If we write $\uline{\ten}$ for the external tensor product  $(U \uline{\ten}V)_{i,j}:= U_i \ten V_j$, then we can characterise the multiplication on $A$ as a map $\mu \co \diag(A \uline{\ten} A) \to A$, so we have a composite
        \[
        NA \ten NA \cong \Tot(NA \uline{\ten} NA)=\Tot(\uline{N}(A \uline{\ten} A))  \xra{\nabla} N\diag(A \uline{\ten} A)\xra{\mu} NA,   
        \]
        giving our graded-commutative multiplication on the chain complex $NA$.
        
        Another consequence of the Alexander--Whitney is to give us a  simplicial ring $N^{-1}A$ associated to any dg-algebra $A$ in non-negative chain degrees, but this does not preserve commutativity. A generalisation of this construction allows us to associate simplicial categories to dg-categories (i.e. categories enriched in chain complexes as in \cite{kellerModelDGCat})\index{dg-category}, after truncation if necessary, as in Lemma \ref{dgModsCat}.
    \end{construction}
    
    
    
    
\subsection{Simplicial mapping spaces}\label{mappingsn}
    
    Given a category $\C$ with weak equivalences, we write $\oR \Map_{\C}$  for the functor $\C^{\op}\by \C \to s\Set$ given by right-deriving $\Hom$ (if it exists).  In model categories, $\oR \Map_{\C}$ always exists, and we now show how to calculate it using  function complexes as in \cite{DKfunction}  or  \cite[\S 5.4]{hovey}.
    
    \begin{definition}
        Given a model category $\C$ and a object $Y \in \C$, we can define a \emph{simplicial fibrant resolution}\index{simplicial!fibrant resolution} of $Y$ to be
        a simplicial diagram $\hat{Y} \co \Delta^{\op} \to \C$ and a map from  the constant diagram $Y$ to  $\hat{Y}$ (equivalently, a map $Y \to \hat{Y}_0$ in $\C$) such that
        
        \begin{enumerate}
            \item the maps $Y \to \hat{Y}_n$ are all weak equivalences,
            
            \item the matching maps $\hat{Y}_n \to M_{\pd \Delta^n}(\hat{Y})$ (defined by the same formulae as \S \ref{matchsn}) are fibrations in $\C$ for all $n \ge 0$;
            in particular, this includes the condition that $\hat{Y}_0$
            be fibrant.
        \end{enumerate}
    \end{definition}
    \begin{exercise}
        $\hat{Y}_1$ is a path object for $\hat{Y}_0$, via $\sigma_0 \co \hat{Y}_0 \to  \hat{Y}_1$ and $\hat{Y}_1 \xra{(\pd_0,\pd_1)} \hat{Y}_0 \by \hat{Y}_0$.
    \end{exercise}
    \begin{examples}\label{pathexamples}\ 
        \begin{enumerate}
            \item In $\Top$, we can take $\hat{Y}_n$ to be the space $Y^{|\Delta^n|}$  of maps from $|\Delta^n|$ to $Y$.
            \item In $s\Set$, if $Y$ is fibrant, we can take $\hat{Y}_n:= Y^{\Delta^n}$, where $(Y^K)_i:= \Hom_{s\Set}(\Delta^i \by K,Y)$.
            \item In cochain complexes, we can take $\hat{V}_n:= V\ten \CC^{\bt}(\Delta^n,\Z)$ (simplicial cochains on the $n$-simplex).
            \item In cdgas $\catdga$, we can take $\hat{A}_n:= \tau_{\ge 0}(A\ten\Omega^{\bt}(\Delta^n))$, where
                \[
                    \Omega^{\bt}(\Delta^n)= \Q[x_0, \ldots,x_n, \delta x_0, \ldots, \delta x_n]/(\sum x_i -1, \sum \delta x_i),
                \]
                for $x_i$ of degree $0$ (the polynomial de Rham complex of the $n$-simplex).

                The reason this works is that the matching object $M_{\pd \Delta^n}(\hat{A})$ is isomorphic to the cdga $\tau_{\ge 0}(A\ten\Omega^{\bt}(\pd \Delta^n))$, where
                \[
                \Omega^{\bt}(\pd\Delta^n)=  \Omega^{\bt}(\Delta^n)/(\prod_i x_i, \delta(\prod_i x_i));
                \]
                since $ \Omega^{\bt}(\Delta^n)\to \Omega^{\bt}(\pd\Delta^n)$ is surjective, the matching map $\hat{A}_n\to M_{\pd \Delta^n}(\hat{A})$ is surjective in strictly positive degrees, so a fibration.
        \end{enumerate}
    \end{examples}
    \begin{theorem}\label{Rmapthm}
        If $X$ is cofibrant and $\hat{Y}$
        is a fibrant simplicial resolution of $Y$, then the right function complex $\oR\Map_r(X,Y)$, given by
        \[
        n \mapsto \Hom_{\C}(X, \hat{Y}_n)
        \]
        gives a model for the right-derived functor $\oR \Map_{\C}$ of $\Hom \co \C^{\op} \by \C \to s\Set$.\index{mapping space $\oR\Map$}
    \end{theorem}
    \begin{proof}[Proof (sketch).]
        By \cite{DKfunction} or \cite[\S 5.4]{hovey},  function complexes  preserve weak equivalences, and are independent of the choice of resolution (so in particular we may assume $\hat{Y}$ is chosen functorially). There is an obvious natural transformation  $\Hom_{\C}  \to \oR \Map_{\C}$, so it suffices to prove universality.
        
        If we have a natural transformation $\Hom_{\C} \to F$ with $F$ preserving weak equivalences, then the maps $F(X,Y) \to F(X,\hat{Y}_i)$ are weak equivalences for all $i$, so the map  $F(X,Y) \to \diag (i \mapsto F(X,\hat{Y}_i))$ is a weak equivalence.  But we have a map  $\Hom_{\C}(X,\hat{Y}_i)\to F(X,\hat{Y}_i)$, so taking diagonals gives
        \[
            \oR\Map_r(X,Y) \to \diag (i \mapsto F(X,\hat{Y}_i)) \xla{\sim} F(X,Y),
        \]
        hence the required morphism in the homotopy category.
    \end{proof}
    
    Note that  derived functors send $\oR\Map_{\C}$ to $\oR\Map_{\cD}$, and that Quillen equivalences induce  weak equivalences on $\oR\Map$.

      \begin{examples}\label{mapex}
 Here are some explicit examples of mapping spaces of cdgas:
        \begin{enumerate}
            \item Consider the affine line $\bA^1= \Spec k[x]$.
                A model for $\oR\Map_{\catdga}(k[x],B_\bt)$ is given  by $n  \mapsto \z_0((\Omega^{\bt}(\Delta^n)\ten B)_\bt)$, since $k[x]$ is cofibrant.
                However, a smaller model is given by Dold--Kan denormalisation: $ \oR\Map_{\catdga}(k[x],B_\bt) \simeq N^{-1}B_\bt$.
            \item Consider the affine group $\GL_n= \Spec A$, where 
                \[
                    A := k[ x_{ij}]_{1 \le i,j, \le n}[\det(x_{ij})^{-1}].
                \]
                A cofibrant replacement for $A$ is given by $\tilde{A_\bt} := k[x_{ij},y,t]$ with $t$ in degree $1$ satisfying  $\delta t = y\det(x_{ij})-1$, so 
                \[
                    \Hom_{\catdga}(\tilde{A_\bt},B_\bt)= \{(M,c,h) \in \Mat_n(B_0)\by B_0 \by B_1 ~:~ \delta h = c\det M -1\}, 
                \]   
                and then $\oR\Map_{\catdga}(A,B_\bt)$ is given in simplicial level $r$   by applying this to $\tau_{\ge 0}((\Omega^{\bt}(\Delta^r)\ten B)_\bt)$. 
                
                However, when $B_\bt$ Noetherian, we may just take $\GL_n(\widehat{(\z_0(\Omega^{\bt}(\Delta^r)\ten B))_\bt})$ in level $r$, where $\widehat{(-)}$ is completion along $(\z_0(\Omega^{\bt}(\Delta^r)\ten B)_\bt)\to \H_0(B_\bt)$, using the (Quillen equivalent) complete model structure of \cite[Proposition \ref{drep-cNhat}]{drep}. 
                
                In fact,  since $\GL_n$ is Zariski locally affine space, instead of completing we can just localise away from $\H_0(B_\bt)$, and drop the Noetherian hypothesis; this follows by using the local model structure,  a special case of  \cite[Proposition \ref{DStein-locmodelprop}]{DStein}.
        \end{enumerate}
    \end{examples}
    
    %
    %

    \begin{remark}
        The expressions above for cdgas  adapt to dg  $\C^{\infty}$ and EFC algebras, using $\odot$ instead of $\ten$ and $\C^{\infty}(\R^n)$ or $\sO^{\hol}(\Cx^n)$ instead of $\Omega^0(\Delta^n) \cong R[x_1, \dots, x_n]$.
        
        
    \end{remark}
    \begin{lemma}\label{dgModsCat}
        In categories like $\catdgmU$ or $\catdgm$,  the simplicial abelian groups $\oR\Map(M,P)$ normalise to give $N\oR\Map(M,P) \simeq \tau_{\ge 0}\oR\HHom_A(M,P)$ for $\HHom$ the dg $\Hom$ functor.   
    \end{lemma}
    \begin{proof}
        One approach is just to take the function complex $\hat{P}(n) \cong P\ten \bar{C}^{\bt}(\Delta^n)$ (normalised \emph{chains}\index{chains} on the $n$-simplex).

        Alternatively, note that   
        $\tau_{\le 0}\oR\HHom$ is the right-derived bifunctor of the composition of $\Hom$ with the inclusion of abelian groups  in non-negatively graded chain complexes. 
        
        Normalisation preserves weak equivalences, as does the forgetful functor from simplicial abelian groups to simplicial sets, so Dold--Kan denormalisation $N^{-1}$ gives the simplicial set-valued functor $N^{-1}\tau_{\ge 0}\oR\HHom_A(M,P)$ as the right-derived functor of $\Hom$, and thus 
        \[
            \oR\Map(M,P) \simeq N^{-1}\tau_{\ge 0}\oR\HHom_A(M,P).
        \]
    \end{proof}
    
    The following is a consequence of  Theorem \ref{Rmapthm}:
    
    \begin{corollary}
        If $F \co \cC \to \cD$ is left Quillen, with right adjoint $G$, then 
        \[
        \oR\Map_{\C}(A,\oR GB) \simeq \oR\Map_{\cD}(\oL F A,B);
        \]
        in particular, the derived functors $\oL F, \oR G$ give an adjunction of the associated infinity categories. 
    \end{corollary}
    
    \begin{example}
        The homotopy fibre of $\oR\Map_{\catdga}(A,B\oplus I) \to \oR\Map_{\catdga}(A,B)$ over $f$ is $\oR\Map_{\catdgm[A]}(\bL^{A/k},f_*M)$ which is quasi-isomorphic to $ N^{-1}\tau_{\ge 0}\oR\HHom_A(\bL^{A/k},f_*M) $, so 
        \[
        \pi_i\oR\Map_{\catdgm[A]}(\bL^{A/k},f_*M) \cong \EExt^{-i}_A(\bL^{A/k},f_*M). 
        \]
        This accounts for all of the obstruction maps seen in \S \ref{obssn}.
    \end{example}
    
    Another feature of  $\oR\Map$ is that it interacts with homotopy limits in the obvious way, so
    
    \begin{align*}
        \oR\Map(A, \ho\Lim_{i \in I}B(i))&\simeq \ho\Lim_{i \in I}\oR\Map(A,B(i)) \\  \oR\Map(\ho\LLim_{j\in J}A(j), B)&\simeq \ho\Lim_{j \in J}\oR\Map(A(j),B).
    \end{align*}

    
    
    
\subsection{Simplicial algebras}\label{salgsn}

    If we don't want our base $k$ to contain $\Q$, then we have to use simplicial rings instead of dg-algebras, giving the primary viewpoint of \cite{Q}.

    
\subsubsection{Definitions}
   
    \begin{definition}
        For a commutative ring $R$, define the category $s\Alg_R$ to consist of \emph{simplicial commutative $R$-algebras},\index{simplicial!algebras} i.e. functors $A \to \Delta^{\op} \to \Alg_R$.
    \end{definition}
    
    Thus each $A_n$ is a commutative $R$-algebra and the operations $\pd_i,\sigma_i$ are $R$-algebra homomorphisms.
    
    \smallskip    
    
    Quillen \cite{QRat,QHA} gives $s\Alg_R$ a model structure in which fibrations and weak equivalences are inherited from the corresponding properties for the underlying simplicial sets.
    
    
    \begin{theorem}[Quillen]\label{normalgthm}
        For $\Q \subseteq R$, Dold--Kan denormalisation gives a right Quillen {\bf equivalence} $N \co s\Alg_R \to \catdga[R]$, where the multiplication on $NA$ is defined using shuffles (\S \ref{shufflesn}). 
    \end{theorem}
    
    \begin{remarks}
    The theorem tells us that cdgas and simplicial algebras have equivalent homotopy theory in characteristic $0$, but simplicial algebras still work in finite and mixed characteristic. Our focus has been on cdgas, though, because they give much more manageable objects --- the degeneracies in a simplicial diagram generate a lot of elements.
    
    For an explicit homotopy inverse to $N \co s\Alg_R \to \catdga[R]$, instead of taking the derived left Quillen functor, we can just take the model for $\oR\Map(R[x],-)$ from Examples \ref{mapex}.
    \end{remarks}
    
    \begin{remark}
    We can also consider simplicial   EFC-algebras and $\C^{\infty}$-algebras (i.e. simplicial diagrams in the respective categories of algebras, so all structures are defined levelwise). Dold--Kan normalisation again gives a right Quillen functor to dg EFC or dg $\C^{\infty}$-algebras, and  this is a right Quillen equivalence by \cite{nuitenThesis}, hence our focus on the dg incarnations.
    
    Cotangent complexes are formulated for any algebraic theory in \cite{Q}, so the results there can be applied directly to EFC and $\C^\infty$ settings, but again they reduce to the differential graded constructions by \cite{nuitenThesis}.
    \end{remark}

    \subsubsection{Simplicial modules}
    
    \begin{definition}
    Given a simplicial ring $A$, we define the category $s\Mod_A$ of \emph{simplicial $A$-modules} \index{simplicial!modules} to consist of $A$-modules $M$ in simplicial sets.
    \end{definition}
    Thus each $M_n$ is an $A_n$-module, with the obvious compatibilities between the face and degeneracy maps $\pd_i,\sigma_i$ on $A$ and on $M$.

    \begin{theorem}[Quillen] 
    For $A \in s\Alg_R$, Dold--Kan denormalisation gives a right Quillen {\bf equivalence} $N \co s\Mod_A \to \catdgm[NA]$, where the multiplication of $NA$ on $NM$ is defined using shuffles (\S \ref{shufflesn}). 
    \end{theorem}
    Note that this statement does {\it not} need any restriction on the characteristic, essentially because modules do not care whether an algebra is commutative; the analogous statement for $A$ a simplicial non-commutative ring is also true.

    \subsubsection{Consequences}
    The various constructions we have seen for dg-algebras carry over to simplicial algebras, extending results beyond characteristic $0$.  Such constructions include the cotangent complex \index{cotangent complex} $\bL^{S/R} \in s\Mod_S$ (equivalently, $\catdgm[NS]$), which  has the same properties for smooth morphisms, \'etale morphisms and regular embeddings as before, though the calculation in the proof of Theorem \ref{smoothcotthm} becomes a little dirtier. The cotangent complex  is then used to define Andr\'e--Quillen  cohomology $D^*$\index{Andr\'e--Quillen cohomology!$D^*(A,M)$}. In characteristic  $0$, these are all (quasi-)isomorphic to our earlier cdga constructions. For details, see \cite{Q}.
    
    Mapping spaces for simplicial algebras are in fact simpler to describe than those for dg-algebras, since a fibrant simplicial resolution of $A$ is given by $n \mapsto A^{\Delta^n}$, defined in the same way as for simplicial sets in Examples \ref{pathexamples}.    

    \subsection{\texorpdfstring{$n$}{n}-Hypergroupoids}\label{hgpdsn}
    
    References for this section include \cite{duskin,glenn,getzler}, or \cite{stacks2} for  relative and trivial hypergroupoids; we follow the treatment in \cite{stacksintro}.

    %
    %
    %
    %
    \begin{definition}\label{relhyp}
    Given   $Y\in s\Set$, define a \emph{relative  $n$-hypergroupoid}\index{relative hypergroupoid}  over $Y$ to be a morphism $f:X\to Y$ in $s\Set$, such that  the   relative partial matching maps
    $$
    X_m \to M_{\L^{m,k}} (X)\by_{M_{\L^{m,k}}(Y)}Y_m 
    $$
    are surjective for all $k,m$ (i.e. $f$ is a Kan fibration), and isomorphisms for all $m> n$. In the terminology of \cite{glenn}, this says that $f$ is a Kan fibration which is an exact fibration in all dimensions $>n$.
    
    When $Y=*$ (the constant diagram on a point), we simply say that $X$ is an \emph{$n$-hypergroupoid}.\index{hypergroupoid}
    \end{definition}

    In other words, the definition says,  ``Relative horn fillers exist for all $m$, and are unique for $m>n$'': 
    the dashed arrows in  figure \ref{fig:hypergroupoids} making the triangles commute always exist, and are unique for $m>n$.
    \begin{figure}[H]
    \centering

    \begin{tikzpicture}
    
    \draw[fill=black!20] (-.5,0) -- (.5, 0) -- (0,.707) --cycle;
    \draw (-.5, 2) -- (0, 2.707) -- (.5, 2);
    \node[rotate=-90] at (0, 1.3535) {$\subset$};
    
    \node (Lam) at (2,2.35) {$\Lambda^{m,k}$};
    \node (Del) at (2,.35) {$\Delta^m$};
    \node (X) at (4,2.35) {$X$};
    \node (Y) at (4,.35) {$Y$};
    
    \draw[->] (Lam) -- (X);
    \draw[right hook->] (Lam) -- (Del);
    \draw[->, dashed] (Del) -- (X);
    \draw[->] (Del) -- (Y);
    \draw[->] (X) -- (Y);
    
    \end{tikzpicture}
    \caption{Horn-filling conditions}
    \label{fig:hypergroupoids}
    \end{figure}
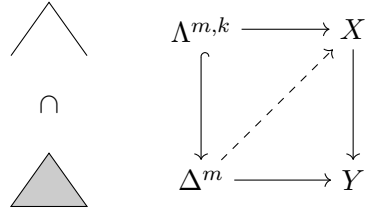

    \begin{examples}\label{ex:hgpd}\
      \begin{enumerate}
    \item
    A $0$-hypergroupoid $X$ is just a set $X=X_0$ regarded as a constant simplicial object, in the sense that we set $X_m=X_0$ for all $n$.
    
    \item
    \cite[\S 2.1]{glenn} (see also \cite[Lemma I.3.5]{sht}):  $1$-hypergroupoids are precisely nerves \index{nerve of a groupoid}
    $B\Gamma$ of groupoids $\Gamma$,  given by
    \[
    (B\Gamma)_n= \coprod_{x_0, \ldots, x_n}\Gamma(x_0,x_1)\by \Gamma(x_1,x_2)\by\ldots  \by \Gamma(x_{n-1},x_n),       
    \]
    with the face maps $\pd_i$ given by multiplications or discarding the ends, and the degeneracy maps $\sigma_i$ by inserting identity maps.
    %
    
    \item\label{ex:0hgpd} A relative $0$-hypergroupoid $f\co X \to Y$ is a Cartesian morphism,\index{Cartesian!morphism}
    in the sense that the maps
    $$
    X_n\xra{(\pd_i, f)} X_{n-1}\by_{Y_{n-1}, \pd_i}Y_n
    $$
    are all isomorphisms.
    \[
    \begin{CD}
    X_n @>f>> Y_n\\
    @V{\pd_i}VV @VV{\pd_i}V\\
    X_{n-1} @>f>> Y_{n-1}.
    \end{CD}
    \]
    
    Given $y \in Y_0$, we can write $F(y):= f_0^{-1}\{y\}$, and observe that $f$ is equivalent to a local system on $Y$ with fibres $F$.
    \end{enumerate}
    \end{examples}

    \begin{properties}\label{hypergroupoidprops}\
      \begin{enumerate}
    \item  For an $n$-hypergroupoid $X$, we have $\pi_mX=0$ for all $m>n$.
    
    \item Conversely, if $Y \in s\Set$ with $\pi_mY=0$ for all $m>n$, then  there exists a weak equivalence $Y \to X$ for some $n$-hypergroupoid $X$ (given by taking applying the fundamental $n$-groupoid construction of \cite{glenn} to a fibrant replacement).
    
    \item \label{truncate}\cite[Lemma \ref{stacks2-truncate}]{stacks2}: An $n$-hypergroupoid $X$ is completely determined by its truncation $X_{\le n+1}$. Explicitly, $X= \cosk_{n+1}X$, where the $m$-coskeleton $\cosk_mX$  has the universal property that $\Hom_{s\Set}(Y,\cosk_mX)\cong \Hom(Y_{\le m},X_{\le m})$ for all $Y \in s\Set$, so in particular has  $(\cosk_mX)_i= \Hom((\Delta^i)_{\le m}, X_{\le m})$. \index{coskeleton $\cosk$}
    
    Moreover, a simplicial set of the form $\cosk_{n+1}X$ is an $n$-hypergroupoid if and only if  it satisfies the conditions of Definition \ref{relhyp} up to level $n+2$. 
    
    When $n=1$, these statements amount to saying that a groupoid is uniquely determined by its objects (simplicial level $0$), morphisms and identities (level $1$) and multiplication (level $2$). However, we do not know we have a groupoid until we check associativity (level $3$). 
    
    \item Under the Dold--Kan correspondence between non-negatively graded chain complexes and simplicial abelian groups, $n$-hypergroupoids in abelian groups correspond to chain complexes concentrated in degrees $[0,n]$. One implication is easy to see because all simplicial groups are fibrant and $N_mA= \ker(A_m \to M_{\L^{m,0}}(A))$; the reverse implication uses the characterisation
    $N_mA \cong A_m/\sum \sigma_iA_{m-1}$. 
    
    \end{enumerate}
    \end{properties}
    %
    %
    %

    \begin{digression}
    There are  also versions for categories instead of groupoids, with just inner horns --- drop the conditions for $\L^{m,0}$ and $\L^{m,m}$. These give a model for $n$-categories (i.e. $(n,1)$-categories) instead of $n$-groupoids (i.e. $(n,0)$-categories). Taking $n = \infty$ then gives  Boardman and Vogt's weak Kan complexes \index{weak Kan complex} \cite{BoardmanVogt}, called quasi-categories \index{quasi-category} by Joyal \cite{joyalQCatKan}.
    
    Nowadays, these are often known simply as $\infty$-categories following the usage in \cite{luriehighertopoi,lurieSAG}, whereas \cite{lurieInftyTopoi,lurie} use that term exclusively for simplicial categories, which give an equivalent theory by \cite{joyalQCatSCat}. While quasi-categories lead to efficient proofs in the general theory of $\infty$-categories, they tend to be less convenient when working in a specific $\infty$-category. 
    \end{digression}
    
    \subsubsection{Trivial hypergroupoids}
    
    When is a groupoid contractible? When does a relative hypergroupoid correspond to an equivalence?
    
    \begin{definition}\label{trelhyp}
    Given   $Y\in s\Set$, define a \emph{trivial  relative  $n$-hypergroupoid} \index{hypergroupoid!trivial relative} over $Y$ to be a morphism $f:X\to Y$ in $s\Set$, such that  the   relative  matching maps
    $$
    X_m \to M_{\pd \Delta^m} (X)\by_{M_{\pd \Delta^m} (Y)}Y_m 
    $$
    are surjective for all $m$ (i.e. $f$ is a trivial Kan fibration), and isomorphisms for all $m\ge n$. 
    \end{definition}
    
    In other words, 
    the dashed arrows in figure \ref{fig:trivhypergroupoids}  making the triangles commute always exist, and are unique for $m\ge n$.
    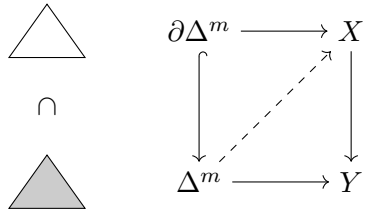
\begin{figure}[H]
    \centering

    \begin{tikzpicture}
    
    \draw[fill=black!20] (-.5,0) -- (.5, 0) -- (0,.707) --cycle;
    \draw (-.5, 2) -- (0, 2.707) -- (.5, 2) --cycle;
    \node[rotate=-90] at (0, 1.3535) {$\subset$};
    
    \node (Lam) at (2,2.35) {$\pd\Delta^{m}$};
    \node (Del) at (2,.35) {$\Delta^m$};
    \node (X) at (4,2.35) {$X$};
    \node (Y) at (4,.35) {$Y$};
    
    \draw[->] (Lam) -- (X);
    \draw[right hook->] (Lam) -- (Del);
    \draw[->, dashed] (Del) -- (X);
    \draw[->] (Del) -- (Y);
    \draw[->] (X) -- (Y);
    
    \end{tikzpicture}
    \caption{Simplex-filling conditions}
    \label{fig:trivhypergroupoids}
    \end{figure}

    Note that if $X$ is a trivial $n$-hypergroupoid over a point, then $X= \cosk_{n-1}X$, so $X$ is determined by $X_{<n}$. The converse needs conditions to hold for $X_{<n}$. 
    
    \begin{examples}\
    \begin{enumerate}
    \item A trivial relative $0$-hypergroupoid is an isomorphism.
    
    \item A trivial $1$-hypergroupoid over a point is the nerve of a contractible groupoid.
    \end{enumerate}
    \end{examples}

    \clearpage  
    
    \section{Geometric \texorpdfstring{$n$}{n}-stacks}\label{stacksn}
    
    References for this section are \cite{stacksintro, stacks2}. The  approach we will be taking was first postulated by Grothendieck in \cite{pursuingstacks}. Familiarity with the theory of algebraic stacks \cite{DeligneMumford,Artin,Champs} is not essential to follow this section, as we will construct everything from scratch in  a more elementary way.
    
    \medskip
    
    So far, we've mostly looked at derived affine schemes; they arise as homotopy limits of affine schemes.
    
    Now, we want to glue or take quotients, so we want homotopy colimits, which means we look to enrich objects in the opposite direction.

    \begin{warning}
    Whereas the simplicial algebras of \S \ref{salgsn}  correspond to covariant functors from $\Delta$ to affine schemes, i.e. to \emph{cosimplicial}\index{cosimplicial} affine schemes
    \[
    \xymatrix@1{ X^0 \ar@<1ex>[r]^{\pd^0} \ar@<-1ex>[r]_{\pd^1} & \ar@{.>}[l]|{\sigma^0} X^1 \ar[r] \ar@/^/@<0.5ex>[r] \ar@/_/@<-0.5ex>[r] & \ar@{.>}@<0.75ex>[l] \ar@{.>}@<-0.75ex>[l]   
    X^2\ar@/^1pc/[rr] \ar@/_1pc/[rr] \ar@{}[rr]|{\cdot} \ar@{}@<1ex>[rr]|{\cdot} \ar@{}@<-1ex>[rr]|{\cdot} && X^3{}   \ar@/^1.2pc/[rr] \ar@/_1.2pc/[rr]\ar@{}[rr]|{\cdot} \ar@{}@<1ex>[rr]|{\cdot} \ar@{}@<-1ex>[rr]|{\cdot}& & {\phantom{A}}\cdots,}
    \]    
    (one model for derived affine schemes), we now look at simplicial affine schemes\index{simplicial!affine schemes $s\Aff$}
    \[
    \xymatrix@1{ Y_0 \ar@{.>}[r]|{\sigma_0}& \ar@<1ex>[l]^{\pd_0} \ar@<-1ex>[l]_{\pd_1} Y_1 \ar@{.>}@<0.75ex>[r] \ar@{.>}@<-0.75ex>[r]  & \ar[l] \ar@/^/@<0.5ex>[l] \ar@/_/@<-0.5ex>[l] 
    Y_2 &  &\ar@/^1pc/[ll] \ar@/_1pc/[ll] \ar@{}[ll]|{\cdot} \ar@{}@<1ex>[ll]|{\cdot} \ar@{}@<-1ex>[ll]|{\cdot}  Y_3 & \ldots&\ldots}
    \]   
    as models for higher stacks. These constructions behave {\it very} differently from each other.
    \end{warning}
    
    \begin{digression}[Quasi-isomorphisms and rational homotopy theory]
    At this point, someone usually asks whether we can replace these simplicial schemes with cdgas, and the answer is no. Although denormalisation gives a right Quillen equivalence from cdgas in non-negative {\it cochain} degrees to cosimplicial algebras (an analogue of Theorem \ref{normalgthm}), this would only be applicable if we were willing to declare morphisms $X \to Y$ to be equivalences whenever they induce isomorphisms $\H^*(Y,\sO_Y)\to \H^*(X,\sO_X)$.
    
    Infinitesimally, there is a correspondence between pro-Artinian cdgas and derived higher stacks, 
    as in \cite{hinstack} and \cite[\S\S 4.5--4.6]{ddt1}, but even this requires a more subtle notion of equivalence than quasi-isomorphism (see \Cref{dglacomp}).
    
    Anyone thinking that taking cohomology isomorphisms sounds as harmless as rational homotopy theory \cite{QRat,Sullivan} should reflect that it would  force the projective spaces $\bP^n$ and the stacks $B\GL_n$ to all be equivalent to points. If you have to do that, please don't try to call it algebraic geometry. While it is possible to embed rational homotopy types (and their generalisation to schematic or pro-algebraic homotopy types) into categories of higher stacks as in \cite{chaff,lurieDAG13}, this tends to be inefficient, with Koszul duality and pro-nilpotent Lie models \cite{htpy,BFMTLieModels}  providing more tractable constructions. 
    \end{digression}
    
    \medskip  
    Simplicial resolutions of schemes will be familiar to anyone who has computed \v Cech cohomology. Given a quasi-compact   scheme $Y$ which is semi-separated \index{semi-separated} (i.e. the diagonal map $Y \to Y \by Y$ is affine
    ), 
    we may take a finite affine cover $U=\coprod_i U_i$ of $Y$, and define the simplicial affine scheme $\check{Y}$ to be the \v Cech nerve $\check{Y}:= \cosk_0(U/Y)$. Explicitly,  \index{Cech nerve@\v Cech nerve}
    \begin{eqnarray*}
    \check{Y}_n= \overbrace{U\by_{Y}U \by_{Y} \ldots  \by_{Y}U}^{n+1}
    = \coprod_{i_0, \ldots, i_n} U_{i_0}\cap \ldots \cap U_{i_n},
    \end{eqnarray*}
    so $\check{Y}_n$ is an affine scheme, and $\check{Y}$ is the unnormalised \v Cech resolution of $Y$.\footnote{The nerve of a groupoid that we saw in \S \ref{hgpdsn} is also a form of \v Cech nerve, since  $B\Gamma \cong \cosk_0(\Ob \Gamma/\Gamma)$ provided all fibre products in the \v Cech nerve are taken as $2$-fibre products of groupoids; here, $\Ob \Gamma$ is the  groupoid of objects of $\Gamma$ and only identity morphisms.}
    
    Given  a quasi-coherent sheaf $\sF$ on $Y$, we can then form a cosimplicial abelian group $\check{C}^n(Y, \sF):= \Gamma(\check{Y}_n, \sF)$, and of course Zariski cohomology is given by\footnote{The Dold--Kan normalisation gives a quasi-isomorphic subcomplex, restricting to terms for which the indices $i_0, \ldots,i_n$ are all distinct. The standard \v Cech complex (with $i_0<\ldots <i_n$) is a quasi-isomorphic quotient of that.}
    \[
    \H^i(Y, \sF) \cong \H^i (\check{C}^{\bt}(Y, \sF), \sum_i (-1)^i\pd^i ).      
    \]

    Likewise, if $\fY$ is a  quasi-compact semi-separated Artin stack, we can choose a presentation $U \to \fY$ with $U$ an affine scheme, and take the \v Cech nerve $\check{Y}:= \cosk_0(U /\fY)$, so \index{coskeleton $\cosk$}
    \[
    \check{Y}_n= \overbrace{U\by_{\fY}U \by_{\fY} \ldots  \by_{\fY}U}^{n+1}.       
    \]
    For example, if $G$ is an affine group scheme acting on an affine scheme $U$, we can take the quotient stack $\fY=[U/G]$ with presentation $U \to [U/G]$, and then we get  $\check{Y}_n\cong U \by G^n$:
    \[
    U \Leftarrow U \by G \Lleftarrow U \by G \by G  \ldots .
    \]    
    Resolutions of this sort were used by Olsson in \cite{olssartin} to study quasi-coherent sheaves on Artin stacks, fixing an error in \cite{Champs}. They also appear extensively in the theory of cohomological descent \cite[Expos\'e vbis]{SGA4.2}. The analogous notion in differential geometry is a differentiable stack, with a specific presentation of the form $U\by_{\fY}U \Rightarrow U$ corresponding to a Lie groupoid; Deligne--Mumford stacks roughly correspond to orbifolds. \index{orbifold}

    This motivates the following questions, which we will address in the remainder of the notes:
    \begin{itemize}
    \item Which simplicial affine schemes correspond to schemes, Artin stacks or Deligne--Mumford stacks in this way?
    
    \item What about higher stacks? 
    
    \item What about derived schemes and derived stacks?
    
    \item How do we then define morphisms?
    
    \item How can we characterise quasi-coherent sheaves in terms of these resolutions?
    \end{itemize}
 
    \subsection{Definitions}
    
    Given a simplicial set $K$ and a simplicial affine scheme \index{simplicial!affine schemes $s\Aff$} $X$ (i.e. a functor $\Delta^{\op} \to \Aff$), there is an affine scheme $M_K(X)$ (the $K$-matching object) \index{matching!object} with the property that for all rings $A$, we have $M_K(X)(A)= M_K(X(A))$, i.e. $\Hom_{s\Set}(K,X(A))$. Explicitly, when $K= \L^{m,k}$ this is given by the equaliser of a diagram
    \[
    \prod_{\substack{0\le i \le m\\i \ne k}} X_{m-1} \rightrightarrows    \prod_{\substack{0\le i<j \le m\\i,j \ne k}} X_{m-2},    
    \]
    and when $K=\pd\Delta^m$ it is given by the equaliser of a diagram
    \[
    \prod_{0\le i \le m} X_{m-1} \rightrightarrows    \prod_{0\le i<j \le m} X_{m-2}  
    \]
    for $m>0$,
    the idea being that we have to specify a value for each face of $\L^{n,k}$ or $\pd\Delta^n$ in such a way that they agree on the overlaps. We also have $M_{\pd\Delta^0}X=M_{\emptyset}X\cong \ast$.
    
    The following definition gives  objects which can be used  to model higher stacks, an idea  originally due to Grothendieck \cite[\S 112, p.463]{pursuingstacks}:

    \begin{definition}\label{npdef}
        Define an \emph{Artin} (resp. \emph{Deligne--Mumford})  \emph{$n$-hypergroupoid}\index{Artin $n$-hypergroupoid}\index{Deligne--Mumford  $n$-hypergroupoid} to be a simplicial affine scheme $X$ for which the    partial matching maps
        $$
        X_m \to M_{\L^{m,k}}( X)
        $$
        are   smooth (resp. \'etale) surjections   for all $m,k$ (i.e. $m \ge 1$ and $0\le k \le m$), and isomorphisms  for all $m>n$ and all $k$.
    \end{definition}
    

    \begin{remark}[Generalisations to other geometries]
    Note that hypergroupoids can be defined in any category containing pullbacks along covering morphisms.   
    
    In \cite{zhu},
    Zhu uses this 
    to define Lie $n$-groupoids \index{Lie $n$-groupoid}
    (taking the category of manifolds, with coverings given by surjective submersions), and hence differentiable $n$-stacks. \index{differentiable $n$-stack}
    A similar approach could be used to define higher topological stacks \index{topological stack} (generalising \cite{Noohi1}), taking surjective local fibrations as the coverings in the category of topological spaces.\footnote{This is in marked contrast to the derived story, there being no non-trivial notion of a derived topological space: see \href{https://mathoverflow.net/questions/291093/derived-topological-stacks}{mathoverflow.net/questions/291093/derived-topological-stacks}.} 
    
    Similar constructions can be made in  non-commutative geometry  \cite{NCstacks} (the main difficulty is in deciding what a surjection should be), and in synthetic differential geometry and analytic geometry.
    In the last two, descent can become more complicated than for algebraic geometry, essentially because affine objects are no longer compact.

    We could also  extend our category to allow formal affine schemes as building blocks, allowing us to model functors  such as  the de Rham stack $X_{\dR}$ of \cite[\S 7]{Simfil}.

    The reasons we take affine schemes as our building blocks in these notes, rather than schemes or algebraic spaces, are twofold:
    firstly, we know what a derived affine scheme is, but the other two are tricky, so this will generalise readily; secondly,
    quasi-coherent sheaves and quasi-coherent cohomology work much  better if we can reduce to affine objects. From a conceptual point of view, it also feels more satisfying to reduce to an algebraic theory in an elementary way.
    \end{remark}
    
    \begin{digression}
    The  main reason for affine objects behaving  differently in algebraic geometry than in other geometries is that the Zariski topology has more points than the analytic and smooth topologies. In analytic (resp. differential) geometry, the 
    EFC- (resp $\C^{\infty}$-) ring\footnote{Note that these are even  finitely presented in these settings, being isomorphic to  $\sO^{\hol}(\Cx)/(\exp z -1)$ and $\C^{\infty}(\R)/(\sin x)$, respectively.} 
    $\Cx^{\N}$ (resp. $\R^{\N}$)
    usually corresponds to the discrete space $\N$. By contrast,  $\Spec (\Cx^{\N})$ is the Stone--\v Cech compactification $\beta \N$ of $\N$, 
    with the corona $\beta \N \setminus \N$ being $\Spec(\Cx^{\N}/\Cx^{\infty}) 
    $,
    for 
    the ideal 
    $\Cx^{\infty}$
    of finite sequences. 
    %
    A solution in the analytic setting is to take the  building blocks to be compact Stein spaces \cite[Proposition 11.9.2]{TaylorCV} endowed with overconvergent functions.  
    
    This seems a lot of effort to exclude points Grothendieck taught us to embrace, 
    so a speculative alternative solution might 
    allow  a compact building block for every EFC-ring, with the space associated to a Stein algebra $\sO^{\hol}(X)$ perhaps being $\im(\beta X \to \Spec( \sO^{\hol}(X)))$  with the quotient topology
    ; 
classical     Stein spaces could still be built as countable nested unions of compact Stein spaces.
    \end{digression}

    \begin{remark}
    Other generalisations of higher stacks exist by taking more structured objects than simplicial sets as the foundation; for details see \cite{balchinAugmentedHAG}.
    \end{remark}

    \begin{examples}\label{hgpdex}\
     \begin{enumerate}
    \item The \v Cech nerve  of a quasi-compact semi-separated scheme, as considered at the start of this section, gives a DM (in fact Zariski) $1$-hypergroupoid. The same construction  for an affine \'etale cover of a  quasi-compact  semi-separated algebraic space also gives a DM $1$-hypergroupoid. (Imposing the extra condition that $X_1 \to M_{\pd \Delta^1}(X)$ be an immersion characterises the simplicial affine schemes corresponding to such nerves.) 
    
    \item  The \v Cech nerve of a quasi-compact  semi-separated DM stack is a DM $1$-hypergroupoid.
    
    \item The \v Cech nerve of a quasi-compact  semi-separated Artin stack is an Artin $1$-hypergroupoid. This applies to $BG$ or $[U/G]$ for smooth affine group schemes $G$ (e.g. $\GL_n$).
    
    \item  Given a smooth affine commutative group scheme $A$ (e.g. $\bG_m$, $\bG_a$), we can form a simplicial affine scheme $K(A,n)$ as follows. First take $A[n]$, regarded as a chain complex of commutative group schemes, then apply the Dold--Kan denormalisation functor to give a simplicial commutative group scheme 
    \[
    K(A,n):= N^{-1}A[n],
    \]
    which is given in level $m$ by $K(A,n)_m \cong A^{\binom{m}{n}}$. This is an example of an Artin $n$-hypergroupoid, and will give rise to an Artin $n$-stack.

    \end{enumerate}
    \end{examples}

    We also have a relative notion: 
    \begin{definition}\label{npreldef}
    Given  $Y\in s\Aff$, define a (relative) \emph{Artin (resp. DM) $n$-hypergroupoid}  over $Y$ to be a morphism $X\to Y$ in $s\Aff$ for which  the   partial matching maps \index{Artin $n$-hypergroupoid!relative} \index{Deligne--Mumford  $n$-hypergroupoid!relative}
    $$
    X_m \to M_{\L^{m,k}} (X)\by_{M_{\L^{m,k}} (Y)}Y_m 
    $$
    are  smooth (resp. \'etale) surjections for all  $k,m$, and are isomorphisms for all $m>n$ and all $k$. 
    \end{definition}
    
    The following gives rise a notion of equivalence for hypergroupoids:   
    
    \begin{definition}\label{npreldeftriv}
    Given     $Y \in s\Aff$, define a \emph{trivial Artin (resp. DM)  $n$-hypergroupoid}  over $Y$ to be a morphism $X\to Y$ in  $s\Aff$ 
    for which  the matching maps \index{Artin $n$-hypergroupoid!trivial} \index{Deligne--Mumford  $n$-hypergroupoid!trivial}
    $$
    X_m \to M_{\pd \Delta^m} (X)\by_{M_{\pd \Delta^m} (Y)}Y_m 
    $$
    are   smooth (resp. \'etale) surjections  for all  $m\ge 0$, and are isomorphisms for all $m\ge n$.
    
    When $n=\infty$,  this is called a smooth (resp. \'etale) simplicial hypercover. \index{simplicial!hypercover}
    \end{definition}
    Note in particular that the $m=0$ term above implies that $X_0 \to Y_0$ is a smooth (resp. \'etale) surjection.
    
    \begin{example}
    
    Let $\{V_j\}_j$  and $\{U_i\}_i$ be finite  open affine covers  of a semi-separated quasi-compact scheme $Y$, and set $V:=\coprod_j V_j$ and $U:=\coprod_i U_i$. Then for  $W:=U\by_YV= \coprod_{i,j} U_i \cap V_j$, the morphisms
    \begin{align*}
    \cosk_0( W/Y) &\to \cosk_0(U/Y) \\
    (W \Leftarrow W\by_YW \Lleftarrow W\by_YW\by_YW \ldots) &\to (U \Leftarrow U\by_YU \Lleftarrow U\by_YU\by_YU \ldots)
    \end{align*}
    and $  \cosk_0( W/Y) \to \cosk_0(V/Y)$ are
    trivial relative DM (in fact Zariski) $1$-hypergroupoids.
    \end{example}
    
    \begin{example}\label{calcHomSch}
    If we think about how we calculate morphisms between schemes or algebraic spaces, we first take affine covers $U$ of $X$ and $V$ of $Y$, then seek a refinement $U'$ of $U$ and a map $U' \to V$ such that values on overlaps agree. 
    
    A morphism of simplicial affine schemes from the \v Cech nerve $\check{X}:= \cosk_0(U/X)$ to $\check{Y}:= \cosk_0(V/Y)$ is just such a map in the case $U'=U$, i.e. a map  $X \to Y$ which admits a lift to a map $U \to V$.  We can then account for refinements $U'$ of $U$ by replacing $\hat{X}$ with trivial Zariski or DM $1$-hypergroupoids   $X' \to \check{X}$ over $\check{X}$, giving a filtered colimit expression
    \[
    \Hom_{\Sch}(X,Y)= \LLim_{X' \to \check{X}}\Hom_{s\Aff}(X', \check{Y}).       
    \]
    \end{example}
    
    \subsection{Main results}
    
    For our purposes, we can use the following as the definition of an $(n-1)$-geometric stack. It is a special case of \cite[Theorem \ref{stacks2-bigthm}]{stacks2}.
    
    \begin{warning}[Terminology]
    Beware that \cite{stacks2,stacksintro} use the terminology from an earlier version of \cite{hag2} in which the indices were $1$ higher for strongly quasi-compact objects, so that $n$-geometric in \cite{stacks2} corresponds to $(n-1)$-geometric in the final version of  \cite{hag2}, whereas our terminology in these notes conforms with the latter.
    \end{warning}

    \begin{theorem}\label{bigthm}\index{n-geometric@$n$-geometric!Artin stack} \index{higher Artin stack}
        The homotopy category of  strongly quasi-compact\index{strongly quasi-compact}
        $(n-1)$-geometric  Artin  stacks is given by taking the full subcategory of $s\Aff$ consisting of  Artin  $n$-hypergroupoids $X$, and formally inverting the trivial relative   Artin  $n$-hypergroupoids $X \to Y$. 
        
        \medskip
        In fact, a model for the $\infty$-category of strongly quasi-compact $(n-1)$-geometric   Artin  stacks is given by the relative category $(\C,\cW)$ with $\C$ the full subcategory of $s\Aff$ on  Artin  $n$-hypergroupoids $X$ and $\cW$ the subcategory of trivial relative Artin $n$-hypergroupoids  $X \to Y$.
        
        \medskip 
        The same results hold true if we substitute ``Deligne--Mumford'' for ``Artin''  throughout.\index{n-geometric@$n$-geometric!Deligne--Mumford stack}\index{higher Deligne--Mumford stack}
    \end{theorem}
    In particular, this means we obtain the simplicial category of such stacks  by simplicial localisation of Artin/DM $n$-hypergroupoids at the class of  trivial relative Artin/DM $n$-hypergroupoids.

    \begin{examples}\label{hgpdex2}
        We have already seen several fairly standard examples of hypergroupoid resolutions in Examples \ref{hgpdex}, and we now describe some more involved constructions, for those who are interested.
    \begin{enumerate}
    
    \item
    The method of split resolutions from \cite[\S 6.2]{Hodge3} can be
    adapted to give resolutions for schemes, algebraic spaces, and even Deligne--Mumford $n$-stacks, but not for Artin stacks because the diagonal of a smooth morphism is not smooth.
    
    \item Take a smooth group scheme $G$ which is quasi-compact and  semi-separated, but not affine, for instance an elliptic curve. The simplicial scheme $BG$ (given by $(BG)_m=G^m$) is an Artin $1$-hypergroupoid in non-affine schemes which resolves the classifying stack $BG^{\sharp}$ of $G$.
    
    Next, we have to take a finite affine cover $\{U_i\}_{i \in I}$ for $G$ and set $U=\coprod_i U_i$, writing $p \co U \to G$. To proceed further, we introduce the simplicial schemes $U^{\Delta_r^{\bt}}$ ({\it not} to be confused with the simplicial schemes $U^{\Delta^r}$ we meet in \S \ref{dcotsn}), which are given by $(U^{\Delta_r^{\bt}})_m:= U^{\Delta^m_r}\cong U^{\binom{m+r+1}{m}}$, and have the property that maps $X \to U^{\Delta_r^{\bt}}$ correspond to maps $X_r \to U$.
    
    We can then form an affine Artin  $2$-hypergroupoid resolving $BG^{\sharp}$ by taking
    \[
    BG\by_{G^{\Delta_1^{\bt}}}U^{\Delta_1^{\bt}}
    \]       
    Explicitly, this looks like
    \[
    p^{-1}(e) \Leftarrow p^{-1}(e)^2 \by U \Lleftarrow p^{-1}(e)^3 \by \{(x,y,z) \in U^3 ~:~ p(x)p(y)=p(z)\} \quad \cdots ,
    \]
    with the affine scheme in level $m$ being 
    $$
    \{ \uline{x} \in \prod_{0 \le i \le j \le m} U ~:~ p(x_{ij})p(x_{jk})=p(x_{ik}) ~\forall i \le j \le k\}$$
    (in particular, $p(x_{ii})=e$ for all $i$).
    
    \item As a higher generalisation of the previous example, if $G$ is moreover commutative, then we can form the simplicial scheme $K(G,n)=N^{-1}G[n]$, which is an Artin $n$-hypergroupoid in non-affine schemes, and then form the resolution
    \[
    K(G,n)\by_{G^{\Delta_n^{\bt}}}U^{\Delta_n^{\bt}}
    \]
    to give an affine Artin  $(n+1)$-hypergroupoid.
    \end{enumerate}
    Most examples are however not so simple, and the algorithm from \cite{stacks2} takes $2^n -1$ steps to construct an $n$-hypergroupoid resolution in general.
    \end{examples}
    
    \bigskip    An $\infty$-stack over $R$ is a functor $\Alg_R \to s\Set$ satisfying various conditions, so we need to associate such functors to Artin/DM hypergroupoids. The solution (not explicit) is to take
    \[
    X^{\sharp}(A)= \oR\Map_{\cW}(\Spec A, X),       
    \]
    where $\oR\Map_{\cW}$ is the right-derived functor of $\Hom_{s\Aff}$, with respect to trivial  Artin/DM $n$-hypergroupoids.

    \subsubsection{Morphisms}\label{nstackmorphismsn}
    
    We now give a more explicit description of mapping spaces, generalising Example \ref{calcHomSch}.
    
    \begin{definition}
        Define the \emph{simplicial $\Hom$ functor}\index{simplicial!Hom@$\Hom$ functor}    on    simplicial affine schemes by letting $\HHom_{s\Aff}(X,Y)$ be the simplicial set given by 
        \[
        \HHom_{s\Aff}(X,Y)_n:= \Hom_{s\Aff}(\Delta^n \by X, Y),        
        \]
        where $(X\by \Delta^n)_i$ is given by the coproduct of $\Delta^n_i$ copies of $X_i$.
    \end{definition}
    
    %
    %

    \begin{definition}\label{univcoverdef}
        Given $X \in s\Aff$, say that an inverse system $\tilde{X}= (\tilde{X}(0) \la \tilde{X}(0) \la \ldots) $ (possibly transfinite) over $X$ is an \emph{$n$-Artin} (resp. \emph{$n$-DM}) \emph{universal cover}\footnote{Think of this as being somewhat similar to  a universal cover of a topological space.} if:\index{universal cover}
        \begin{enumerate}
            \item the morphisms $\tilde{X}_0 \to X$ and $\tilde{X}_{a+1}\to \tilde{X}_a$ are trivial Artin (resp. DM) $n$-hypergroupoids;
            \item for any limit ordinal $a$, we have $\tilde{X}(a) = \Lim_{b<a}\tilde{X}(b)$;
            \item for every $a$ and every trivial  Artin (resp. DM)  $n$-hypergroupoid $Y \to \tilde{X}(a)$, there exists $b \ge a$ and a
            factorisation
            \[
            \xymatrix{ \tilde{X}(b) \ar[rr] \ar@{-->}[dr] && \tilde{X}(a)\\  & Y. \ar[ur]  }
            \]
        \end{enumerate}
    \end{definition}
    
    These always exist, by \cite[Proposition \ref{stacks2-procofibrant}]{stacks2}. Moreover,  \cite[Corollary \ref{stacks2-procofibrantet}]{stacks2} shows that every   $n$-DM universal cover is in fact an $n$-Artin universal cover.
    
    \begin{definition}\label{hyperdef}
        Given an  Artin $n$-hypergroupoid $Y$ and $X \in s\Aff$, define 
        \[
        \HHom_{s\Aff}^{\sharp}(X,Y):= \LLim_i  \HHom_{s\Aff}(\tilde{X}(i), Y),
        \]
        \index{HomsAffsharp@$\HHom_{s\Aff}^{\sharp}$}
        where the colimit runs over the objects $\tilde{X}(i)$ of any $n$-Artin universal cover $\tilde{X} \to X$. 
    \end{definition}

    The following is \cite[Corollary \ref{stacks2-duskinmor}]{stacks2}:
    \begin{theorem}\label{duskinmor}
    If $X \in s\Aff$ and $Y$ is an  Artin $n$-hypergroupoid, then the derived $\Hom$ functor on the associated hypersheaves (a.k.a. $n$-stacks) $X^{\sharp}, Y^{\sharp}$  is given 
    by
    \[
    \oR \Map(X^{\sharp}, Y^{\sharp}) \simeq  \HHom_{s\Aff}^{\sharp}(X,Y).   
    \]      
    \end{theorem}
    
    \begin{remarks}\label{higherrks}
    Given a ring $A$, set $X=\Spec A$ and note that $\HHom_{s\Aff}^{\sharp}(X,Y)= Y^{\sharp}(A)$, the hypersheafification of the functor $Y: \Alg_R \to s\Set$, so we can take Definition \ref{hyperdef} as a definition of hypersheafification for Artin hypergroupoids, giving an explicit description of $Y^{\sharp}$.
    The $n=1$ case should be familiar as the standard definition of sheafification.      
    
    Between them,  Theorems \ref{bigthm} and \ref{duskinmor} recover the simplicial category  of strongly quasi-compact $(n-1)$-geometric Artin stacks, with Theorem \ref{bigthm}
    giving the objects and Theorem \ref{duskinmor} the morphisms. We could thus take those theorems to be a definition of that simplicial category. 
    \end{remarks}

    
    %
    %
    
    \subsubsection{Truncation considerations} \label{nstacktruncationsn}
    
    \begin{remark}
    Recall from Properties \ref{hypergroupoidprops} that an $n$-hypergroupoid $Y$ is determined by $Y_{\le n+1}$, and in fact that $Y\cong \cosk_{n+1}Y$. This implies that   
    \[
    \Hom_{s\Aff}(X,Y) \cong  \Hom_{s_{\le n}\Aff}(X_{\le n+1}, Y_{\le n+1})
    \]
    for any $X$, where $s_{\le n}\Aff$ is the category of \emph{$n$-truncated simplicial schemes},\index{truncated simplicial schemes $s_{\le n}\Aff$}\footnote{This terminology should not be confused with the notion of $n$-truncated stacks in \Cref{truncated}.}
    i.e. $\Aff$-valued contravariant functors from the full subcategory of $\Delta$ on objects $\oO, \ldots \on$.
    
    Thus
    \[
    \Hom_{s\Aff}(X\by \Delta^m,Y) \cong  \Hom_{s_{\le n}\Aff}( (X \by \Delta^m)_{\le n+1}, Y_{\le n+1}),
    \]
    which greatly simplifies the calculation of $\HHom(X,Y)$.
    \end{remark}
    
    \begin{example}
    A trivial  $0$-hypergroupoid is just an isomorphism, so the $0$-DM universal cover of any $X$ is just $X$.
    For  $Y$ an affine scheme and any $X \in s\Aff$, this means that
    \begin{align*}
    \HHom_{s\Aff}^{\sharp}(X,Y) &= \HHom_{s\Aff}(X,Y),\\
    &=\HHom_{s_{\le 1}\Aff} (X_{\le 1},Y_{\le 1}):
    \end{align*}
    \[
    \xymatrix{ X_0 \ar@{.>}[r]\ar[d] & \ar@<1ex>[l]^{\pd_0} \ar@<-1ex>[l]_{\pd_1} X_1 \ar[d]
    \\
    Y \ar@{.>}[r]  & \ar@<1ex>[l]^{\id} \ar@<-1ex>[l]_{\id} Y. 
    }
    \]

    In level $0$, this is just 
    the equaliser of $\Hom(X_0,Y) \Rightarrow \Hom(X_1,Y)$. Thus  $\Hom_{s\Aff}(X,Y)=\Hom(\pi_0^{\Aff}X,Y)$, where $\pi_0^{\Aff}X$ is the equaliser of $X_1 \Rightarrow X_0$ in the category of affine schemes, which when $X$ is the \v Cech nerve of a scheme, algebraic space, or even algebraic stack $Z$ is given by  $\pi_0^{\Aff}X \cong \Spec \Gamma(Z,\sO_Z)$.

    For higher $n$, we get 
    \[
    \HHom(X,Y)_n \cong \Hom(\pi_0^{\Aff}(X\by \Delta^n),Y)\cong\Hom(\pi_0^{\Aff}X,Y) ,
    \]
    so we have a simplicial set with constant simplicial structure:
    \[
    \HHom_{s\Aff}^{\sharp}(X,Y)\cong \Hom(\pi_0^{\Aff}X,Y).
    \]
    \end{example}

    %
    %
    
    \begin{example}
    Take a $1$-hypergroupoid $Y$ and  an affine scheme $U$, then look at $\HHom_{s\Aff}^{\sharp}(U,Y)$. Relative $1$-hypergroupoids over $U$ are just \v Cech nerves $\cosk_0(U'/U) \to U$ for \'etale surjections $U' \to U$:
    \[
    (U' \Leftarrow U'\by_UU' \Lleftarrow U'\by_UU'\by_UU' \ldots)\to U.
    \]
    Then
    \begin{align*}
    \HHom_{s\Aff}^{\sharp}(U,Y) &= \LLim_{U'}\HHom_{s\Aff}(\cosk_0(U'/U),Y),\\
    &= \LLim_{U'} \HHom_{s_{\le 2}\Aff}(\cosk_0(U'/U)_{\le 2},Y_{\le 2}),
    \end{align*}
    so an element $f$ of the mapping space consists of maps $f_i$ making the  diagram
    \[
    \xymatrix{ U' \ar@{.>}[r]\ar[d]^{f_0} & \ar@<1ex>[l] \ar@<-1ex>[l] U'\by_UU' \ar@{.>}@<0.75ex>[r] \ar@{.>}@<-0.75ex>[r] \ar[d]^{f_1}  & \ar[l] \ar@/^/@<0.5ex>[l] \ar@/_/@<-0.5ex>[l] 
    U'\by_UU'\by_UU'\ar[d]^{f_2}  
    \\
    Y_0 \ar@{.>}[r] & \ar@<1ex>[l] \ar@<-1ex>[l] Y_1 \ar@{.>}@<0.75ex>[r] \ar@{.>}@<-0.75ex>[r]  & \ar[l] \ar@/^/@<0.5ex>[l] \ar@/_/@<-0.5ex>[l] 
    Y_2
    }
    \]
    commute.
    
    When $Y$ is a \v Cech nerve  $\cosk_0(V/Z)$ (for $V$ a cover of a scheme, algebraic space or algebraic stack $Z$), the target becomes
    \[
    \xymatrix@1{
    V \ar@{.>}[r] & \ar@<1ex>[l] \ar@<-1ex>[l] V\by_ZV \ar@{.>}@<0.75ex>[r] \ar@{.>}@<-0.75ex>[r]  & \ar[l] \ar@/^/@<0.5ex>[l] \ar@/_/@<-0.5ex>[l] 
    V\by_ZV\by_ZV,
    } 
    \]
    giving the expected data for a morphism $f$, as in the expression of \cite[Lemma 3.2]{Champs} for  $U$-valued points of the stack $Z\simeq [V\by_ZV \Rightarrow V]$. Note that when $Y$ is an algebraic space, $f_0$ determines $f_1$ because then the map $Y_1 \to M_{\pd\Delta^1}(Y)$ (i.e. $V\by_ZV  \to V \by V$) is an immersion, so even $f_2$ then becomes redundant. 
    \end{example}
    
    \begin{example}\label{BGsharpex}
    
    
    \smallskip    
    Now let $X$ be a semi-separated scheme and set $\check{X}:= \cosk_0(U/X)$, for $U \to X$ an \'etale cover. Let $Y=BG$, for $G$ a smooth affine group scheme (e.g. $\GL_n$, $\bG_m$ or $\bG_a$ but not $\mu_p$ in characteristic $p$).
    Set
    \[
    \check{\CC}^n(X,G):= \Gamma(\check{X}_n,G(\sO_X)) = \Hom(\overbrace{U\by_X \ldots \by_XU}^{n+1}, G);     
    \]
    for $\bG_a$ and $\bG_m$, this gives $\Gamma(\check{X}_n,\O_X)$ and $\Gamma(\check{X}_n, \O_X^*)$ respectively. 
    
    \smallskip    Then 
    $
    \Hom_{s\Aff}(\check{X}, BG)
    $
    is isomorphic to 
    \[
    \check{Z}^1(U,G)= \{\omega \in \check{C}^1(U,G)\,:\, \pd^2\omega \cdot \pd^0\omega = \pd^1\omega \in \check{C}^2(U,G)\};
    \]
    in other words, $\omega$ satisfies the cocycle condition,
    so determines a $G$-torsor $P$ on $X$ with  $P\by_XU \cong G\by U$. Meanwhile, 
    \[
    \Hom_{s\Aff}(\Delta^1 \by \check{X}, BG)\cong \{(\omega_0,g,\omega_1) \in \check{Z}^1(U,G)\by \check{C}^0(U,G)\by \check{Z}^1(U,G)~:~ 
    \pd^1(g)\cdot \omega_0=\omega_1 \cdot \pd^0(g)\},
    \]
    so $g$ is a  gauge transformation between $\omega_0$ and $\omega_1$; on the corresponding $G$-torsors, this amounts to giving an isomorphism $\phi \co P_1 \cong P_2$ (which need not respect the trivialisation on $U$).
    
    \smallskip Thus  $\HHom_{s\Aff}(\check{X},BG)$ is the nerve of the groupoid $[\check{Z}^1(U,G)/\check{C}^0(U,G)]$ of $G$-torsors  on $X$ which become trivial on pullback to $U$, and   passing to the colimit over all \'etale covers $U'$ of $X$, we get that $\HHom^{\sharp}_{s\Aff}(\check{X},BG)$ is equivalent to the nerve of the groupoid of all \'etale $G$-torsors on $X$, as expected. 
    \end{example}
    
    \begin{example}
    For $E$ an elliptic curve, the \v Cech complex of $BE$ is an  Artin $2$-hypergroupoid, but  $\Map(X, (BE)^{\sharp})$ still classifies $E$-torsors on $X$.      
    \end{example}
    
    \begin{examples}
    Example  \ref{BGsharpex} tells us that
    \[
    \pi_i(BG)^{\sharp}(X)= \varinjlim_{X'} \H^{1-i}(X', G) = \H^{1-i}_{\et}(X,G)       
    \]
    for $i=0,1$, where $X' \to X$ runs over \'etale hypercovers.
    
    If $A$ is a smooth commutative affine group scheme (such as $\bG_m, \bG_a$ or $\mu_n$ for $n^{-1} \in R$), we can generalise this to the higher $K(A,n)$. We have
    \[
    \Hom_{s\Aff}(\check{X}, K(A,n))\cong \z^nN\check{C}(X,A),       
    \]
    and using a path object for $K(A,n)$, we then get
    \[
    \pi_i\HHom_{s\Aff}(X,K(A,n)^{\sharp})\cong \pi_iK(A,n)^{\sharp}(X)\cong \H^{n-i}_{\et}(X,A),      
    \]
    which is reminiscent of Brown representability in topology. Note that for $A=\bG_a$, this is just $ \H^{n-i}(X,\sO_X)$.
    \end{examples}

    \begin{definition}\label{truncated}
        We say that a functor $F \co \Alg_k \to s\Set$ is \emph{$n$-truncated}\index{n-truncated@$n$-truncated} if $\pi_iF(A)=0
        $
        for all $i>n$ and all $A \in \Alg_k$.
    \end{definition}
    For a hypergroupoid $X$, $n$-truncatedness of $X^{\sharp}$ amounts to saying that the  the maps
    \[
    X_i \to M_{\pd\Delta^i}(X)
    \]
    are monomorphisms (i.e. immersions) of affine schemes for all $i>n$.
    
    \subsubsection{Terminological warnings}
    \begin{warning}\label{cflurie}
        Beware that
        there are slight differences in terminology between \cite{hag2} and \cite{lurie}. In the former,  affine schemes are $(-1)$-representable, so arbitrary schemes might only be $1$-geometric, while Artin stacks are $0$-geometric stacks if and only if they have affine diagonal. In the latter, algebraic spaces are $0$-stacks.

        An $n$-stack $\fX$ in the sense of \cite{lurie} is an $n$-truncated functor which is $m$-geometric for some $m$. 
        
        It follows easily from Property \ref{hypergroupoidprops}.(\ref{truncate}) that every $n$-geometric stack in \cite{hag2} is $(n+1)$-truncated;
        conversely,  any  $n$-truncated stack $\fX$ is $(n+1)$-geometric.\footnote{This is {\it not} a typo for $(n-1)$; a non-semi-separated scheme such as $\bA^2\cup_{\bA^2\setminus\{0\}}\bA^2$ is $0$-truncated but $1$-geometric, while an affine scheme is $0$-truncated and $(-1)$-geometric.} Any Artin stack with affine diagonal (in particular any separated algebraic space) is $0$-geometric.
        
        The conditions can be understood in terms of higher diagonals. If for simplicial sets (equivalently, topological spaces) $K$, we write $\fX^{h K}$ for the functor $A \mapsto \oR\Map_{s\Set}(K,\fX(A))$, then we get $\fX^{h S^n} \simeq \fX\by^h_{\fX^{h S^{n-1}}}\fX$, and we think of the diagonal map $\fX \to \fX^{h S^n}$ as the $n$th higher diagonal. \index{higher diagonal}
        
        Being $n$-geometric then amounts to saying that the higher diagonal morphism $\fX \to \fX^{h S^n}$ to the iterated loop space \index{loop space!iterated}
        is affine (where $S^{-1}=\emptyset$ and $S^0=\{-1,1\}$), while being $n$-truncated amounts to saying that the morphism $\fX \to \fX^{S^{n+1}}$ is an equivalence.
        
        If we took quasi-compact, quasi-separated algebraic spaces instead of  affine schemes in Definition \ref{npdef}, then Theorem \ref{bigthm} would adapt to give a characterisation of $n$-truncated Artin stacks. Our main motivation for using affine schemes as the basic objects is that they will be easier to translate to a derived setting.
    \end{warning}  
    
    \begin{remark}\label{qcmpt}
    The strong quasi-compactness \index{strongly quasi-compact} condition in Theorem \ref{bigthm} is terminology from \cite{toenvaquie} which amounts to saying that the objects are quasi-compact, quasi-separated, and so on (all the higher diagonals $\fX \to \fX^{h S^n}$ in  \Cref{cflurie}
    are quasi-compact). We can drop that assumption if we expand our category of building blocks by allowing arbitrary disjoint unions of affine schemes.
    \end{remark}

    \begin{warning}[More obscure]
        The situation is further complicated by earlier versions of \cite{hag2} using higher indices, and the occasional use as in \cite{toencrm} of $n$-algebraic stacks,  intermediate between $(n-1)$-geometric and $n$-truncated stacks. In \cite{toenseattle}, $n$-geometric Artin stacks are simply called $n$-Artin stacks, and distinguished from Lurie's Artin $n$-stacks. Finally, beware that \cite{lurietannaka} and its derivatives refer to geometric stacks, by which they mean $0$-geometric stacks (apparently in the belief this is standard algebro-geometric terminology for semi-separated); incidentally, the duality there also does not carry its customary meaning, not being an equivalence.    
    \end{warning}

\subsection{Quasi-coherent sheaves and complexes}\label{qucohsn1}
    
 The notion of quasi-coherent sheaves on a scheme, and complexes of such sheaves is generalised to higher stacks  in \cite{hag2,lurie}. Rather than repeat those definitions, we now give more concrete characterisations of the same concepts.
     
    The following is part of  \cite[Corollary \ref{stacks-qcohequiv}]{stacks2}:
    \begin{proposition}
        For an Artin $n$-hypergroupoid $X$, giving a quasi-coherent sheaf (a.k.a Cartesian sheaf) 
        on the associated $n$-geometric stack $X^{\sharp}$ is equivalent to giving a Cartesian quasi-coherent sheaf on $X$, i.e.:\index{Cartesian!module}
        \begin{enumerate}
            \item a quasi-coherent sheaf $\sF^n$ on $X_n$ for each $n$, and
            \item isomorphisms $\pd^i\co \pd_i^*\sF^{n-1} \to \sF^n$ for all $i$ and $n$, satisfying the usual cosimplicial identities
            \[
            \pd^j\pd^i=\pd^{i+1}\pd^j \qquad j\le i.
            \]
        \end{enumerate}
    \end{proposition}
    
    
    It is not too hard to see that $\sF$ is determined by the sheaf $\sF^0$ on $X_0$  and the isomorphism $\theta \co \pd_0^*\sF^0 \cong \pd_1^*\sF^0$ on $X_1$, which satisfies the  cocycle condition $(\pd_2^*\theta) \circ (\pd_0^*\theta) = \pd_1^*\theta$  on $X_2$.
    
    This has the following generalisation (again, \cite[Corollary \ref{stacks-qcohequiv}]{stacks2}):
    \begin{proposition}\label{hCartprop}
        For an Artin $n$-hypergroupoid $X$, giving a quasi-coherent complex on the associated $n$-geometric stack $X^{\sharp}$ is equivalent to giving a homotopy-Cartesian module, i.e.:\index{homotopy-Cartesian!module}
        \begin{enumerate}
            \item a complex $\sF_{\bt}^n$ of quasi-coherent sheaves  on $X_n$ for each $n$, 
            \item quasi-isomorphisms $\pd^i\co \pd_i^*\sF_{\bt}^{n-1} \to \sF_{\bt}^n$ for all $i$ and $n$, and
            \item morphisms  $\sigma^i\co \sigma_i^*\sF_{\bt}^{n+1} \to \sF_{\bt}^n$ for all $i$ and $n$,
        \end{enumerate}
        where the operations $\pd^i, \sigma^i$ satisfy the usual cosimplicial identities,\footnote{\label{nosenote}These identities are all required to hold on the nose, so that for instance $\sigma^i \circ \pd^i $ must equal the identity, not just be homotopic to it.}
        and a morphism $\{\sE^n_{\bt}\}_n \to \{\sF^n_{\bt}\}_n$ is  a weak  equivalence if the maps $\sE^n_{\bt} \to \sF^n_{\bt}$ are all quasi-isomorphisms. 
    \end{proposition}
    
    Note that because the maps $\pd_i$ are all smooth, they are flat, so we do not need to left-derive the pullback functors $\pd_i^*$. Also  note that because the maps $\pd^i$ are only quasi-isomorphisms, they do not have inverses, which is why we have to include the morphisms $\sigma^i$ as additional data. The induced morphisms $\sigma^i \co \oL \sigma_i^*\sF_{\bt}^{n+1} \to \sF_{\bt}^n$ are then automatically quasi-isomorphisms. We can also rephrase the quasi-isomorphism condition as saying that $\{\H_i(\sF_{\bt}^n)\}_n$ is a Cartesian sheaf on $X$ for all $i$.

    
    \begin{remark} 
        Inclusion gives a canonical functor $\cD(\QCoh(\sO_X)) \to \Ho(\mathrm{Cart}(\sO_X))$ from the derived category of complexes of quasi-coherent sheaves on $X$  to the homotopy category of homotopy-Cartesian complexes, and this is an equivalence when $X$ is 
        a quasi-compact semi-separated scheme, by \cite[Theorem 4.5.1]{huettemann}.
        
        Under the same hypotheses, $\cD(\QCoh(\sO_X))$  is in turn  equivalent to the derived category  $\cD_{\QCoh}(\sO_X)$ of complexes of sheaves of $\sO_X$-modules with quasi-coherent homology {\it sheaves}, by \cite[Corollary 5.5]{bokstedtneeman} (or just \cite[Exp. II, Proposition 3.7b]{sga6} if $X$ is Noetherian).
        
        To compare  $\Ho(\mathrm{Cart}(\sO_X))$ and  $\cD_{\QCoh}(\sO_X)$ directly, observe that since  sheafification is exact, it gives  us a functor $\Ho(\mathrm{Cart}(\sO_X))\to \cD_{\QCoh}(\sO_X)$. This is always an equivalence, with quasi-inverse given by the derived right adjoint, sending each  complex of sheaves to an injective resolution (or more precisely, to a fibrant replacement in the injective model structure).
        
        
        
    \end{remark}
    
\subsubsection{Inverse images}\label{Lfstarsn}
    
    Given a morphism $f\co X \to Y$ of Artin $n$-hypergroupoids, inverse images of quasi-coherent sheaves are easy to compute: if $\sF$ is a Cartesian quasi-coherent sheaf on $Y$, then the sheaf $f^*\sF$ on $X$ given levelwise by $(f^*\sF)^n:= f_n^*\sF^n$ is also Cartesian. Similarly, if  $\sF_{\bt}$ is a homotopy-Cartesian quasi-coherent complex on $Y$, then the there is a complex $\oL f^*\sF_{\bt}$ on $X$, given levelwise by $(\oL f^*\sF_{\bt})^n \simeq \oL f_n^*\sF_{\bt}^n$, which is also homotopy-Cartesian.
    \index{derived pullback functor $\oL f^*$}
    
\subsubsection{Derived global sections}
    
    Direct images are characterised as right adjoints to inverse images, but are much harder to construct, because taking $f_*$ levelwise destroys the Cartesian property. There is an explicit description in  \cite[\S \ref{stacks-directsn}]{stacks2}  of the derived direct image functor $\oR f_*^{\cart}$.
    
    The special case of derived global sections is much easier to describe. If $\sF_{\bt}$ is homotopy-Cartesian on $X$, then $\oR\Gamma(X^{\sharp},\sF_{\bt})$ is just the {\it product} total complex of the double complex
    \[
        \Gamma(X_0,\sF_{\bt}^0) \xra{\pd^0-\pd^1} \Gamma(X_1,\sF_{\bt}^1)\xra{\pd^0-\pd^1+\pd^2} \Gamma(X_2,\sF_{\bt}^2) \to \ldots 
    \]
    (or equivalently of its Dold--Kan normalisation, restricting to $\bigcap_i \ker \sigma^i$ in each level).
    
\subsection{Hypersheaves}
    
    \begin{definition}
        A functor 
        \[
        \sF \co \Aff^{\op} \to \C
        \]
        to a model category (or more generally an $\infty$-category) $\C$
        is said to be an \emph{\'etale hypersheaf} \index{hypersheaf} if for any trivial DM $\infty$-hypergroupoid (a.k.a.   \'etale hypercover) $U_{\bt}\to X$, the map
        \[
        \sF(X) \to \ho\Lim_{n \in \Delta} \sF(U_n)
        \]
        is a weak equivalence, and for any  $X,Y$ the map
        \[
        \sF(X \sqcup Y) \to \sF(X) \by \sF(Y)
        \]
        is a weak equivalence, with $\sF(\emptyset)$ contractible.
        
    \end{definition}
    %
    Note that  since $U_{\bt}$ is a simplicial scheme, contravariance means that the functor  $n \mapsto \sF(U_n)$ is a {\it cosimplicial} diagram in $\C$ (i.e. a functor $\Delta  \to \C$). On the category $\C^{\Delta}$ of cosimplicial objects, there is a functor $\pi^0$ which sends a cosimplicial object $A^{\bt}$ to the equaliser of $\pd^0,\pd^1\co A^0 \to A^1$, and then 
    $\holim _{n \in \Delta}$ is its right-derived functor $\oR \pi^0$.

    \begin{examples}\label{categs}\
        \begin{enumerate}
            \item 
            If $\C$ is a category with trivial model structure (all morphisms are fibrations and cofibrations, isomorphisms are the only weak equivalences), then $\ho\Lim_{n \in \Delta} A^n=\pi^0A$, 
            so  hypersheaves in $\C$ are precisely $\C$-sheaves.
            
            \item  
            For the category $\Ch$ of unbounded chain complexes and $V \in \Ch^{\Delta}$, 
            \[
            \ho\Lim_{n \in \Delta}V^n\simeq \Tot^{\Pi}(V^0 \xra{\pd^0-\pd^1} V^1 \to \ldots), 
            \]
            the product total complex of the associated double complex, with reasoning as in \S \ref{sympsn}. 
            
            \item\label{categsSet}
            On the category $(s\Set)^{\Delta}$ of cosimplicial simplicial sets, $ \ho\Lim_{n \in \Delta}$ is the functor $\oR\Tot_{s\Set} $, where $\oR$ means ``take (Reedy) fibrant replacement first'', and $\Tot_{s\Set} $ is the total complex $\Tot$ of a cosimplicial space defined in \cite[\S VII.5]{sht};   see \Cref{fig:TotsSet}.
            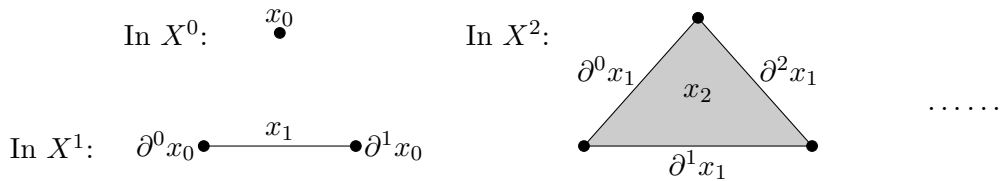
\begin{figure}[H]
            \begin{tikzpicture}[scale=1]
            \newcommand\ptSize{2pt}
            \node at (1.5, 1.5) {In $X^0$:};
            \draw[fill=black] (3, 1.5) circle (\ptSize); 
            \node at (3, 1.7) {$x_0$};
            \node at (0,0) {In $X^1$:}; 
            \draw[fill=black] (2,0) circle (\ptSize);
            \draw[fill=black] (4,0) circle (\ptSize);
            \draw (2, 0) -- (4, 0);
            \node at (1.5, 0) {$\pd^0x_0$};
            \node at (4.5, 0) {$\pd^1x_0$};
            \node at (3, 0.2) {$x_1$};
            
            \draw[fill=black!20] (8.5, 1.7) -- (7, 0) -- (10, 0) -- cycle;
            \node at (6, 1.5) {In $X^2$:};
            \draw[fill=black] (7, 0) circle (\ptSize);
            \draw[fill=black] (10, 0) circle (\ptSize);
            \draw[fill=black] (8.5,1.7) circle (\ptSize);
            \node at (8.5, -0.25) {$\pd^1x_1$};
            \node at (7.3, 1) {$\pd^0x_1$};
            \node at (9.7, 1) {$\pd^2x_1$};
            \node at (8.5, 0.7) {$x_2$};
            \node at (12, 0.5) {$\ldots\ldots$};
            \end{tikzpicture}
            \caption[An element $x$ of $\Tot_{s\Set}X^{\bt}$.]{an element $x$ of $\Tot_{s\Set}X^{\bt}$.\footnotemark}
            \label{fig:TotsSet}
            \end{figure}
            \footnotetext{Note that the vertices of the $2$-simplex match up because $\pd^0\pd^0x_0=\pd^1\pd^0x_0$,\, $\pd^0\pd^1x_0 =\pd^2\pd^0x_0$ and $\pd^1\pd^1x_0=\pd^2\pd^0x_0$.}
            
            Explicitly, 
            $$
            \Tot_{s\Set} X^{\bt} =\{ x \in \prod_n (X^n)^{\Delta^n}\,:\, \pd^i_Xx_n = (\pd^i_{\Delta})^*x_{n+1},\,\sigma^i_Xx_n = (\sigma^i_{\Delta})^*x_{n-1}\}.
            $$ 
            Homotopy groups $\pi_{i} \oR\Tot_{s\Set} X^{\bt}$ of the total space are related to the homotopy groups $\pi_{i+n}X^n$ by  a spectral sequence given in \cite[\S VII.6]{sht}.

            %
            %
            \item   For the 
            category 
            $\catdga$ of  non-negatively graded cdgas,  a model for $\ho\Lim_{n \in \Delta}$ is given by taking good truncation of the functor of Thom--Sullivan (a.k.a Thom--Whitney) cochains \cite[\S 4]{HinSch}, defined using de Rham polynomial forms. 
        \end{enumerate}
    \end{examples}
    \begin{remarks}
        
        A  sheaf $\sF$ of modules is a hypersheaf when regarded as a presheaf of non-negatively graded chain complexes, but not a hypersheaf  when regarded as a presheaf of unbounded or non-positively graded chain complexes unless $\H^i(U, \sF)=0$ for all $i>0$ and all $U$. Beware that the sheafification (as opposed to hypersheafification) of a hypersheaf will not, in general, be a hypersheaf.
    \end{remarks}

    The construction $X \leadsto X^{\sharp}$ introduced after Theorem \ref{bigthm} is an example of hypersheafification:
    
    \begin{definition}
        Given a functor $\sF\co \Aff^{\op} \to s\Set$, the \emph{\'etale hypersheafification} $\sF^{\sharp}$ of $\sF$ \index{hypersheafification $\sF^{\sharp}$}
        is the universal  \'etale hypersheaf $\sF^{\sharp}$ equipped with a map
        $
        \sF\to \sF^{\sharp}
        $ in the homotopy category of simplicial presheaves.
    \end{definition}
    
    In other words, hypersheafification is the derived left adjoint to the forgetful functor from hypersheaves to presheaves.
    
    \begin{warning}
        Terminology is disastrously inconsistent between references. Hypersheaves are often known as $\infty$-stacks or $\infty$-sheaves, but are referred to as stacks in \cite{hag1}, and as sheaves in \cite{luriehighertopoi,lurie} (where {\it stacks} refer only to algebraic stacks). They are also sometimes known as homotopy sheaves, but we avoid this terminology for fear of  confusion  with homotopy groups of a simplicial sheaf.     
    \end{warning}
    
\subsection{The conventional approach to higher stacks}
    
    Instead of defining $n$-stacks using hypergroupoids, To\"en--Vezzosi and Lurie \cite{hag2,lurie} use an
    inductive definition, building on the key properties established by Hirschowitz and Simpson in \cite{hirschowitzsimpson} to flesh out the approach set out by Simpson in \cite{simpsonalggeomnstacks}, which he attributed to Walter. 
    
    %
    %
    
    
    In that approach, $n$-geometric stacks are  defined inductively by saying that an \'etale hypersheaf $F$ is  $n$-geometric  if 
    \begin{enumerate}
        \item there exists a smooth covering  $\coprod_i U_i \to F$ from a family $\{U_i\}_i$ of affine schemes, and
        \item the diagonal $F \to F\by F$ is relatively representable by $(n-1)$-geometric stacks.
    \end{enumerate}
    
    If we take the families $\{U_i\}_i$ to be finite at  each stage in the definition above, then we obtain the definition of a strongly quasi-compact $n$-geometric stack. In the final version of \cite{hag2}, the induction starts by setting affine schemes to be $(-1)$-geometric.

    
    Beware that the  $n$-stacks of \cite{lurie} are indexed slightly differently, taking $0$-stacks to be  algebraic spaces, leading to the differences explained in  \Cref{cflurie}.
    
    By \cite[Theorem \ref{stacks2-bigthm}]{stacks2} (see Theorem \ref{bigthm} above),  $n$-geometric stacks correspond to Artin $(n+1)$-hypergroupoids.
    
    \medskip
    In practice, this inductive definition feels like a halfway house which iterates the least satisfactory aspects of the definition of an algebraic stack. To prove a general statement about geometric $n$-stacks, it is usually easier to work with hypergroupoids,  while representability theorems (see \S \ref{repsn}) tend to be the simplest means of proving that a given hypersheaf is a geometric $n$-stack.
    
\clearpage
\section{Derived geometric \texorpdfstring{$n$}{n}-stacks}\label{dstacksn} 
    
    There is nothing special about affine schemes as building blocks, so now we will use derived affine schemes. The only real change is that we have to replace all limits with homotopy limits, but since homotopy limits are practically useless without a means to compute them, we will start out with a more elementary characterisation. 
    
    The constructions and results prior of \S\S \ref{dstackdefsn}--\ref{dcotsn},\ref{addendum} 
    all have natural analogues in differential and analytic contexts, by using dg $\C^{\infty}$ or EFC rings instead of derived commutative rings, replacing smooth morphisms with submersions and \'etale morphisms with local diffeomorphisms or local biholomorphisms, respectively.
    
\subsection{Definitions}\label{dstackdefsn}

    \begin{notation}
        For the entire section we let $R$ be a commutative ring.
        Since we want statements applying in all characteristics, we will let the category  $d\Aff_R$ of derived affine schemes \index{derived affine schemes $d\Aff_R$}  be the opposite category to either cdgas  $\catdga[R]$ (if $\Q \subseteq R$) or to simplicial algebras $s\Alg_R$; denote that opposite category by $d\Alg_R$.\index{derived algebras $d\Alg_R$}  
    \end{notation}
    
    We give this the opposite model structure, so a morphism $\Spec B \to \Spec A$ in $d\Aff_R$ is a \emph{fibration/cofibration/weak equivalence}\index{weak equivalence!in $d\Aff_R$}\index{fibration!in $d\Aff_R$}\index{cofibration!in $d\Aff_R$} if and only if
    $A \to B$ is a cofibration/fibration/weak equivalence of cdgas or simplicial algebras;
    in particular, a morphism is a fibration if the corresponding map of cdgas or simplicial algebras is a retract of a quasi-free map\footnote{or at least an ind-quasi-smooth map in the original sense of \Cref{rem:qusmooth}, if you prefer to use the Henselian model structure of \cite[Proposition \ref{DStein-locmodelprop}]{DStein}, as discussed in  \Cref{hensel}}.
    
    \begin{notation}
        From now on, in order to have uniform notation for the simplicial and dg settings, we denote the homotopy groups $\pi_iA = \H_i(A, \sum (-1)^i \pd_i)$ of a simplicial algebra $A$  by $\H_iA$. Beware that this notation is highly abusive, since these are the homotopy groups, {\it not} the homology groups, of the underlying simplicial set.   
    \end{notation} 
    
    \begin{definition}
        We  define the category   
        $sd\Aff_R$ of \emph{simplicial derived affine schemes}\index{simplicial!derived affine schemes $sd\Aff$} to be $(d\Aff_R)^{\Delta^{\op}}$, so it consists of  diagrams
        \[
        \xymatrix@1{ X_0 \ar@{.>}[r]|{\sigma_0}& \ar@<1ex>[l]^{\pd_0} \ar@<-1ex>[l]_{\pd_1} X_1 \ar@{.>}@<0.75ex>[r] \ar@{.>}@<-0.75ex>[r]  & \ar[l] \ar@/^/@<0.5ex>[l] \ar@/_/@<-0.5ex>[l] 
        X_2 &  &\ar@/^1pc/[ll] \ar@/_1pc/[ll] \ar@{}[ll]|{\cdot} \ar@{}@<1ex>[ll]|{\cdot} \ar@{}@<-1ex>[ll]|{\cdot}  X_3 & \ldots&\ldots}
        \] 
        for derived affine schemes $X_i$.
        Equivalently, this is the opposite category to the category of cosimplicial cdgas or of cosimplicial simplicial algebras.
    \end{definition}
    
    Recall that we write $\pi^0\Spec (A_\bt) = \Spec \H_0(A_\bt)$ for $A_\bt \in \catdga[R]$, and similarly write $\pi^0\Spec A := \Spec (\pi_0A)$ for $A \in s\Alg_R$. 
    

    \begin{definition}\label{dnpdefh}
        We say that a simplicial derived affine scheme $X$ is a \emph{homotopy 
        derived Artin (resp. DM) $n$-hypergroupoid}  \index{derived Artin $n$-hypergroupoid!homotopy} \index{derived DM $n$-hypergroupoid!homotopy} 
        if:
        \begin{enumerate}
        
        \item the  simplicial affine scheme $\pi^0X$ is an Artin (resp. DM) $n$-hypergroupoid (Definition \ref{npdef});
        
        \item the sheaves $\H_j(\sO_{X})$ on $\pi^0X$ are all Cartesian;  explicitly, for the morphisms $\pd_i \co X_{m+1} \to X_m$, we have isomorphisms
        \[
        \pd^i \co  \pd_i^{-1}\H_j(\sO_{X_m})\ten_{\pd_i^{-1}\H_0(\sO_{X_m})}\H_0(\sO_{X_{m+1}}) \to  \H_j(\sO_{X_{m+1}})
        \]
        for all $i,m,j$. 
        \end{enumerate}
    \end{definition} 
    
    Equivalently, the second condition  says that 
    the morphisms $\pd_i \co X_{m+1} \to X_m$ are strong  for all $i,m$, hence homotopy-smooth (resp. homotopy-\'etale).
    
    As in the underived setting, we have a relative notion:
    
    \begin{definition}\label{dnpreldefh}
        Given   $Y \in sd\Aff$, a morphism $f \co X\to Y$ of simplicial derived affine schemes is said to be a \emph{homotopy 
        derived Artin (resp. DM) $n$-hypergroupoid}  \index{derived Artin $n$-hypergroupoid!homotopy} \index{derived DM $n$-hypergroupoid!homotopy} 
        over $Y$ if:
        \begin{enumerate}
        
        \item the morphism $\pi^0X \to \pi^0Y$ of simplicial affine schemes is an Artin (resp. DM) $n$-hypergroupoid (Definition \ref{npreldef});
        
        \item \label{dnpreldefh:strong} the morphisms $(f,\pd_i) \co X_{m+1} \to Y_{m+1}\by^h_{\pd_i,Y_m}X_m$ are strong (Definition \ref{strongdef}) for all $i,m$.
        
        \end{enumerate}
        
        
        The morphism $X\to Y$ is then said to be homotopy-smooth (resp. homotopy-\'etale) if in addition $X_0 \to Y_0$ is  homotopy-smooth (resp. homotopy-\'etale), and surjective if $\pi^0X_0 \to \pi^0Y_0$ is a surjective morphism of affine schemes.
    \end{definition}  
    
    
    \begin{remarks}\
        \begin{enumerate}
            %
            
            \item If $Y$ is itself a  homotopy derived Artin  $N$-hypergroupoid, then Condition (\ref{dnpreldefh:strong}) in Definition \ref{dnpreldefh} reduces to saying that the morphisms $\pd_i \co X_{m+1} \to X_m$ are all strong, by Exercises \ref{strongex}. 
            
            \item If $X_0$ is underived  in the sense that $X_0 \simeq \pi^0X_0$, and if $X$ is a homotopy derived Artin $n$-hypergroupoid, then the morphism  $\pi^0 X \to X$ is a weak equivalence
            by homotopy-smoothness, because  anything homotopy-smooth  over an underived base is itself underived so the maps $\pi^0X_m \to X_m$ are all quasi-isomorphisms. 
            
            \item Similarly, if $Y_0$ is underived then for $X \to Y$ to be homotopy-smooth (resp. homotopy-\'etale)   just means that $X_0$ is quasi-isomorphic to a smooth (resp. \'etale) underived affine scheme over $Y_0$. 
        \end{enumerate} 
    \end{remarks}
    \begin{examples}\label{htpyhgpdex}\
        \begin{enumerate}
            \item  Every Artin/DM $n$-hypergroupoid is a homotopy derived Artin/DM $n$-hypergroupoid.
            
            \item Saying that $X$ is a homotopy $0$-hypergroupoid is equivalent to saying that $X_0 \to X$ 
            is a weak equivalence, i.e. that  $X$ 
            is equivalent to a derived affine scheme  (with constant simplicial structure).
            
            \item
            If a smooth affine group scheme $G$ acts on a derived affine scheme $U$, then the simplicial derived affine scheme
            \[
            B[U/G]:= (U \Leftarrow U \by G \Lleftarrow U \by G \by G \ldots )
            \]   
            is a homotopy derived Artin
            $1$-hypergroupoid. 
            
            \item\label{derschex} 
            If $(\pi^0X,\sO_X)$ is a derived scheme (Definition \ref{def:dersch}), with $\pi^0X$ quasi-compact and semi-separated, take a finite cover $\cU =\{U_i\}_{i \in I}$ of $\pi^0X$ by open affine subschemes, and consider the simplicial derived affine scheme $\check{X}_{\cU}$ given by the \v Cech nerve \index{Cech nerve@\v Cech nerve}
            \[
            (\check{X}_{\cU})_m := \Spec (\prod_{i_0, \ldots, i_m \in I}\Gamma(U_{i_0}\cap \ldots \cap U_{i_m},\sO_X))
            \]
            with the obvious face and degeneracy maps.
            
            Since $\pi^0\check{X}_{\cU}$ is the \v Cech nerve of $\coprod U_i$ over $\pi^0X$ and $\sO_X$ is homotopy-Cartesian by definition, it follows that $\check{X}_{\cU}$ is a homotopy derived DM (in fact Zariski) $1$-hypergroupoid. Similar statements hold for semi-separated derived algebraic spaces and derived DM stacks with affine diagonal.
        \end{enumerate}
    \end{examples}
    
    As in the underived setting,  we have the following notion giving rise to equivalences for hypergroupoids: 
    
    \begin{definition}\label{dnpreldefhtriv}
        Given   $Y \in sd\Aff$, a morphism $f \co X\to Y$ in $sd\Aff$ is a \emph{homotopy  trivial
        derived Artin (resp. DM) $n$-hypergroupoid}   \index{derived Artin $n$-hypergroupoid!homotopy trivial} \index{derived DM $n$-hypergroupoid!homotopy trivial}
        over $Y$ if and only if:
        \begin{enumerate}
        
            \item the morphism $\pi^0f \co \pi^0X \to \pi^0Y$ of simplicial affine schemes is a trivial Artin (resp. DM) $n$-hypergroupoid;
            
            \item for all $j,m$, the maps $\H_0(\sO_{X_m})\ten_{f^{-1}\H_0(\sO_{Y_m})}f^{-1}\H_j(\sO_{Y_m})\to \H_j(\sO_{X_m})$ are isomorphisms.
        \end{enumerate}
    \end{definition}
    
    Because the morphisms $\pi^0f_m$ are all smooth (resp. \'etale), note that the second condition is equivalent to saying that the maps $f_m$ are all strong, hence homotopy-smooth (resp. homotopy-\'etale).
    
    \begin{example}\label{ex:derschhgpdtriv}
        If $(\pi^0X,\sO_X)$ is a derived scheme, with $\pi^0X$ quasi-compact and semi-separated, take finite covers $\cU =\{U_i\}_{i \in I}$ and $\cV=\{V_j\}_{j \in J}$ of $\pi^0X$ by open affine subschemes, and let $\cW:=\{U_i \cap V_j\}_{(i,j) \in I \by J}$. Then in the notation of Example \ref{htpyhgpdex}.(\ref{derschex}),  the resulting morphisms
        \begin{align*}
        \check{X}_{\cU} \la \check{X}_{\cW} \to \check{X}_{\cV} 
        \end{align*}
        are both trivial  DM (in fact Zariski) $1$-hypergroupoids.
    \end{example}

    \begin{warning}\label{reedyfibrant}
        A homotopy {\it derived} Artin/DM $n$-hypergroupoid $X$ 
        isn't determined by $X_{\le n+1}$ (whereas Property \ref{hypergroupoidprops}.(\ref{truncate}) gives $X \cong \cosk_{n+1}(X)$ for underived Artin $n$-hypergroupoids).\footnote{The reason the previous, underived argument fails is that  a section need not be a weak equivalence if its left inverse is the homotopy pullback of a section.}
        
        However, a homotopy {\it trivial} derived Artin $n$-hypergroupoid $X$ over $Y$ does satisfy $X \simeq \cosk_{n-1}^h(X)\by^h_{\cosk_{n-1}^h(Y)}Y$ for the homotopy $(n-1)$-coskeleton $\cosk^h_{n-1}$ (the right-derived functor of $\cosk_{n-1}$), so is determined by $X_{< n}$ over $Y$, up to homotopy.
    \end{warning}
    
\subsection{Main results}   
    
\subsubsection{Derived stacks}
    For our purposes, we can use the following as the definition of a derived $(n-1)$-geometric stack. It is a special case of \cite[Theorem \ref{stacks2-bigthm}]{stacks2}, as strengthened in \cite[Theorem 
    5.11]{stacksintro}.\footnote{As in Theorem \ref{bigthm}, we are using the terminology from later versions of \cite{hag2}, so indices are $1$ higher
    than in  \cite{stacks2,stacksintro}.}

    \begin{theorem}\label{bigthm2}\index{n-geometric@$n$-geometric!derived Artin stack} \index{derived Artin stack}
    The homotopy category of  strongly quasi-compact $(n-1)$-geometric derived   Artin  stacks is given by taking the full subcategory of $sd\Aff$ consisting of homotopy derived   Artin  $n$-hypergroupoids $X$, and formally inverting the homotopy trivial relative   Artin  $n$-hypergroupoids $X \to Y$. 
    
    \medskip
    In fact, a model for the $\infty$-category of strongly quasi-compact $(n-1)$-geometric derived  Artin  stacks is given by the relative category $(\C,\cW)$ with $\C$ the full subcategory of $sd\Aff$ consisting of homotopy derived Artin  $n$-hypergroupoids $X$ and $\cW$ the subcategory of homotopy trivial relative derived Artin $n$-hypergroupoids  $X \to Y$.
    
    \medskip
    The same results hold true if we substitute ``Deligne--Mumford'' for ``Artin''  throughout.\index{n-geometric@$n$-geometric!derived Deligne--Mumford stack}\index{derived Deligne--Mumford stack}
    \end{theorem}       
    %
    %
    In particular, this means we obtain the simplicial category of such derived stacks  by simplicial localisation of homotopy derived $n$-hypergroupoids at the class of homotopy trivial relative derived $n$-hypergroupoids.
    
    \begin{remark}\label{qcmpt2}
    We can extend Theorem \ref{bigthm2} to non-quasi-compact objects if we expand our 
    $\infty$-category of building blocks by allowing arbitrary disjoint unions of derived affine schemes (which form a  full subcategory of the ind-category $\ind(d\Aff)$). 
    \end{remark}
    
    An derived $\infty$-stack over $R$ is a functor $d\Alg_R \to s\Set$ satisfying various conditions, so we need to associate such functors to homotopy derived Artin/DM $n$-hypergroupoids. Similarly to the  underived setting,  the solution (not explicit) is to take
    \[
    X^{\sharp}(A)= \oR\Map_{\cW}(\Spec A, X),       
    \]\index{Xsharp@$X^{\sharp}$}
    where $\oR\Map_{\cW}$ is the right-derived functor of $\Hom_{sd\Aff}$ with respect to homotopy trivial derived  Artin/DM $n$-hypergroupoids.
    When $X$ is a $0$-hypergroupoid, we simply write $\oR \Spec A := (\Spec A)^{\sharp}$; this is just given by the functor $\oR\Map_{d\Alg}(A,-)$.
    \index{RSpec@$\oR \Spec$}
    
    We will describe the structure of the mapping spaces  $\oR \Map(X^{\sharp}, Y^{\sharp})$ in \S \ref{morphismsrevsn}, then give 
    explicit  formulae for the \S \ref{dmorphismsn}.

    \begin{definition}
        A derived stack $F\co d\Alg_R \to s\Set$ is said to be \emph{$n$-truncated} \index{n-truncated@$n$-truncated} if the restriction $\pi^0F\co \Alg_R \to s\Set$ is so.     \end{definition}
    
    \begin{warning}
        Beware that this does {\it not} mean that $\pi_iF(A)=0$ for $A \in d\Alg_R$ and $i>n$; that is only true if $A$ is underived, i.e. $A \in \Alg_R$. We have already seen in Examples \ref{mapex} that the first statement fails even for the affine line, with $\pi_i(\bA^1)^{\sharp}(A)\cong \H_i(A)$. 
    \end{warning}
    
    Since $n$-truncation is a condition on the restriction to underived algebras, \Cref{cflurie} (on the relation between $n$-geometric and $n$-truncated) applies in the derived setting in exactly the same way.
    
    \subsubsection{Quasi-coherent complexes}\label{qucohsn2}
    
    For derived $n$-stacks, the behaviour of quasi-coherent complexes is entirely similar to that for $n$-stacks in \S \ref{qucohsn1}.    
    
    We take a  homotopy derived Artin $n$-hypergroupoid $X$:
    \[
    \xymatrix@1{ X_0\, \ar@{.>}[r]|{\sigma_0}& \ar@<1ex>[l]^{\pd_0} \ar@<-1ex>[l]_{\pd_1}\, X_1\, \ar@{.>}@<0.75ex>[r] \ar@{.>}@<-0.75ex>[r]  & \ar[l] \ar@/^/@<0.5ex>[l] \ar@/_/@<-0.5ex>[l] 
    \,X_2\, &  &\ar@/^1pc/[ll] \ar@/_1pc/[ll] \ar@{}[ll]|{\cdot} \ar@{}@<1ex>[ll]|{\cdot} \ar@{}@<-1ex>[ll]|{\cdot} \, X_3\, & \ldots&\ldots,}
    \] 
    for derived affine schemes $X_i$.
    
    Equivalently, writing $O(X)^i_{\bt}$ for the cdga $O(X_i)_{\bt}$ associated to $X_i$,\footnote{Note that contravariance produces cosimplicial objects from simplicial objects, so turns subscripts into superscripts.} we have a cosimplicial cdga
    \[
    \xymatrix@1{ O(X)_{\bt}^0 \,\ar@<1ex>[r]^{\pd^0} \ar@<-1ex>[r]_{\pd^1}&\ar@{.>}[l]|{\sigma^0}\,  O(X)_{\bt}^1\, \ar[r] \ar@/^/@<0.5ex>[r] \ar@/_/@<-0.5ex>[r]   & \ar@{.>}@<0.75ex>[l] \ar@{.>}@<-0.75ex>[l] 
    \,O(X)_{\bt}^2 \,  \ar@/^1pc/[rr] \ar@/_1pc/[rr] \ar@{}[rr]|{\cdot} \ar@{}@<1ex>[rr]|{\cdot} \ar@{}@<-1ex>[rr]|{\cdot} & & \, O(X)_{\bt}^3 \,& \ldots&\ldots,}
    \]
    so we can look at modules 
    \[
    \xymatrix@1{ M_{\bt}^0 \ar@<1ex>[r]^{\pd^0} \ar@<-1ex>[r]_{\pd^1}&\ar@{.>}[l]|{\sigma^0}  M_{\bt}^1\ar[r] \ar@/^/@<0.5ex>[r] \ar@/_/@<-0.5ex>[r]   & \ar@{.>}@<0.75ex>[l] \ar@{.>}@<-0.75ex>[l] 
    M_{\bt}^2   \ar@/^1pc/[rr] \ar@/_1pc/[rr] \ar@{}[rr]|{\cdot} \ar@{}@<1ex>[rr]|{\cdot} \ar@{}@<-1ex>[rr]|{\cdot} & & M_{\bt}^3 & \ldots&\ldots,}
    \]
    over it, with each $M^r_{\bt}$ being an $O(X)^r_{\bt}$-module in chain complexes.
    
    As in the underived setting of Proposition \ref{hCartprop}, \cite[Corollary \ref{stacks-qcohequiv}]{stacks2} says that giving a quasi-coherent complex on the associated derived $n$-geometric stack $X^{\sharp}$ is equivalent to giving a module on $X$ which is homotopy-Cartesian:
    
    \begin{definition}\label{hcartdef2}
   We define a \emph{homotopy-Cartesian module} \index{homotopy-Cartesian!module} $\sF$ on our homotopy derived Artin $n$-hypergroupoid $X$ to consist of
    \begin{enumerate}
    \item an $O(X)^m_{\bt}$-module $\sF_{\bt}^m$ in chain  complexes 
    for each $m$, and
    \item morphisms $\pd^i\co \pd_i^*\sF_{\bt}^{m-1} \to \sF_{\bt}^m$ and $\sigma^i\co\sigma_i^*\sF_{\bt}^{m+1} \to \sF_{\bt}^m$, for all $i$ and $m$, satisfying the usual cosimplicial identities\cref{nosenote}, such that
    
    \item the quasi-coherent sheaves
    $(m \mapsto \H_j(\sF^m_{\bt}))$ on the simplicial scheme $(m \mapsto \pi^0X_m)$
    are Cartesian for all $j$, i.e. the maps
    \[
    \pd^i\co  (\pi^0\pd_i)^*\H_j(\sF^{
    m-1}_{\bt}) \to \H_j(\sF^m_{\bt}),
    \]
    for $\pi^0\pd_i\co \pi^0X_m \to \pi^0X_{m-1}$,
    are isomorphisms of quasi-coherent sheaves on $\pi^0X_m$ (i.e. of $\H_0O(X_m)$-modules) for all $i,j,m$.
    \end{enumerate}
    A morphism $\{\sE^m_{\bt}\}_m \to \{\sF^m_{\bt}\}_m$ is a weak  equivalence if the maps $\sE^m_{\bt} \to \sF^m_{\bt}$ are all quasi-isomorphisms. 
    \end{definition}
    
    \begin{notation}
    Here, we are writing $(\pi^0\pd_i)^*\H_j(\sF^{
    m-1}_{\bt})$ for
    \[
    \pd_i^{-1}\H_j(\sF^{
    m-1}_{\bt})\ten_{\pd_i^{-1}\H_0(\sO_{X_{m-1,\bt}})}\H_0(\sO_{X_m,\bt}).
    \]
    
    Since $\pi^0X_m=\Spec \H_0(O(X)^m_{\bt})$, we are also associating quasi-coherent sheaves on the simplicial scheme
    \[
    (\pi^0X_0 \Leftarrow \pi^0X_1 \Lleftarrow \pi^0X_2 \cdots) 
    \]
    to the cosimplicial module 
    \[
    (\H_j(\sF^0_{\bt}) \Rightarrow \H_j(\sF^1_{\bt}) \Rrightarrow  \H_j(\sF^2_{\bt}) \cdots)
    \]
    over the cosimplicial ring
    \[
    (\H_0(O(X)^0_{\bt}) \Rightarrow \H_0(O(X)^1_{\bt}) \Rrightarrow  \H_0(O(X)^2_{\bt}) \cdots).
    \]
    \end{notation}
    
    \begin{remark}
    Note that because the maps $\pd_i$ are homotopy-smooth, the Cartesian condition in \Cref{hcartdef2} is equivalent to saying that the composite maps 
    $$ 
    \oL\pd_i^*\sF_{\bt}^{m} \to \pd_i^*\sF_{\bt}^{m} \xra{\pd^i} \sF_{\bt}^{m+1}
    $$  are quasi-isomorphisms, which implies that the  morphisms  $\sigma^i\co \oL\sigma_i^*\sF_{\bt}^{m+1} \to \sF_{\bt}^m$ are also automatically quasi-isomorphisms. We have  to  left-derive the pullback functors $\pd_i^*$ in this version of the statement because homotopy-smoothness does not imply quasi-flatness. 
    \end{remark}
    

    
    \bigskip
    When $X$ is an Artin $n$-hypergroupoid with no derived structure, observe that the statement above just recovers Proposition \ref{hCartprop}. We now consider simple examples with derived structure.
 
    \begin{example}
    Take a derived scheme $(\pi^0X,\sO_{X})$ with $\pi^0X$ quasi-compact and semi-separated, and let $\sF_{\bt}$ be a homotopy-Cartesian presheaf of $\sO_{X}$-modules in chain  complexes, in the sense of  Definition \ref{derschCart}. Then for any finite affine cover $\cU:=\{U_i\}_{i \in I}$ of $\pi^0X$, we can form chain complexes
    \[
    \check{\CC}^m(\cU,\sF_{\bt}):= \prod_{i_0, \ldots, i_m \in I}\Gamma(U_{i_0}\cap \ldots \cap U_{i_m},\sF_{\bt}),
    \]
    and these fit together to give a cosimplicial chain complex
    \[
    \xymatrix@1{ \check{\CC}^0(\cU,\sF_{\bt})\, \ar@<1ex>[r]^{\pd^0} \ar@<-1ex>[r]_{\pd^1}&\ar@{.>}[l]|{\sigma^0}\,   \check{\CC}^1(\cU,\sF_{\bt}) \,\ar[r] \ar@/^/@<0.5ex>[r] \ar@/_/@<-0.5ex>[r]   & \ar@{.>}@<0.75ex>[l] \ar@{.>}@<-0.75ex>[l] 
    \,\check{\CC}^2(\cU,\sF_{\bt})\,   \ar@/^1pc/[rr] \ar@/_1pc/[rr] \ar@{}[rr]|{\cdot} \ar@{}@<1ex>[rr]|{\cdot} \ar@{}@<-1ex>[rr]|{\cdot} & & \,\check{\CC}^3(\cU,\sF_{\bt})\, & \ldots&\ldots,}
    \]
    which is a module over the cosimplicial cdga
    \[
    \xymatrix@1{ \check{\CC}^0(\cU,\sO_{X,\bt})\, \ar@<1ex>[r]^{\pd^0} \ar@<-1ex>[r]_{\pd^1}&\ar@{.>}[l]|{\sigma^0}   \,\check{\CC}^1(\cU,\sO_{X,\bt})\, \ar[r] \ar@/^/@<0.5ex>[r] \ar@/_/@<-0.5ex>[r]   & \ar@{.>}@<0.75ex>[l] \ar@{.>}@<-0.75ex>[l] 
    \,\check{\CC}^2(\cU,\sO_{X,\bt})\,  \ar@/^1pc/[rr] \ar@/_1pc/[rr] \ar@{}[rr]|{\cdot} \ar@{}@<1ex>[rr]|{\cdot} \ar@{}@<-1ex>[rr]|{\cdot} & & \,\check{\CC}^3(\cU,\sO_{X,\bt})\, & \ldots&\ldots.}
    \]
    The latter is just $O(\check{X}_{\cU})$ for the homotopy derived  DM $1$-hypergroupoid $\check{X}_{\cU}$
    from   Example \ref{htpyhgpdex}.(\ref{derschex}), and the homotopy-Cartesian hypothesis on $\sF_{\bt}$ ensures that $\check{\CC}^{\bt}(\cU,\sF_{\bt})$
    is a  homotopy-Cartesian module  on  $\check{X}_{\cU}$.
    \end{example}  
    
    \begin{example}
    Let's look at what happens when $X$ is a homotopy derived  $0$-hypergroupoid, so the morphisms $\pd_i \co X_m \to X_{m-1}$, $\sigma_i\co X_{m-1} \to X_m$ are all quasi-isomorphisms. Then \Cref{hcartdef2} simplifies to say that  a homotopy-Cartesian module on $X$ is an $\{O(X)^m_{\bt}\}_m$-module $\{\sF_{\bt}^m\}_m$ for which the morphisms $\pd^i \co \sF^m_{\bt} \to \sF^{m+1}_{\bt}$ (and hence $\sigma^i \co\sF^{m+1}_{\bt}\to \sF^m_{\bt}$) are all quasi-isomorphisms.
    
    \smallskip  This gives us an equivalence of $\infty$-categories between homotopy-Cartesian modules on $X$ and $O(X)^0_{\bt}$-modules in chain complexes. The correspondence sends a module $\{\sF^m_{\bt}\}_m$ over $X$ to the $O(X_0)_{\bt}$-module $\sF^0_{\bt}$, with quasi-inverse functor given by the right adjoint, which sends an $O(X_0)_{\bt}$-module $\sE_{\bt} $ to $(\sE_{\bt} \Rightarrow \sE_{\bt} \Rrightarrow \sE_{\bt}\cdots)$, i.e. to itself, given constant cosimplicial structure, with the $O(X_n)_{\bt}$-actions coming via the degeneracy maps in $X$. The unit $\{\sF^m_{\bt}\}_m \to \{\sF^0_{\bt}\}_m$ of the adjunction is then manifestly a levelwise quasi-isomorphism by the reasoning above, because $\sF$ is homotopy-Cartesian. 
    \end{example}  
    
    \begin{definition}
    As in the underived setting of \S \ref{Lfstarsn}, for any morphism $f \co X_{\bt} \to Y_{\bt}$ of homotopy derived Artin $n$-hypergroupoids we have a \emph{derived pullback functor} $\oL f^*$ \index{derived pullback functor $\oL f^*$} on quasi-coherent complexes, given levelwise by $(\oL f^*\sF_{\bt})^m \simeq \oL f_m^*\sF_{\bt}^m$.
    \end{definition}
    

    
    %
    
    

    \subsection{Tangent and obstruction theory}
    We follow the treatment in \cite[\S \ref{drep-tgtsn}]{drep}.


    \begin{lemma}\label{hgslemma}
    Given a derived $n$-geometric Artin  stack $F \co d\Alg_R \to s\Set$ and maps $A \to B \la C$ in $d\Alg_R$, with $A \onto B$ a surjection with nilpotent kernel, we have a weak equivalence
    \[
    F(A\by_BC) \xra{\sim} FA\by^h_{FB}FC.       
    \]
    \end{lemma}
    For a proof, see Corollary \ref{hgscor}.
    %
\begin{definition}
    As in \cite{drep}, we call functors satisfying the conclusion of Lemma \ref{hgslemma}  \emph{homotopy-homogeneous}\index{homotopy-homogeneous}, by analogy with the notion of homogeneity \cite{Man2}; it is best thought of as a derived form of Schlessinger's conditions.
    \end{definition}
    
    \begin{remark}[Terminology] Recent sources  tend to use phrases like ``infinitesimally cohesive on one factor'' for this notion (or a slight variant), because the
    notion of infinitesimally cohesive \index{infinitesimally cohesive} in \cite{lurie} imposes the nilpotent surjectivity condition to $C \to B$ as well;  our notion of homotopy-homogeneity more closely resembles Artin's generalisation \cite[2.2 (S1)]{Artin} of Schlessinger's conditions,\index{Schlessinger's conditions} which unsurprisingly leads to a more usable representability theorem.
    \end{remark}
    
    The long exact sequence of a homotopy fibre product (\Cref{lesfibreprod}) immediately gives the following: 
    \begin{lemma}
    If $F$ is homotopy-homogeneous, then we have a surjection
    \[
    \pi_0F(A\by_BC) \onto \pi_0FA\by_{\pi_0FB}\pi_0FC
    \]
    for all  maps $A \to B \la C$ in $d\Alg_R$ with $A \onto B$ a nilpotent surjection, and a weak equivlaence
    \[
    \pi_0F(A\by C) \xra{\sim} \pi_0FA \by \pi_0FC.
    \]
    \end{lemma}
    You may recognise these as  generalisations of two  of Schlessinger's conditions \cite{Sch}\index{Schlessinger's conditions}, the third being a finiteness constraint.

    \begin{example}
    For a homotopy derived $0$-hypergroupoid given by a derived affine scheme $U$ (with constant simplicial structure), the associated derived stack is given by $U^{\sharp} \simeq \oR\Map_{d\Aff}(-,U)$, a functor for which we've seen these results before, with tangent and obstruction theory as in \S \ref{tgtsn1}.  
    \end{example}

\subsubsection{Tangent spaces}

    Now, take $A \in \catdga[R]$ and $M \in  \catdgm[A]$, with $F \co \catdga[R] \to s\Set$. We regard $A \oplus M$ as an object of $\catdga[R]$ by setting the product of elements $M$ to be $0$.
    \begin{definition}
        For $x \in F(A)_0$, define the \emph{tangent space}\index{tangent space} of $F$ at $x$ with coefficients in $M$ to be the homotopy fibre  $T_x(F,M):= F(A\oplus M)\by^h_{F(A)}\{x\}$. 
    \end{definition}

    If $F$ is homotopy-homogeneous, then we have an additive structure on tangent spaces $T_x(F,M) \in \Ho(s\Set)$ via the composition
    \begin{align*}
    &F(A\oplus M)\by^h_{FA}F(A\oplus M)\\ 
    &\simeq F((A\oplus M)\by_A(A\oplus M))\\
    &=F(A\oplus (M\oplus M)) \xra{\phantom{xxx}+\phantom{xxx}} F(A \oplus M).
    \end{align*}
    
    Moreover we have a short exact sequence $0 \to M \to \cone(M \to M) \to M[-1] \to 0$, so $M= \cone(M \to M)\by_{M[-1]}0$, and thus  $$
    F(A \oplus M)\simeq F(A \oplus \cone(M \to M))\by_{F(A\oplus M[-1])}^hF(A),
    $$
    since $F$ is  homotopy-homogeneous and  $A \oplus \cone(M \to M)\to A \oplus M[-1]$ is a square-zero extension.
    
    If $F$ is also \emph{homotopy-preserving}\index{homotopy-preserving} in the sense that it preserves weak equivalences,\index{homotopy-preserving} then  $F(A \oplus \cone(M \to M))\simeq F(A)$, so we have 
    \[
    F(A \oplus M)\simeq F(A)\by_{F(A\oplus M[-1])}^hF(A).
    \]
    
    Taking homotopy fibres over $x \in F(A)$,  we then get 
    $$
    T_x(F,M) =  0\by_{T_x(F,M[-1])}^h0,
    $$      
    which is a loop space,
    so $T_x(F,M[-1])$ deloops $T_x(F,M)$ and 
    $$
    \pi_i T_x(F,M)\cong \pi_{i+n} T_x(F,M[-n]).
    $$
    
    \begin{definition}\label{totcohodef}
        We can thus define \emph{tangent cohomology groups} by\index{tangent cohomology}\index{Andr\'e--Quillen cohomology!generalised $D^*_x(F,M)$}
        $$
        D^{n-j}_x(F,M):= \pi_jT_x(F, M[-n]).
        $$
    \end{definition}
    These generalise the Andr\'e--Quillen cohomology groups of derived affine schemes.

    \begin{definition}
        For a homotopy-preserving homotopy-homogeneous functor $F \co d\Alg \to s\Set$ and an element $x \in F(k)$ for $k$ a field, define the \emph{dimension of $F$ at $x$}\index{dimension} to be the Euler characteristic  $\dim_x(F):= \sum (-1)^i\dim (D^i(F,k))$, when finite.       
    \end{definition}
    
    \begin{examples}\
      \begin{enumerate}
    \item If $F$ is the derived stack associated to 
    a dg-scheme $X$ and $x \in X(A)$, then $D^i_x(F,M)\cong \Ext^i_{\sO_X}(\bL^X, \oR x_* M) \cong \Ext^i_A(\oL x^*\bL^X, M)$. 
    
    When $X$ is a dg-manifold, the dimension of $F$ at $x \in X(k)$ is therefore  the Euler characteristic $\dim_x(F)=\chi(x^*\Omega^1_X)$, the alternating sum of the number of generators of $(\sO_{X,\bt})_x$  in each degree.  When  $X$ is just a smooth underived scheme, this is simply $\dim x^*T_X =\dim X$ .
    
    \item If $V$ is a cochain complex in degrees $\ge -n$, finite-dimensional over $k$, then 
    \[
    F \co A \mapsto N^{-1}\tau_{\ge 0} \Tot^{\Pi}(V\ten_k A)       
    \]
    (Dold--Kan denormalisation of good truncation) is represented by an $n$-hypergroupoid over $k$, with $D^i_x(F^{\sharp},M)\cong \H^i(\Tot^{\Pi}(V\ten_kM))$ for all $i$. At all points $x \in F(k)= N^{-1}\tau_{\ge 0}V$, we thus have $D^i_x(F^{\sharp},k)\cong \H^i(V)$, so  $\dim_x(F) \cong \chi(V)$, when finite. 
    \end{enumerate}

    \end{examples}

\subsubsection{The long exact sequence of obstructions}\label{lesobssn}

Take a square-zero extension $g \co A \to B$ of commutative rings,  with kernel $I$.
If $F$ is a homotopy-preserving homotopy-homogeneous functor, then  there is a long exact sequence of groups and sets:\index{long exact sequence of groups and sets} 
    
    $$\xymatrix@R=0ex{
    \cdots \ar[r]^-{} &\pi_n(FA,y) \ar[r]^-{g_*}&\pi_n(FB,x) \ar[r]^-{u}& D^{1-n}_{y}(F,I) \ar[r]^-{e_*} &\pi_{n-1}(FA,y)\ar[r]^-{g_*}&\cdots\\ &\cdots \ar[r]^-{g_*}&\pi_1(FB,x) \ar[r]^-{u}& D^0_{y}(F,I)  \ar[ddll]|-{\,-*y\,}\\
    \\
    & \pi_0(FA)\ar[r]_{g_*} & \pi_0(FB) \ar[r]_-{u} & \Gamma(FB,D^1(F, I)).
    }
    $$
    The first part is the sequence associated to   the homotopy fibre sequence $ T_x(F,I) \to F(A) \to F(B)$ as in \cite[Lemma I.7.3]{sht}, but
    the non-trivial content here is in the final map $u$ which gives rise to obstructions.\footnote{This phenomenon of central and abelian extensions giving rise to such obstructions arises in many branches of algebra and topology --- see \cite{obsrec2} for a more general algebraic formulation.}
    
    \medskip
    Here are the details of the construction (following  \cite[Proposition \ref{drep-obs}]{drep}). 
    Let $C=C(A,I)$ be the mapping cone of $ I \to A$. Then $C \to B$ is a square-zero acyclic surjection, so $FC \to FB$ is a weak equivalence, and thus $\pi_i(FC)  \to \pi_i(FB)$ is an isomorphism for all $i$.
    Now,
    \[
    A = C\by_{B \oplus I[-1]}B,
    \]
    and since $C\to B \oplus I[-1]$ is surjective this gives, for $y \in FB$,
      a map
    $$
    p': (FC)\by^h_{(FB)}\{y\} \to T_y(F,I[-1])
    $$
    in the homotopy category of simplicial sets,     with homotopy fibre $(FA)\by^h_{(FB)}\{y\} $ over $0$. The sequence above is just induced by the long exact sequences \cite[Lemma I.7.3]{sht} associated to the homotopy fibre sequences $ (FA)\by^h_{(FB)}\{y\} \to \{y\} \xra{p'}  T_y(F,I[-1])$.

    \begin{lemma}\label{detectweak}
    A morphism $F \to G$ of  $n$-geometric derived stacks over $R$ 
    is a weak equivalence if and only if
    \begin{enumerate}
    \item $\pi^0f\co \pi^0F(B) \to \pi^0G(B)$ is a weak equivalence of functors $\Alg_R \to s\Set$,  
    and 
    \item for all discrete $A$ (i.e. $A \in \Alg_R$), all $A$-modules $M$  and all $x \in F(A)$, the maps $f\co D^i_x(F,M)\to D^i_{f(x)}(G,M)$ are isomorphisms for all $i>0$. 
    (Note that for $i\le 0$, we already know that these maps are isomorphisms,  from the first condition.)
    \end{enumerate}
    \end{lemma}
    \begin{proof} 
    We need to show that $F(B) \to G(B)$ is a weak equivalence for all $B \in d\Alg$, which we do by  working up the Postnikov tower $B=\Lim_i P_iB$. 
    
Since     $P_{i+1}B \to P_iB$ is weakly equivalent to  a square-zero extension with kernel $(\H_{i+1}B)[i+1]$ by \Cref{postfact}, the long exact sequence of obstructions gives inductively (on $j$) that $\pi_jF(P_iB) \cong \pi_jG(P_iB)$ for all $i,j$. 
    To complete the proof, note that $F(B) \simeq \ho\Lim_k F(P_kB)$ and similarly for $G$.
    \end{proof}

    Note that we could relax both conditions in Lemma \ref{detectweak} by asking that they only hold for reduced discrete algebras, and then apply a further induction to the quotients of $\H_0(B)$ by powers of its nilradical. Also note that the proof applies to any homotopy-preserving homotopy-homogeneous functors $F$ which satisfy $F(B) \simeq \ho\Lim_k F(P_kB) $, a condition called \emph{nilcompleteness} \index{nilcomplete} in \cite{lurie}.
    
    \subsubsection{Sample application of derived deformation theory --- semiregularity}

    We now give an application from \cite{semireg2}.\footnote{explanatory slides available at \href{http://www.maths.ed.ac.uk/~jpridham/semiregslide.pdf}{www.maths.ed.ac.uk/$\sim$jpridham/semiregslide.pdf}}
    Take:
    \begin{itemize}
    \item a smooth proper variety $X$ over a field $k$ of characteristic $0$,
    \item a square-zero extension  $A \to B$ of $k$-algebras with kernel $I$,
    \item a closed LCI subscheme $Z \subset X\ten B$  of codimension $p$, flat over $B$.
    \end{itemize}
    Then the obstruction to lifting $Z$ to a subscheme of $X \ten A$ lies in $\H^1(Z, \sN_{Z/X})\ten I$. Bloch \cite{blochSemiregularity} defined a semiregularity map
    \[
    \tau\co \H^1(Z, \sN_{Z/X}) \to \H^{p+1}(X, \Omega^{p-1}_X), 
    \]
    and conjectured that it annihilates all obstructions, giving a reduced obstruction space. There is also a generalisation where $X$ deforms, and then $\tau$ measures the obstruction to deforming the Hodge class $[Z]$; this corresponds to relaxing the requirement that $k$ be a field.
    These conjectures were extended to perfect complexes in place of $\sO_Z$ by \cite{BuchweitzFlenner}.
    
    \medskip 
    
    In \cite{semireg2}, the conjectures were proved, and extended to more general $X$, by interpreting $\tau$ as the tangent map of a  morphism of homotopy-preserving homotopy-homogeneous functors, then factoring through something unobstructed. In more detail, the Chern character $\ch_p$  gives a map from the moduli functor to
    \[
    \cJ_X^p(A):=    ( \oR\Gamma(X, \Tot^{\Pi}\oL\Omega^{\bt}_{X_A/k})\by^h_{ \oR\Gamma(X, \Omega^{\bt}_{X_A/A})}\oR\Gamma(X, F^p\Omega^{\bt}_{X_A/A}))[2p],
    \]
    where $X_A=X\ten_kA$, and this functor has derived tangent complex
    \[
    ( \{0\}\by^h_{\oR\Gamma(X, \Omega^{\bt}_{X})}\oR\Gamma(X, F^p\Omega^{\bt}_{X}))[2p] \simeq \oR\Gamma(X, \Omega^{< p}_X)[2p-1].
    \]
    The map $\tau$ on obstruction spaces then comes from applying $\H^1$ to  the derived tangent maps
    \[\xymatrix@R=0ex{
    T_{[Z]}\oR\Hilb_X \ar[r] & T_{[\sO_Z]}\oR\cP\!\mathit{erf}_X  \ar[r]^{d\ch_p} & T_{\ch_p(\sO_Z)}\cJ_X^p\\
    \oR\Gamma(Z,\sN_{Z/X}) \ar[r] & \oR\HHom_{\sO_X}(\sO_Z,\sO_Z)[1] \ar[r]^{d\ch_p} & \oR\Gamma(X, \Omega^{< p}_X)[2p-1],
    }
    \]
    from the derived Hilbert scheme to the derived moduli stack of perfect complexes and then to $\cJ_X^p$,
    since $d\ch_p$ factors through $\oR\Gamma(X, \Omega^{p-1}_X)[p]$ (a summand via the Hodge decomposition). The obstruction in  $\cJ_X^p$ then vanishes when $k$ is a field, or more generally measures obstructions to deforming $[Z]$ as a Hodge class. 
    
    \begin{remark}
      The key geometric difference between this and earlier approaches is not so much the language of derived deformation theory, which already tended to feature in disguise,
            but rather  the use of derived de Rham cohomology $\oR\Gamma(X, \Tot^{\Pi}\oL\Omega^{\bt}_{X_A/k})$ over the fixed base $k$ to generate horizontal sections, instead of a more classical cohomology theory.   
    \end{remark}

\subsection{Cotangent complexes}\label{dcotsn}

    The cotangent complex (when it exists) of a functor $F \co d\Alg_R \to s\Set$
    represents the tangent functor. Explicitly, it is a quasi-coherent complex\footnote{Explicitly, this means we have an $A$-module $\bL_{F,x}$ for each $x \in F(A)$, such that the maps $\bL_{F,x}\ten^{\oL}_AB \to \bL_{F,fx}$ are quasi-isomorphisms for all $f \co A \to B$.} 
    $\bL_F$
    on $F$ such that for all  $A \in d\Alg_R$, all points $x \in F(A)$ and all $A$-modules $M$, we have 
    \begin{align*}
    T_x(F,M) &\simeq \oR\Map_{\catdgmU[A]}(\oL x^*\bL_F,M) \\
    &\simeq N^{-1}\tau_{\ge 0} \oR\HHom_{A}(\oL x^*\bL_F,M),
    \end{align*}
    so in particular $D^i_x(F,M) \cong \Ext^i_A(\oL x^*\bL_F,M)$.
    
    \smallskip 
    For homotopy derived DM $n$-hypergroupoids $X$, the cotangent complex $\bL^{X^{\sharp}}$ of the associated stack $X^{\sharp}$ 
    corresponds via \S \ref{qucohsn2} to the complex $ m\mapsto \bL^{X_m}$ on $X$, which is homotopy-Cartesian because  the maps $\pd_i \co X_{m+1} \to X_m$ are homotopy-\'etale, so $\bL^{X_{m+1}}\simeq \oL\pd_i^* \bL^{X_m}$.
    
    \smallskip
    For homotopy derived Artin $n$-hypergroupoids $X$, that doesn't work, but it turns out that the iterated derived loop space $X^{h S^n}$ \index{loop space!iterated derived}
    is a derived $0$-hypergroupoid for $n>0$, and then the complex $m \mapsto (\bL^{X^{h S^n}})_m[-n]$ is homotopy-Cartesian on $X^{h S^n}$  and pulls back to give a model for $\bL^{X^{\sharp}}$ on $X$.
    
    Explicitly, writing $X^K \in sd\Aff$ for the functor $(X^K)_i(A):=\Hom_{s\Set}(K \by \Delta^i, X(A))$, 
    when $X$ is Reedy fibrant as in \S \ref{strictderhgpdsn}, a model for the cotangent complex $\bL^{X^{\sharp}}$
    is given by $\oL i^*\Omega^1_{X^{\Delta^n}/X^{\pd \Delta^n}}[-n]$, for the natural map $i \co X \to X^{\Delta^n}$.
    
    \begin{example}
    If $X= B[U/G]$ (a homotopy derived Artin $1$-hypergroupoid), then
    \begin{align*}
    X^{\Delta^1}&= B[(U \by G)/(G \by G)]\\
    X^{\pd \Delta^1}&=X \by X = [(U \by U)/(G \by G)]
    \end{align*}
    In level $0$ (i.e. on $X_0$),  the complex $\oL i^* \oL \Omega^{\bt}_{X^{\Delta^1}/X^{\pd \Delta^1}}[-1]$ is then  $\oL e^*\oL\Omega^1_{(U \by G)/(U \by U)}[-1]$ for $e \co U\to U \by G$ given by $u \mapsto (u,e)$, where $e$ is the identity element of the group $G$. This therefore recovers the formula 
    \[
    \bL^{[U/G]}|_U \simeq \cone(\bL^U  \to \g^*\ten \sO_U)[-1],
    \]
    which readers familiar with cotangent complexes of Artin stacks will recognise. 
    \end{example}
    
    %
    
    \subsubsection{Morphisms revisited}\label{morphismsrevsn}
    
    Given homotopy derived Artin  $n$-hypergroupoids $X$ and $Y$,\footnote{In the terminology of \S \ref{repsn}, the description we use here in fact adapts to  $ \oR \Map(X^{\sharp},F)$ for any $X \in sd\Aff$ and any  homotopy-homogeneous nilcomplete functor $F \co d\Alg_R \to s\Set$.}
    what does the space $\oR \Map(X^{\sharp}, Y^{\sharp})$ of maps $f \co X \to Y$ between the associated derived $(n-1)$-geometric stacks look like?
    
    For a start, we have a morphism $ \oR \Map(X^{\sharp}, Y^{\sharp}) \to\oR \Map(\pi^0X^{\sharp}, Y^{\sharp}) $, and the latter is just the space of maps $(\pi^0X)^{\sharp} \to (\pi^0Y)^{\sharp} $ of underived $(n-1)$-geometric stacks, as described in \S \ref{nstackmorphismsn}. In particular, this is $m$-truncated whenever $Y^{\sharp}$ is so.
    
    \medskip
    By the universal property of hypersheafification, we can replace $X^{\sharp}$ with $X$. Since $\oR \Map(X, Y^{\sharp})\simeq \ho\Lim_{m \in \Delta}\oR\Map(X_m,Y^{\sharp})$,  any homotopy limit expressions for $Y^{\sharp}$ as a functor on $d\Alg$ thus apply to the contravariant functor $\oR\Map(-,Y^{\sharp})$ on $sd\Aff$ as well.
    
    \medskip
    We can now work our way up the Postnikov tower of \S \ref{postnikovsn}, writing $\tau^{\le k}\Spec A:= \Spec P_kA$ and $(\tau^{\le k}X)_m:=\tau^{\le k}(X_m)$ (so in particular $\tau^{\le 0}X=\pi^0X$) to give a tower
    \[
    \ldots \to  \oR \Map(\tau^{\le k+1}X, Y^{\sharp}) \to \oR \Map(\tau^{\le k}X, Y^{\sharp}) \to \ldots \to \oR \Map(\pi^0X, Y^{\sharp}). 
    \]
    Lemma \ref{postfact} and \S \ref{obssn} then give  an expression for $P_{k+1}\sO_X $ as a homotopy pullback of a diagram $P_k\sO_X \xra{u}  \H_0(\sO_X) \oplus \H_{k+1}(\sO_X)[k+2] \xla{(\id,0)}\H_0(\sO_X) $ in the homotopy category,
    giving a homotopy pullback square
    \[
    \xymatrix{
    \oR \Map(\tau^{\le k+1}X, Y^{\sharp}) \ar[d] \ar[r] & \oR \Map(\tau^{\le k}X, Y^{\sharp}) \ar[d]^u\\
    \oR \Map(\pi^0X, Y^{\sharp}) \ar[r] &  \oR \Map(\oSpec_{\pi^0X}(\sO_{\pi^0X} \oplus \H_{k+1}(\sO_X)[k+2]), Y^{\sharp}).
    }
    \]
    
    For a fixed element $[g] \in \pi_0\oR \Map(\pi^0X, Y^{\sharp})$, with  $\pi_0\oR \Map(\tau^{\le k}X, Y^{\sharp})_{[g]}$ the homotopy  fibre over $[g]$, we thus have   a long exact sequence
    \[
    \xymatrix@R=0ex{
    \ldots \ar[r] 
    &\pi_1\oR \Map(\tau^{\le k}X, Y^{\sharp})_{[g]} \ar[r] & \Ext^{k+1}_{\sO_{\pi^0X}}(\oL g^*\bL^Y, \H_{k+1}(\sO_X)) \ar[dll] \\
    \pi_0\oR \Map(\tau^{\le k+1}X, Y^{\sharp})_{[g]} \ar[r] &\pi_0 \oR \Map(\tau^{\le k}X, Y^{\sharp})_{[g]} \ar[r]^-u & \Ext^{k+2}_{\sO_{\pi^0X}}(\oL g^*\bL^Y, \H_{k+1}(\sO_X))
    }
    \]
    of homotopy groups and sets .\index{long exact sequence of groups and sets}
    Explicitly, this means that
    \begin{itemize}
    \item a class $[g^{(k)}] \in \pi_0 \oR \Map(\tau^{\le k}X, Y^{\sharp})_{[g]}$ lifts to  a class
    $[g^{(k+1)}]\in \pi_0\oR \Map(\tau^{\le k+1}X, Y^{\sharp})$ if and only if $u([g^{(k)}])=0$;
    \item the group  $\Ext^{k+1}_{\sO_{\pi^0X}}(\oL g^*\bL^Y, \H_{k+1}(\sO_X))$ acts transitively on the fibre over $[g^{(k)}]$;
    \item taking homotopy groups at basepoints $g^{(k)}$ and  $g^{(k+1)} $,  the rest of the sequence is a long exact sequence of groups, ending with the stabiliser of $[g^{(k+1)}]$ in $\Ext^{k+1}_{\sO_{\pi^0X}}(\oL g^*\bL^Y, \H_{k+1}(\sO_X))$.
    \end{itemize}

    In particular, since   $Y$ is $n$-truncated,  we have $\Ext^{<-n}_{\sO_{\pi^0X}}(\oL g^*\bL^Y, \H_{k+1}(\sO_X))=0$, so it follows by induction that $\pi_i\oR\Map(\tau^{\le k}X, Y^{\sharp}) =0$ for $i>k+n$.
    
    \bigskip Finally, we have 
    $$
    \oR \Map(X, Y^{\sharp}) \simeq \ho\Lim_k  \oR \Map(\tau^{\le k}X, Y^{\sharp}).
    $$
    %
    These homotopy limits behave exactly like derived inverse limits in homological algebra,  with the Milnor exact sequence  of \cite[Proposition VI.2.15]{sht} giving us exact sequences
    \[
    \ast \to {\Lim_k}^1\, \pi_{i+1}\oR \Map(\tau^{\le k}X, Y^{\sharp}) \to \pi_i\oR \Map(X^{\sharp}, Y^{\sharp})
    \to \Lim_k \pi_{i}\oR \Map(\tau^{\le k}X, Y^{\sharp}) \to \ast
    \]
    of groups and pointed sets (basepoints omitted from the notation, but must be compatible).

    \subsubsection{Derived de Rham complexes}\label{derivedDRsn}
    
    
    The module $m \mapsto \bL^{X_m}$  is not homotopy-Cartesian when $X$ is  a derived Artin $n$-hypergroupoid, so it does not give a quasi-coherent complex on the associated derived stack $\fX:=X^{\sharp}$. However, \cite[Lemma 7.8]{stacks2} implies that when $X$ is levelwise fibrant (so $\bL^{X_m} \simeq \Omega^1_{X_m}$), the natural map from the homotopy-Cartesian complex $ \bL^X $ to $\Omega^1_X$ does induce a quasi-isomorphism on global sections 
    \[
    \oR\Gamma(\fX, \bL^{\fX}) \simeq  \oR\Gamma(X, \Omega^1_X):= \Tot^{\Pi}(i \mapsto \Gamma(X_i, \Omega^1_{X_i}))
    \] 
    and similarly on tensor powers, including
    \[
    \oR\Gamma(\fX, \L^p\bL^{\fX}) \simeq  \oR\Gamma(X, \Omega^p_X).
    \]
    
    Derived de Rham cohomology can then just be defined as
    \[
    \H^*\Tot^{\Pi}(i \mapsto \Gamma(X_i, \Tot^{\Pi}\Omega^{\bt}_{X_i}));
    \]
    over $\Cx$, this agrees with $\H^*(|\pi^0X(\Cx)_{\an}|,\Cx)$, for $|\pi^0X(\Cx)_{\an}|$ the realisation of the simplicial topological space $ \pi^0X(\Cx)_{\an}$.
    
    \begin{example}
        For $X=B\bG_m$ over $\Cx$, this gives derived de Rham cohomology as  $\H^*(|B\Cx^*|,\Cx) \cong \H^*(|BS^1|,\Cx) \cong \H^*(K(\Z,2),\Cx)\cong \H^*(\Cx\bP^{\infty},\Cx)\cong \Cx[u]$, for $u$ of degree $2$.
    \end{example}
    
    There is a Hodge filtration $F^p\Omega^{\bt}_X$ given by brutal truncation. Since $\Tot^{\Pi}F^p$ is the right derived functor of $\z^p$, this leads to:
    
    \begin{definition}[\cite{PTVV}\footnote{this characterisation essentially \cite{poisson}; see Remark \ref{PTVVterminology} for terminology}]\label{sympdefstack}
    
        The complex of    \emph{$n$-shifted pre-symplectic structures} \index{shifted pre-symplectic structure} on $X$  is $\tau^{\le 0}\oR\Gamma(X,(\Tot^{\Pi}F^2\Omega^{\bt}_X))[n+2])$. Hence homotopy classes are in $\H^{n+2}(\Tot^{\Pi}F^2\Omega^{\bt}_{X})$. 
    
        We say    
        $\omega$ is  \emph{symplectic}  \index{shifted symplectic structure} if it is non-degenerate in the sense that  the map $\oR\hom_{\sO_{\fX}}(\bL^{\fX},\sO_{\fX}) \to \bL^{\fX}[n]$ induced by $\omega_2 \in \H^n(X, \Omega^2_X) \cong  \oR\Gamma(\fX, \L^p\bL^{\fX})$ is a quasi-isomorphism of quasi-coherent complexes on $\fX$.
    \end{definition}
    
    \begin{example}
    The trace on $\GL_n$ gives rise to a $2$-shifted symplectic structure on $B\GL_n$. 
    \end{example}
    
    There is an equivalent characterisation of shifted symplectic structures in  \cite[\S 3]{poisson}  better suited for comparisons with Poisson structures, effectively replacing the derived Artin  hypergroupoid $X$ with a form of derived Deligne--Mumford hypergroupoid $\Spec D^*O(X^{\Delta})$, but built from  \emph{double} complexes with a graded-commutative product, with the extra cochain grading modelling stacky structure as a form of higher Lie algebroid
    \index{derived Lie algebroid}, similarly to \Cref{LRinfty}; also see \cite{DQDG, smallet}. 
    
    Shifted Poisson structures are then given by  shifted $L_{\infty}$ structures on these stacky cdgas, with the brackets all being multiderivations; see \cite[Examples 3.31
    ]{poisson} and \cite{safronovPoissonLie} for explicit descriptions of the resulting  $2$-shifted structures on quotient stacks $[Y/G]$ and of $1$-shifted structures on $BG$, respectively.

    
    
    \subsection{Artin--Lurie representability}\label{repsn}
    
    Anyone familiar with Artin representability for algebraic stacks \cite{Artin} will know that in the underived setting, axiomatising and constructing obstruction theories was one of the hardest steps; also see \cite{BehrendFantechi}. However, derived algebraic geometry produces obstruction theory for free as in \cite{Man2} or \S \ref{lesobssn}, giving rise  to derived representability theorems which can be significantly simpler than their underived counterparts.
    
    The landmark result is the representability theorem of \cite{lurie}, but it is formulated in a way which can make the conditions onerous to verify, so we will be presenting it in the simplified form established in \cite{drep}. 
    
    \begin{remark}
    The results from now on have only been developed in the setting of algebraic geometry. There are much weaker derived representability theorems in differential and analytic settings given by adapting \cite[Appendix C]{hag2}. Such results only apply when the underlying underived moduli functor is already known to  be representable; the main obstacle is in formulating an analogue of algebraisation for formal moduli, since differential and analytic moduli functors are usually only defined on finitely presented objects.
    \end{remark}

    \begin{definition}
    A functor $F\co d\Alg_R \to s\Set$ is said to be \emph{locally of finite presentation} (l.f.p.)\index{locally of finite presentation}\index{l.f.p.} if it preserves filtered colimits, or equivalently colimits indexed by directed sets, i.e. if the natural map 
    \[
    \LLim_i F(A(i)) \to F(\LLim_i A(i))
    \]
    is a weak equivalence.\footnote{Note that we do not need to write these as homotopy colimits, since filtered colimits are already exact, so are their own left-derived functors.}
    
    A functor $F\co d\Alg_R \to s\Set$ is said to be \emph{almost of finite presentation} (a.f.p.)\index{almost of finite presentation}\index{a.f.p.} if it preserves filtered colimits (or equivalently directed colimits) of objects which are uniformly bounded  in the sense that there exists some $n$ for which the underlying chain complexes are all concentrated in degrees $\le n$. 
    \end{definition}
    
    \begin{example}
    If $U =\Spec S$ is a derived affine scheme, then $U^{\sharp}=\oR\Map_{d\Alg_R}(S,-)$ is l.f.p. if and only if $S$ has a finitely generated cofibrant model, whereas $U^{\sharp}$ is a.f.p.  if and only if $S$
    has cofibrant model with finitely many generators in each degree.
    
    Beware that a finitely presented algebra is not in general l.f.p as a cdga unless its cotangent complex is perfect, though it will be always be a.f.p. if the base is Noetherian.
    
    More generally, if  $X$ is an Artin $n$-hypergroupoid for which the derived affine scheme  $X_0$ is l.f.p. or a.f.p., then the functor $X^{\sharp}$ will be l.f.p. or a.f.p., essentially because smooth morphisms are l.f.p.
    \end{example}

    In order to state the representability theorems, from now on we will work over a base cdga $R$ which is  a derived G-ring admitting a dualising module (in the sense of \cite[Definition 3.6.1]{lurie}). Examples satisfying this hypothesis are any field, the integers, any Gorenstein local ring, and anything of finite type over any of these.  
    
    Our first formulation of the representability theorem is \cite[Corollary \ref{drep-lurierep2}]{drep}, substantially simplifying \cite{lurie}:
    
    \begin{theorem}\label{lurierep2}
    A functor  $F: d\Alg_R \to s\Set$ is an $n$-truncated  geometric derived stack which is almost of finite presentation if and only if 
    the following conditions hold:
    \begin{enumerate}
    
    \item\label{htpypreservingitem}  $F$ is homotopy-preserving:\index{homotopy-preserving} it maps quasi-isomorphisms 
    to weak equivalences.
    \item\label{discreteitem} For all discrete $\H_0(R)$-algebras $A$, $F(A)$ is $n$-truncated, i.e. $\pi_iF(A)=0$ for all $i>n$.
    
    \item\label{htpyhgsitem}
    $F$ is homotopy-homogeneous, i.e. for all square-zero extensions $A \onto C$ and all maps $B \to C$, the map
    $$
    F(A\by_CB) \to F(A)\by_{F(C)}^hF(B)
    $$
    is an equivalence.
    
    \item\label{nilcompleteitem} $F$ is nilcomplete,\index{nilcomplete} i.e. for all $A$, the map 
    $$
    F(A) \to {\Lim}^h F(P_kA)
    $$
    is an equivalence, for $\{P_kA\}$ the Postnikov tower of $A$.
    
    \item\label{colimitem} $\pi^0F:  \Alg_{\H_0(R)}\to s\Set$ preserves filtered colimits (equivalently colimits indexed by directed sets), i.e. 
    \begin{enumerate}  \item\label{colim1item} $\pi_0\pi^0F:  \Alg_{\H_0(R)} \to \Set$  preserves filtered colimits.
    
    \item\label{colim2item} For all $A \in \Alg_{\H_0(R)}$ and all $x \in F(A)$, the functors $\pi_i(\pi^0F,x): \Alg_A \to \Set$  preserve filtered colimits for all $i>0$.
    \end{enumerate}
    
    \item\label{shf1item} $\pi^0F:\Alg_{\H_0(R)} \to s\Set$ is a hypersheaf for the \'etale topology. 
    
    \item \label{shf2item}
    for all finitely generated integral domains $A \in \Alg_{\H_0(R)}$, all $x \in F(A)_0$ and all \'etale morphisms $f:A \to A'$, the maps
    \[
    \DD_x^*(F, A)\ten_AA' \to \DD_{f(x)}^*(F, A')
    \]
    on tangent cohomology groups are isomorphisms.
    
    \item\label{colimbitem}  for all finitely generated $A \in \Alg_{\H_0(R)}$  and all $x \in F(A)_0$, the functors $\DD^i_x(F, -): \Mod_A \to \Ab$ preserve filtered colimits for all $i>0$.
    
    \item\label{fgitem} for all finitely generated integral domains $A \in \Alg_{\H_0(R)}$  and all $x \in F(A)_0$, the groups $\DD^i_x(F, A)$ are all  finitely generated $A$-modules.
    
    \item\label{completeitem} formal effectiveness:\index{formally effective}
    for all complete discrete local Noetherian  $\H_0(R)$-algebras $A$, with maximal ideal $\m$, the map
    $$
    F(A) \to {\Lim_n}^h F(A/\m^r)
    $$
    is a weak equivalence (see \cite[Remark \ref{drep-formalexistrk}]{drep} for a reformulation).
    \end{enumerate}
    
    $F$ is moreover strongly quasi-compact (so built from $d\Aff$, not $\coprod d\Aff$)
    if and only if for all sets $S$  of separably closed fields, the map
    $$
    F(\prod_{k \in S} k) \to (\prod_{k \in S} F(k))
    $$
    is a weak equivalence in $s\Set$.
    
    
    \end{theorem}

    \begin{remarks}
    Note that of the conditions in the theorem as stated in this form, only conditions (\ref{htpypreservingitem}), (\ref{htpyhgsitem}) and (\ref{nilcompleteitem}) are fully derived. The conditions (\ref{discreteitem}) (\ref{colimitem}), (\ref{shf1item}) and (\ref{completeitem}) are purely underived in nature, taking only discrete input, and in particular are satisfied if the underived truncation $\pi^0F$ is representable, while the conditions (\ref{shf2item}), (\ref{colimbitem}) and (\ref{fgitem}) relate to tangent cohomology groups.  The hardest conditions to check are  usually homotopy-homogeneity (\ref{htpyhgsitem}) and formal effectiveness (\ref{completeitem}).
    
    Because derived algebraic geometry automatically takes care of obstructions, it is often easier to establish representability of the underived  moduli functor $\pi^0F$ by  checking the conditions of Theorem \ref{lurierep2} for $F$, rather than checking Artin's conditions   \cite{Artin} and their higher analogues for $\pi^0F$. Beware that a natural equivalence of moduli functors does not necessarily give an equivalence of the corresponding derived moduli functors, a classical example being the derived Quot and Hilbert schemes  of \cite{Quot,Hilb}. 
    \end{remarks}
    
    \bigskip 
    As we saw back in  \S \ref{postnikovsn}, derived structure is infinitesimal in nature, and this now motivates a variant of the representability theorem which just looks at functors on derived rings which are bounded nilpotent extensions of discrete rings.     
    
    \begin{definition}
    Define $d\cN_R^{\flat}$ \index{dNR@$d\cN_R^{\flat}$} to be the full subcategory of $d\Alg_R$  consisting of objects $A$
    for which 
    \begin{enumerate}
    \item the map $A \to \H_0(A)$  has nilpotent kernel.
    \item $A_i=0$ (or $N_iA=0$ in the simplicial case $A \in s\Alg_R$) for all $i \gg 0$.
    \end{enumerate}
    \end{definition}
    
    \begin{exercise}
    Show that any homotopy-preserving a.f.p. nilcomplete functor $F \co d\Alg_R \to s\Set$ is determined by its restriction to $d\cN_R^{\flat}$, bearing in mind that $R$ is Noetherian.
    \end{exercise}

    The following is  \cite[Theorem \ref{drep-lurierep3}]{drep}; it effectively says that we can restrict to functors on $d\cN_R^{\flat}$ and drop the nilcompleteness condition.
    
    \begin{theorem}\label{lurierep3}
    Let $R$ be a  Noetherian G-ring admitting a dualising module. 
    
    Take a functor $F: d\cN_R^{\flat} \to s\Set$. Then $F$ is the restriction of an almost finitely presented  derived $n$-truncated geometric stack $F':d\Alg_R \to s\Set$ if and only if 
    the following conditions hold
    
    \begin{enumerate}
    
    \item $F$ maps square-zero acyclic extensions to weak equivalences.
    
    \item\label{ntrunc} For all discrete rings $A$, $F(A)$ is $n$-truncated, i.e. $\pi_iF(A)=0$ for all $i>n$.
    
    \item
    $F$ is homotopy-homogeneous.
    %
    %
    \item $\pi^0F:\Alg_{\H_0(R)} \to s\Set$ is a hypersheaf for the \'etale topology. 
    
    \item\label{afp1} $\pi^0F: \Alg_{\H_0(R)} \to \Ho(s\Set)$  preserves filtered colimits.
    
    
    \item\label{formaleff} for all complete discrete local Noetherian  $\H_0(R)$-algebras $A$, with maximal ideal $\m$, the map
    $
    \pi^0F(A) \to {\Lim}^h \pi^0F(A/\m^r)
    $
    is a weak equivalence. 
    
    \item 
    for all finitely generated integral domains $A \in \Alg_{\H_0(R)}$, all $x \in F(A)_0$ and all \'etale morphisms $f:A \to A'$, the maps
    $
    D_x^*(F, A)\ten_AA' \to D_{f(x)}^*(F, A')
    $
    are isomorphisms.
    
    \item\label{afp2} for all finitely generated $A \in \Alg_{\H_0(R)}$  and all $x \in F(A)_0$, the functors $D^i_x(F, -): \Mod_A \to \Ab$ preserve filtered colimits for all $i>0$.
    
    \item\label{Difg} for all finitely generated integral domains $A \in \Alg_{\H_0(R)}$  and all $x \in F(A)_0$, the groups $D^i_x(F, A)$ are all finitely generated $A$-modules.
    \end{enumerate}
    Moreover, $F'$ is uniquely determined by $F$ (up to weak equivalence).
    \end{theorem}
    
    \begin{remark}
    There is a much simpler  representability theorem for functors on local dg Artinian rings, essentially requiring only the homogeneity condition. Such derived Schlessinger functors\index{derived Schlessinger functors $\cS$}  (denoted $\cS$ in \cite{ddt1}) also tie in with other approaches to derived deformation theory such as \emph{dg Lie algebras}\index{dg Lie algebra} (dglas) and $L_{\infty}$-algebras. The relations between these were proved in \cite[Corollary 4.57, Theorem 2.30 and Remarks 4.28]{ddt1} (largely rediscovered as the main result of \cite{lurieDGLA,lurieDAG10}); for a survey see  \cite{maggiolo}\footnote{Beware that as in the published version of \cite{ddt1}, the statements in \cite{maggiolo}  giving comparisons with Manetti's set-valued extended deformation functors are missing necessary hypotheses;  see \cite[\S 6]{GuanLazarevShengTangII} for the refined statements.}.

        \end{remark}
    
    \begin{digression}[Comparison with dglas]\label{dglacomp}
The comparison between derived Schlessinger functors and dglas combines Heller's representability theorem \cite{heller} with the contravariant Koszul duality adjunction between dglas and local pro-Artinian $\Z$-graded cdgas. 
    
    As in \cite[Theorem 3.2]{hinstack}, that adjunction is a Quillen equivalence, provided we endow cdgas with a stronger notion of equivalence than quasi-isomorphism. By \cite[Proposition 4.36]{ddt1}, these equivalences are cogenerated by acyclic small extensions of local Artinian cdgas: surjections $f\co A \onto B$ with $\ker(f)\cdot \m(A)=0$ and $\H_*\ker(f)=0$. Similar adjunctions exist for any Koszul dual pair of operads, with the commutative--Lie adjunction the other way round used to  compare the Quillen and Sullivan models for rational homotopy theory, an observation which anticipated both derived deformation theory and operadic Koszul duality in \cite{drinfeldtoschechtman} (see \cite{GK} for the latter theory).
  
    The significance of the dgla comparison result has however been somewhat exaggerated in recent years, after \cite{lurieDGLA} conflated moduli problems\index{formal moduli problems@``formal moduli problems''} with derived Schlessinger functors 
    \index{Schlessinger's conditions} 
    (thereby turning a meta-conjecture into a definition). It is hard to imagine an experienced deformation theorist resorting to the theorem to infer the existence of a dgla governing a given deformation problem; it is almost always easier to write down the governing dgla than even to formulate the derived version of a deformation problem, let alone verify Schlessinger's conditions, and very general constructions were available off the shelf as early as \cite{KS,hinichDefsHtpyAlg,paper2} (the second of those includes a counterexample --- see \Cref{hinichcounterex}). Be wary of any source assuming the conditions are satisfied axiomatically by a given moduli problem.
    
    The formulation of \cite{lurieDGLA} is slightly weaker than in \cite{ddt1},  requiring the functor to be defined on a larger category of cdgas which are only homologically Artinian; that our strictly Artinian cdgas give an equivalent theory follows from \cite[Proposition 2.7]{drep} or \cite[Corollary 4.4.4]{boothDefThPt}; the latter  adapts directly to our commutative setting and avoids Noetherian arguments. In the weaker formulation, the Maurer--Cartan formalism no longer behaves adequately, and the link with operadic Koszul duality is obscured, but properties of the cotangent complex can be used off the shelf, without having to establish the homotopy theory of pro-Artinian objects as in \cite{ddt1}.
    \end{digression}
    
    
    \subsection{Examples}
    
    \begin{examples}\label{modex}
    The following simplicial-category valued functors $\C\co d\Alg \to s\Cat$ \index{sCat@$s\Cat$|see {simplicial category}} are homotopy-homogeneous, homotopy-preserving and  \'etale hypersheaves (though too big to be representable). For objects $A \in d\Alg$ and morphisms  $A \to B$  in $d\Alg$: 
    \begin{enumerate}
    \item  Take $\C(A)$ to be  the simplicial  category of strongly quasi-compact  $n$-geometric derived Artin stacks $\fX$ over $\Spec A$, with the simplicial functor $\C(A) \to \C(B)$ given by $\fX \mapsto \fX\by^h_{\oR \Spec A}\oR \Spec B$ (so $\sO_{\fX} \mapsto \sO_{\fX}\ten^{\oL}_AB$). 
    \smallskip
    
    \item\label{modcx} For a fixed  derived Artin stack $\fX$, take $\C(A)$ to be the simplicial  category of bounded below
    quasi-coherent chain complexes\footnote{i.e. bounded above if written as cochain complexes} 
    
    on $\fX \by \Spec A$, with the simplicial functor  $\C(A) \to \C(B)$ given by $\sE \mapsto \sE\ten^{\oL}_AB$. 
    \smallskip
    
    \item Take $\C(A)$ to be the simplicial  category of pairs $(\fX,\sE)$, for $\fX, \sE$ as above.
    \smallskip
    
    \item Given one of the functors $\C'\co d\Alg \to s\Cat$ above and a small category $I$, take $\C(A):= \C'(A)^I$ to be the 
    simplicial category diagrams of diagrams of shape $I$ in $\C'(A)$. More generally, given a functor $J \to I$ of small categories and an object $Z \in \C'(R)^J$, we can define $\C\co d\Alg_R \to s\Cat$ by $\C'(A)^I\by^h_{\C'(A)^J}\{Z\}$.
    
    Examples of this form include moduli of derived Artin stacks over a base $\fY$, or moduli of pairs $\fX \to \fY$, or of maps:
    \[
    A \mapsto \oR\Map(\fX \by \Spec A, \fY).
    \]
    \end{enumerate}
    
    \end{examples}
    \begin{proof}
    These all appear in \cite{dmsch}. The proofs use hypergroupoids intensively.
    \end{proof}

    \begin{digression}\label{hinichcounterex}
        The bounded below condition in Example \ref{modex}.(\ref{modcx}) arises because of the subtleties of derived base change for unbounded complexes as in \cite{spaltenstein}. 
        
        For instance, for any commutative ring $k$ consider as in \cite[Example 4.3]{hinichDefsHtpyAlg}, 
                the complex $M:=\ldots \xra{\eps} k[\eps]  \xra{\eps} k[\eps]  \xra{\eps} \ldots$ over the dual numbers $k[\eps]$. It is an extension of $V:=\bigoplus_{n \in \Z} k[n]$ by itself, so corresponds to a class in $\EExt^1_k(V,V)$, but $M$ is acyclic so $M\ten^{\oL}_{k[\eps]}k \simeq 0$, meaning $M$ is
        not a derived deformation of $V$. If follows that the functor of derived deformations of $V$ is not homotopy-homogeneous, since its tangent space is not additive.
        
        %
        These issues might be bypassed by instead working with derived categories of the second kind, and in particular 
        the projective model structure of the second kind \cite[8.3]{positselskiDerivedCategories}; 
        in practice, few unbounded complexes would  satisfy the finiteness conditions on self-$\Ext$s needed for representability anyway.
    \end{digression}
    
    \begin{construction}
    In order to associate  moduli functors to our simplicial category-valued functors,  we  first discard any morphisms which are not equivalences, so we restrict to the  simplicial subcategory $\cW(\C)\subset \C$ of  homotopy equivalences, given by
    \[
    \cW(C):=\C\by_{\pi_0\C}\mathrm{core}(\pi_0\C), 
    \]
    for $\mathrm{core}(\pi_0\C) \subset \pi_0\C$ the maximal subgroupoid, which contains all the objects of $\pi_0\C$ with morphisms given by the isomorphisms between them.

    Now, the nerve of a category is a simplicial set, and this extends to a construction giving the nerve $B\C$ of a simplicial category $\C$.\footnote{Explicitly, we first form the bisimplicial set $n \mapsto B\C_n$, then  take the diagonal $\diag$, or more efficiently the codiagonal $\bar{W}$ of \cite{CRdiag}, to give a simplicial set.} Taking
    \[
    B\cW(\C) \co d\Alg \to s\Set
    \]
    then gives the moduli stack of objects in $\C$.  
    \end{construction}
    
    \begin{examples}
    For each case in \Cref{modex}, we now  look at tangent cohomology:
    \begin{enumerate}
    \item For moduli of derived Artin $n$-stacks, at a point $[\fX]\in \C(A)$ the tangent cohomology groups are
    \[
    D^i_{[\fX]}(B\cW(\C),M)\cong \Ext^{i+1}_{\fX}(\bL^{\fX}, \O_{\fX}\ten^{\oL}_A M).
    \]
    
    \item For moduli of quasi-coherent complexes on $X$, at a point $[\sE]\in \C(A)$ the tangent cohomology groups are
    \[
    D^i_{[\sE]}(B\cW(\C), M)= \Ext^{i+1}_{\O_{X}}(\sE,\sE\ten^{\oL}_A M).
    \]
    
    \item For moduli of pairs $(\fX,\sE)$, we have a long exact sequence
    \[
    \Ext^{i+1}_{\O_{\fX}}(\sE,\sE\ten^{\oL}_A M) \to D^i_{[(\fX,\sE)]}(B\cW(\C), M) \to  \Ext^{i+1}_{\fX}(\bL^{\fX}, \O_{\fX}\ten^{\oL}_A M) \to \ldots,
    \]
    in which the boundary map is given by the Atiyah class of $\sE$.
    
    \item For moduli of maps $\fX \to \fY$ (both fixed), we have
    \[
    D^i_{[f\co \fX\to \fY]}(B\cW(\C),M) = \Ext^{i}_{\sO_{\fX}}(\oL f^*\bL^{\fY}, \O_{\fX}\ten^{\oL}_A M),
    \]
    similarly to \S \ref{morphismsrevsn}.
    \end{enumerate}
    \end{examples}
    
    These groups are all far too big to satisfy the finiteness conditions in general, so in each case we have to cut down to some suitably open subfunctor with good finiteness properties:
    
    %
    
    
    \begin{example}[Derived moduli of schemes]  Given $A \in d\Alg$, we can look at  derived Zariski $1$-hypergroupoids $X$ which are homotopy-flat over $\Spec A$, with suitable restrictions on $\pi^0X$ (proper, dimension $d$, \ldots). A specific   example of this type  is given by  moduli of  smooth proper curves, representable by a $1$-truncated derived Artin stack.
    \end{example}

    \begin{example}[Moduli of perfect complexes on a proper scheme (or stack) $X$] In this  example, given $A\in d\Alg$ we look at 
    quasi-coherent complexes on $X \by \Spec A$ which are homotopy-flat over $\Spec A$, and perfect on pulling back to $X \by \pi^0\Spec A$.
    This is an $\infty$-geometric \index{infinity-geometric@$\infty$-geometric} derived Artin stack, in the sense that it is a nested union of open $n$-geometric derived Artin substacks, for varying $n$. Explicitly,  restricting the complexes to live in degrees $[a,b]$ gives an open $(b-a+1)$-truncated derived moduli stack.
    \end{example}
    
    \begin{example}[Derived moduli of  polarised schemes]
        Given $A\in d\Alg$, look at  pairs $(X,\sL)$, with $X$ a derived Zariski $1$-hypergroupoid
        homotopy-flat over $\Spec A$ and  $\sL$ a quasi-coherent complex on $X$, such that  $(\pi^0X, \sL\ten_A^{\oL}\H_0(A))$ is a polarised projective  scheme, so $\sL$ is an ample line bundle.
        We can also  fix the Hilbert polynomial to give a smaller open subfunctor.
    \end{example}
    
    \subsection{Examples in detail}
    We still follow \cite{dmsch}, in particular \S \ref{dmsch-modschsn}; the examples are more general than the title of the paper might suggest.
    
    %
    %
    
        Take  a category-valued  functor $\C: \Alg_{\H_0(R)} \to \Cat$ and an   property $\oP$ on isomorphism classes of objects of $\C$ which is functorial in the sense that whenever $x \in \C(A)$ satisfies $\oP$, its image  $\C(f)(x) \in \C(B)$ also satisfies $\oP$, for any morphism $f\co A \to B$ in $\Alg_{\H_0(R)}$. Then:
    \begin{definition}\label{opencatmodified}
    Say that $\oP$ is an \emph{open property} \index{open property} if it is closed under deformations in the sense that for any square-zero extension $A \to B$, an object of $\C(A)$ satisfies $\oP$ whenever its image in $\C(B)$ does.
    \end{definition}

    \begin{definition}\label{etlocalprop}
    Say that  $\oP$ is an \emph{\'etale local property} \index{etale local property@\'etale local property} if for any $A \in \Alg_{\H_0(R)}$ and any  \'etale cover $\{f_i \co A \to B_i\}_{i \in I}$, an object of $\C(A)$ satisfies $\oP$ whenever its images in $\C(B_i)$ all do.
    \end{definition}

    \begin{definition}
    Given  $\C \co d\cN^{\flat}_R  \to s\Cat$ and  a functorial property $\oP$ on objects of $\pi^0\C$, extend $\oP$ to $\C$ by saying that an object of $\C(A)$ satisfies $\oP$ if and only if its image in $\C(\H_0(A))$ does so.\footnote{This convention gives a correspondence between open properties in the sense of Definition \ref{opencatmodified} and the open simplicial subcategories of \cite[Definition \ref{dmsch-opencat}]{dmsch}.}
    \end{definition}
    
    \begin{example}
    For instance, this means that we would declare a derived Artin stack $\fX$ over $\oR \Spec A$ to be an algebraic curve if and only if the derived stack $\fX\by^h_{\oR\Spec A}\Spec \H_0(A)$ (which has structure sheaf $\sO_{\fX}\ten^{\oL}_{A}\H_0(A)$) is an  algebraic curve over $\Spec \H_0(A)$.
    \end{example}

    \begin{lemma}
    Take   a homotopy-preserving and homotopy-homogeneous  \'etale hypersheaf $\C \co d\cN^{\flat}_R \to s\Cat$. If $\oP$ is a functorial property on $\pi^0\C$ which is open and \'etale local, then the full subfunctor 
    $\tilde{\cM}:d\cN^{\flat}_R \to s\Cat$ of $\C$ on objects satisfying $\oP$ is 
    a  homotopy-preserving and homotopy-homogeneous \'etale hypersheaf.
    \end{lemma}
    \begin{proof}
    \cite[Proposition \ref{dmsch-descentlemma}, Lemmas \ref{dmsch-2htpicgood} and \ref{dmsch-2hgsgood}]{dmsch}.
    \end{proof}
    
    In particular, this means that for subfunctors cut out by open, \'etale local properties in any of any of the functors in \Cref{modex}, we need only check conditions \ref{ntrunc}, \ref{afp1}--\ref{Difg} of \Cref{lurierep3} to establish representability.  With the exception of condition \ref{formaleff} (effectiveness of formal deformations), these amount to finiteness properties of the relevant cohomology groups.

    For more general criteria to establish homotopy-homogeneity for 
    simplicial category-valued functors $\C$, see \cite[Proposition \ref{dmsch-wetale}]{dmsch}.
    
    \subsubsection{Moduli of quasi-coherent complexes}
    The following is \cite[Theorem \ref{dmsch-representdmod}]{dmsch}: 
    \begin{theorem}\label{representdmod}
    Take a strongly quasi-compact $m$-geometric  derived Artin stack $\fX$   over $R$.
    
    Assume that we have  
    an 
    open, \'etale local condition $\oP$ for objects of the functor $ A \mapsto 
    \cD^-(\fX\ten_R^{\oL}A)$, the derived category of  quasi-coherent complexes on $\fX\ten_R^{\oL}A $ which are bounded above as cochain  complexes\footnote{As in \S\S \ref{qucohcx}, \ref{qucohsn1}, \ref{qucohsn2} this is 
    a full subcategory of $\Ho(\mathrm{Cart}(\sO_{\fX}\ten_R^{\oL}A))$. It consists of complexes  $\sE$ with $\H_i\sE(U)=0$ for $i \ll 0$ (in homological, not cohomological, grading), for all affine atlases $U$ of $\fX$.}.
    
    \medskip  
    Also assume that this satisfies the following conditions:
    \begin{enumerate}

    \item\label{dmafp2} for all finitely generated $A \in \Alg_{\H_0(R)}$  and all $\sE\in \cD^-(\fX\ten_R^{\oL}A)$ satisfying $\oP$, the functors 
    \[
    \EExt^i_{\fX\ten_R^{\oL}A}(\sE,\sE\ten_A^{\oL}- ) :\Mod_A \to \Ab
    \]
    preserve filtered colimits (equivalently, colimits indexed by directed sets) for all $i$.
    
    \item for all finitely generated integral domains $A \in \Alg_{\H_0(R)}$  and all $\sE\in \cD^-(\fX\ten_R^{\oL}A)$ satisfying $\oP$, the groups $\EExt^i_{\fX\ten_R^{\oL}A}(\sE, \sE )$ are all  finitely generated $A$-modules.
    
    \item
    The functor $|P|:  \Alg_{\H_0(R)} \to \Set$  of isomorphism classes of objects satisfying $\oP$ preserves filtered colimits.
    
    \item\label{Ahatcdn} for all complete discrete local Noetherian  $\H_0(R)$-algebras $A$, with maximal ideal $\m$, the map
    \[
    |P|(A) \to \Lim_r  |P|(A/\m^r)
    \] 
    is an isomorphism, as are the maps 
    \begin{align*}
    \EExt^i_{\fX\ten_R^{\oL}A}(\sE, \sE ) &\to  \EExt^i_{\fX\ten_R^{\oL}A}(\sE, \oR\Lim_r\sE/\m^r )\\
    &\cong\Lim_r \EExt^i_{\fX\ten_R^{\oL}A}(\sE, \sE/\m^r )
    \end{align*}
    for  all $\sE$ satisfying $\oP$ and all $i\le 0$. 
    
    \item For any  $\H_0(R)$-algebra $A$ and $\sE \in \cD^-(\fX\ten_R^{\oL}A)$ satisfying $\oP$, the cohomology groups $\EExt^i_{\fX\ten_R^{\oL}A}(\sE, \sE ) $ vanish for $i\le -n$. 
    
    \end{enumerate}
    
    Let $\tilde{\cM}:d\cN^{\flat}_R \to s\Cat$  be given by sending $A$ to  the simplicial category  of quasi-coherent complexes $\sE$ on $\fX\ten_R^{\oL}A$ for which $\sE\ten^{\oL}_A\H_0(A) \in \cD^-(\fX\ten_R^{\oL}\H_0(A))$ satisfies $\oP$. Let $\cW\tilde{\cM}$ be the full simplicial subcategory of quasi-isomorphisms.
    
    \medskip
    {\bf Then the nerve of $\cW\tilde{\cM}$ 
    is an $n$-truncated derived Artin stack.}
    \end{theorem}
    
    \begin{examples}
    One example of an open, \'etale local condition is to ask that $\sE$ be a perfect complex, 
    and we  could  then impose a further such condition by fixing its Euler characteristic. 
    
    Another open, \'etale local condition would be to impose bounds on $\sE$, asking that it only live in degrees $[a,b]$, provided a flatness condition is imposed to ensure functoriality, as we need the derived pullbacks $\sE\ten^{\oL}_AB$ to satisfy the same constraint. Similar considerations apply for perverse $t$-structures, and in particular the moduli stack  of objects living in 
    the heart of a  $t$-structure will be $1$-truncated because we have no negative $\Ext$s.
    
    We can also apply Theorem \ref{representdmod} to study derived moduli of Higgs bundles, for instance. One interpretation of a Higgs bundle on a smooth proper scheme $X$ is as a quasi-coherent sheaf $\sE$ on the cotangent scheme $T^*X=\oSpec_X\Symm_{\sO_X}\sT_X $, such that $\sE$  is a vector bundle when regarded as a sheaf on $X$. This defines an open, \'etale local condition on the functor  $A \mapsto \cD^-(T^*X\ten A)$, so gives rise to a derived moduli stack, with tangent spaces given by Higgs cohomology.
    
    There is a variant of this example for the derived de Rham moduli space of vector bundles with flat connection, replacing $\Symm_{\sO_X}\sT_X$ with the ring of differential operators. Since the latter is non-commutative, we cannot appeal directly to Theorem \ref{representdmod}, but the same proof adapts verbatim. 
    \end{examples}

    \subsubsection{Moduli of derived Artin stacks}
    
    \cite[Theorem \ref{dmsch-representdaffine}]{dmsch} gives a similar statement for moduli of derived Artin $n$-stacks (and thus any subcategories such as derived DM $n$-stacks, derived schemes, ...),
    taking
    open, \'etale local conditions $\oP$ on
    the homotopy category of $n$-truncated  derived Artin stacks $X$ over a fixed  base $Y$. The relevant cohomology groups are now
    \[
    \EExt^i_{X}(\bL^{X/Y_A},\O_{X}\ten_A^{\oL}- ): \Mod_A \to \Ab,
    \]
    and the resulting moduli stack $\tilde{\cM}$ is $(n+1)$-truncated.
    
    \bigskip  
    In detail, the conditions become:
    \begin{enumerate}
    
    \item
    for all finitely generated $A \in \Alg_{\H_0(R)}$  and all $X$ over $A$ satisfying $\oP$, the functors 
    \[
    \EExt^i_{X}(\bL^{X/Y_A},\O_{X}\ten_A^{\oL}- ): \Mod_A \to \Ab
    \]
    preserve filtered colimits for all $i>1$.
    
    \item for all finitely generated integral domains $A \in \Alg_{\H_0(R)}$  and all $X$ over $A$ satisfying $\oP$, the groups $\EExt^i_{X}(\bL^{X/Y_A}, \O_{X} )$ are all  finitely generated $A$-modules.
    
    \item for all complete discrete local Noetherian  $\H_0(R)$-algebras $A$, with maximal ideal $\m$, the map
    $$
    \tilde{\cM}(A) \to {\Lim_r}^h \tilde{\cM}(A/\m^r)
    $$
    is a weak equivalence.  

    \item $\tilde{\cM}: \Alg_{\H_0(R)} \to s\Cat$  preserves filtered  colimits (i.e. the simplicial functor $\LLim_i\tilde{\cM}( A(i))\to \tilde{\cM}(A)$ is a weak equivalence for all systems $A(i)$ indexed by directed sets.) 
    %
    
    \end{enumerate}
    
    \begin{example}
    If we let $Y$ be the stack $B\bG_m$, then one open, \'etale local condition on derived Artin stacks $X$ over $B\bG_m$ is to ask that $X$ be a projective scheme flat over the base $R$, with $X \to B\bG_m$ the morphism associated to an ample line bundle on $X$. We could also fix the Hilbert polynomial associated to this line bundle. The theorem then gives us representability of derived moduli stacks of polarised projective schemes.
    \end{example}
    
    \begin{remark}
    There are some earlier examples of representable derived moduli functors in the literature: \begin{itemize}
    \item Stable curves, line bundles and closed subschemes were addressed in \cite{lurie}, although for stable curves 
    the derived moduli stack is just the classical underived moduli stack, and the objects parametrised by the   derived Hilbert scheme there were not {\it a priori} derived schemes in  the usual sense, cf. \Cref{dschfootnote}.
    \item In \cite{hag2}, local systems, finite algebras over an operad,\footnote{The methods of \cite{ddt2,higher,dmc} apply to algebras over more general monads.} and mapping stacks were addressed; the last persist as the most popular way to construct  representable functors.
    \item For  associative dg-algebras $A$ of finite type,  representability for moduli of  complexes of $A$-modules perfect over the base ring $k$ was established directly in \cite{toenvaquie}, and hence for moduli of $T$-modules in  perfect $k$-complexes for any dg category $T$ derived Morita equivalent to such a dga $A$. 
    \end{itemize} 
    \end{remark}
    
    
    \subsection{Pre-representability}
    Details for this section appear in \cite[\S \ref{drep-prerepsn}]{drep}.
    
    The idea behind pre-representability is to generalise the way we associate derived functors  to smooth schemes, which can be useful when constructing things like derived  quotient stacks, or morphisms between derived stacks.
    
    \begin{definition}
    Given
    a functor $F: d\cN^{\flat} \to s\Set$, we  define  a functor $\underline{F}: d\cN^{\flat}\to ss\Set$ to  the category of bisimplicial sets by 
    $$
    \underline{F}(A)_{n} :=  F(A^{\Delta^n}),
    $$
    \index{Funderline@$\underline{F}$}
    where for $A \in \catdga[R]$, we set $ A^{\Delta^n}:= \tau_{\ge 0} (A \ten \Omega^{\bt}(\Delta^n))$ as in \Cref{pathexamples}, while for $A \in s\Alg$ the simplicial algebra $A^{\Delta^n}$ is given by $(A^{\Delta^n})_i := \Hom_{s\Set}(\Delta^i \by \Delta^n,A)$.
    
    \end{definition}
    
    The results of \cite[\S \ref{drep-prerepsn}]{drep} and in particular \cite[Theorem \ref{drep-lurieprerep}]{drep} then show that if $F$ satisfies the conditions of Theorem \ref{lurierep3}, but mapping acyclic square-zero extensions to surjections rather than weak equivalences,  then the diagonal  $\diag \uline{F}$ is an $n$-truncated derived Artin stack. We then think of $F$ as being pre-representable, by close analogy with the predeformation functors of \cite{Man2}.

    One way to interpret the construction is that $\diag \uline{F}$  is the right-derived functor of $F$  with respect to quasi-isomorphisms in $d\cN^{\flat}_R$. Note that  if $F$ was already representable, then the natural map $F \to \diag \uline{F}$ is a weak equivalence.
    
    \bigskip
    Constructing morphisms $f \co \fX \to \fY$ between derived stacks  can be cumbersome to attempt directly  because derived stacks encode so much data, but pre-representable functors can provide a simplification. 
    Instead of  constructing the morphism $f$ itself, if we can characterise $\fX$ as equivalent to $\diag \uline{X}$ for some much smaller functor $X$, then it suffices to construct a morphism $X \to \fY$, since 
    \[
    \fX \simeq \diag\uline{X} \to \diag\uline{\fY} \simeq \fY.
    \]

    \begin{example}\label{dgexample}
    If $X$ is a dg-manifold (in the sense of Definition \ref{dgmfddef}), then the functor $X: \operatorname{dg_+}\cN^{\flat}_R \to \Set$ given by $X(A):= \Hom((\Spec (A_0), A), X)$ satisfies the conditions of \cite[Theorem \ref{drep-lurieprerep}]{drep}, so $\uline{X}: \operatorname{dg_+}\cN^{\flat}_R \to s\Set$ is a $0$-truncated derived Artin (or equivalently DM) stack, i.e. a derived algebraic space. \index{derived algebraic space}
    
    However, the space of morphisms $\uline{X} \to F$ to a derived stack $F$, which would be complicated to calculate directly, is just equivalent to the space of  morphisms $X \to F$ for the functor $X$ above, so is given by the simplicial set $\oR\Gamma(X^0, F(\sO_X))$, which can be calculated via a \v Cech complex as in Example \ref{categs}.(\ref{categsSet}).
    \end{example}
    See \cite[\S\S 3--6]{dmc} for many more examples of $1$-truncated derived Artin moduli stacks constructed from pre-representable groupoid-valued functors.

\subsection{Addendum: derived hypergroupoids \texorpdfstring{\`a}{à}
    la  \texorpdfstring{\cite{stacks2}}{[Pri09]}}\label{addendum}
    
\subsubsection{Homotopy derived hypergroupoids}\label{cfstacks2sn}
    
    The definitions  given in \S \ref{dstackdefsn} for homotopy derived Artin and DM hypergroupoids are not the same as those of \cite{stacks2,stacksintro} (which are cast to work in more general settings), but are equivalent. For want of a suitable reference, we now  prove the equivalence of the two sets of definitions, but readers should regard this section as a glorified footnote.
    
    As with simplicial affine schemes, we still have notions of  matching objects $M_{\pd \Delta^m}(X)$ and partial matching objects $M_{\L^{m,k}}(X)$ for simplicial derived affine schemes $X$. Explicitly, $M_{\pd \Delta^m}(X)$ is the equaliser of a diagram
    \[
        \prod_{0\le i \le m} X_{m-1} \rightrightarrows    \prod_{0\le i<j \le m} X_{m-2},    
    \]
    and is characterised by the property that
    \[
        \Hom_{d\Aff}(U, M_{\pd \Delta^m}(X)) \cong \Hom_{sd\Aff}(\pd \Delta^m \by U,X),
    \]
    naturally in $U \in d\Aff$, while $M_{\L^{m,k}}(X)$ is the equaliser of a diagram
    \[
        \prod_{\substack{0\le i \le m \\ i \ne k} } X_{m-1} \rightrightarrows    \prod_{\substack{0\le i<j \le m\\ i,j \ne k}} X_{m-2},    
    \]
    and is characterised by the property that
    \[
        \Hom_{d\Aff}(U, M_{\L^{m,k}}(X)) \cong \Hom_{sd\Aff}(\L^{m,k} \by U,X),
    \]
    naturally in $U \in d\Aff$.

    In order to formulate the key definition from \cite{stacks2}, we now need to replace these limits 
    with homotopy limits: 
    \begin{definition}
        Define the \emph{homotopy matching objects} \index{homotopy matching objects} and  \emph{homotopy partial matching objects} \index{homotopy partial matching objects}
        \begin{align*}
            M_{\pd \Delta^m}^h \co sd\Aff &\to d\Aff\\
            M_{\L^{m,k}}^h \co sd\Aff &\to d\Aff
        \end{align*}
        to be the right-derived functors of the matching and partial matching object functors  $M_{\pd \Delta^m}$ and $M_{\L^{m,k}}$, respectively.
    \end{definition}

    \begin{definition}
        We say that a  morphism $f \co X \to Y$ in $d\Aff$ is \emph{surjective}\index{surjective!in $d\Aff$} if $\pi^0f \co \pi^0X \to \pi^0Y$ is a surjection of affine schemes.  
    \end{definition}

    \begin{definition}\label{dnpreldefh2}
        Given   $Y \in sd\Aff$, a morphism $X\to Y$ in $sd\Aff$ is said to be a \cite{stacks2}-\emph{homotopy 
        derived Artin (resp. DM) $n$-hypergroupoid}  \index{derived Artin $n$-hypergroupoid!homotopy} \index{derived DM $n$-hypergroupoid!homotopy} 
        over $Y$ if:
        \begin{enumerate}
            \item for all  $m\ge 1$ and $0 \le k \le m$, the   homotopy partial matching maps
            $$
            X_m \to M_{\L^{m,k}}^h (X)\by^h_{M_{\L^{m,k}}^h (Y)}Y_m 
            $$
            are homotopy-smooth (resp. homotopy-\'etale) surjections;
            
            \item for all $m>n$ and all $0 \le k\le m$, the  homotopy partial matching maps
            $$
            X_m \to M_{\L^{m,k}}^h (X)\by^h_{M_{\L^{m,k}}^h (Y)}Y_m 
            $$
            are weak equivalences.
        \end{enumerate}
        
        
        The morphism $X \to Y$ is then said to be homotopy-smooth (resp. homotopy-\'etale, resp. surjective) if $X_0 \to Y_0$ is  homotopy-smooth (resp. homotopy-\'etale, resp. surjective). 
    \end{definition}  

    \begin{definition}\label{dnpreldefhtriv2}
        Given   $Y \in sd\Aff$, a morphism $X\to Y$ in $sd\Aff$ is said to be a \cite{stacks2}-\emph{homotopy  trivial
        derived Artin (resp. DM) $n$-hypergroupoid}   \index{derived Artin $n$-hypergroupoid!homotopy trivial} \index{derived DM $n$-hypergroupoid!homotopy trivial}
        over $Y$ if and only if:
        \begin{enumerate}
            \item for each  $m$, the homotopy  matching map
            $$
            X_m \to M_{\pd \Delta^m}^h (X)\by_{M_{\pd \Delta^m}^h (Y)}^hY_m 
            $$
            is a   homotopy-smooth (resp. homotopy-\'etale)  surjection;
            
            \item for all $m \ge n$, the homotopy matching maps
            $$
            X_m \to M_{\pd \Delta^m}^h (X)\by_{M_{\pd \Delta^m}^h (Y)}^h Y_m 
            $$
            are weak equivalences.
        \end{enumerate}
    \end{definition}
    
    
    We now have the following consistency check:  
    
    \begin{lemma}\label{cfstacks2}
        A \cite{stacks2}-homotopy  (trivial)
        derived Artin (resp. DM) $n$-hypergroupoid is precisely the same as a homotopy  (trivial)
        derived Artin (resp. DM) $n$-hypergroupoid in the sense of \S \ref{dstackdefsn}.
    \end{lemma}
    \begin{proof}
        If $f \co X \to Y$ is a \cite{stacks2}-homotopy derived Artin $n$-hypergroupoid, then as in the proof of \cite[Theorem \ref{stacks2-relstrict}]{stacks2}, the morphisms $ (f,\pd_i) \co X_m \to Y_m\by^h_{\pd_iY_{m-1}}X_{m-1}$ are all homotopy-smooth for all $m>0$ and all $i$, since $\Delta^{m-1}$ and $\Delta^m$ are contractible. In particular, those morphisms are strong, satisfying the second condition of Definitions \ref{dnpreldefh}, \ref{dnpreldefhtriv}; the first is automatic.
        
        For the converse, we start by using the following  observation ($\dagger$), as in the end of the proof of \cite[Lemma 2.2.2.8]{hag2}: that a morphism $W \to Z$ in $d\Aff$ is strong if and only if the map $\pi^0W \to W\by^h_Z\pi^0Z$ is a weak equivalence. Thus the second condition of Definition \ref{dnpreldefh} can be rephrased as saying that whenever $f \co X \to Y$ is a homotopy derived Artin $n$-hypergroupoid, the map
        \[
            g \co   \pi^0X \to X\by^h_{Y}\pi^0Y
        \]
        is homotopy-Cartesian \index{homotopy-Cartesian!morphism}
        (the derived analogue of the notion in  Example \ref{ex:hgpd}.(\ref{ex:0hgpd})).
        But this is equivalent to saying that $g$ is a \cite{stacks2}-homotopy derived  $0$-hypergroupoid\footnote{We do not specify Artin or DM, since $0$-hypergroupoids are independent of the notion of covering.}. In particular, the homotopy partial matching maps of $g$ are all quasi-isomorphisms, so the homotopy partial matching maps of $f$ are all strong, by ($\dagger$). Combined with the first condition, this completes the proof for Definition \ref{dnpreldefh}.  
        
        Likewise,  the second condition of Definition \ref{dnpreldefhtriv} says that $g$ is a levelwise equivalence whenever $f$ is a homotopy trivial derived Artin $n$-hypergroupoid, which is equivalent to saying that all the homotopy matching maps of $g$ are quasi-isomorphisms, and hence that the homotopy  matching maps of $f$ are all strong. 
    \end{proof}
    
    %
    %
    
\subsubsection{Derived hypergroupoids}   \label{strictderhgpdsn}
    
    We now introduce an equivalent, but much more restrictive,  model for homotopy derived hypergroupoids, which is especially useful when describing morphisms, but can be unwieldy to construct.
    
    By \cite[Theorem 5.2.5]{hovey},   there is  a model structure (the \emph{Reedy model structure})\index{Reedy model structure} on $sd\Aff$ in which a map $X \to Y$ is a weak equivalence if it is a quasi-isomorphism  in each level $X_m \xra{\sim} Y_m$, 
    a cofibration if it is a cofibration in each level, and a fibration if the matching maps
    \[
    X_m \to Y_m\by_{M_{\pd \Delta^m}(Y)}M_{\pd \Delta^m}(X)
    \]
    are fibrations for all $m \ge 0$.
    
    \begin{example}
        Derived affine schemes $U$   are almost  never Reedy fibrant when regarded as objects of $sd\Aff$ with constant simplicial structure, because then we would have $M_{\pd\Delta^1}(U) \cong U \by U$, with the matching map being the diagonal map $U \to U \by U$, which  can only be a fibration if it is an isomorphism. 
        
        For instance, the diagonal map $\bA^1 \to \bA^2$
        corresponds to $k[x,y] \to k[x,y]/(x-y) \cong k[x]$, which is not quasi-free,
       so $\bA^1$ is not Reedy fibrant, despite being the prototypical fibrant derived affine scheme.
       
       In fact, the homotopy matching objects $M_{\pd\Delta^n}^h(U)$ (see \S \ref{cfstacks2sn}) are given by higher derived loop spaces $U^{h S^{n-1}}$, and in particular $M_{\pd\Delta^2}^h(U) \simeq \cL U$ for the derived loop space $\cL$ of Definition \ref{def:loopspace}.
       
       The smallest Reedy fibrant replacement of the affine line $\bA^1$ is given in level $n$  by $\Spec (\Symm (\bar{C}_{\bt}(\Delta^n)))$ in the $\catdga[]$ setting, where $\bar{C}_{\bt}(\Delta^n)$ denotes
        normalised chains on the $n$-simplex, and by a similar construction with  unnormalised chains in the $s\Alg$ setting.
        \end{example}
    
    For Reedy fibrant simplicial derived affine schemes, the matching and partial matching objects are already {\it homotopy} matching and partial matching objects, leading to the following strictified analogues of Definitions \ref{dnpreldefh2}, \ref{dnpreldefhtriv2}:
    \begin{definition}\label{dnpreldef}
        Given   $Y \in sd\Aff$, define a   \emph{derived Artin} (resp. \emph{DM}) \emph{$n$-hypergroupoid}  \index{derived Artin $n$-hypergroupoid} \index{derived DM $n$-hypergroupoid}  over $Y$ to be a morphism $X\to Y$ in $sd\Aff$, satisfying the following:
        \begin{enumerate}
            \item $X \to Y$ is a Reedy fibration.
                %
            \item for each  $m\ge 1$ and $0 \le k \le m$, the   partial matching map
                $$
                X_m \to M_{\L^{m,k}} (X)\by_{M_{\L^{m,k}} (Y)}Y_m 
                $$
                is a homotopy-smooth (resp. homotopy-\'etale) surjection in $d\Aff$;  
            \item for all $m>n$ and all $0 \le k\le m$, the  partial matching maps
                $$
                X_m \to M_{\L^{m,k}} (X)\by_{M_{\L^{m,k}} (Y)}Y_m 
                $$
                are trivial fibrations in $d\Aff$.
        \end{enumerate}
        
    \end{definition}
    \begin{definition}
        A  \emph{trivial}  derived Artin (resp. DM) $n$-hypergroupoid $X \to Y$ is a morphism  in $sd\Aff$ satisfying the following: \index{derived Artin $n$-hypergroupoid!trivial} \index{derived DM $n$-hypergroupoid!trivial}
        \begin{enumerate}
            \item for each  $m$, the  matching map
                $$
                X_m \to M_{\pd \Delta^m} (X)\by_{M_{\pd \Delta^m} (Y)}^hY_m 
                $$
                is a fibration  and a   homotopy-smooth (resp. homotopy-\'etale)  surjection in $d\Aff$;
            \item for all $m \ge n$, the  matching maps
                $$
                X_m \to M_{\pd \Delta^m} (X)\by_{M_{\pd \Delta^m} (Y)} Y_m 
                $$
                are trivial fibrations.
        \end{enumerate}
    \end{definition}
    
    Since a model structure comes with fibrant replacement, the following is an immediate consequence of Reedy fibrant replacement combined with Lemma \ref{cfstacks2}:
    
    \begin{lemma}
        A  map  $X\to Y$ is  a homotopy  derived Artin $n$-hypergroupoid  if and only if  its Reedy fibrant replacement $\hat{X} \to Y$ is a   derived Artin $n$-hypergroupoid. 
        
        A  map  $X\to Y$ is a homotopy trivial derived Artin $n$-hypergroupoid  if and only if its Reedy fibrant replacement $\hat{X} \to Y$ is a trivial derived Artin $n$-hypergroupoid.
    \end{lemma}
    
    \[
        \xymatrix{ X \ar[rr]^{\sim} \ar[ddrr] && \hat{X} \ar[dd]|-{\substack{\text{Reedy} \\ \text{ fibration}}} \\  \\&& Y}
    \]
    
    Theorem \ref{bigthm2} now has the following refinement (its original form as in \cite[Theorem 
    5.11]{stacksintro}).   
    
    \begin{theorem}\label{bigthm2a}\index{n-geometric@$n$-geometric!derived Artin stack} \index{derived Artin stack}
        The homotopy category of  strongly quasi-compact $(n-1)$-geometric derived   Artin  stacks is given by taking the full subcategory of $sd\Aff$ consisting of  derived   Artin  $n$-hypergroupoids $X$, and formally inverting the  trivial relative   Artin  $n$-hypergroupoids $X \to Y$. 
    
        \medskip 
        In fact, a model for the $\infty$-category of strongly quasi-compact $(n-1)$-geometric derived  Artin  stacks is given by the relative category $(\C,\cW)$ with $\C$ the full subcategory of $sd\Aff$ on derived Artin  $n$-hypergroupoids $X$ and $\cW$ the subcategory of trivial relative derived Artin $n$-hypergroupoids  $X \to Y$.
    
        \medskip
        The same results hold true if we substitute ``Deligne--Mumford'' for ``Artin''  throughout.\index{n-geometric@$n$-geometric!derived Deligne--Mumford stack}\index{derived Deligne--Mumford stack}
    \end{theorem}  
    
    %
    In particular, this means we obtain the simplicial category of such derived stacks  by simplicial localisation of  derived $n$-hypergroupoids at the class of trivial relative derived $n$-hypergroupoids.
    
    We can now give a direct proof of one of the ingredients we saw within the representability theorems:     
    
    \begin{corollary}\label{hgscor}
        Every derived $n$-geometric Artin  stack  $F \co d\Alg_R \to s\Set$ is homotopy-homogeneous.
    \end{corollary}
    \begin{proof} 
        We need to show that for maps  $A \to B \la C$ in $d\Alg_R$, with $A \onto B$ a surjection with nilpotent kernel, we have
        \[
            F(A\by_BC) \xra{\sim} FA\by^h_{FB}FC.       
        \]
    
        For a derived Artin $(n+1)$-hypergroupoid $X$, this is an immediate consequence of the infinitesimal smoothness criterion, because $X(A) \to X(B)$ is then a Kan fibration, so $X(A)\by_{X(B)}X(C)\simeq X(A)\by_{X(B)}^hX(C)$, while we also have an isomorphism $X(A\by_BC)\cong X(A)\by_{X(B)}X(C)$ for any $X \in sd\Aff$. The result passes to hypersheafifications because \'etale morphisms lift nilpotent extensions uniquely.       
    \end{proof} 
    
\subsubsection{Explicit morphism spaces} \label{dmorphismsn}
    
    Definition \ref{hyperdef} and Theorem \ref{duskinmor} now adapt in the obvious way to give a description of the  derived mapping spaces $\oR \Map(X^{\sharp}, Y^{\sharp})$:
    
    \begin{definition}
        Define the \emph{simplicial $\Hom$ functor}\index{simplicial!Hom@$\Hom$ functor}   on    simplicial derived affine schemes by letting $\HHom_{sd\Aff}(X,Y)$ be the simplicial set given by 
        \[
            \HHom_{sd\Aff}(X,Y)_n:= \Hom_{sd\Aff}(\Delta^n \by X, Y),        
        \]
        where $(X\by \Delta^n)_i$ is given by the coproduct of $\Delta^n_i$ copies of $X_i$.
    \end{definition}
    
    There then exist derived $n$-Artin and  $n$-DM universal covers, defined similarly to Definition \ref{univcoverdef}. Every derived $n$-DM universal cover is then also a derived $n$-Artin universal cover,
    and  as in Definition \ref{hyperdef}:
    
    \begin{definition}  
        Given a derived  Artin $n$-hypergroupoid $Y$ and $X \in sd\Aff$, we define 
        \[
            \HHom_{sd\Aff}^{\sharp}(X,Y):= \LLim  \HHom_{sd\Aff}(\tilde{X}(i), Y),
        \]
\index{HomsdAffsharp@$\HHom_{sd\Aff}^{\sharp}$}
        where the colimit runs over the objects $\tilde{X}(i)$ of any $n$-Artin universal cover $\tilde{X} \to X$.
    \end{definition}
    
    The following is a case of \cite[Corollary \ref{stacks2-duskinmor}]{stacks2}:
    
    \begin{theorem}\label{duskinmord}
        If $X \in sd\Aff$ and $Y$ is a derived  Artin $n$-hypergroupoid, then the derived $\Hom$ functor on the associated   hypersheaves (a.k.a. derived $n$-stacks) $X^{\sharp}, Y^{\sharp}$  is given (up to weak equivalence)   by
        \[
            \oR \Map(X^{\sharp}, Y^{\sharp}) \simeq  \HHom_{sd\Aff}^{\sharp}(X,Y).   
        \]      
    \end{theorem}
    
    In particular, this means the functor $Y^{\sharp} \co (d\Aff)^{\op} \to s\Set$ is given by $ \HHom_{sd\Aff}^{\sharp}(-,Y)$.     
    
    \begin{warning}
        Beware that the truncation formulae of \S \ref{nstacktruncationsn} do not have derived analogues, following Warning \ref{reedyfibrant}. Also note that Theorem \ref{duskinmord} cannot be relaxed by taking $Y$ to be  a homotopy derived  Artin $n$-hypergroupoid.
    \end{warning}

\clearpage
\bibliographystyle{alpha}
\addcontentsline{toc}{section}{Bibliography}

\bibliography{references}

\newcommand{\etalchar}[1]{$^{#1}$}
 \newcommand{\noop}[1]{} \def\cprime{$'$}
\begin{thebibliography}{{Sim}96a}

\bibitem[AKSZ95]{AKSZ}
M.~Alexandrov, M.~Kontsevich, A.~Schwarz, and O.~Zaboronsky.
\newblock The geometry of the master equation and topological quantum field
  theory.
\newblock {\em International Journal of Modern Physics A}, 12(07):1405--1429,
  March 1997\noop{1995}.
\newblock arXiv:hep-th/9502010v2.

\bibitem[Ane09]{anel}
Mathieu Anel.
\newblock Grothendieck topologies from unique factorisation systems.
\newblock arXiv:0902.1130v2 [math.AG], 2009.

\bibitem[Art74]{Artin}
M.~Artin.
\newblock Versal deformations and algebraic stacks.
\newblock {\em Invent. Math.}, 27:165--189, 1974.

\bibitem[Bal17]{balchinAugmentedHAG}
Scott Balchin.
\newblock Augmented {H}omotopical {A}lgebraic {G}eometry.
\newblock arXiv:1711.02640, 2017.

\bibitem[Bar04]{barrearlycoho}
Michael Barr.
\newblock Algebraic cohomology: the early days.
\newblock In {\em Galois theory, {H}opf algebras, and semiabelian categories},
  volume~43 of {\em Fields Inst. Commun.}, pages 1--26. Amer. Math. Soc.,
  Providence, RI, 2004.

\bibitem[Bec67]{beckThesis}
Jonathan~Mock Beck.
\newblock Triples, algebras and cohomology.
\newblock {\em Repr. Theory Appl. Categ.}, (2):1--59, 2003\noop{1967}.
\newblock Reprint of Ph.D. thesis, Columbia University, 1967.

\bibitem[Beh02]{behrendDGSch2}
Kai Behrend.
\newblock Differential graded schemes {II}: The 2-category of differential
  graded schemes.
\newblock arXiv:math/0212226, 2002.

\bibitem[BF96]{BehrendFantechi}
K.~Behrend and B.~Fantechi.
\newblock The intrinsic normal cone.
\newblock {\em Invent. Math.}, 128:45--88, 1997\noop{1996}.
\newblock arXiv:alg-geom/9601010v1.

\bibitem[BF03]{BuchweitzFlenner}
Ragnar-Olaf Buchweitz and Hubert Flenner.
\newblock A semiregularity map for modules and applications to deformations.
\newblock {\em Compositio Math.}, 137(2):135--210, 2003.

\bibitem[BFMT20]{BFMTLieModels}
Urtzi Buijs, Yves F{\'e}lix, Aniceto Murillo, and Daniel Tanr{\'e}.
\newblock {\em Lie Models in Topology}.
\newblock Progress in mathematics. Springer Nature, Cham, Switzerland, 1
  edition, December 2020.

\bibitem[{Bha}12]{BhattDerivedDR}
B.~{Bhatt}.
\newblock {Completions and derived de Rham cohomology}.
\newblock arXiv: 1207.6193 [math.AG], 2012.

\bibitem[BJ15]{BorisovJoyce}
D.~{Borisov} and D.~{Joyce}.
\newblock {Virtual fundamental classes for moduli spaces of sheaves on
  Calabi-Yau four-folds}.
\newblock {\em Geom. Topol.}, to appear. 2015.
\newblock arXiv: 1504.00690 [math.AG].

\bibitem[BK10]{BarwickKanRelative}
C.~Barwick and D.~M. Kan.
\newblock Relative categories: another model for the homotopy theory of
  homotopy theories.
\newblock {\em Indag. Math. (N.S.)}, 23(1-2):42--68, 2012\noop{2010}.
\newblock arXiv:1011.1691v2 [math.AT].

\bibitem[Blo72]{blochSemiregularity}
Spencer Bloch.
\newblock Semi-regularity and de{R}ham cohomology.
\newblock {\em Invent. Math.}, 17:51--66, 1972.

\bibitem[BN93]{bokstedtneeman}
Marcel B{\"o}kstedt and Amnon Neeman.
\newblock Homotopy limits in triangulated categories.
\newblock {\em Compositio Math.}, 86(2):209--234, 1993.

\bibitem[Boo20]{boothDefThPt}
Matt Booth.
\newblock The derived deformation theory of a point.
\newblock {\em Math. Z.}, 300:3023--3082, 2022\noop{2020}.
\newblock arXiv:2009.01590.

\bibitem[Bru10]{bruceGeomObjects}
A.~J. Bruce.
\newblock {\em Geometric objects on natural bundles and supermanifolds}.
\newblock PhD thesis, Manchester, 2010.
\newblock
  \href{https://www.researchgate.net/publication/202940003}{https://www.researchgate.net/publication/202940003}.

\bibitem[BV73]{BoardmanVogt}
J.~M. Boardman and R.~M. Vogt.
\newblock {\em Homotopy Invariant Algebraic Structures on Topological Spaces}.
\newblock Springer Berlin Heidelberg, 1973.

\bibitem[BZN07]{benzvinadlerlooplanglands}
David Ben-Zvi and David Nadler.
\newblock Loop spaces and {L}anglands parameters.
\newblock arXiv:0706.0322 [math.RT], 2007.

\bibitem[CFK99]{Quot}
Ionu{\c{t}} Ciocan-Fontanine and Mikhail Kapranov.
\newblock Derived {Q}uot schemes.
\newblock {\em Ann. Sci. {\'E}cole Norm. Sup. (4)}, 34(3):403--440,
  2001\noop{1999}.
\newblock arXiv:math/9905174v2 [math.AG].

\bibitem[CFK00]{Hilb}
Ionu{\c{t}} Ciocan-Fontanine and Mikhail~M. Kapranov.
\newblock Derived {H}ilbert schemes.
\newblock {\em J. Amer. Math. Soc.}, 15(4):787--815 (electronic),
  2002\noop{2000}.
\newblock arXiv:math/0005155v1 [math.AG].

\bibitem[Cis19]{cisinskiHigherCats}
Denis-Charles Cisinski.
\newblock {\em Higher Categories and Homotopical Algebra}.
\newblock Cambridge University Press, April 2019.

\bibitem[CPT{\etalchar{+}}15]{CPTVV}
D.~Calaque, T.~Pantev, B.~To{\"e}n, M.~Vaqui{\'e}, and G.~Vezzosi.
\newblock Shifted {P}oisson structures and deformation quantization.
\newblock {\em J. Topol.}, 10(2):483--584, 2017\noop{2015}.
\newblock arXiv:1506.03699v4 [math.AG].

\bibitem[CR05]{CRdiag}
A.~M. Cegarra and Josu{\'e} Remedios.
\newblock The relationship between the diagonal and the bar constructions on a
  bisimplicial set.
\newblock {\em Topology Appl.}, 153(1):21--51, 2005.

\bibitem[CR12]{CarchediRoytenberg}
David Carchedi and Dmitry Roytenberg.
\newblock On theories of superalgebras of differentiable functions.
\newblock {\em Theory Appl. Categ.}, 28:No. 30, 1022--1098, 2013\noop{2012}.
\newblock arXiv:1211.6134v2 [math.DG].

\bibitem[Del74]{Hodge3}
Pierre Deligne.
\newblock Th{\'e}orie de {H}odge. {III}.
\newblock {\em Inst. Hautes {\'E}tudes Sci. Publ. Math.}, (44):5--77, 1974.

\bibitem[DHKS04]{DwyerHirschhornKanSmith}
W.~G. Dwyer, P.~S. Hirschhorn, D.~M. Kan, and J.~H. Smith.
\newblock {\em Homotopy limit functors on model categories and homotopical
  categories}, volume 113 of {\em Mathematical Surveys and Monographs}.
\newblock American Mathematical Society, Providence, RI, 2004.

\bibitem[DK80a]{simploc2}
W.~G. Dwyer and D.~M. Kan.
\newblock Calculating simplicial localizations.
\newblock {\em J. Pure Appl. Algebra}, 18(1):17--35, 1980.

\bibitem[DK80b]{DKfunction}
W.G. Dwyer and D.M. Kan.
\newblock Function complexes in homotopical algebra.
\newblock {\em Topology}, 19(4):427--440, 1980.

\bibitem[DK84]{DubucKock}
E.~J. Dubuc and A.~Kock.
\newblock On {$1$}-form classifiers.
\newblock {\em Comm. Algebra}, 12(11-12):1471--1531, 1984.

\bibitem[DK87]{DKEquivsHtpyDiagrams}
W.~G. Dwyer and D.~M. Kan.
\newblock {VIII}. equivalences between homotopy theories of diagrams.
\newblock In {\em Algebraic Topology and Algebraic K-Theory ({AM}-113)}, pages
  180--205. Princeton University Press, 1987.

\bibitem[DM69]{DeligneMumford}
P.~Deligne and D.~Mumford.
\newblock The irreducibility of the space of curves of given genus.
\newblock {\em Publications math{\'{e}}matiques de l'{IH}{\'{E}}S},
  36(1):75--109, January 1969.

\bibitem[DM99]{DeligneMorganSupersymm}
Pierre Deligne and John~W. Morgan.
\newblock Notes on supersymmetry (following {J}oseph {B}ernstein).
\newblock In {\em Quantum fields and strings: a course for mathematicians,
  {V}ol. 1, 2 ({P}rinceton, {NJ}, 1996/1997)}, pages 41--97. Amer. Math. Soc.,
  Providence, RI, 1999.

\bibitem[Dri88]{drinfeldtoschechtman}
Vladimir Drinfeld.
\newblock A letter from {K}harkov to {M}oscow.
\newblock {\em EMS Surveys in Mathematical Sciences}, 1(2):241--248,
  2014\noop{1988}.

\bibitem[Dub81]{dubuc}
Eduardo~J. Dubuc.
\newblock {$C^{\infty }$}-schemes.
\newblock {\em Amer. J. Math.}, 103(4):683--690, 1981.

\bibitem[Dus75]{duskin}
J.~Duskin.
\newblock Simplicial methods and the interpretation of ``triple'' cohomology.
\newblock {\em Mem. Amer. Math. Soc.}, 3(issue 2, 163):v+135, 1975.

\bibitem[Emm95]{emmanouil}
Ioannis Emmanouil.
\newblock The cyclic homology of affine algebras.
\newblock {\em Invent. Math.}, 121(1):1--19, 1995.

\bibitem[FT85]{FeiginTsygan}
B.L. Feigin and B.L. Tsygan.
\newblock Additive {K}-theory and crystalline cohomology.
\newblock {\em Functional Analysis and Its Applications}, 19(2):124--132, 1985.

\bibitem[Gai11]{GaitsgoryIndCoh}
Dennis Gaitsgory.
\newblock Ind-coherent sheaves.
\newblock {\em Mosc. Math. J.}, 13(3):399--528, 553, 2013\noop{2011}.
\newblock arXiv: arXiv:1105.4857 [math.AG].

\bibitem[Get04]{getzler}
Ezra Getzler.
\newblock Lie theory for nilpotent {$L_\infty$}-algebras.
\newblock {\em Ann. of Math. (2)}, 170(1):271--301, 2009\noop{2004}.
\newblock arXiv:math/0404003v4 [math.AT].

\bibitem[GJ99]{sht}
Paul~G. Goerss and John~F. Jardine.
\newblock {\em Simplicial homotopy theory}, volume 174 of {\em Progress in
  Mathematics}.
\newblock Birkh{\"a}user Verlag, Basel, 1999.

\bibitem[GK94]{GK}
Victor Ginzburg and Mikhail Kapranov.
\newblock Koszul duality for operads.
\newblock {\em Duke Math. J.}, 76(1):203--272, 1994.

\bibitem[Gle82]{glenn}
Paul~G. Glenn.
\newblock Realization of cohomology classes in arbitrary exact categories.
\newblock {\em J. Pure Appl. Algebra}, 25(1):33--105, 1982.

\bibitem[GLST19]{GuanLazarevShengTangII}
Ai~Guan, Andrey Lazarev, Yunhe Sheng, and Rong Tang.
\newblock Review of deformation theory {II}: a homotopical approach.
\newblock 2019.
\newblock arXiv:1912.04028 [math.AT].

\bibitem[Gro60]{descent}
Alexander Grothendieck.
\newblock Technique de descente et th{\'e}or{\`e}mes d'existence en
  g{\'e}om{\'e}trie alg{\'e}brique. {II}. {L}e th{\'e}or{\`e}me d'existence en
  th{\'e}orie formelle des modules.
\newblock In {\em S{\'e}minaire Bourbaki, Vol.\ 5}, pages Exp.\ No.\ 195,
  369--390. Soc. Math. France, Paris, 1995\noop{1960}.

\bibitem[Gro83]{pursuingstacks}
A.~Grothendieck.
\newblock Pursuing stacks.
\newblock unpublished manuscript, 1983.

\bibitem[Har72]{HartshorneAlgDeRham}
Robin Hartshorne.
\newblock Algebraic de {R}ham cohomology.
\newblock {\em Manuscripta Math.}, 7:125--140, 1972.

\bibitem[Hel81]{heller}
Alex Heller.
\newblock On the representability of homotopy functors.
\newblock {\em J. London Math. Soc. (2)}, 23(3):551--562, 1981.

\bibitem[Hin98]{hinstack}
Vladimir Hinich.
\newblock D{G} coalgebras as formal stacks.
\newblock {\em J. Pure Appl. Algebra}, 162(2-3):209--250, 2001\noop{1998}.
\newblock https://arxiv.org/abs/math/9812034.

\bibitem[Hin99]{hinichDefsHtpyAlg}
Vladimir Hinich.
\newblock Deformations of homotopy algebras.
\newblock {\em Communications in Algebra}, 32(2):473--494, 2004\noop{1999}.
\newblock arXiv:math/9904145.

\bibitem[Hin17]{hinichLectInftyCats}
Vladimir Hinich.
\newblock Lectures on infinity categories.
\newblock arXiv:1709.06271, 2017.

\bibitem[Hir03]{Hirschhorn}
Philip~S. Hirschhorn.
\newblock {\em Model categories and their localizations}, volume~99 of {\em
  Mathematical Surveys and Monographs}.
\newblock American Mathematical Society, Providence, RI, 2003.

\bibitem[Hov99]{hovey}
Mark Hovey.
\newblock {\em Model categories}, volume~63 of {\em Mathematical Surveys and
  Monographs}.
\newblock American Mathematical Society, Providence, RI, 1999.

\bibitem[HS87]{HinSch}
V.~A. Hinich and V.~V. Schechtman.
\newblock On homotopy limit of homotopy algebras.
\newblock In {\em {$K$}-theory, arithmetic and geometry ({M}oscow,
  1984--1986)}, volume 1289 of {\em Lecture Notes in Math.}, pages 240--264.
  Springer, Berlin, 1987.

\bibitem[HS98]{hirschowitzsimpson}
Andr\'e Hirschowitz and Carlos Simpson.
\newblock Descente pour les $n$-champs.
\newblock arXiv:math.AG/9807049, 1998.

\bibitem[H{\"u}t08]{huettemann}
Thomas H{\"u}ttemann.
\newblock On the derived category of a regular toric scheme.
\newblock {\em Geom. Dedicata}, 148:175--203, 2010\noop{2008}.
\newblock arXiv:0805.4089v2 [math.KT].

\bibitem[Ill71]{Ill1}
Luc Illusie.
\newblock {\em Complexe cotangent et d{\'e}formations. {I}}.
\newblock Springer-Verlag, Berlin, 1971.
\newblock Lecture Notes in Mathematics, Vol. 239.

\bibitem[Ill72]{Ill2}
Luc Illusie.
\newblock {\em Complexe cotangent et d{\'e}formations. {I}{I}}.
\newblock Springer-Verlag, Berlin, 1972.
\newblock Lecture Notes in Mathematics, Vol. 283.

\bibitem[Joy02]{joyalQCatKan}
A.~Joyal.
\newblock Quasi-categories and {K}an complexes.
\newblock {\em Journal of Pure and Applied Algebra}, 175(1):207--222, 2002.
\newblock Special Volume celebrating the 70th birthday of Professor Max Kelly.

\bibitem[Joy07]{joyalQCatSCat}
Andr{\'e} Joyal.
\newblock Quasi-categories vs simplicial categories.
\newblock preprint, 2007.

\bibitem[JT07]{joyaltierney}
Andr{\'e} Joyal and Myles Tierney.
\newblock Quasi-categories vs {S}egal spaces.
\newblock In {\em Categories in algebra, geometry and mathematical physics},
  volume 431 of {\em Contemp. Math.}, pages 277--326. Amer. Math. Soc.,
  Providence, RI, 2007.

\bibitem[Kan58]{kanAdjointFunctors}
Daniel~M Kan.
\newblock Adjoint functors.
\newblock {\em Transactions of the American Mathematical Society},
  87(2):294--329, 1958.

\bibitem[Kap15]{kapranovSupergeometry}
Mikhail Kapranov.
\newblock Supergeometry in mathematics and physics.
\newblock In Mathieu Anel and Gabriel Catren, editors, {\em New Spaces in
  Physics: Formal and Conceptual Reflections}, volume~2, page 114–152.
  Cambridge University Press, 2021\noop{2015}.
\newblock arXiv:1512.07042 [math.AG].

\bibitem[Kel06]{kellerModelDGCat}
Bernhard Keller.
\newblock On differential graded categories.
\newblock In {\em International {C}ongress of {M}athematicians. {V}ol. {II}},
  pages 151--190. Eur. Math. Soc., Z\"urich, 2006.

\bibitem[Kon94a]{Kon}
Maxim Kontsevich.
\newblock Topics in algebra --- deformation theory.
\newblock Lecture Notes, available at
  \href{http://www.math.brown.edu/~abrmovic/kontsdef.ps}{http://www.math.brown.edu/$\sim$abrmovic/kontsdef.ps},
  1994.

\bibitem[Kon94b]{Kon2}
Maxim Kontsevich.
\newblock Enumeration of rational curves via torus actions.
\newblock In {\em The moduli space of curves (Texel Island, 1994)}, volume 129
  of {\em Progr. Math.}, pages 335--368. Birkh{\"a}user Boston, Boston, MA,
  1995\noop{1994}.
\newblock arXiv:hep-th/9405035v2.

\bibitem[Kon97]{kontsevichPoisson}
Maxim Kontsevich.
\newblock Deformation quantization of {P}oisson manifolds.
\newblock {\em Lett. Math. Phys.}, 66(3):157--216, 2003\noop{1997}.
\newblock arXiv:q-alg/9709040v1.

\bibitem[Kon99]{kontsevichOperads}
Maxim Kontsevich.
\newblock Operads and motives in deformation quantization.
\newblock {\em Lett. Math. Phys.}, 48(1):35--72, 1999.
\newblock arXiv:math/9904055v1 [math.QA].

\bibitem[KS00]{KS}
Maxim Kontsevich and Yan Soibelman.
\newblock Deformations of algebras over operads and the {D}eligne conjecture.
\newblock In {\em Conf{\'e}rence Mosh{\'e} Flato 1999, Vol. I (Dijon)},
  volume~21 of {\em Math. Phys. Stud.}, pages 255--307. Kluwer Acad. Publ.,
  Dordrecht, 2000.
\newblock arXiv:math/0001151v2 [math.QA].

\bibitem[KV08]{KhudaverdianVoronov}
H.~M. Khudaverdian and Th.~Th. Voronov.
\newblock Higher {P}oisson brackets and differential forms.
\newblock In {\em Geometric methods in physics}, volume 1079 of {\em AIP Conf.
  Proc.}, pages 203--215. Amer. Inst. Phys., Melville, NY, 2008.
\newblock arXiv:0808.3406v2 [math-ph].

\bibitem[LGLS18]{LaurentGengouxLavauStroblLieInftyAlgdFoliation}
Camille Laurent-Gengoux, Sylvain Lavau, and Thomas Strobl.
\newblock The universal lie
  {\textdollar}{\textbackslash}infty{\textdollar}-algebroid of a singular
  foliation.
\newblock {\em Doc. Math.}, 25:1571--1652, 2020\noop{2018}.
\newblock arXiv:1806.00475[math.DG].

\bibitem[LMB00]{Champs}
G{\'e}rard Laumon and Laurent Moret-Bailly.
\newblock {\em Champs alg{\'e}briques}, volume~39 of {\em Ergebnisse der
  Mathematik und ihrer Grenzgebiete. 3. Folge. A Series of Modern Surveys in
  Mathematics [Results in Mathematics and Related Areas. 3rd Series. A Series
  of Modern Surveys in Mathematics]}.
\newblock Springer-Verlag, Berlin, 2000.

\bibitem[LS67]{LS}
S.~Lichtenbaum and M.~Schlessinger.
\newblock The cotangent complex of a morphism.
\newblock {\em Trans. Amer. Math. Soc.}, 128:41--70, 1967.

\bibitem[Lur03]{lurieInftyTopoi}
Jacob Lurie.
\newblock On infinity topoi.
\newblock arXiv:math/0306109v2 [math.CT], 2003.

\bibitem[Lur04a]{lurie}
J.~Lurie.
\newblock {\em Derived Algebraic Geometry}.
\newblock PhD thesis, M.I.T., 2004.
\newblock
  \href{http://hdl.handle.net/1721.1/30144}{http://hdl.handle.net/1721.1/30144}
  or
  \href{https://www.math.ias.edu/~lurie/papers/DAG.pdf}{https://www.math.ias.edu/$\sim$lurie/papers/DAG.pdf}.

\bibitem[Lur04b]{lurietannaka}
Jacob Lurie.
\newblock Tannaka duality for geometric stacks.
\newblock arXiv:math/0412266v2 [math.AG], 2004.

\bibitem[Lur07]{lurieellipticsurvey}
Jacob Lurie.
\newblock A survey of elliptic cohomology.
\newblock Available at:
  \href{http://www-math.mit.edu/~lurie/papers/survey.pdf}{http://www-math.mit.edu/$\sim$lurie/papers/survey.pdf},
  2007.

\bibitem[Lur09a]{lurieDAG5}
Jacob Lurie.
\newblock Derived algebraic geometry {V}: Structured spaces.
\newblock arXiv:0905.0459v1 [math.CT], 2009.

\bibitem[Lur09b]{luriehighertopoi}
Jacob Lurie.
\newblock {\em Higher topos theory}, volume 170 of {\em Annals of Mathematics
  Studies}.
\newblock Princeton University Press, Princeton, NJ, 2009.

\bibitem[Lur10]{lurieDGLA}
Jacob Lurie.
\newblock Moduli problems for ring spectra.
\newblock In {\em Proceedings of the {I}nternational {C}ongress of
  {M}athematicians. {V}olume {II}}, pages 1099--1125, New Delhi, 2010.
  Hindustan Book Agency.

\bibitem[Lur11a]{lurieDAG9}
Jacob Lurie.
\newblock Derived algebraic geometry {IX}: Closed immersions.
\newblock available at
  \href{http://www.math.harvard.edu/~lurie/papers/DAG-IX.pdf}{www.math.harvard.edu/$\sim$lurie/papers/DAG-IX.pdf},
  2011.

\bibitem[Lur11b]{lurieDAG10}
Jacob Lurie.
\newblock Derived algebraic geometry {X}: Formal moduli problems.
\newblock available at
  \href{http://www.math.harvard.edu/~lurie/papers/DAG-X.pdf}{www.math.harvard.edu/$\sim$lurie/papers/DAG-X.pdf},
  2011.

\bibitem[Lur11c]{lurieDAG13}
Jacob Lurie.
\newblock Derived algebraic geometry {XIII}: Rational and $p$-adic homotopy
  theory.
\newblock available at
  \href{http://www.math.harvard.edu/~lurie/papers/DAG-XIII.pdf}{www.math.harvard.edu/$\sim$lurie/papers/DAG-XIII.pdf},
  2011.

\bibitem[Lur18]{lurieSAG}
Jacob Lurie.
\newblock Spectral algebraic geometry.
\newblock available at
  \href{https://www.math.ias.edu/~lurie/papers/SAG-rootfile.pdf}{www.math.ias.edu/$\sim$lurie/papers/SAG-rootfile.pdf},
  2018.

\bibitem[Mag10]{maggiolo}
Stefano Maggiolo.
\newblock Derived deformation functors, 2010.
\newblock Lecture notes available at
  \href{http://blog.poormansmath.net/docs/WDTII_Pridham.pdf}{blog.poormansmath.net/docs/WDTII\_Pridham.pdf}.

\bibitem[Mah88]{waldhausenBN}
Mark~E Mahowald.
\newblock {\em Algebraic topology: proceedings of the international conference
  held March 21-24, 1988}.
\newblock Number~96. American Mathematical Soc., 1989\noop{1988}.

\bibitem[Man99]{Man2}
Marco Manetti.
\newblock Extended deformation functors.
\newblock {\em Int. Math. Res. Not.}, (14):719--756, 2002\noop{1999}.
\newblock arXiv:math/9910071v2 [math.AG].

\bibitem[Mat89]{Mat}
Hideyuki Matsumura.
\newblock {\em Commutative ring theory}.
\newblock Cambridge University Press, Cambridge, second edition, 1989.
\newblock Translated from the Japanese by M. Reid.

\bibitem[Noo05]{Noohi1}
Behrang Noohi.
\newblock Foundations of topological stacks {I}.
\newblock arXiv:math/0503247v1 [math.AG], 2005.

\bibitem[Nui18]{nuitenThesis}
Joost Nuiten.
\newblock {\em Lie algebroids in derived differential topology}.
\newblock PhD thesis, Utrecht, 2018.

\bibitem[Ols07]{olssartin}
Martin Olsson.
\newblock Sheaves on {A}rtin stacks.
\newblock {\em J. Reine Angew. Math.}, 603:55--112, 2007.

\bibitem[Pos09]{positselskiDerivedCategories}
Leonid Positselski.
\newblock Two kinds of derived categories, {K}oszul duality, and
  comodule-contramodule correspondence.
\newblock {\em Mem. Amer. Math. Soc.}, 212(996):vi+133, 2011\noop{2009}.
\newblock arXiv:0905.2621v12 [math.CT].

\bibitem[Pri03]{paper2}
J.~P. Pridham.
\newblock Deformations of schemes and other bialgebraic structures.
\newblock {\em Trans. Amer. Math. Soc.}, 360(3):1601--1629, 2008\noop{2003}.
\newblock mostly contained in arXiv:math/0311168v6 [math.AG].

\bibitem[Pri06]{htpy}
J.~P. Pridham.
\newblock Pro-algebraic homotopy types.
\newblock {\em Proc. London Math. Soc.}, 97(2):273--338, 2008\noop{2006}.
\newblock arXiv math.AT/0606107 v8.

\bibitem[Pri07a]{ddt1}
J.~P. Pridham.
\newblock Unifying derived deformation theories.
\newblock {\em Adv. Math.}, 224(3):772--826, 2010\noop{2007}.
\newblock corrigendum 228 (2011), no. 4, 2554--2556, arXiv:0705.0344v6
  [math.AG].

\bibitem[Pri07b]{ddt2}
J.~P. Pridham.
\newblock Derived deformations of schemes.
\newblock {\em Comm. Anal. Geom.}, 20(3):529--563, 2012\noop{2007}.
\newblock arXiv:0908.1963v1 [math.AG].

\bibitem[Pri08]{higher}
J.~P. Pridham.
\newblock Derived deformations of {A}rtin stacks.
\newblock {\em Comm. Anal. Geom.}, 23(3):419--477, 2015\noop{2008}.
\newblock arXiv:0805.3130v3 [math.AG].

\bibitem[Pri09]{stacks2}
J.~P. Pridham.
\newblock Presenting higher stacks as simplicial schemes.
\newblock {\em Adv. Math.}, 238:184--245, 2013\noop{2009}.
\newblock arXiv:0905.4044v4 [math.AG].

\bibitem[Pri10a]{dmsch}
J.~P. Pridham.
\newblock Derived moduli of schemes and sheaves.
\newblock {\em J. K-Theory}, 10(1):41--85, 2012\noop{2010}.
\newblock arXiv:1011.2189v5 [math.AG].

\bibitem[Pri10b]{drep}
J.~P. Pridham.
\newblock Representability of derived stacks.
\newblock {\em J. K-Theory}, 10(2):413--453, 2012\noop{2010}.
\newblock arXiv:1011.2742v4 [math.AG].

\bibitem[Pri11a]{stacksintro}
J.~P. Pridham.
\newblock Notes characterising higher and derived stacks concretely.
\newblock arXiv:1105.4853v3 [math.AG], 2011.

\bibitem[Pri11b]{dmc}
J.~P. Pridham.
\newblock Constructing derived moduli stacks.
\newblock {\em Geom. Topol.}, 17(3):1417--1495, 2013\noop{2011}.
\newblock arXiv:1101.3300v2 [math.AG].

\bibitem[Pri12]{semireg2}
J.~P. Pridham.
\newblock Semiregularity as a consequence of {G}oodwillie's theorem.
\newblock arXiv:1208.3111v2 [math.AG], 2012.

\bibitem[Pri15]{poisson}
J.~P. Pridham.
\newblock Shifted {P}oisson and symplectic structures on derived {$N$}-stacks.
\newblock {\em J. Topol.}, 10(1):178--210, 2017\noop{2015}.
\newblock arXiv:1504.01940v5 [math.AG].

\bibitem[Pri17]{obsrec2}
J.~P. Pridham.
\newblock Non-abelian reciprocity laws and higher {B}rauer--{M}anin
  obstructions.
\newblock {\em Algebr. Geom. Topol.}, 20(2):699--756, 2020\noop{2017}.
\newblock arXiv:1704.03021v3 [math.AT].

\bibitem[Pri18a]{DQDG}
J.~P. Pridham.
\newblock An outline of shifted {P}oisson structures and deformation
  quantisation in derived differential geometry.
\newblock arXiv: 1804.07622v3 [math.DG], 2018.

\bibitem[Pri18b]{DStein}
J.P. Pridham.
\newblock A differential graded model for derived analytic geometry.
\newblock {\em Advances in Mathematics}, 360:106922, 2020\noop{2018}.
\newblock arXiv: 1805.08538v1 [math.AG].

\bibitem[Pri19]{smallet}
J.~P. Pridham.
\newblock A note on \'etale atlases for {A}rtin stacks, {P}oisson structures
  and quantisation.
\newblock arXiv:1905.09255, 2019.

\bibitem[Pri20]{NCstacks}
J.~P. Pridham.
\newblock Non-commutative derived moduli prestacks.
\newblock {\em Adv. Math.}, to appear\noop{2020}.
\newblock arXiv:2008.11684 [math.AG].

\bibitem[PTVV11]{PTVV}
T.~Pantev, B.~To{\"e}n, M.~Vaqui{\'e}, and G.~Vezzosi.
\newblock Shifted symplectic structures.
\newblock {\em Publ. Math. Inst. Hautes \'Etudes Sci.}, 117:271--328,
  2013\noop{2011}.
\newblock arXiv: 1111.3209v4 [math.AG].

\bibitem[Qui67]{QHA}
Daniel~G. Quillen.
\newblock {\em Homotopical algebra}.
\newblock Lecture Notes in Mathematics, No. 43. Springer-Verlag, Berlin, 1967.

\bibitem[Qui69]{QRat}
Daniel Quillen.
\newblock Rational homotopy theory.
\newblock {\em Ann. of Math. (2)}, 90:205--295, 1969.

\bibitem[Qui70]{Q}
Daniel Quillen.
\newblock On the (co-) homology of commutative rings.
\newblock In {\em Applications of Categorical Algebra (Proc. Sympos. Pure
  Math., Vol. XVII, New York, 1968)}, pages 65--87. Amer. Math. Soc.,
  Providence, R.I., 1970.

\bibitem[Ren06]{renaudin}
Olivier Renaudin.
\newblock Plongement de certaines théories homotopiques de quillen dans les
  dérivateurs.
\newblock {\em Journal of Pure and Applied Algebra}, 213(10):1916--1935,
  2009\noop{2006}.
\newblock arXiv:math/0603339v2 [math.AT].

\bibitem[Rie19]{riehlHtpicalCats}
Emily Riehl.
\newblock Homotopical categories: from model categories to
  $(\infty,1)$-categories.
\newblock In {\em Stable categories and structured ring spectra}, MSRI book
  series. CUP, to appear\noop{2019}.
\newblock arxiv:1904.00886v3[math.CT].

\bibitem[{Saf}17]{safronovPoissonLie}
P.~{Safronov}.
\newblock {Poisson-Lie structures as shifted Poisson structures}.
\newblock arXiv: 1706.02623v2 [math.AG], 2017.

\bibitem[Sch68]{Sch}
Michael Schlessinger.
\newblock Functors of {A}rtin rings.
\newblock {\em Trans. Amer. Math. Soc.}, 130:208--222, 1968.

\bibitem[Ser56]{GAGA}
Jean-Pierre Serre.
\newblock G{\'e}om{\'e}trie alg{\'e}brique et g{\'e}om{\'e}trie analytique.
\newblock {\em Ann. Inst. Fourier, Grenoble}, 6:1--42, 1955--1956.

\bibitem[Ser65]{serreAlgebreLocale}
Jean-Pierre Serre.
\newblock {\em Alg\`ebre locale. {M}ultiplicit\'{e}s}, volume~11 of {\em
  Lecture Notes in Mathematics}.
\newblock Springer-Verlag, Berlin-New York, 1965.
\newblock Cours au Coll\`ege de France, 1957--1958, r\'{e}dig\'{e} par Pierre
  Gabriel, Seconde \'{e}dition, 1965.

\bibitem[SGA4b]{SGA4.2}
{\em Th{\'e}orie des topos et cohomologie {\'e}tale des sch{\'e}mas. {T}ome 2}.
\newblock Lecture Notes in Mathematics, Vol. 270. Springer-Verlag, Berlin,
  1972\noop{1964}\noop{4b}.
\newblock S{\'e}minaire de G{\'e}om{\'e}trie Alg{\'e}brique du Bois-Marie
  1963--1964 (SGA 4), Dirig{\'e} par M. Artin, A. Grothendieck et J. L.
  Verdier. Avec la collaboration de N. Bourbaki, P. Deligne et B. Saint-Donat.

\bibitem[SGA6]{sga6}
{\em Th{\'e}orie des intersections et th{\'e}or{\`e}me de {R}iemann--{R}och}.
\newblock Lecture Notes in Mathematics, Vol. 225. Springer-Verlag, Berlin,
  1971\noop{1967}\noop{A6}.
\newblock S{\'e}minaire de G{\'e}om{\'e}trie Alg{\'e}brique du Bois-Marie
  1966--1967 (SGA 6), Dirig{\'e} par P. Berthelot, A. Grothendieck et L.
  Illusie. Avec la collaboration de D. Ferrand, J. P. Jouanolou, O. Jussila, S.
  Kleiman, M. Raynaud et J. P. Serre.

\bibitem[{Sim}96a]{simpsonalggeomnstacks}
C.~{Simpson}.
\newblock {Algebraic (geometric) $n$-stacks}.
\newblock arXiv:alg-geom/9609014, 1996.

\bibitem[Sim96b]{Simfil}
Carlos Simpson.
\newblock The {H}odge filtration on nonabelian cohomology.
\newblock In {\em Algebraic geometry---Santa Cruz 1995}, volume~62 of {\em
  Proc. Sympos. Pure Math.}, pages 217--281. Amer. Math. Soc., Providence, RI,
  1997\noop{1996}.
\newblock arXiv:alg-geom/9604005v1.

\bibitem[Spa88]{spaltenstein}
N.~Spaltenstein.
\newblock Resolutions of unbounded complexes.
\newblock {\em Compositio Mathematica}, 65(2):121--154, 1988.

\bibitem[Sul77]{Sullivan}
Dennis Sullivan.
\newblock Infinitesimal computations in topology.
\newblock {\em Inst. Hautes {\'E}tudes Sci. Publ. Math.}, (47):269--331 (1978),
  1977.

\bibitem[Tay02]{TaylorCV}
Joseph~L Taylor.
\newblock {\em Several complex variables with connections to algebraic geometry
  and lie groups}.
\newblock Graduate studies in mathematics. American Mathematical Society,
  Providence, RI, May 2002.

\bibitem[To{\"e}00]{chaff}
Bertrand To{\"e}n.
\newblock Champs affines.
\newblock {\em Selecta Math. (N.S.)}, 12(1):39--135, 2006\noop{2000}.
\newblock arXiv math.AG/0012219.

\bibitem[To{\"e}06]{toenseattle}
Bertrand To{\"e}n.
\newblock Higher and derived stacks: a global overview.
\newblock In {\em Algebraic geometry---{S}eattle 2005. {P}art 1}, volume~80 of
  {\em Proc. Sympos. Pure Math.}, pages 435--487. Amer. Math. Soc., Providence,
  RI, 2009\noop{2006}.
\newblock arXiv math/0604504v3.

\bibitem[To{\"e}08]{toencrm}
Bertrand To{\"e}n.
\newblock Simplicial presheaves and derived algebraic geometry.
\newblock In {\em Simplicial Methods for Operads and Algebraic Geometry}, pages
  119--186. Springer Basel, 2010\noop{2008}.
\newblock hal-00772850v2.

\bibitem[To{\"e}14]{toenICM}
Bertrand To{\"e}n.
\newblock Derived algebraic geometry and deformation quantization.
\newblock In {\em Proceedings of the {I}nternational {C}ongress of
  {M}athematicians (Seoul 2014), Vol. II}, pages 769--752, 2014.
\newblock arXiv:1403.6995v4 [math.AG].

\bibitem[TV02]{hag1}
Bertrand To{\"e}n and Gabriele Vezzosi.
\newblock Homotopical algebraic geometry. {I}. {T}opos theory.
\newblock {\em Adv. Math.}, 193(2):257--372, 2005\noop{2002}.
\newblock arXiv:math/0207028 v4.

\bibitem[TV03]{TVbn}
Bertrand To{\"e}n and Gabriele Vezzosi.
\newblock Brave new algebraic geometry and global derived moduli spaces of ring
  spectra.
\newblock In Haynes~R. Miller and Douglas~C.Editors Ravenel, editors, {\em
  Elliptic Cohomology: Geometry, Applications, and Higher Chromatic Analogues},
  London Mathematical Society Lecture Note Series, pages 325--359. Cambridge
  University Press, 2007\noop{2003}.
\newblock arXiv:math/0309145v2 [math.AT].

\bibitem[TV04]{hag2}
Bertrand To{\"e}n and Gabriele Vezzosi.
\newblock Homotopical algebraic geometry. {II}. {G}eometric stacks and
  applications.
\newblock {\em Mem. Amer. Math. Soc.}, 193(902):x+224, 2008\noop{2004}.
\newblock arXiv math.AG/0404373 v7.

\bibitem[TV05]{toenvaquie}
Bertrand To{\"e}n and Michel Vaqui{\'e}.
\newblock Moduli of objects in dg-categories.
\newblock {\em Ann. Sci. {\'E}cole Norm. Sup. (4)}, 40(3):387--444,
  2007\noop{2005}.
\newblock arXiv:math/0503269v5.

\bibitem[TV09]{TVchern}
Bertrand To{\"e}n and Gabriele Vezzosi.
\newblock Chern character, loop spaces and derived algebraic geometry.
\newblock In {\em Algebraic topology}, volume~4 of {\em Abel Symp.}, pages
  331--354. Springer, Berlin, 2009.
\newblock arXiv:0903.3292v2 [math.AG].

\bibitem[Vor07]{voronovMackenzie}
Theodore~Th. Voronov.
\newblock {$Q$}-manifolds and {M}ackenzie theory.
\newblock {\em Comm. Math. Phys.}, 315(2):279--310, 2012\noop{2007}.
\newblock arXiv:0709.4232v1 [math.DG].

\bibitem[Wei94]{W}
Charles~A. Weibel.
\newblock {\em An introduction to homological algebra}.
\newblock Cambridge University Press, Cambridge, 1994.

\bibitem[Zhu08]{zhu}
Chenchang Zhu.
\newblock {$n$}-groupoids and stacky groupoids.
\newblock {\em Int. Math. Res. Not. IMRN}, (21):4087--4141, 2009\noop{2008}.
\newblock arXiv:0801.2057.

\end{thebibliography}

\addcontentsline{toc}{section}{Index}
\printindex

\end{document}